\documentclass[11pt,twoside,final]{article}

\usepackage{hyperref}

\hypersetup{colorlinks, linkcolor={blue}, citecolor={purple},
  urlcolor={red}}

\usepackage{fullpage}

\usepackage[mathscr]{eucal}
\usepackage{amsbsy}

\DeclareMathAlphabet{\mathpzc}{OT1}{pzc}{m}{it}

\usepackage{epsf}
\usepackage{fancyheadings}
\usepackage{graphics}
\usepackage{graphicx}
\usepackage{enumitem}
\setitemize{leftmargin=0.15in,itemsep=0.02in,topsep=0.1in}
\setenumerate{leftmargin=0.2in,itemsep=0.02in,topsep=0.1in}

\usepackage{float}

\usepackage{color}

\usepackage{amsthm} \usepackage{amsfonts} \usepackage{amsmath}
\usepackage{amssymb}
\numberwithin{equation}{section}

\usepackage[linesnumbered,ruled,vlined]{algorithm2e}

\SetCommentSty{mycommfont}
\usepackage{algorithmic}

\usepackage{bbm}

\usepackage{caption} \usepackage{subcaption}

\usepackage{bm}
\usepackage{mathtools}
\usepackage{dsfont}

\usepackage{macros}

\usepackage[noabbrev,nameinlink]{cleveref}

\newcommand{\NORMAL}{\ensuremath{\mathcal{N}}}

\DeclareMathOperator{\ind}{\mathds{1}}  
\newcommand{\cnoise}{\ensuremath{c_{\mathsf{n}}}}
\newcommand{\Cnoise}{\ensuremath{c_{\mathsf{noise}}}}
\newcommand{\PlainEnoise}{\ensuremath{\rho}}
\newcommand{\Enoise}{\ensuremath{\PlainEnoise_{N,T}}}
\newcommand{\EN}{\ensuremath{\PlainEnoise_N}}
\newcommand{\ET}{\ensuremath{\PlainEnoise_T}}
\newcommand{\PlainEnoisebar}{\ensuremath{\bar{\rho}}}
\newcommand{\Ebar}{\ensuremath{\PlainEnoisebar_{N,T}}}
\newcommand{\ENbar}{\ensuremath{\PlainEnoisebar_N}}
\newcommand{\ETbar}{\ensuremath{\PlainEnoisebar_T}}
\newcommand{\ball}{\ensuremath{\mathcal{B}}}
\newcommand{\noise}{\ensuremath{\omega}}
\newcommand{\cb}{\ensuremath{C_{\mathsf{b}}}}
\newcommand{\cu}{\ensuremath{C_{\mathsf{upper}}}}
\newcommand{\csig}{\ensuremath{C_{\sigma}}}
\newcommand{\cinc}{\ensuremath{c_{\mathsf{inc}}}}
\newcommand{\cgamma}{\ensuremath{C_{\gamma}}}

\DeclareMathOperator{\myvdots}{\vphantom{\int^{0^{0^0}}}\smash[t]{\vdots}}
\DeclareMathOperator{\myddots}{\vphantom{\int^{0^{0^0}}}\smash[t]{\ddots}}

\usepackage{tikz}
\usetikzlibrary{matrix,decorations.pathreplacing}
\usepackage{adjustbox}
\definecolor{yly}{RGB}{125,0,0}

\newcommand{\Varplain}{\ensuremath{\gamma}}
\newcommand{\Varhat}{\ensuremath{\widehat{\Varplain}}}
\newcommand{\Varstar}{\ensuremath{\Varplain^*}}

\newcommand{\delbound}{\ensuremath{b}}

\newcommand{\clow}{\ensuremath{c_\ell}}
\newcommand{\cupper}{\ensuremath{c_u}}
\newcommand{\Ustar}{\ensuremath{\bm{U}^\star}}
\newcommand{\Vstar}{\ensuremath{\bm{V}^\star}}
\newcommand{\SigStar}{\ensuremath{\bm{\Sigma}^\star}}

\newcommand{\Rescale}{\ensuremath{R}}

\newcommand{\hackgam}{\lambda}

\usepackage[firstinits=true,  
  maxnames = 99,  
  backref=true,
  backend=bibtex,
  style=alphabetic,
  sortlocale=de_DE,
  natbib=true,
  doi=false,
  isbn=false,
  url=false]{biblatex}

\newcommand{\figdir}{figs}
\begin{document}


\begin{center}


  {\bf{\LARGE{Entrywise Inference for Missing Panel Data: \\ A Simple
        and Instance-Optimal Approach}}}
  
\vspace*{.2in}

{\large{
\begin{tabular}{ccc}
Yuling Yan$^{\star,\dagger}$ && Martin J. Wainwright$^{\star,\dagger,
  \ddagger,+}$
\end{tabular}
}}


\vspace*{.2in}

\begin{tabular}{c}
  Institute for Data, Systems, and Society$^\star$ \\
  Laboratory for Information and Decision Systems$^\dagger$ \\
  Department of Electrical Engineering and Computer
  Sciences$^\ddagger$\\
  Department of Mathematics$^+$ \\
  Massachusetts Institute of Technology \\
  \texttt{\{yulingy,mjwain\}@mit.edu}
\end{tabular}

\date{}

\vspace*{.2in}

\begin{abstract}
  Longitudinal or panel data can be represented as a matrix with rows
  indexed by units and columns indexed by time.  We consider
  inferential questions associated with the missing data version of
  panel data induced by staggered adoption.  We propose a
  computationally efficient procedure for estimation, involving only
  simple matrix algebra and singular value decomposition, and prove
  non-asymptotic and high-probability bounds on its error in
  estimating each missing entry.  By controlling proximity to a
  suitably scaled Gaussian variable, we develop and analyze a
  data-driven procedure for constructing entrywise confidence
  intervals with pre-specified coverage.  Despite its simplicity, our
  procedure turns out to be instance-optimal: we prove that the width
  of our confidence intervals match a non-asymptotic instance-wise
  lower bound derived via a Bayesian Cram\'{e}r--Rao argument.  We
  illustrate the sharpness of our theoretical characterization on a
  variety of numerical examples. Our analysis is based on a general
  inferential toolbox for SVD-based algorithm applied to the matrix
  denoising model, which might be of independent interest.
\end{abstract}


\end{center}


\section{Introduction}

Longitudinal or panel data consists of a collection of observations of
units (e.g., individuals, companies, countries) that are collected
over time.  In many applications, a subset of units are exposed to a
``treatment'' (e.g., drugs, regulations, or governmental
interventions) beginning at some time.  Given data of this type, it is
frequently of interest to estimate counterfactual quantities, such as
what would have been their response if they had not been treated.
Such estimates underpin inferential methods for the treatment effect,
corresponding to the difference between the treated and untreated
responses.  From the perspective of the untreated observations,
performing treatment can be modeled as inducing \emph{missing data}:
we no longer have observations of a given unit's untreated response at
all times after treatment.  This problem and its variants, known as
(causal) \emph{inference with panel data}, find wide applications in
economics, social sciences, and biomedical research
(e.g.,~\citep{hu2008ownership,lewis2008economics,huang2008causal,imbens2015causal}).

In this paper, we assume that the treatment assignments follow the
staggered adoption
design~\citep{athey2022design,shaikh2021randomization}, meaning that
units may begin treatment at possibly different times, but that once
initiated, the treatment is irreversible.  We can collect data for the
untreated unit/period pairs into a matrix, whose rows index units and
columns index periods.  Performing treatment can be viewed as inducing
a form of missingness in this matrix; we refer the reader
to~\Cref{fig:intro} for an illustration of the induced missing
pattern.

Broadly speaking, in the literature on causal panel data, there are at
least two main approaches for imputing missing entries in panel data:
those based on unconfoundedness
(e.g.,~\citep{rosenbaum1983central,imbens2015causal}), and those based
on synthetic controls
(e.g.,~\citep{abadie2003economic,abadie2015comparative,abadie2021using}).
In order to understand how these methods work, suppose that we are
interested in estimating the missing outcome for unit $i$ at period
$t$. Approaches based on unconfoundedness seek to identify a subset of
untreated units whose outcomes (before unit $i$ was treated) are
similar to those of unit $i$, and use their observed outcomes at
period $t$ to estimate unit $i$'s missing outcome. On the other hand,
approaches based on synthetic controls estimate unit $i$'s missing
outcome at time $t$ by using a weighted average of all untreated
units' observed outcomes at that time, where the weights are typically
determined by solving a re problem using a set of predictors
that are unaffected by the treatment.

\newlength{\mywidth}
\setlength{\mywidth}{0.55cm}
\begin{figure}[t]
  \centering 	\begin{tikzpicture}
		\matrix (m) [matrix of math nodes,
		nodes={draw, minimum height=\mywidth, minimum width=\mywidth, anchor=center},
		column sep=-\pgflinewidth, row sep=-\pgflinewidth, font=\small]{
			|[fill=gray!20]| \mathrm{C} & |[fill=gray!20]| \mathrm{C} & |[fill=gray!20]| \mathrm{C} & |[fill=gray!20]| \mathrm{C} & |[fill=gray!20]| \mathrm{C} & |[fill=gray!20]| \mathrm{C} & |[fill=gray!20]| \mathrm{C} & |[fill=gray!20]| \mathrm{C} & |[fill=gray!20]| \mathrm{C} & |[fill=gray!20]| \mathrm{C} & |[fill=gray!20]| \mathrm{C} & |[fill=gray!20]| \mathrm{C} & |[fill=gray!20]| \mathrm{C} & |[fill=gray!20]| \mathrm{C} & |[fill=gray!20]| \mathrm{C} &   \\ 
			|[fill=gray!20]| \mathrm{C} & |[fill=gray!20]| \mathrm{C} & |[fill=gray!20]| \mathrm{C} & |[fill=gray!20]| \mathrm{C} & |[fill=gray!20]| \mathrm{C} & |[fill=gray!20]| \mathrm{C} & |[fill=gray!20]| \mathrm{C} & |[fill=gray!20]| \mathrm{C} & |[fill=gray!20]| \mathrm{C} & |[fill=gray!20]| \mathrm{C} & |[fill=gray!20]| \mathrm{C} & |[fill=gray!20]| \mathrm{C} & |[fill=gray!20]| \mathrm{C} & |[fill=gray!20]| \mathrm{C} & |[fill=gray!20]| \mathrm{C} &  \\ 
			|[fill=gray!20]| \mathrm{C} & |[fill=gray!20]| \mathrm{C} & |[fill=gray!20]| \mathrm{C} & |[fill=gray!20]| \mathrm{C} & |[fill=gray!20]| \mathrm{C} & |[fill=gray!20]| \mathrm{C} & |[fill=gray!20]| \mathrm{C} & |[fill=gray!20]| \mathrm{C} & |[fill=gray!20]| \mathrm{C} & |[fill=gray!20]| \mathrm{C} & |[fill=gray!20]| \mathrm{C} & |[fill=gray!20]| \mathrm{C} & |[fill=gray!20]| \mathrm{C} & |[fill=gray!20]| \mathrm{C} & |[fill=gray!20]| \mathrm{C} &  \\ 
			|[fill=gray!20]| \mathrm{C} & |[fill=gray!20]| \mathrm{C} & |[fill=gray!20]| \mathrm{C} & |[fill=gray!20]| \mathrm{C} & |[fill=gray!20]| \mathrm{C} & |[fill=gray!20]| \mathrm{C} & |[fill=gray!20]| \mathrm{C} & |[fill=gray!20]| \mathrm{C} & |[fill=gray!20]| \mathrm{C} & |[fill=gray!20]| \mathrm{C} & |[fill=gray!20]| \mathrm{C} & |[fill=gray!20]| \mathrm{C} & |[fill=gray!20]| \mathrm{C} & |[fill=gray!20]| \mathrm{C} & |[fill=gray!20]| \mathrm{C} &  \\ 
			|[fill=gray!20]| \mathrm{C} & |[fill=gray!20]| \mathrm{C} & |[fill=gray!20]| \mathrm{C} & |[fill=gray!20]| \mathrm{C} & |[fill=gray!20]| \mathrm{C} & |[fill=gray!20]| \mathrm{C} & |[fill=gray!20]| \mathrm{C} & |[fill=gray!20]| \mathrm{C} & |[fill=gray!20]| \mathrm{C} & |[fill=gray!20]| \mathrm{C} & |[fill=gray!20]| \mathrm{C} & \mathrm{T} & \mathrm{T} & \mathrm{T} & \mathrm{T} &  \\ 
			|[fill=gray!20]| \mathrm{C} & |[fill=gray!20]| \mathrm{C} & |[fill=gray!20]| \mathrm{C} & |[fill=gray!20]| \mathrm{C} & |[fill=gray!20]| \mathrm{C} & |[fill=gray!20]| \mathrm{C} & |[fill=gray!20]| \mathrm{C} & |[fill=gray!20]| \mathrm{C} & |[fill=gray!20]| \mathrm{C} & |[fill=gray!20]| \mathrm{C} & |[fill=gray!20]| \mathrm{C} & \mathrm{T} & \mathrm{T} & \mathrm{T} & \mathrm{T} &  \\ 
			|[fill=gray!20]| \mathrm{C} & |[fill=gray!20]| \mathrm{C} & |[fill=gray!20]| \mathrm{C} & |[fill=gray!20]| \mathrm{C} & |[fill=gray!20]| \mathrm{C} & |[fill=gray!20]| \mathrm{C} & |[fill=gray!20]| \mathrm{C} & |[fill=gray!20]| \mathrm{C} & \mathrm{T} & \mathrm{T} & \mathrm{T} & \mathrm{T} & \mathrm{T} & \mathrm{T} & \mathrm{T} &  \\ 
			|[fill=gray!20]| \mathrm{C} & |[fill=gray!20]| \mathrm{C} & |[fill=gray!20]| \mathrm{C} & |[fill=gray!20]| \mathrm{C} & |[fill=gray!20]| \mathrm{C} & |[fill=gray!20]| \mathrm{C} & |[fill=gray!20]| \mathrm{C} & |[fill=gray!20]| \mathrm{C} & \mathrm{T} & \mathrm{T} & \mathrm{T} & \mathrm{T} & \mathrm{T} & \mathrm{T} & \mathrm{T} &  \\ 
			|[fill=gray!20]| \mathrm{C} & |[fill=gray!20]| \mathrm{C} & |[fill=gray!20]| \mathrm{C} & |[fill=gray!20]| \mathrm{C} & |[fill=gray!20]| \mathrm{C} & |[fill=gray!20]| \mathrm{C} & |[fill=gray!20]| \mathrm{C} & |[fill=gray!20]| \mathrm{C} & \mathrm{T} & \mathrm{T} & \mathrm{T} & \mathrm{T} & \mathrm{T} & \mathrm{T} & \mathrm{T} &  \\ 
			|[fill=gray!20]| \mathrm{C} & |[fill=gray!20]| \mathrm{C} & |[fill=gray!20]| \mathrm{C} & |[fill=gray!20]| \mathrm{C} & \mathrm{T} & \mathrm{T} & \mathrm{T} & \mathrm{T} & \mathrm{T} & \mathrm{T} & \mathrm{T} & \mathrm{T} & \mathrm{T} & \mathrm{T} & \mathrm{T} &  \\ 
			|[fill=gray!20]| \mathrm{C} & |[fill=gray!20]| \mathrm{C} & |[fill=gray!20]| \mathrm{C} & |[fill=gray!20]| \mathrm{C} & \mathrm{T} & \mathrm{T} & \mathrm{T} & \mathrm{T} & \mathrm{T} & \mathrm{T} & \mathrm{T} & \mathrm{T} & \mathrm{T} & \mathrm{T} & \mathrm{T} &  \\ 
		};
		
		\node[above=-1mm, font=\small] at (m.north) {time periods};
		\node[left=0mm, font=\small] at (m.west) {units};
		
		\draw[decorate,decoration={brace,amplitude=6pt,raise=1pt}, font=\small]
		(m-1-15.north east) -- (m-4-15.south east) node [black,midway,xshift=30pt,yshift=-8pt,anchor=south] {group 1};
		\draw[decorate,decoration={brace,amplitude=6pt,raise=1pt}, font=\small]
		(m-5-15.north east) -- (m-6-15.south east) node [black,midway,xshift=30pt,yshift=-8pt,anchor=south] {group 2};
		\draw[decorate,decoration={brace,amplitude=6pt,raise=1pt}, font=\small]
		(m-7-15.north east) -- (m-9-15.south east) node [black,midway,xshift=30pt,yshift=-8pt,anchor=south] {group 3};
		\draw[decorate,decoration={brace,amplitude=6pt,raise=1pt}, font=\small]
		(m-10-15.north east) -- (m-11-15.south east) node [black,midway,xshift=30pt,yshift=-8pt,anchor=south] {group 4};
		
		\draw[decorate, decoration={brace, amplitude=6pt, raise=1pt, mirror}, font=\small]
		(m-11-1.south west) -- (m-11-4.south east) node [black, midway, yshift=-15pt] {stage 1};
		\draw[decorate, decoration={brace, amplitude=6pt, raise=1pt, mirror}, font=\small]
		(m-11-5.south west) -- (m-11-8.south east) node [black, midway, yshift=-15pt] {stage 2};
		\draw[decorate, decoration={brace, amplitude=6pt, raise=1pt, mirror}, font=\small]
		(m-11-9.south west) -- (m-11-11.south east) node [black, midway, yshift=-15pt] {stage 3};
		\draw[decorate, decoration={brace, amplitude=6pt, raise=1pt, mirror}, font=\small]
		(m-11-12.south west) -- (m-11-15.south east) node [black, midway, yshift=-15pt] {stage 4};
	\end{tikzpicture}
  \caption{An illustration of a panel data with missingness induced by
    staggered adoption design. The labels C and T refer to ``control''
    (i.e., untreated) and ``treated'' respectively.  From the
    perspective of the control group, all entries marked with T
    correspond to missing data, and a key problem is to impute these
    entries.} \label{fig:intro}
\end{figure}

Athey et al.~\citep{athey2021matrix} observed that both the
unconfoundedness and synthetic control methods can be related to the
literature on low-rank matrix completion
(e.g.,~\citep{ExactMC09,Negahban2012restricted}).  Based on this
connection, they proposed to impute the missing potential outcomes via
the standard convex relaxation for low-rank matrix
completion---namely, minimizing a least-squares objective with a
nuclear norm regularization.  When the potential outcome matrix is
approximately low-rank, this approach can be equipped with attractive
guarantees.  Following this general avenue, the past few years have
witnessed substantial progress in the development and analysis of
matrix completion algorithms tailored for panel data with missingness
(e.g.,~\citep{bai2021matrix,cahan2023factor,agarwal2023causal,choi2023matrix}).
Matrix completion methods based on nuclear norm relaxation are known
to be optimal when entries are missing completely at random
(e.g.,~\cite{wainwright2019high}), but there are currently no such
optimality guarantees for missingness induced by staggered adoption.

\subsection{Our contributions}

With this context, the main contribution of this paper is to propose a
simple procedure, substantially less computationally intensive than
nuclear norm relaxation, for inferring individual treatment effects in
panel data.  Most importantly, we are able to show that our
procedure---despite its simplicity---is unimprovable in a sharp
instance-optimal sense.  More precisely:
\begin{itemize}
\item \textbf{Entrywise guarantees:} We give non-asymptotic bounds and
  distributional characterizations of the error in estimating each
  matrix entry. Using the distributional characterization, we provide
  a data-driven procedure for computing confidence intervals for the
  unobserved entries with prescribed coverage.
\item \textbf{Sharp instance-wise optimality:} We demonstrate that the
  length of our confidence intervals match the Bayesian Cram\'er-Rao
  lower bound associated with a genie-aided version of the problem (in
  which an oracle provides partial information about the unknown
  truth).
\item \textbf{Computational simplicity:} Our algorithm only involves
  basic matrix operations and singular value decompositions, making it
  computationally faster than standard convex relaxations (which
  typically involve solving semi-definite programs).
\end{itemize}
Our procedure and theory is built upon an inferential toolbox
developed for the SVD algorithm in the context of low-rank matrix
denoising, which might be of independent
interest. See~\Cref{appendix:proof-thm-denoising} for details.

\paragraph{A preview:}
\Cref{FigPreview} provides a high-level preview of some of the
consequences of our results, and the sharpness of our theoretical
predictions.  In particular, let $\Mstar \in \real^{N \times T}$ be
the underlying matrix of outcomes for the untreated group.
\begin{figure}[h!]
  \begin{center}
    \begin{tabular}{ccc}
      \widgraph{0.45\textwidth}{\figdir/fig_qqplot_entry1} &&
      \widgraph{0.45\textwidth}{\figdir/fig_incoherence_case1} \\
      (a) && (b)
    \end{tabular}
    \caption{(a) Illustration of the asymptotic normality of the
      rescaled error $\Rescale_{i,t}$ from
      equation~\eqref{EqnDefnZscore} for entry $(i,t) = (500, 500)$
      from a $500 \times 500$ matrix.  Shown is a Q-Q plot of the
      empirical quantile of the suitably rescaled estimation error
      versus the standard normal quantiles; consistent with our
      theory, excellent agreement is shown.  (b) Illustration of the
      sharpness of our theoretical predictions: we construct a family
      of matrix missing data problems, whose difficultly is calibrated
      by an ``incoherence parameter''.  Plots of actual mean-squared
      error (MSE) of our procedure (red circles) to a lower bound
      derived using a Bayesian Cram\'{e}r-Rao lower bound (blue solid
      line).  See~\Cref{sec:numerical} for the details of these
      simulation studies.}
      \label{FigPreview}
  \end{center}
\end{figure}
Our methodological innovation is to propose an
estimate $\Mhat$ that is attractively simple: its computation involves
only SVDs and other basic matrix operations.  For each unobserved
entry $(i, t)$, we define the rescaled error term
\begin{align}
\label{EqnDefnZscore}
\Rescale_{i,t} = ( \widehat{M}_{i,t} - M_{i,t}^\star ) /
\sqrt{\Varhat_{i,t}},
\end{align}
where $\Varhat_{i,t}$ is a data-dependent scaling factor.  On the
achievable side, we provide a non-asymptotic decomposition of the
rescaled error $\Rescale_{i,t}$, a particular consequence of which is
that it converges in distribution to a standard $N(0,1)$ variate;
see~\Cref{thm:distribution,thm:distribution-general} for statements of
these results.  Panel (a) in~\Cref{FigPreview} provides an empirical
demonstration of this predicted entrywise asymptotic normality.
Moreover, we prove that these estimation-theoretic and inferential
guarantees are \emph{optimal in a strong sense:} our variance
estimates $\Varhat_{i,t}$ converge to a population quantity
$\Varstar_{i,t}$ that is defined via a non-asymptotic Bayesian
Cram\'{e}r--Rao lower bound; see~\Cref{thm:CRLB} and the discussion
following~\Cref{thm:distribution-general} for these claims.  Panel (b)
provides empirical confirmation of the instance optimality of our
method; we construct an ensemble of problems whose difficulty is
indexed by an ``incoherence'' parameter.  Panel (b) compares Monte
Carlo estimates of the mean-squared error obtained by our procedure
with the lower bound predicted by~\Cref{thm:CRLB}; note the excellent
agreement.  We refer the reader to~\Cref{sec:numerical} for full
details on the simulation set-ups that underlie these empirical
studies.


\subsection{Related work}

So as to put our contributions in context, let us discuss some related
lines of work.  Low-rank matrix completion has been studied in great
detail when entries are assumed to be missing completely at random;
the behavior of the standard convex relation based on the nuclear norm
minimization is now very well-understood.  Initial
investigations~\citep{ExactMC09,recht2010guaranteed,CanTao10}
primarily focused on the noiseless scenario and examined the minimal
sample size required for exact recovery. In the context of noisy
matrix completion (where the observed entries are corrupted by random
noise), the last decade has seen the establishment of optimal
estimation guarantees for convex
relaxation~\citep{Negahban2012restricted,klopp2014noisy,chen2019noisy}. A
more recent line of work~\cite{chen2019inference,xia2021statistical}
has focused on debiasing techniques for convex relaxation, showing how
it is possible to construct confidence intervals for each entry of the
unknown matrix.  We refer to the reader to the
survey~\cite{chi2018nonconvex} for an overview of other non-convex
approaches to matrix completion
(e.g.,~\citep{burer2003nonlinear,KesMonSew2010,srebro2004learning}).

Athey et al.~\cite{athey2021matrix} pioneered the use of convex matrix
completion for imputing potential outcomes in panel data.  They
provided Frobenius norm bounds on the error in estimating the full
matrix; such a bound can be viewed as providing estimation guarantees
for averaged treatment effects. More recently, Choi and
Yuan~\cite{choi2023matrix} proposed a more refined approach that first
divides the missing entries into smaller groups, and then applies the
convex relaxation to each group.  They derived entrywise error bounds
that are sharper than the Frobenius norm error bound from the
paper~\cite{athey2021matrix}, and also established asymptotic
normality for some statistics of interest. Agarwal et
al.~\cite{agarwal2023causal} proposed an algorithm based on synthetic
nearest-neighbors, for which they established error bounds and
asymptotic normality.  There is also a line of
work~\cite{bai2021matrix,cahan2023factor} on models satisfying certain
factor constraints, to which we compare in more detail
in~\Cref{sec:comparison}.  Other work on related but distinct
estimation problems for causal panel data include the
papers~\cite{farias2021learning,lei2023estimating,choi2023inference,agarwal2020synthetic}.

Our algorithm involves spectral techniques, and it is natural that our
analysis of it has connections to past work on spectral
methods~\cite{abbe2017entrywise,cai2019subspace,yan2021inference,zhou2023deflated}.
Notably, some past
work~\citep{el2015impact,sur2017likelihood,ma2017implicit,chen2020bridging}
has made effective use of ``leave-one-out'' methods; see the
monograph~\citep{chen2020spectral} for an overview.  Among the
technical contributions of this paper is a natural generalization of
this idea, which we refer to as ``leave-one-block-out''.  As will be
clarified by our analysis, this extension is essential in providing a
sharp characterization of the subspace perturbation error that arises
in the panel data model.

\paragraph{Notation:} For a positive integer $n$, we adopt the
shorthand $[n] \coloneqq \{ 1 , \ldots , n \}$, For a matrix $\bm{A}
\in \mathbb{R}^{n_1 \times n_2}$ and subsets $\mathcal{I} = \{ i_1,
i_2, \ldots, i_I\} \subseteq [n_1]$ and $\mathcal{J} = \{ j_1, j_2,
\ldots, j_J \} \subseteq [n_2]$, define $\bm{A}_{\mathcal{I} ,
  \mathcal{J}} \in \mathbb{R} ^ { I \times J}$ to be the submatrix of
$\bm{A}$ comprising the intersection of its $I$ rows indexed by
$\mathcal{I}$ and $J$ columns indexed by $\mathcal{J}$, specifically $
(\bm{A}_{\mathcal{I}, \mathcal{J}})_{k, l}=A_{i_k, j_l} $ for any $k
\in [ I ]$ and $l \in [ J ]$. In addition, we abbreviate
$\bm{A}_{\mathcal{I}, \cdot} \coloneqq \bm{A}_{ \mathcal{I} , [ n_2 ]
}$ and $\bm{A}_{ \cdot, \mathcal{J} } \coloneqq \bm{A}_{ [n_1] ,
  \mathcal{J}}$. If a subset $\mathcal{I} = \{ i \}$ is a singleton,
we may directly use the index $i$ to replace $\mathcal{I}$ for
simplicity. We use $\Vert \bm{A} \Vert$, $\Vert \bm{A}
\Vert_{\mathrm{F}}$ and $\Vert \bm{A} \Vert_{\infty}$ (respectively)
to denote the spectral norm, Frobenius norm and entrywise
$\ell_{\infty}$ norms of a matrix $\bm{A}$. For any symmetric matrices
$\bm{A}$ and $\bm{B} \in \mathbb{R}^{n\times n}$, the relation $\bm{A}
\succeq\bm{B}$ (resp.~$\bm{A}\preceq \bm{B}$) means that
$\bm{A}-\bm{B}$ (resp.~$\bm{B}-\bm{A}$) is positive semidefinite. For
any invertible matrix $\bm{H} \in \mathbb{R}^{r\times r}$ with SVD
$\bm{U} \bm{\Sigma} \bm{V}^\top$, define its sign matrix $\mathsf{sgn}
( \bm{H} ) \coloneqq \bm{U} \bm{V}^\top$. We also define the function
$\mathsf{svds}(\bm{A},r)$ to output the truncated rank-$r$ SVD
$(\bm{U}, \bm{\Sigma}, \bm{V})$ of $\bm{A}$, where the columns of
$\bm{U} \in \mathbb{R}^{n_1 \times r}$ and $\bm{V} \in \mathbb{R}^{n_2
  \times r}$ are the top-$r$ left/right singular vectors, and
$\bm{\Sigma} \in \mathbb{R}^{r \times r}$ is a diagonal matrix
containing the top-$r$ singular values.



\section{Problem set-up}

In this section, we provide a more precise description of the
observation model, and its reformulation in terms of matrix estimation
with missing entries.

\subsection{Basic observation model}
\label{sec:model}

We consider a panel data setting in which there are $\Nunit$ units
over $\Time$ periods.  For each unit $i \in [N]$ and at each time
period $t \in [T]$, we observe a response $Y_{i,t}$.  The
interpretation of this observation depends on whether or not unit $i$
has undergone treatment by time $t$.  We assume that treatment follows
the \emph{staggered adoption design}, meaning that each unit $i \in
[\Nunit]$ may differ in the time that it is first exposed to
treatment, and the treatment is irreversible
(e.g.,~\cite{athey2022design,athey2021matrix,shaikh2021randomization}).

More precisely, for each unit $i$, we let the integer $t_i$ denote the
time at which unit $i$ was first exposed to the treatment; for
completeness, we set $t_i = \infty$ if unit $i$ never undergoes the
treatment.  Using the notation of potential outcomes, let $Y_{i,t}(0)$
be the mean outcome of unit $i$ at time $t$ \emph{if} the unit is
never exposed to the treatment.  For a given unit $i$ and for all
times $t = 1, 2, \ldots, t_i-1$, we assume that the observed response
$Y_{i,t}$ is a noisy version of this mean outcome---that is
\begin{align}
Y_{i,t} & = Y_{i,t} (0) + E_{i,t} \qquad \mbox{for each $t = 1,
  \ldots, t_i-1$,}
\end{align}
where $\{ E_{i,t} \}$ are independent noise variables. On the other
hand, for any time index $t \geq t_i$, we do not model any joint
structure between the observation $Y_{i,t}$ and $Y_{i,t}(0)$, because
we expect that the former might also depend on the adoption time
$t_i$.

For each unit $i$ and each time $t \geq t_i$, our goal is to estimate
the individual treatment effect (ITE)
\begin{align}
\tau_{i,t} \coloneqq Y_{i,t} - Y_{i,t}(0),
\end{align}
corresponding to the difference between the treated response $Y_{i,t}$
(that we observe) and and the (unobserved) untreated mean
$Y_{i,t}(0)$.  Without further assumptions, none of our observations
give any information about $Y_{i,t}(0)$ for all $t \geq t_i$; thus,
the ITEs are not identifiable in general.  Additional assumptions are
required to ensure identifiability, and following a line of previous
work, out approach is to assume that the untreated mean outcomes have
low-rank structure.

In order to state this assumption more precisely, we introduce an
$\Nunit \times \Time$ matrix that is populated by the untreated mean
outcomes---that is, matrix $\bm{M}^\star \in \RR^{N \times T}$ with
entries $M_{i,t}^\star = Y_{i,t}(0)$.  Given the previously described
observation model, we observe
\begin{align}
Y_{i,t} & = M_{i,t}^\star + E_{i,t} \qquad \mbox{for each $i \in [N]$
  and $t = 1, 2, \ldots, t_i-1$.}
\end{align}
Consequently, the problem of estimating each individual treatment
effect $\tau_{i,t}$ can reduced to estimating the missing entries
$\{M_{i,t}^\star \mid i\in[N], t \geq t_i\}$.  As in a line of past
work~\cite{athey2021matrix,choi2023matrix}, we adopt the assumption
that $\bm{M}^\star$ is a low-rank matrix (i.e., has rank much smaller
than $\min \{N, T \}$), under which it becomes possible to estimate
$\bm{M}^\star$ even in the missing data setting given here.


\subsection{Formulation as low-rank matrix completion}

We now give a precise formulation in terms of low-rank matrix
completion.
\begin{figure}[t]
\begin{center}
  	\begin{tabular}{cc}
		\begin{tikzpicture}
			\matrix (m) [matrix of math nodes,
			nodes={draw, minimum height=0.8cm, minimum width=1.2cm, anchor=center, inner sep=0.2pt},
			column sep=-\pgflinewidth, row sep=-\pgflinewidth]{
				\bm{1} & \bm{1}  & \cdots & \bm{1} & \bm{1} & \bm{1}  \\
				\bm{1} & \bm{1} & \cdots & \bm{1} & \bm{1} & \bm{0}  \\
				\bm{1} & \bm{1} & \cdots & \bm{1} & \bm{0} & \bm{0}  \\
				\myvdots & \myvdots & \myddots & \myvdots & \myvdots & \myvdots & \\
				\bm{1} & \bm{0} & \cdots & \bm{0} & \bm{0} & \bm{0} \\
			};
			
			\node[above, font=\footnotesize] at (m-1-1.north) {$T_1$};
			\node[above, font=\footnotesize] at (m-1-2.north) {$T_2$};
			\node[above, font=\footnotesize] at (m-1-4.north) {$T_{k-2}$};
			\node[above, font=\footnotesize] at (m-1-5.north) {$T_{k-1}$};
			\node[above, font=\footnotesize] at (m-1-6.north) {$T_{k}$};

			\node[left, font=\footnotesize] at (m-1-1.west) {$N_{1}$};
			\node[left, font=\footnotesize] at (m-2-1.west) {$N_{2}$};
			\node[left, font=\footnotesize] at (m-3-1.west) {$N_{3}$};
			\node[left, font=\footnotesize] at (m-5-1.west) {$N_{k}$};
			
			\node[below=1mm] at (m.south) {$\bm{\Omega}$};
		\end{tikzpicture}  & 
		\begin{tikzpicture}
			\matrix (m) [matrix of math nodes,
			nodes={draw, minimum height=0.8cm, minimum width=1.2cm, anchor=center, font=\footnotesize, inner sep=0.2pt},
			column sep=-\pgflinewidth, row sep=-\pgflinewidth]{
				|[fill=gray!20]| \bm{M}_{1,1}^\star & |[fill=gray!20]| \bm{M}_{1,2}^\star  & |[fill=gray!20]| \cdots & |[fill=gray!20]| \bm{M}_{1,k-2}^\star & |[fill=gray!20]| \bm{M}_{1,k-1}^\star & |[fill=gray!20]| \bm{M}_{1,k}^\star  \\
				|[fill=gray!20]| \bm{M}_{2,1}^\star & |[fill=gray!20]| \bm{M}_{2,2}^\star & |[fill=gray!20]| \cdots & |[fill=gray!20]| \bm{M}_{2,k-2}^\star & |[fill=gray!20]| \bm{M}_{2,k-1}^\star & \bm{M}_{2,k}^\star  \\
				|[fill=gray!20]| \bm{M}_{3,1}^\star & |[fill=gray!20]| \bm{M}_{3,2}^\star & |[fill=gray!20]| \cdots & |[fill=gray!20]| \bm{M}_{3,k-2}^\star & \bm{M}_{3,k-1}^\star & \bm{M}_{3,k}^\star  \\
				\myvdots & \myvdots & \myddots & \myvdots & \myvdots & \myvdots & \\
				|[fill=gray!20]| \bm{M}_{k,1}^\star & \bm{M}_{k,2}^\star & \cdots & \bm{M}_{k,k-2}^\star & \bm{M}_{k,k-1}^\star & \bm{M}_{k,k}^\star \\
			};
			
			\node[above, font=\footnotesize] at (m-1-1.north) {$T_1$};
			\node[above, font=\footnotesize] at (m-1-2.north) {$T_2$};
			\node[above, font=\footnotesize] at (m-1-4.north) {$T_{k-2}$};
			\node[above, font=\footnotesize] at (m-1-5.north) {$T_{k-1}$};
			\node[above, font=\footnotesize] at (m-1-6.north) {$T_{k}$};
			
			\node[left, font=\footnotesize] at (m-1-1.west) {$N_{1}$};
			\node[left, font=\footnotesize] at (m-2-1.west) {$N_{2}$};
			\node[left, font=\footnotesize] at (m-3-1.west) {$N_{3}$};
			\node[left, font=\footnotesize] at (m-5-1.west) {$N_{k}$};
			
			\node[below=1mm] at (m.south) {$\bm{M}^\star$};
		\end{tikzpicture}  
	\end{tabular}
  \caption{An illustration of the the indicator matrix $\bm{\Omega}$
    and the potential outcome matrix $\bm{M}^\star$ under staggered
    adoption design. Here $\bm{1}$ (resp.~$\bm{0}$) denotes a matrix
    whose entries are all one (resp.~zero). The gray blocks in the
    right panel are the potential outcomes associated with the
    untreated unit/period pairs.}
  \label{fig:setup}
\end{center}
\end{figure}
We assume that the matrix $\MatTre{}$ has rank $r \ll \min
\{\Nunit,\Time \}$.  Define the subset $\Omega \subseteq [N] \times
        [T]$ of indices where $(i, t) \in \Omega$ if and only if unit
        $i$ has not yet been treated at time $t$.  In terms of this
        notation, we can write our observation model in the compact
        form
\begin{align}
\label{EqnObservationModel}  
\bm{M} \mydefn \mathcal{P}_{ \Omega } \left ( \MatTre{} + \bm{E}
\right ),
\end{align}
where $\mathcal{P}_{\Omega}(\bm{X})$ denotes the Euclidean projection
of a matrix $\bm{X}$ onto the subspace of matrices supported on
$\Omega$, and $\bm{E}$ is a random noise matrix with
i.i.d.~$\mathcal{N} \big(0, \noise^{2}/(NT) \big)$ entries.  Our
scaling of the noise variance ensures that $\bm{E}$ has Frobenius norm
close to $\noise$ with high probability, and facilitates
interpretation of results in the sequel.

Under the staggered adoption design, we can sort the $N$ units
according to the time $t_{i}$ at which they were first exposed to the
treatment.  With this sorting, we have $t_{1} \geq t_{2} \geq \cdots
\geq t_{N}$, and we are guaranteed the existence of an integer $k \geq
1$ and partitions $N_{1}, \ldots N_{k}$ and $T_{1}, \ldots, T_{k}$
with $\sum_{i = 1}^{k}N_{i} = N$ and $\sum_{j = 1}^{k}T_{j} = T$ such
that the Boolean matrix
\begin{align}
  \label{EqnDefnOmega}
  \bm{\Omega} \in \{0,1\}^{N \times T} \quad \mbox{with entries
    $\bm{\Omega}_{i,t} \mydefn \ind_{(i, t) \in \Omega}$,}
\end{align}
associated with the subset $\Omega$ has the block structure as in the
left panel of~\Cref{fig:setup}. We also partition $\MatTre{}$
according to the pattern in $\bm{\Omega}$ to get the submatrices
$\bm{M}_{i, j}^{\star} \in \mathbb{R}^{N_{i} \times T_{j}}$, as shown
in the right panel of~\Cref{fig:setup}. For each $i\in[k]$ and $1 \leq
j \leq k-i$ define $\bm{M}_{i, j} = \bm{M}_{i, j}^{\star} + \bm{E}_{i,
  j}$, where $\bm{E}_{i,j}$ is the corresponding submatrix of
$\bm{E}$.

Our goal is to design an algorithm that provides point estimates,
along with confidence intervals, for each unseen entry $M_{i,t}^\star$
as well as corresponding ITE $\tau_{i,t} = Y_{i,t} - M_{i,t}^\star$,
where $(i,t) \notin \Omega$. Since $Y_{i,t}$ is known, the estimation
and inference for these two sets of quantities are equivalent, hence
we will only present algorithms and results for $M_{i,t}^\star$ in the
following sections.


\subsection{A simple four-block structure}

Our estimator for the general case can be understood most easily by
understanding how it applies to a special four-block structure.  More
precisely, suppose that a subset of $N_{1}$ units \emph{never} receive
the treatment, while the other $N_{2} = N - N_{1}$ units are exposed
to the irreversible treatment at time $T_{1} + 1$. Then our general
observation model can be written in the four-block form
\begin{align}
\MatTre{} = \left [ \begin{array}{cc} \MatTre{a} & \MatTre{b}
  \\ \MatTre{c} & \MatTre{d}
\end{array} \right ], \qquad  \bm{E} =  \left [ \begin{array}{cc}
  \bm{E}_{a} & \bm{E}_{b} \\
\label{eq:four-block-structure}  
  \bm{E}_{c} & \text{N/A}
\end{array} \right ].
\end{align}
Here the subscript $a$ denotes matrices of size $N_{1}$ by $T_{1}$;
the subscript $b$ denotes matrices of size $N_1$ by $T - T_1$; and so
on for the subscripts $c$ and $d$.  As shown in the sequel, it
suffices to understand how to perform estimation and inference in this
scenario because our algorithm for the general staggered design first
reduces the problem to a simpler problem with four-block structure. In
view of the discussion at the end of the previous section, our goal is
to estimate and construct confidence intervals for the entries of
$\bm{M}_{d}^{\star}$, since they correspond to the counterfactual
outcomes associated with the control group.




\section{Estimation algorithms}
\label{sec:alg}

In this section, we first design an algorithm that is applicable to
the simpler four-block model~\eqref{eq:four-block-structure}.  This
algorithm is relatively easy to describe, and is the basic building
block for our algorithm that applies to the general staggered design.

\subsection{Algorithm for the four-block design}

\begin{algorithm}[t]
  \DontPrintSemicolon \SetNoFillComment
  \textbf{Input:} Data matrix $ \bm{M}$, rank $r$ \\
\tcc{Step 1: Subspace Estimation} Compute the truncated rank-$r$ SVD
$(\bm{U}_{\mathsf{left}}, \bm{\Sigma}_{\mathsf{left}},
\bm{V}_{\mathsf{left}})$ of $\bm{M}_{\mathsf{left}}$. \\
Partition $\bm{U}_{\mathsf{left}}$ into two submatrices $\bm{U}_{1}$
and $\bm{U}_{2}$, where $\bm{U}_{1} \in \mathbb{R}^{N_{1}\times r}$
consists of its top $N_{1}$ rows and $\bm{U}_{2} \in
\mathbb{R}^{N_{2}\times r}$ consists of its bottom $N_{2}$
rows. \\
\tcc{Step 2: Matrix Denoising} Compute the truncated rank-$r$ SVD $(
\bm{U}_{\mathsf{upper}}, \bm{\Sigma}_{\mathsf{upper}},
\bm{V}_{\mathsf{upper}}) $ of $ \bm{M}_{\mathsf{upper}}$ \\
Partition $\bm{V}_{\mathsf{upper}}$ into two submatrices $\bm{V}_{1}$
and $\bm{V}_{2}$, where $\bm{V}_{1} \in \mathbb{R}^{T_{1}\times r}$
consists of its top $T_{1}$ rows and $\bm{V}_{2} \in
\mathbb{R}^{T_{2}\times r}$ consists of its bottom $T_{2}$ rows. \\
Compute the estimate $\widehat{ \bm{M}}_{b} \mydefn
\bm{U}_{\mathsf{upper}} \bm{\Sigma}_{\mathsf{upper}}
\bm{V}_{2}^{\top}$ of $\bm{M}_{b}^{\star}$. \\
\tcc{Step 3: Linear Regression} Compute the matrix $\widehat{
  \bm{M}}_{d} \mydefn \bm{U}_{2}( \bm{U}_{1}^{\top} \bm{U}_{1})^{-1}
\bm{U}_{1}^{\top}\widehat{ \bm{M}}_{b}$. \\
\textbf{Output:} $\widehat{ \bm{M}}_{d}$ as estimate of
$\bm{M}_{d}^{\star}$
\caption{Estimating counterfactual outcomes: four-block
  design \label{alg:4-blocks-Md}}
\end{algorithm}

We first define some necessary notation for the four-block model.  We
begin by writing our observation matrix $\bm{M}$ in the form
\begin{align*}
 \bm{M} = \left[\begin{array}{cc} \bm{M}_{a} & \bm{M}_{b}\\ \bm{M}_{c}
     & \bm{?}
\end{array} \right] = \left[\begin{array}{cc}
 \bm{M}_{a}^{\star} + \bm{E}_{a} & \bm{M}_{b}^{\star} +
 \bm{E}_{b}\\ \bm{M}_{c}^{\star} + \bm{E}_{c} & \bm{?}
\end{array} \right].
\end{align*}
The unknown matrix $\bm{M}^{\star}$ has a (compact) singular value
decomposition, which we write in the form \mbox{$\bm{M}^\star =
  \bm{U}^{\star} \bm{\Sigma}^{\star} \bm{V}^{\star\top}$,} where
$\bm{U}^{\star} \in \mathbb{R}^{N\times r}$ and $\bm{V}^{\star} \in
\mathbb{R}^{T\times r}$ have orthonormal columns, and
\mbox{$\bm{\Sigma}^{\star} = \mathsf{diag}(\sigma_{1}^{\star}, \ldots,
  \sigma_{r}^{\star})$} is a diagonal matrix consisting of the ordered
singular values \mbox{$\sigma_{1}^{\star} \geq \sigma_{2}^{\star} \geq
  \cdots \geq \sigma_{r}^{\star}$.} We partition the matrices of
singular vectors $\bm{U}^{\star} \in \real^{N \times r}$ and
$\bm{V}^{\star} \in \real^{T \times r}$ into two blocks as
\begin{align}
\bm{U}^{\star} = \left[\begin{array}{c} \bm{U}_{1}^{\star} \in
    \mathbb{R}^{N_1 \times r} \\ \bm{U}_{2}^{\star} \in \real^{N_2
      \times r}
  \end{array} \right] \qquad
\text{and} \qquad \bm{V}^{\star} = \left[\begin{array}{c}
    \bm{V}_{1}^{\star} \in \real^{T_1 \times r} \\ \bm{V}_{2}^{\star}
    \in \real^{T_2 \times r}
\end{array} \right]. \label{eq:UV1-star}
\end{align}
It is also convenient to define
\begin{align*}
 \bm{M}_{\mathsf{left}} = \left[\begin{array}{c} \bm{M}_{a}
     \\ \bm{M}_{c}
   \end{array} \right],  \qquad
\bm{M}_{\mathsf{left}}^{\star} = \left[\begin{array}{c}
    \bm{M}_{a}^{\star} \\
\bm{M}_{c}^{\star}
  \end{array} \right],
\qquad
\bm{M}_{\mathsf{upper}} = \left[\, \bm{M}_{a}\;\; \bm{M}_{b}\,
  \right], \qquad \bm{M}_{\mathsf{upper}}^{\star} = \left[\,
  \bm{M}_{a}^{\star}\;\; \bm{M}_{b}^{\star}\, \right].
\end{align*}

To motivate our algorithm, consider the ideal setting where there is
no noise (i.e.~$ \bm{E} = \bm{0}$). Given that $\bm{M}_{a}^{\star}$ is
rank-$r$, we can first compute the matrix $\bm{U}^{\star}$ (up to a
rotation) by using the SVD of the left submatrix
$\bm{M}_{\mathsf{left}}^{\star}$, and then recover
$\bm{M}_{d}^{\star}$ by
\begin{align*}
 \bm{M}_{d}^{\star} = \bm{U}_{2}^{\star}( \bm{U}_{1}^{\star\top}
 \bm{U}_{1}^{\star})^{-1} \bm{U}_{1}^{\star\top} \bm{M}_{b}^{\star}.
\end{align*}
The idea behind the above formula is that, each column of $
\bm{M}^{\star}$ can be written uniquely as a linear combination of the
columns of $ \bm{U}^{\star}$. When $ \bm{U}_{1}^{\star}$ has full
rank, we can find the coefficients by regressing each column of $
\bm{M}_{b}^{\star}$ over the columns of $ \bm{U}_{1}^{\star}$. These
coefficients in turn allow reconstruction of $\bm{M}_{d}^{\star}$
using the columns of $ \bm{U}_{2}^{\star}$.

In the presence of random noise, the na\"{i}ve procedure needs to be
adjusted.  Doing so leads to the following three-step procedure:
\begin{itemize}
\item (\textbf{Subspace estimation}) The goal of this step is to
  estimate $\bm{U}^{\star}$. We compute the truncated rank-$r$ SVD of
  $\bm{M}_{\mathsf{left}}$ as
\begin{align*}
 \left( \bm{U}_{\mathsf{left}}, \bm{\Sigma}_{\mathsf{left}},
 \bm{V}_{\mathsf{left}} \right)\;\longleftarrow\;\mathsf{svds} \left(
 \bm{M}_{\mathsf{left}}, r \right),
\end{align*}
and use the top-$r$ column subspace $\bm{U}_{\mathsf{left}}$ as an
estimate of $\bm{U}^{\star}$ up to rotations. We partition $
\bm{U}_{\mathsf{left}}$ into two submatrices $ \bm{U}_{1}$ and $
\bm{U}_{2}$, where $ \bm{U}_{1} \in \mathbb{R}^{N_{1}\times r}$
consists of its top $N_{1}$ rows and $ \bm{U}_{2} \in
\mathbb{R}^{N_{2}\times r}$ consists of its bottom $N_{2}$ rows.
\item (\textbf{Matrix denoising}) Then we need an estimate of $
  \bm{M}_{b}^{\star}$.  The most natural idea is to use $ \bm{M}_{b}$,
  which is an unbiased estimator for $ \bm{M}_{b}^{\star}$. However,
  due to the noise in the observation, this naive estimate will incur
  much larger estimation error, leading to suboptimal statistical
  efficiency. Instead we compute the truncated rank-$r$ SVD of
  $\bm{M}_{\mathsf{upper}}$, given by
\begin{align*}
 \left( \bm{U}_{\mathsf{upper}}, \bm{\Sigma}_{\mathsf{upper}},
 \bm{V}_{\mathsf{upper}} \right)\;\longleftarrow\;\mathsf{svds} \left(
 \bm{M}_{\mathsf{upper}}, r \right),
\end{align*}
We partition $ \bm{V}_{\mathsf{upper}}$ into two submatrices
$\bm{V}_{1}$ and $\bm{V}_{2}$, where $ \bm{V}_{1} \in
\mathbb{R}^{T_{1}\times r}$ consists of its top $T_{1}$ rows and $
\bm{V}_{2} \in \mathbb{R}^{T_{2}\times r}$ consists of its bottom
$T_{2}$ rows. We use $\widehat{ \bm{M}}_{b} \mydefn
\bm{U}_{\mathsf{upper}} \bm{\Sigma}_{\mathsf{upper}}
\bm{V}_{2}^{\top}$ as an estimate of $ \bm{M}_{b}^{\star}$.
\item (\textbf{Linear regression}) For each $1 \leq t \leq T_{2}$, we
  use the OLS solution to estimate the coefficient of the linear
  combination of the $t$-th column of $ \bm{M}_{b}^{\star}$ with
  respect to the basis $ \bm{U}_{1}^{\star}$
\begin{align*}
\widehat{ \bm{\beta}}_{t} \mydefn \mathop{\arg\min}_{ \bm{\beta} \in
  \mathbb{R}^{r}}\big \Vert(\widehat{ \bm{M}}_{b})_{\cdot,
  t}- \bm{U}_{1} \bm{\beta}\big
\Vert_{2}^{2} =( \bm{U}_{1}^{\top} \bm{U}_{1})^{-1} \bm{U}_{1}^{\top}(\widehat{ \bm{M}}_{b})_{\cdot,
  t} \quad \text{for all }1 \leq t \leq T_{2}.
\end{align*}
This allows us to estimate the $t$-th column of $ \bm{M}_{d}^{\star}$
by $ \bm{U}_{2}\widehat{ \bm{\beta}}_{t}$. In matrix form,  our final
estimate for $ \bm{M}_{d}^{\star}$ is $\widehat{ \bm{M}}_{d} \mydefn  \bm{U}_{2}( \bm{U}_{1}^{\top} \bm{U}_{1})^{-1} \bm{U}_{1}^{\top}\widehat{ \bm{M}}_{b}$.
\end{itemize}
This procedure is summarized in~\Cref{alg:4-blocks-Md}.
As we will see in the following sections,  under mild assumptions the
estimate $\widehat{ \bm{M}}_{d}$ achieves full statistical efficiency
for estimating $ \bm{M}_{d}^{\star}$ (even including the preconstant).

\subsection{Algorithm for the staggered adoption design}

\begin{figure}[t]
	\centering
	
	\begin{tabular}{cc}
	\begin{tikzpicture}
		\matrix (m) [matrix of math nodes, 
		nodes ={draw,  minimum height =0.75cm,  minimum width =1.1cm,  anchor =center}, 
		column sep =-\pgflinewidth,  row sep =-\pgflinewidth]{
			|[fill =red!20]|  \bm{M}_{1, 1} & |[fill =red!20]|  \bm{M}_{1, 2} & |[fill =blue!20]|  \bm{M}_{1, 3} & |[fill =blue!20]|  \bm{M}_{1, 4} & |[fill =gray!20]|  \bm{M}_{1, 5} & |[fill =gray!20]|  \bm{M}_{1, 6} \\
			|[fill =red!20]|  \bm{M}_{2, 1} & |[fill =red!20]|  \bm{M}_{2, 2} & |[fill =blue!20]|  \bm{M}_{2, 3} & |[fill =blue!20]|  \bm{M}_{2, 4} & |[fill =gray!20]|  \bm{M}_{2, 5} & \,  \\
			|[fill =red!20]|  \bm{M}_{3, 1} & |[fill =red!20]|  \bm{M}_{3, 2} & |[fill =blue!20]|  \bm{M}_{3, 3} & |[fill =blue!20]|  \bm{M}_{3, 4} & \,  & \,  \\
			|[fill =yellow!20]|  \bm{M}_{4, 1} & |[fill =yellow!20]|  \bm{M}_{4, 2} &  |[fill =gray!20]|  \bm{M}_{4, 3} & \,  & \,  & \,  \\
			|[fill =yellow!20]|  \bm{M}_{5, 1} & |[fill =yellow!20]|  \bm{M}_{5, 2} & \,  &  \bm{?} & \,  & \,  \\
			|[fill =gray!20]|  \bm{M}_{6, 1} & \,  & \,  & \,  & \,  & \,  \\
		};
		
		\foreach \n in {1, ..., 6} {
			\node[above] at (m-1-\n.north) {$T_{\n}$};
		}
		
		\foreach \n in {1, ..., 6} {
			\node[left] at (m-\n-1.west) {$N_{\n}$};
		}
		
		\draw[line width =0.5mm] (m-1-1.north west) rectangle (m-5-4.south east);
		\draw[line width =0.5mm] (m-1-1.north west) rectangle (m-3-4.south east);
		\draw[line width =0.5mm] (m-1-1.north west) rectangle (m-3-2.south east);
		\draw[line width =0.5mm] (m-1-1.north west) rectangle (m-5-2.south east);
	\end{tikzpicture}  & 
	\begin{tikzpicture}
		\matrix (m) [matrix of math nodes, 
		nodes ={draw,  minimum height =0.75cm,  minimum width =1.1cm,  anchor =center}, 
		column sep =-\pgflinewidth,  row sep =-\pgflinewidth]{
			|[fill =red!20]|  \bm{M}_{1, 1} & |[fill =blue!20]|  \bm{M}_{1, 2} & |[fill =blue!20]|  \bm{M}_{1, 3} & |[fill =blue!20]|  \bm{M}_{1, 4} & |[fill =blue!20]|  \bm{M}_{1, 5} & |[fill =blue!20]|  \bm{M}_{1, 6} \\
			|[fill =yellow!20]|  \bm{M}_{2, 1} & |[fill =gray!20]|  \bm{M}_{2, 2} & |[fill =gray!20]|  \bm{M}_{2, 3} & |[fill =gray!20]|  \bm{M}_{2, 4} & |[fill =gray!20]|  \bm{M}_{2, 5} & \,  \\
			|[fill =yellow!20]|  \bm{M}_{3, 1} & |[fill =gray!20]|  \bm{M}_{3, 2} & |[fill =gray!20]|  \bm{M}_{3, 3} & |[fill =gray!20]|  \bm{M}_{3, 4} & \,  & \,  \\
			|[fill =yellow!20]|  \bm{M}_{4, 1} & |[fill =gray!20]|  \bm{M}_{4, 2} &  |[fill =gray!20]|  \bm{M}_{4, 3} & \,  & \,  & \,  \\
			|[fill =yellow!20]|  \bm{M}_{5, 1} & |[fill =gray!20]|  \bm{M}_{5, 2} & \,  & \,  & \,  & \,  \\
			|[fill =yellow!20]|  \bm{M}_{6, 1} & \,  & \,  & \,  & \,  &  \bm{?} \\
		};
		
		\foreach \n in {1, ..., 6} {
			\node[above] at (m-1-\n.north) {$T_{\n}$};
		}
		
		\foreach \n in {1, ..., 6} {
			\node[left] at (m-\n-1.west) {$N_{\n}$};
		}
		
		\draw[line width =0.5mm] (m-1-1.north west) rectangle (m-6-6.south east);
		\draw[line width =0.5mm] (m-1-1.north west) rectangle (m-6-1.south east);
		\draw[line width =0.5mm] (m-1-1.north west) rectangle (m-1-6.south east);
	\end{tikzpicture}
	\tabularnewline
	$ \quad \;\;$(a)  & $ \quad \;\;\;$(b)\tabularnewline
\end{tabular}
	
	\caption{Two examples on how to construct a four-block data matrix $ \bm{M}^{(i_0, j_0)}$ to estimate $ \bm{M}^\star_{i_0, j_0}$. We consider $k =6$,  and the left and right panel correspond to $(i_0, j_0) =(5, 4)$ and $(6, 6)$ respectively. We use the question mark to denote the unobserved block that we want to estimate,  and we use the bold line to single out the corresponding four-block design. For each observed block,  we use different colors to distinguish their roles: red,  blue and yellow blocks constitute $ \bm{M}_a$,  $ \bm{M}_b$ and $ \bm{M}_c$ (cf.~(\ref{eq:reduction})),  while gray blocks are those unimportant data that we discard while estimating $ \bm{M}^\star_{i_0, j_0}$.} \label{fig:alg-general}
\end{figure}

We move on to discuss the general staggered adoption design. At a high
level, our procedure is based on reducing to the four-block case.
Concretely, for any block $(i_{0}, j_{0})$ in~\Cref{fig:setup}, we
define an associated four-block estimation problem.  The solution to
this problem---which can be obtained using the algorithm from the
preceding section---allows us to the block $\bm{M}_{i_{0},
  j_{0}}^{\star}$.  We obtain the four-block problem by removing a
subset of data so as to simplify the observation structure;
importantly, our theory to be described in the next section shows that
the left-out data is ultimately not of significant statistical
utility.

Let us now describe how to estimate a block $ \bm{M}_{i_{0},
  j_{0}}^{\star}$, for some index $(i_{0}, j_{0})$ such that
\mbox{$j_{0} > k + 1-i_{0}$}, so that it is not observed. Defining the
indices $k_{1} =k + 1 - j_{0}$ and $k_{2} = k + 1 - i_{0}$, we
construct a four-block data matrix $\bm{M}^{(i_{0}, j_{0})}$ as
follows
\begin{align}
\label{eq:reduction}  
 \bm{M}^{(i_{0}, j_{0})} \mydefn \left[\begin{array}{cc} \bm{M}_{a} &
     \bm{M}_{b} \\
     \bm{M}_{c} & \bm{?}
   \end{array} \right],  \qquad  \text{where} \quad
 \begin{cases}
   \bm{M}_{a} = \left[ \bm{M}_{i, j} \right]_{1 \leq i \leq k_{1}, 1
     \leq j \leq k_{2}}, \\
   \bm{M}_{b} = \left[ \bm{M}_{i, j} \right]_{1 \leq i \leq k_{1},
     k_{2} < j \leq j_{0}}, \\
   \bm{M}_{c} = \left[ \bm{M}_{i, j} \right]_{k_{1} < i \leq i_{0}, 1
     \leq j \leq k_{2}}.
\end{cases}
\end{align}
Next we call~\Cref{alg:4-blocks-Md} with input $\bm{M}^{(i_{0},
  j_{0})}$ to estimate its bottom-right block. Although the output
provides estimation jointly for all blocks $ \bm{M}_{i, j}^{\star}$
where $k_{1} < i \leq i_{0}$ and $k_{2} < j \leq j_{0}$, we only use
it to estimate the block $ \bm{M}_{i_{0}, j_{0}}^{\star}$.  (For other
blocks, this estimate may not be statistically efficient.)
\Cref{fig:alg-general} and its caption provide two examples to help
illustrate how we construct the four-block data matrix for estimating
each unobserved block. The complete estimation procedure is summarized
in~\Cref{alg:general}.

As noted above, we show in the sequel that this algorithm achieves
full statistical efficiency (including the constant pre-factors).
Consequently, the data discarded while estimating each block are
indeed unimportant for estimating that block, due to the lack of
increase in statistical error.

\begin{algorithm}[t]
  \DontPrintSemicolon \SetNoFillComment
  \textbf{Input:} Data matrix $\bm{M}$, Boolean matrix $\bm{\Omega}$,
  rank $r$. \\
  Extract the dimension information $\{N_i\}_{1 \leq i \leq k}$ and
  $\{T_j\}_{1 \leq j \leq k}$ from $ \bm{\Omega}$. \\
  \For{$ i_0 = 1$ \KwTo $k$}{ \For{$j_0 = k + 1 - i_0$ \KwTo $k$}{
      Construct the data matrix $ \bm{M}^{(i_0, j_0)}$ via
      equation~\eqref{eq:reduction}. \\
Call~\Cref{alg:4-blocks-Md} with input $ \bm{M}^{(i_0, j_0)}$ and rank
$r$ to compute an estimate $\widehat{ \bm{M}}_d$ of its bottom left
block. \\
Extract the submatrix $\widehat{ \bm{M}}_{i_0, j_0}$ from $\widehat{
  \bm{M}}_d$ corresponding to $\bm{M}_{i_0, j_0}^\star$. \\ } }
\textbf{Output:} $\widehat{ \bm{M}}_{i_0, j_0}$ as an estimate of
$\bm{M}_{i_0, j_0}^{\star}$ for each $(i_0, j_0)$
\caption{Estimating counterfactual outcomes: staggered adoption
  design \label{alg:general}}
\end{algorithm}




\section{Main results}
\label{sec:main-results}

We are now positioned to present our main theoretical guarantees for
the algorithms described in~\Cref{sec:alg}.  We begin with theory for
the four-block design in~\Cref{sec:theory-four}, before extending it
to the general case in~\Cref{sec:theory-general}.


\subsection{Theory for four-block design}
\label{sec:theory-four}

As noted previously, low-rank matrix recovery is an under-determined
problem, and certain regularity conditions are required to provide
theoretical guarantees.  We begin by stating the conditions used in
our analysis; see the discussion in~\Cref{SecInterpret} for some
interpretation and intuition.

\paragraph{Sub-block conditioning:}
Our first set of conditions involve the top sub-blocks $\Ustar_1
\in\real^{N_1 \times r}$ and $\Vstar_1 \in \real^{T_1 \times r}$ of
$\Ustar$ and $\Vstar$, respectively.  In particular, we assume that
there are constants $0 < \clow \leq \cupper < \infty$ such that
\begin{align}
\label{EqnSubmatrix}
\clow \tfrac{\Nunit_{1}}{N} \bm{I}_{r} \preceq \bm{U}_{1}^{\star\top}
\bm{U}_{1}^{\star} \preceq \cupper \tfrac{\Nunit_{1}}{N} \bm{I}_{r},
\qquad \text{and} \qquad \clow \tfrac{T_{1}}{T} \bm{I}_{r} \preceq
\bm{V}_{1}^{\star \top} \bm{V}_{1}^{\star} \preceq \cupper
\tfrac{T_{1}}{T} \bm{I}_{r}.
\end{align}

\paragraph{Noise level:}
Throughout, so as to avoid degenerate corner cases, we assume that
$\min \{\Nunit_{1}, T_{1}\} \geq c \log (N+ T)$.  In addition, some of
our results require an upper bound on the noise level, as measured by
the ratio $\noise/\sigma^*_r$. In particular, we define the noise
functional
\begin{align}
  \label{eq:noise-level}
  \Enoise \coloneqq \frac{\noise}{\sigma^*_r} \frac{1}{ \sqrt{ \min
      \{\Nunit_1, \Time_1\}}}
\end{align}
For a given target error level $\delta > 0$, we require that $\min
\{\Nunit_1, \Time_1\}$ is sufficiently large so as to ensure that
\begin{align}
  \label{eq:noise-condition-est}
  \Enoise \; \sqrt{r+\log(N+T)} & \leq c_{\mathsf{noise}} \delta
\end{align}
for sufficiently small constant $c_{\mathsf{noise}} > 0$.

\paragraph{Local incoherence:}
For each $i \in [N]$ and $t \in [T]$, we define the local incoherence
parameters
\begin{subequations}
\label{eq:incoherence}
\begin{align}
  \mu_{i} \coloneqq \sqrt{ \tfrac{N}{r} }\left\Vert \bm{U}_{i,
    \cdot}^{\star} \right \Vert _{2} \qquad \text{and} \qquad \nu_{t}
  = \sqrt{\tfrac{T}{r}} \left \Vert \bm{V}_{t,
    \cdot}^{\star}\right\Vert _{2},
\end{align}
along with the noise levels $\EN \coloneqq
\frac{\sigma}{\sigma_{r}^{\star}} \; \sqrt{\frac{1}{ \Nunit_1 }}$ and
$\ET \coloneqq \frac{\sigma}{\sigma_{r}^{\star}} \; \sqrt{\frac{1}{
    \Time_1 }}$.  For any given target error level $\delta > 0$, in
order to provide recovery guarantees for entry $(i, t)$, we require
that
\begin{align}
\label{eq:incoherence-est}  
  \max \left\{ \mu_{i} \ET, \nu_{t} \EN \right\} \sqrt{r} \leq
  \cinc \delta \qquad \text{and} \qquad \min \left\{
  \tfrac{\mu_{i}}{\sqrt{N_1}} , \tfrac{\nu_{t} }{\sqrt{T_1}} \right\}
  \; \sqrt{r \log(N+T)} \leq \cinc \delta
\end{align}  
for some sufficiently small constant $\cinc > 0$.  For
proving proximity to a Gaussian variable (as stated in the main text),
we also require that
\begin{align}
\label{eq:signal-lb}  
\min \left \{ \tfrac{\ET}{\mu_i}, \tfrac{\EN}{\nu_t} \right \} \; \Big
\{ \sqrt{\log(N+T)} + \tfrac{\log(N +T)}{\sqrt{r}} \Big \} & \leq
\cinc \delta.
\end{align} 
\end{subequations}
We refer the reader to~\Cref{lem:inference} for some analysis not
requiring this last condition.


\subsubsection{An achievable result and its consequences}

With these conditions in hand, we are now ready to state an achievable
result for our estimator.  It provides a non-asymptotic bound, for any
unobserved index $(i,t)$, on the rescaled error of our estimator in
terms of a standard Gaussian variable $G_{i,t} \sim \NORMAL(0,1)$.

\begin{theorem}[\textsf{Non-asymptotics and distributional theory}]
  \label{thm:distribution}
Under the sub-matrix condition~\eqref{EqnSubmatrix}, consider some
$\delta > 0$ for which the noise
condition~\eqref{eq:noise-condition-est} and incoherence
conditions~\eqref{eq:incoherence} hold. Then with probability at least
\mbox{$1 - O((N + T)^{-10})$,} \Cref{alg:4-blocks-Md} produces an
estimate $\widehat{\bm{M}}_d$ that for any unobserved index $(i,t)$,
we have
\begin{subequations}
  \begin{align}
\label{EqnMainFourBound}    
 \Big| \tfrac{1}{ \smash[b]{\sqrt{\Varstar_{i,t}}} }
 \big(\widehat{\bm{M}}_{d} - \bm{M}_{d}^{\star}\big)_{i, t} - G_{i, t}
 \Big| & \leq \delta,
  \end{align}
and the variance takes the form
\begin{align}
\label{eq:variance-defn}  
\Varstar_{i, t} & \mydefn \frac{\noise^{2}}{N T} \Big \{ \bm{U}_{i,
  \cdot}^{\star} (\bm{U}_{1}^{\star\top} \bm{U}_{1}^{\star})^{-1}
\bm{U}_{i,\cdot}^{\star \top} + \bm{V}_{t, \cdot}^{\star}
(\bm{V}_{1}^{\star\top} \bm{V}_{1}^{\star})^{-1} \bm{V}_{t,
  \cdot}^{\star\top} \Big \}.
\end{align}
\end{subequations}
\end{theorem}

Let us explore some consequences of~\Cref{thm:distribution}, beginning
with its uses in constructing confidence intervals.  Observe that the
bound~\eqref{EqnMainFourBound} guarantees that the rescaled error is
close to a standard Gaussian variable.  The scaling factor
$\Varstar_{i,t}$, corresponding to the error variance, depends on
unknown population quantities.  However, it can be estimated
accurately by standard plug-in approach, leading to data-driven
construction of entrywise confidence intervals.  In particular, we
define
\begin{align*}
\Varhat_{i,t} & \coloneqq \frac{\widehat{\noise}^{2}}{NT} \Big \{
\bm{U}_{i, \cdot}(\bm{U}_{1}^{\top} \bm{U}_{1})^{-1} \bm{U}_{i,
  \cdot}^{\top} + \bm{V}_{t,\cdot} (\bm{V}_{1}^{\top} \bm{V}_{1})^{-1}
\bm{V}_{t, \cdot}^{\top} \Big \} \quad \text{where} \quad
\widehat{\noise}^{2} = \frac{N}{N_1} \big\Vert \bm{M}_{\mathsf{upper}}
- \widehat{\bm{M}}_{\mathsf{upper}} \big\Vert_{\mathrm{F}}^{2} .
\end{align*}

Here all of the matrices $\bm{U} = \bm{U}_{\mathsf{left}}$, $\bm{V} =
\bm{V}_{\mathsf{left}}$ and $\widehat{\bm{M}}_{\mathsf{upper}} =
\bm{U}_{\mathsf{upper}} \bm{\Sigma}_{\mathsf{upper}}
\bm{V}_{\mathsf{upper}}^\top$ can be computed from
\Cref{alg:4-blocks-Md}.  Letting $\Phi$ denote the CDF of standard
Gaussian distribution, we then define the interval
\begin{subequations}
\begin{align}
\mathsf{CI}_{i, t}^{1-\alpha} \coloneqq \big[ \widehat{M}_{i,t} \pm
  \Phi^{-1} (1-\alpha/2) \sqrt{\Varhat_{i,t}} \big],
\end{align}
With this definition, we have the following guarantee:
\begin{proposition}
  \label{prop:CI}
In addition to the conditions of \Cref{thm:distribution}, suppose that
$\min\{ N_1, T \} \geq \delta^{-2} r \log( N + T)$. Then the interval
$\mathsf{CI}_{i,t}^{1-\alpha}$ has the coverage guarantee
\begin{align}
\label{EqnCI}
  \mathbb{P} \Big( \mathsf{CI}_{i,t}^{1-\alpha} \ni M_{i,t}^\star
  \Big) = 1 - \alpha + O \big( \delta + (N+T)^{-10} \big).
\end{align}
\end{proposition}
\end{subequations}
\noindent See~\Cref{sec:proof-CI} for the proof. \\

Second, while~\Cref{thm:distribution} is stated in a way that
facilitates its inferential use, it also allows us to derive entrywise
error bounds for $\widehat{\bm{M}}_d$. If the conditions
of~\Cref{thm:distribution} hold with some $\delta$ decreasing to zero
as $\min \{N_1, T_1 \}$ grows, then it can be verified that
\begin{subequations}
\begin{align}
\mathbb{E}\big[(\widehat{M}_{i,t} - M_{i,t}^\star)^2 \big] = (1 +
o(1)) \; \Varstar_{i, t}
\lesssim \frac{\noise^2}{N_1 T} \Vert \bm{U}_{i,\cdot}^\star \Vert_2^2
+ \frac{\noise^2}{N T_1} \Vert \bm{V}_{t,\cdot}^\star \Vert_2^2 .
\end{align}
To provide intuition for the convergence rate, suppose that the local
incoherence parameters $\mu_i$ and $\nu_t$ are viewed as quantities of
constant order; in this case, we have
\begin{align}
  \mathbb{E}\big[(\widehat{M}_{i,t} - M_{i,t}^\star)^2 \big] \lesssim
  \frac{\noise^2 r}{NT \min\{ N_1, T_1\}} \asymp \frac{r}{N T} \;
  \Enoise^2.
\end{align}
\end{subequations}
In fact, even without the incoherence
conditions~\eqref{eq:incoherence} and with a relaxation of the noise
condition~\eqref{eq:noise-condition-est}, our analysis still provides
bounds on the estimation error; see~\Cref{lem:master} for details.
%


\subsubsection{Matching lower bound}

Is it possible to improve the variance term~\eqref{eq:variance-defn}
in our achievable result?  In this section, we provide a negative
answer to this question by showing that it matches with a Bayesian
Cram\'er--Rao lower bound.

We decompose the unknown matrix as $\bm{M}^{\star} = \bm{X}^{\star}
\bm{Y}^{ \star \top }$, where $\bm{X}^{\star} \coloneqq
\bm{U}^{\star}( \bm{\Sigma}^{\star} )^{1/2}$ and $\bm{Y}^{\star}
\coloneqq \bm{V}^{\star}( \bm{\Sigma}^{\star} )^{1/2}$ are $N \times
r$ and $T \times r$ matrices, respectively.  Now suppose that there is
a genie that reveals all rows of $\bm{X}^{\star}$ \emph{except} for
the $i$-th one $\xstar \coloneqq \bm{X}_{i, \cdot}^{\star}$, and also
all rows of $\bm{Y}^{\star}$ \emph{except} for its $t$-th one $\ystar
\coloneqq \bm{Y}_{t, \cdot}^{\star}$.  Given this information, the
statistician is left with the following problem:
\begin{enumerate}
\item[(a)] The only unknown parameters are the $r$-dimensional row
  vectors $\bm{x}^\star \coloneqq \bm{X}_{i, \cdot}^{\star}$ and
  $\bm{y}^\star \coloneqq \bm{Y}_{t, \cdot}^{\star}$, and the goal is
  to estimate their inner product $(\bm{M}_{d}^{\star})_{i, t} \equiv
  \inprod{\bm{x}^\star}{\bm{y}^{\star}}$.
\item[(b)] The remaining observations of any use are the
  $\Nunit_{1}+T_{1}$ linear measurements
\begin{align}
  \label{eq:lb-samples}
\big\{ \bm{M}_{i, s} = \inprod{\xstar}{\bm{Y}_{s, \cdot}^{\star}} +
E_{i, s} \big\}_{s=1}^{T_1} \qquad \text{and} \quad \big\{ \bm{M}_{k,
  t} = \inprod{\bm{X}_{k, \cdot}^{\star}}{\ystar} + E_{k, t}
\big\}_{k=1}^{\Nunit_1}.
\end{align}
All other observations are uninformative given the information
revealed by the genie.
\end{enumerate}
Thus, with the aid from the genie, the causal panel data model reduces
to a linear regression model with dimension $2 r$.

We use this genie-aided problem to compute lower bounds for the
estimation task---both the Cram\'er-Rao lower bound as well as a local
minimax risk.  The local minimax risk at scale $\varepsilon > 0$ is
given by
\begin{align}
\label{eq:local-minimax-risk}
R_{\mathsf{loc}} (\varepsilon) \coloneqq \inf_{\widehat{M}_{i,t}} \;
\; \sup_{ \substack{\bm{x} \in \ball_\infty (\bm{x}^\star ,
    \varepsilon ) \\ \bm{y} \in \ball_\infty (\bm{y}^\star ,
    \varepsilon )} } \mathbb{E} \left[ \big (\widehat{M}_{i,t} -
  \inprod{\bm{x}}{\bm{y}} \big)^2 \right].
\end{align}
Here the infimum is taken over all estimators of the scalar
$\inprod{\bm{x}}{\bm{y}}$ based on samples from the
model~\eqref{eq:lb-samples}, except that the true parameters are
$\bm{x}$ and $\bm{y}$ instead of $\bm{x}^\star$ and $\bm{y}^\star$.
\begin{theorem}[\textsf{Optimality}]
  \label{thm:CRLB}
In the genie-aided setting, the Cram\'er-Rao lower bound for
estimating $(\bm{M}_d^\star)_{i, t}$ is given by
$\gamma_{i,t}^\star$. More generally, under the sub-block
condition~\eqref{EqnSubmatrix}, the local minimax risk at scale
$\varepsilon_{N,T} = \frac{\sqrt{\sigma_r^\star}}{\sqrt{ \min \{N, T
    \} }}$ is lower bounded by
\begin{align}
\label{EqnLocalMinimax}    
R_{\mathsf{loc}} ( \varepsilon_{N,T} ) \geq \Big( 1 -
\frac{\cupper \pi^2}{\clow^2} \Enoise^2 \Big ) \; \gamma_{i,t}^\star,
\end{align}
where $\Enoise$ is the noise level previously
defined~\eqref{eq:noise-level}.
\end{theorem}
\noindent See~\Cref{sec:proof-thm-crlb} for the proof. \\

Under the noise assumption~\eqref{eq:noise-condition-est}, we have
$\Enoise^2 \leq \frac{c^2_{\mathsf{noise}} \delta^2}{r + \log(N +
  T)}$, whence $\tfrac{R_{\mathsf{loc}}
  (\varepsilon_{N,T})}{\gamma_{i,t}^\star} \geq 1 - O(\delta^2)$ as
$\delta$ goes to zero.  Thus, we see that our our estimator
(cf.~\Cref{alg:4-blocks-Md}) matches the local minimax risk aided by
the genie, including a sharp pre-factor. Thus, our procedure is
optimal in a strong instance-dependent sense.



\subsubsection{Interpretation of structural conditions}
\label{SecInterpret}

Let us interpret the structural conditions---namely, the sub-block and
local incoherence conditions---that underlie~\Cref{thm:distribution}.

\paragraph{Sub-block condition:}
Since $\bm{M}^{\star}_a = \Ustar_1 \SigStar \Vstar_1$, the existence
of the constant $\clow > 0$ ensures that $\MatTre{a}$ has full rank.
This condition is necessary, because otherwise it would be impossible
to recover the missing block $\MatTre{d}$ even in the noiseless
setting.\footnote{For example, if the row rank of $\MatTre{a}$ is less
than $r$, then we may find a row of $\MatTre{b}$ that is not in the
subspace spanned by the rows of $\MatTre{a}$. In this case, the
corresponding row of $\MatTre{d}$ is unidentifiable.}.  Otherwise, we
observe that $\Ustar_1$ is an $N_1 \times T_1$ sub-matrix of the
orthonormal matrix $\Ustar \in \real^{N \times T}$.  Since $\Ustar$
has orthonormal columns, we expect that the typical size of each entry
is $1/\sqrt{N}$.  Thus, the typical squared norm of each column of the
sub-matrix $\Ustar_1$ is of the order $N_1/N$, which accounts for the
normalization in the condition~\eqref{EqnSubmatrix} on $\Ustar_1$.
Similar comments apply to the condition on $\Vstar_1$.

\paragraph{Local incoherence conditions:}
Since the Frobenius norm of $\bm{U}^{\star}$ and $\bm{V}^{\star}$ are
both $\sqrt{r}$, the ``typical size'' of their rows (in $\ell_{2}$
norm) are $\sqrt{r/N}$ and $\sqrt{r/T}$ respectively. We can see that
$\mu_{i}$ (resp.~$\nu_{i}$) evaluates the ratio between $\Vert
\bm{U}_{i, \cdot}^{\star} \Vert_{2}$ (resp.~$\Vert \bm{V}_{t,
  \cdot}^{\star} \Vert_{2}$) and its typical size. This definition
connects to the (global) incoherence condition that is widely assumed
in the low-rank matrix completion literature~\cite{ExactMC09,
  chen2015incoherence, chi2018nonconvex}, where the (global)
incoherence parameters $\mu \coloneqq\max_{1 \leq i \leq N} \mu_{i}$
and $\nu \coloneqq \max_{1\leq t\leq T}\nu_{t}$ defined therein cannot
be too large in order for the recovery to be possible. In contrast,
our theory does not assume any global incoherence condition to hold.
If our goal is to estimate and infer the counterfactual outcome or
treatment effect of unit $i$ at time $t$, we just need the local
incoherence parameters $\mu_{i}$ and $\nu_{t}$ to be reasonably
well-behaved.

Our constraints involving the local incoherence
conditions~\eqref{eq:incoherence-est} are reasonably mild. When
$\min\{\Nunit_{1}, T_{1}\}\gg r$ and the noise condition
(\ref{eq:noise-condition-est}) is satisfied, we can allow both
$\Vert\bm{U}_{i, \cdot}^{\star}\Vert_{2}$ and $\Vert \bm{V} _ {t ,
  \cdot}^{\star} \Vert_{2}$ to be much larger than their typical
size. The other condition \eqref{eq:signal-lb} requires one of
$\Vert\bm{U}_{i, \cdot}^{\star}\Vert_{2}$ and $\Vert\bm{V}_{t,
  \cdot}^{\star}\Vert_{2}$ to be not vanishingly small compared to
their typical size. In fact, even when (\ref{eq:signal-lb}) does not
hold, we can still characterize the entrywise distribution of
$\widehat{\bm{M}}_{d}-\bm{M}_{d}^{\star}$, although it is not
approximately Gaussian; we refer interested readers to
\Cref{lem:inference} for this general result.


\subsection{Theory for staggered adoption design}
\label{sec:theory-general}

In this section, we extend our theoretical analysis to the general
(non-four-block) setting to which \Cref{alg:general} applies.  Without
loss of generality, we focus on estimating and inferring entries of an
unobserved submatrix $\bm{M}_{i_{0}, j_{0}}^{\star}$. In
\Cref{sec:alg}, we described how to construct a data matrix
$\bm{M}^{(i_{0}, j_{0})}$ for estimating $\bm{M}_{i_{0},
  j_{0}}^{\star}$ that admits the four-block design; see
equation~\eqref{eq:reduction} for details.  \Cref{fig:theory} provides
a simple example so as to help visualize the set-up and to facilitate
later discussion.

At the left of~\Cref{fig:theory}, we depict a panel data problem with
units separated into $k = 6$ subsets, as indicated by the large
blocks.  The $i$-th subset has $N_i$ units that all received the
treatment at time $1+\sum_{j=1}^{i} T_j$; recall that a time larger
than $T$ means that the unit never received the treatment. All the
colored blocks are observed, while the uncolored blocks are
missing. Suppose we want to estimate the sub-matrix
$\bm{M}_{i_0,j_0}^\star$ indexed by the pair $(i_0, j_0)=(5,4)$, as
indicated by a the question mark in the figure. The data matrix
$\bm{M}_{(i_{0}, j_{0})}^\star$ consists of the $i_0=5$ by $j_0=4$
block matrix on the top left corner of $\bm{M}$. The gray blocks,
although observed, are discarded to enforce the four-block
structure. We use three colors (red, blue and yellow) to show the
sub-matrices that play the role of $\bm{M}_a^\star$, $\bm{M}_b^\star$
and $\bm{M}_c^\star$ in the four-block model. The red component
consists of the $k_1=3$ by $k_2=2$ block matrix on the top left
corner; in general, these two integers satisfy the relations $k_1 = k
+ 1 - j_0$ and $k_2 = k + 1 - i_0$.

Now we define the dimension parameters $\bar{N}_{1}$, $\bar{N}_{2}$,
$\bar{T}_{1}$ and $\bar{T}_{2}$ with respect to the four-block
structure, given by
\begin{align*}
\bar{N}_{1} \coloneqq \sum_{i=1}^{k_{1}}N_{i}, \quad\bar{N}_{2}
\coloneqq \sum_{i=k_{1}+1}^{i_{0}}N_{i}, \quad\bar{T}_{1} \coloneqq
\sum_{j=1}^{k_{2}} T_{j}, \quad \mbox{and} \quad \bar{T}_{2} \coloneqq
\sum_{j=k_{2}+1}^{j_{0}} T_{j},
\end{align*}
We also define $\bar{N} \coloneqq \bar{N}_{1}+\bar{N}_{2}$ and
$\bar{T} \coloneqq \bar{T}_{1}+\bar{T}_{2}$.  Let $\bm{U}_{1}^{\star}
\in \mathbb{R}^{\bar{N}_{1}\times r}$ and $\bm{U}_{2}^{\star} \in
\mathbb{R}^{\bar{N}_{2}\times r}$ be submatrices of $\bm{U}^{\star}
\in \mathbb{R}^{N\times r}$, consisting of its first $\bar{N}_{1}$
rows and the next $\bar{N}_{2}$ rows respectively. Similarly, let
$\bm{V}_{1}^{\star} \in \mathbb{R}^{\bar{T}_{1}\times r}$ and
$\bm{V}_{2}^{\star} \in \mathbb{R}^{\bar{T}_{2}\times r}$ be
submatrices of $\bm{V}^{\star} \in \mathbb{R}^{T\times r}$, consisting
of its first $\bar{T}_{1}$ rows and the next $\bar{T}_{2}$ rows
respectively.  These dimension parameters are also illustrated
in~\Cref{fig:theory}.

\begin{figure}[h]
  \centering 	\begin{tabular}{ccc}
		\begin{tikzpicture}
			\matrix (m) [matrix of math nodes, 
			nodes={draw,  minimum height=0.75cm,  minimum width=0.9cm,  anchor=center,  font=\small}, 
			column sep=-\pgflinewidth,  row sep=-\pgflinewidth]{
				|[fill=red!20]| \,  & |[fill=red!20]| \,  & |[fill=blue!20]| \,  & |[fill=blue!20]| \,  & |[fill=gray!20]| \,  & |[fill=gray!20]| \,  \\
				|[fill=red!20]| \,  & |[fill=red!20]| \,  & |[fill=blue!20]| \,  & |[fill=blue!20]| \,  & |[fill=gray!20]| \,  & \,  \\
				|[fill=red!20]| \,  & |[fill=red!20]| \,  & |[fill=blue!20]| \,  & |[fill=blue!20]| \,  & \,  & \,  \\
				|[fill=yellow!20]| \,  & |[fill=yellow!20]| \,  &  |[fill=gray!20]| \,  & \,  & \,  & \,  \\
				|[fill=yellow!20]| \,  & |[fill=yellow!20]| \,  & \,  & \bm{?} & \,  & \,  \\
				|[fill=gray!20]| \,  & \,  & \,  & \,  & \,  & \,  \\
			};
			
			\draw[decorate, decoration={brace, amplitude=6pt, raise=1pt}]
			(m-1-1.north west) -- (m-1-2.north east) node [black, midway, yshift=13pt] {$\bar{T}_1$};
			\draw[decorate, decoration={brace, amplitude=6pt, raise=1pt}]
			(m-1-3.north west) -- (m-1-4.north east) node [black, midway, yshift=13pt] {$\bar{T}_2$};
			\draw[decorate, decoration={brace, amplitude=6pt, raise=1pt, mirror}]
			(m-6-1.south west) -- (m-6-4.south east) node [black, midway, yshift=-13pt] {$\bar{T}$};
			
			\draw[decorate, decoration={brace, amplitude=6pt, raise=1pt, mirror}]
			(m-1-1.north west) -- (m-3-1.south west) node [black, midway, xshift=-15pt, yshift=-8pt, anchor=south] {$\bar{N}_1$};
			\draw[decorate, decoration={brace, amplitude=6pt, raise=1pt, mirror}]
			(m-4-1.north west) -- (m-5-1.south west) node [black, midway, xshift=-15pt, yshift=-8pt, anchor=south] {$\bar{N}_2$};
			\draw[decorate, decoration={brace, amplitude=6pt, raise=1pt}]
			(m-1-6.north east) -- (m-5-6.south east) node [black, midway, xshift=13pt, yshift=-8pt, anchor=south] {$\bar{N}$};
			
			\draw[line width=0.5mm] (m-1-1.north west) rectangle (m-5-4.south east);
			\draw[line width=0.5mm] (m-1-1.north west) rectangle (m-3-4.south east);
			\draw[line width=0.5mm] (m-1-1.north west) rectangle (m-3-2.south east);
			\draw[line width=0.5mm] (m-1-1.north west) rectangle (m-5-2.south east);
		\end{tikzpicture}  
		& 
		\begin{adjustbox}{valign=t, raise=56mm}
		\begin{tikzpicture}
			\matrix (m) [matrix of math nodes, 
			nodes={draw,  minimum width=0.7cm,  anchor=center,  font=\small}, 
			column sep=-\pgflinewidth,  row sep=-\pgflinewidth]{
				|[minimum height=3*0.75cm,  fill=red!20]| \bm{U}_{1}^\star  \\
				|[minimum height=2*0.75cm,  fill=yellow!20]| \bm{U}_{2}^\star  \\
				|[minimum height=0.75cm,  fill=gray!20]| \,   \\
			};
			
			\draw[decorate, decoration={brace, amplitude=6pt, raise=1pt}]
			(m-1-1.north west) -- (m-1-1.north east) node [black, midway, yshift=13pt] {$r$};
			
			\draw[decorate, decoration={brace, amplitude=6pt, raise=1pt, mirror}]
			(m-1-1.north west) -- (m-1-1.south west) node [black, midway, xshift=-15pt, yshift=-8pt, anchor=south] {$\bar{N}_1$};
			\draw[decorate, decoration={brace, amplitude=6pt, raise=1pt, mirror}]
			(m-2-1.north west) -- (m-2-1.south west) node [black, midway, xshift=-15pt, yshift=-8pt, anchor=south] {$\bar{N}_2$};
			\draw[decorate, decoration={brace, amplitude=6pt, raise=1pt}]
			(m-1-1.north east) -- (m-2-1.south east) node [black, midway, xshift=13pt, yshift=-8pt, anchor=south] {$\bar{N}$};
			
			\draw[line width=0.5mm] (m-1-1.north west) rectangle (m-1-1.south east);
			\draw[line width=0.5mm] (m-1-1.north west) rectangle (m-2-1.south east);
			
		\end{tikzpicture} 
		\end{adjustbox}
		& 
		\begin{adjustbox}{valign=t, raise=37mm}
			\begin{tikzpicture}
				\matrix (m) [matrix of math nodes, 
				nodes={draw,  minimum height=0.7cm,  anchor=center,  font=\small}, 
				column sep=-\pgflinewidth,  row sep=-\pgflinewidth]{
					|[minimum width=2*0.9cm,  fill=red!20]| \bm{V}_{1}^{\star\top}  &
					|[minimum width=2*0.9cm,  fill=blue!20]| \bm{V}_{2}^{\star\top}  &
					|[minimum width=2*0.9cm,  fill=gray!20]| \,   \\
				};
				
				\draw[decorate, decoration={brace, amplitude=6pt, raise=1pt, mirror}]
				(m-1-1.north west) -- (m-1-1.south west) node [black, midway, xshift=-11pt, yshift=-6pt, anchor=south] {$r$};
				
				\draw[decorate, decoration={brace, amplitude=6pt, raise=1pt}]
				(m-1-1.north west) -- (m-1-1.north east) node [black, midway, yshift=13pt] {$\bar{T}_1$};
				\draw[decorate, decoration={brace, amplitude=6pt, raise=1pt}]
				(m-1-2.north west) -- (m-1-2.north east) node [black, midway, yshift=13pt] {$\bar{T}_2$};
				\draw[decorate, decoration={brace, amplitude=6pt, raise=1pt, mirror}]
				(m-1-1.south west) -- (m-1-2.south east) node [black, midway, yshift=-13pt] {$\bar{T}$};
				
				\draw[line width=0.5mm] (m-1-1.north west) rectangle (m-1-1.south east);
				\draw[line width=0.5mm] (m-1-1.north west) rectangle (m-1-2.south east);
			\end{tikzpicture}
		\end{adjustbox}
		\tabularnewline
		$\;\;$  $\bm{M}$  & $\;\;$ $\bm{U}^\star$ & $\quad\quad$ $\bm{V}^{\star\top}$ \tabularnewline
	\end{tabular}
  \caption{An illustration of the dimension parameters and the
    partition used in~\Cref{sec:theory-general}. }
  \label{fig:theory}
\end{figure}

In analogy with the four-block design, we first introduce the
conditions required for establishing theoretical guarantees in the
general staggered adoption design.

\paragraph{Sub-block conditioning:}
We assume that there are constants $0 < \clow \leq \cupper < \infty$
such that
\begin{subequations}
\label{EqnSubmatrix-general}
\begin{align}
\clow \frac{\bar{N}_{1}}{N} \bm{I}_{r} \preceq \bm{U}_{1}^{\star\top}
\bm{U}_{1}^{\star} \preceq \cupper \frac{\bar{N}_{1}}{N} \bm{I}_{r}, &
\qquad \clow \frac{\bar{N}}{N} \bm{I}_{r} \preceq
\bm{U}_{1}^{\star\top} \bm{U}_{1}^{\star} + \bm{U}_{2}^{\star\top}
\bm{U}_{2}^{\star} \preceq \cupper \frac{\bar{N}}{N} \bm{I}_{r}, \\
\clow \frac{\bar{T}_{1}}{T} \bm{I}_{r} \preceq \bm{V}_{1}^{\star\top}
\bm{V}_{1}^{\star} \preceq \cupper \frac{\bar{T}_{1}}{T} \bm{I}_{r}, &
\qquad \clow \frac{\bar{T}}{T} \bm{I}_{r} \preceq
\bm{V}_{1}^{\star\top} \bm{V}_{1}^{\star} + \bm{V}_{2}^{\star\top}
\bm{V}_{2}^{\star} \preceq \cupper \frac{\bar{T}}{T} \bm{I}_{r}.
\end{align}	
\end{subequations}
Recall the sub-block condition \eqref{EqnSubmatrix} for the four-block
design and the discussion in \Cref{SecInterpret}, here we are assuming
in addition that $\bm{U}_{1}^{\star\top} \bm{U}_{1}^{\star} +
\bm{U}_{2}^{\star\top} \bm{U}_{2}^{\star}$ and $\bm{V}_{1}^{\star\top}
\bm{V}_{1}^{\star} + \bm{V}_{2}^{\star\top} \bm{V}_{2}^{\star}$ do not
deviate too much from their typical size. This additional requirement
connects the spectrum of the ground truth of $\bm{M}^{(i_{0}, j_{0})}$
(which is a submatrix of $\bm{M}^{\star}$) with that of the full
ground truth $\bm{M}^{\star}$, which allows us to present cleaner
results that depend on the spectrum of $\bm{M}^{\star}$ instead of its
submatrices.

\paragraph{Noise level:}
The noise level in this general setting takes the following form:
\begin{align}
  \label{eq:noise-level-general}
  \Ebar \coloneqq \max \{ \ENbar , \ETbar \}, \qquad \text{where}
  \qquad \ENbar \coloneqq \frac{\noise}{\sigma_{r}^{\star}} \;
  \sqrt{\frac{1}{ \bar{\Nunit}_1 }} \qquad \text{and} \qquad \ETbar
  \coloneqq \frac{\noise}{\sigma_{r}^{\star}} \; \sqrt{\frac{1}{
      \bar{\Time}_1 }}.
\end{align}
For any prescribed target error level $\delta > 0$, we assume that it
is upper bounded by
\begin{align}
	\label{eq:noise-condition-est-general}    
	\Ebar \sqrt{r+\log(N+T)} \leq c_{\mathsf{noise}} \delta ,
\end{align}
for some sufficiently small constant $c_{\mathsf{noise}} > 0$.

\paragraph{Local incoherence:}
For each $i \in [N]$ and $t \in [T]$, the local incoherence parameters
$\mu_i$ and $\nu_t$ are defined as in equation~\eqref{eq:incoherence}.
For any given target error level $\delta>0$, we assume that there is
some sufficiently small constant $\cinc>0$ such that,
\begin{subequations}
\label{eq:incoherence-general}
\begin{align}
  \label{eq:incoherence-est-general}  
  \max \left\{ \mu_{i} \ETbar, \nu_{t} \ENbar \right\} \sqrt{r} \leq
  \cinc \delta \qquad \text{and} \qquad \min \left\{
  \tfrac{\mu_{i}}{\sqrt{\bar{N}_1}} , \tfrac{\nu_{t}
  }{\sqrt{\bar{T}_1}} \right\} \; \sqrt{r \log(N+T)} \leq
  \cinc \delta
\end{align}
and that at least one of the following conditions holds
\begin{align}
  \label{eq:signal-lb-1}  
  \min \left \{ \tfrac{\ETbar}{\mu_i}, \tfrac{\ENbar}{\nu_t} \right \}
  \; \Big \{ \sqrt{\log(N+T)} + \tfrac{\log(N +T)}{\sqrt{r}} \Big \} &
  \leq \cinc \delta.
\end{align} 
\end{subequations}

With this set-up, we let us describe how our theory applies to this
more general setting, in particular by showing proximity of the
rescaled error to a standard normal $G_{i,t} \sim \NORMAL(0,1)$.
\begin{theorem}
[\textsf{Non-asymptotics and distribution: General setting}]
\label{thm:distribution-general}
For a given $\delta > 0$, suppose that the noise
condition~\eqref{eq:noise-condition-est-general} and the incoherence
conditions~\eqref{eq:incoherence-general} hold. For any index $(i,t)$
belonging to any unobserved submatrix $\bm{M}_{i_0,j_0}^\star$, with
probability at least \mbox{$1 - O((N + T)^{-10})$,} \Cref{alg:general}
produces an estimate $\widehat{\bm{M}}_{i_0,j_0}$ such that
\begin{subequations}
\begin{align}
\label{EqnMainStaggeredBound}    
\Big| \tfrac{1}{ \smash[b]{\sqrt{\Varstar_{i,t}}} }
\big(\widehat{\bm{M}}_{i_0,j_0} - \bm{M}_{i_0,j_0}^{\star}\big)_{i, t}
- G_{i, t} \Big| & \leq \delta,
\end{align}
where the variance is given by
\begin{align}
  \label{eq:variance-defn-general}  
  \Varstar_{i, t} & \mydefn \frac{\noise^{2}}{NT} \Big \{ \bm{U}_{i,
    \cdot}^{\star} (\bm{U}_{1}^{\star\top} \bm{U}_{1}^{\star})^{-1}
  \bm{U}_{i}^{\star\top} + \bm{V}_{t}^{\star} (\bm{V}_{1}^{\star\top}
  \bm{V}_{1}^{\star})^{-1} \bm{V}_{t, \cdot}^{\star\top} \Big \}.
\end{align}
\end{subequations}
\end{theorem} 
\noindent 
See~\Cref{sec:proof-general-distribution} for the proof. \\

Observe that this result is the natural generalization of our previous
achievable result (\Cref{thm:distribution}) for the simpler four-block
setting.  This correspondence makes sense, since our algorithm for the
general case involves ``reducing'' the problem to instances of the
simpler four-block case.  Of course, one might wonder whether or not
this reduction---which involves discarding some data that might be
relevant---leads to a sub-optimal guarantee.

Interestingly, our lower bound---establishing the optimality of the
variance $\gamma^\star_{i,t}$, as stated in~\Cref{thm:CRLB}---also
extends to this multi-block setting.  (We do not state this lower
bound formally here, but note that our proof is actually given for the
multi-block setting; see~\Cref{sec:proof-thm-crlb} for details.)  As a
consequence, under the conditions of~\Cref{thm:distribution-general},
there is \emph{no appreciable loss} of accuracy in solving the problem
by successive reduction to the four-block setting.


\subsection{Comparison with previous work}
\label{sec:comparison}

Having given precise statements of our results, let us now provide a
detailed comparison with related results established in past
work~\cite{athey2021matrix,bai2021matrix,agarwal2023causal,choi2023matrix}.
To simplify discussion, we ignore any log factors in making these
comparisons, and use $\lesssim$ to reflect inequalities that ignore
constants and such log factors.  For making comparisons, we record
here the estimation error bound that can can be deduced
from~\Cref{thm:distribution-general}: for any entry $(i,t)$, the error
is upper bounded as
\begin{align}
  \label{eq:our-bound}
  \big | \widehat{M}_{i , t} - M_{i , t}^\star \big | \lesssim \noise
  \sqrt{ \frac{1}{\bar{N}_1 T} } \Vert \bm{U}_{i, \cdot}^\star \Vert_2
  + \noise \sqrt{ \frac{1}{N \bar{T}_1} } \Vert \bm{V}_{t,
    \cdot}^\star \Vert_2
\end{align}
with high probability.
\begin{itemize}
\item Athey et al.~\citep{athey2021matrix} proposed to use
  convex relaxation (nuclear norm minimization) to estimate
  the missing potential outcomes. They showed that with high
  probability
  \begin{equation}
    \label{eq:athey-bound}
    \frac{1}{ \sqrt{ N T } } \big \Vert \widehat{ \bm{M}
    }^{\mathsf{ABD}} - \bm{M}^\star \big
    \Vert_{\mathrm{F}} \lesssim \frac{ N^{1/4} \Vert
      \bm{M}^\star \Vert_\infty }{ \sqrt{N_1} } + 
    \frac{ \noise }{ N_1 T } \sqrt{\frac{ r }{ \min \{ N, T \} }
    } .
  \end{equation}
This is a Frobenius norm bound, which measures the average error over
the full matrix and does not provide entrywise guarantees like
(\ref{eq:our-bound}). In addition, their bound (\ref{eq:athey-bound})
does not vanish as the noise level $\sigma$ goes to zero, so that even
when there is no noise, it does not ensure exact recovery.
\item In recent work, Choi and Yuan~\citep{choi2023matrix} built upon
  the paper~\cite{athey2021matrix}, providing an algorithm with
  superior estimation guarantees.  In particular, they suggested a
  more sophisticated and multi-round use of convex relaxation, based
  on dividing the missing entries into smaller groups, and then
  solving a nuclear-norm regularized problem for each of these groups.
  They showed that with high probability, their estimator
  $\widehat{\bm{M}}^{\mathsf{CY}}$ satisfies the entrywise
  $\ell_{\infty}$-bound
\begin{align}
  \label{eq:choi-bound}
  \big \Vert \widehat{ \bm{M} }^{\mathsf{CY}} - \bm{M}^\star \big
  \Vert_{\infty} \lesssim \frac{\noise}{\sqrt{NT}} \sqrt{\frac{ \kappa^5 \mu^2 r^3 }{
      \min \{ \bar{N}_1, \bar{T}_1 \} } }.
\end{align}
Here $\kappa=\sigma_1^\star/\sigma_r^\star$ is the condition number of
$\bm{M}^\star$, and $\mu$ is the \emph{global} incoherence parameter
(i.e., the maximum over all the local parameters $\mu_i$ and $\nu_t$
defined in equation~\eqref{eq:incoherence}).  In comparison, our
bound~\eqref{eq:our-bound}) is a localized improvement of their
bound~\eqref{eq:choi-bound}; in particular, it has better dependency
on the rank $r$, and no dependence on the condition number.  In
addition, we also provide data-driven confidence intervals for each
entry of $\bm{M}^\star$ with optimal width.  It is also worth
mentioning that our spectral algorithm is more computationally
efficient than the SDP-based algorithm in \cite{choi2023matrix}.
	
\item Agarwal et al.~\citep{agarwal2023causal} proposed a procedure
  called ``synthetic nearest neighbors'' (SNN), for which they proved
  entrywise error bound and established asymptotic normality. Their
  algorithm works for a wider range of missing mechanism and is more
  general than our algorithm. However, their error rate is not optimal
  under the staggered adoption design. For example, if $N=T$ and there
  is only one missing entry, their estimation error scales as
  $N^{-1/4}$, which is slower than the $N^{-1/2}$ rate for our
  algorithm and convex relaxation.
	
\item Another line of work~\citep{bai2021matrix,cahan2023factor} is
  based on assuming a factor structure for the full panel data matrix,
  and applying spectral algorithms that share similar spirit with
  ours. In general, the analyses in these papers impose stronger
  assumptions than those used in our analysis.  First, they model the
  factors and factor loadings---roughly speaking, these correspond to
  the rows of $\bm{U}^\star$ and $\bm{V}^\star$--- as being random,
  and required to satisfy certain asymptotic moment conditions.
  Second, they require that the observation sizes satisfy the
  constraint $\min\{ \bar{N}_1, \bar{T}_1 \} \gg \sqrt{ \max \{ N, T
    \}}$, as well as certain eigengap conditions.  Neither of these
  conditions are imposed in our analysis.  Third, their underlying
  noise condition (used in the asymptotic analysis) is
  $\sqrt{\min\{\bar{N}_1, \bar{T}_1\}}$ times more stringent than our
  noise condition~\eqref{eq:noise-condition-est}. On the other hand,
  we note that their set-up is general enough to handle weakly
  correlated noise, which is not covered by our framework.
\end{itemize}


\section{Numerical experiments}
\label{sec:numerical}

In this section, we report the results of some numerical experiments
to corroborate the computational simplicity of our algorithm, and to
study the sharpness of our theoretical predictions.  For certain
random ensembles, we draw matrices $\bm{U} \in \mathbb{R}^{N\times r}$
from the Haar distribution over the Grassmann manifold; for
simplicity, we say that any such $\bm{U}$ is a random orthonormal
matrix.

\subsection{Scaling with incoherence}

All experiments in this section were based on matrices with dimensions
$N = T = 500$, rank $r = 3$, and using the staggered adoption design
with $k = 3$ groups with dimension parameters $N_{1} = T_{1} = 200$,
$N_2 = T_2 = 200$ and $N_3 = T_3 = 100$.  We used the fixed noise
level $\omega = 1$.

Our goal in this set of experiments was to examine the necessity and
role of incoherence
conditions~\eqref{eq:incoherence-general}. Focusing on estimation of
the entry $(i, t) =(500, 500)$ for concreteness, we generated random
problems of the following type.  For an incoherence parameter $\mu \in
(0, \sqrt{N/r})$, we generated an orthonormal matrix as
\begin{subequations}
\label{eq:ensemble}
\begin{align}
\label{eq:U-mu}
\bm{U}^\star (\mu) = \left[\begin{matrix} \sqrt{1 - \mu^2 \frac{r}{N}}
    \, \bm{U}^{-} \\ \mu \sqrt{ \frac{r}{N}} \, \bm{O}_1
\end{matrix}\right]
\end{align}
where $\bm{U}^{-} \in \mathbb{R}^{(N-r) \times r}$ and $\bm{O}_1 \in
\mathbb{R}^{r\times r}$ are random orthonormal matrices. For any $\mu
\in (0,\sqrt{N/r})$, it is straightforward to check that $\bm{U}^\star
(\mu) \in \mathbb{R}^{N\times r}$ is an orthonormal matrix, and it has
associated local incoherence parameter $\mu_i \equiv \sqrt{N/r} \Vert
\bm{e}_i^\top \bm{U}^\star (\mu) \Vert_2 = \mu$.  Similarly, for any
$\nu \in (0,\sqrt{T/r})$, we can construct an orthonormal matrix
$\bm{V}^\star (\nu)$ with a given $\nu_t = \nu$ by
\begin{align}
\label{eq:V-nu}
\bm{V}^\star(\nu) = \left[\begin{matrix} \sqrt{1 - \nu^2 \frac{r}{T}}
    \, \bm{V}^{-} \\ \nu \sqrt{ \frac{r}{T}} \, \bm{O}_2
\end{matrix}\right]
\end{align}
\end{subequations}
where $\bm{V}^{-} \in \mathbb{R}^{(T-r) \times r}$ and $\bm{O}_2 \in
\mathbb{R}^{r\times r}$ are random orthonormal matrices. We consider
two cases:
\begin{itemize}
\item {\bf Case 1:} We generate $\bm{M}^{\star}=\bm{U}^{\star} (\mu)
  \bm{V}^{\star\top}$ where $\bm{U}^\star(\mu) \in \mathbb{R}^{N\times
    r}$ is defined in equation~\eqref{eq:U-mu} with
  $\mu$ varying between $10^{-4}$ and $10^2$, while $\bm{V}^\star \in
  \mathbb{R}^{T\times r}$ is a random orthonormal matrix.
\item {\bf Case 2:} We generate $\bm{M}^{\star}=\bm{U}^{\star} (\mu)
  \bm{V}^{\star}(\nu)^\top$ where $\bm{U}^\star (\mu) \in
  \mathbb{R}^{N\times r}$ and $\bm{V}^\star (\nu) \in \mathbb{R}^{T
    \times r}$ are defined in equations~\eqref{eq:U-mu}
  and~\eqref{eq:V-nu}, respectively, with $\mu=\nu$ varying between
    $10^{-4}$ and $10^2$.
\end{itemize}

\begin{figure}[h]
	\begin{center}
		\begin{tabular}{ccc}
			\widgraph{0.45\textwidth}{\figdir/fig_incoherence_case1} &&
			\widgraph{0.45\textwidth}{\figdir/fig_incoherence_case2} \\
			(a) && (b)
		\end{tabular}
		\caption{Red circles: Mean-squared error (MSE) of the estimate
			$\widehat{M}_{i,t}$ of entry $(i,t) = (500, 500)$ returned by
			Algorithm~\ref{alg:general} versus the the incoherence parameters.
			Blue lines: theoretically predicted scaling $\gamma_{i,t}^\star$
			of the MSE versus incoherence.  Empirical MSE was obtained as an
			average of $100$ Monte Carlo trials.  (a) Results for the matrix
			ensemble in Case 1.  We see good agreement between simulation and
			theory across the full range of incoherence. (b) Results for Case
			2.  Simulation/theory agreement is good for incoherence parameters
			above $10^{-2}$.  Theory breaks down below this level because
			higher-order terms start to dominate the MSE for a fixed matrix
			dimension.}
		\label{fig:incoherence}
	\end{center}
\end{figure}

For each case, we performed a total of $100$ Monte Carlo trials, each
time computing our estimate and forming a Monte Carlo approximation of
the mean-squared error
$\mathbb{E}[(\widehat{M}_{i,t}-M_{i,t}^\star)^2]$ associated with
estimating entry $(i,t)$ of $M_{i,t}^\star$.  Panels (a) and (b),
respectively, of~\Cref{fig:incoherence} give plots of these empirical
MSEs versus the incoherence parameter for Cases 1 and 2, respectively.
For comparison, we also plot the theoretically predicted variance
$\gamma^\star_{i,t}$ from equation~\eqref{eq:variance-defn} for these
two cases.  For Case 1 in panel (a), we see excellent agreement
between the empirical behavior and the theoretical prediction for the
full range of incoherence parameters.  For Case 2 in panel (b), the
agreement is excellent for moderate and large values of the
incoherence parameters $\mu_i$ and $\nu_t$, and we observe some
discrepancies for very small values of incoherence (below $10^{-2}$).
This difference arises because in this extreme regime, certain
``higher-order'' terms---not part of the theoretical
prediction---become dominant.  Overall, we see the entrywise
estimation error matches the theoretical scaling (hence is
statistically optimal) even when $\mu_i$ and/or $\nu_t$ are large, or
at least one of $\mu_i$ and $\nu_t$ is not vanishingly small.  These
results support the practical validity of our algorithm and theory
over a wide range of problem instances.

\subsection{Scaling with rank}

In this section, we devised some experiments that expose the effect of
varying the matrix rank on our procedures.  In order to do so, we set
$N = T = 500$, the noise level $\omega=1$, and consider the same
staggered adoption design with $k = 3$ groups with $N_{1} = T_{1} =
200$, $N_2 = T_2 = 200$ and $N_3 = T_3 = 100$.

\begin{figure}[h]
  \begin{center}
    \begin{tabular}{ccc}
      \widgraph{0.45\textwidth}{\figdir/fig_rank_case1} &&
      \widgraph{0.45\textwidth}{\figdir/fig_rank_case2} \\
			(a) && (b)
		\end{tabular}
    \caption{ Red circles: Mean-squared error (MSE) of the estimate
      $\widehat{M}_{i,t}$ of entry $(i,t) = (500, 500)$ returned by
      Algorithm~\ref{alg:general} versus the rank $r$. Empirical MSE
      was obtained as an average of $100$ Monte Carlo trials.  Blue
      lines: theoretically predicted scaling $\gamma_{i,t}^\star$ of
      the MSE versus rank.  (a) Plots for case $1$ where the theory
      predicts scaling as $r$.  (b) Plots for Case 2 where theory
      predicts scaling as $r^{3/2}$.
			\label{fig:rank}}
	\end{center}
\end{figure}

As before, we consider estimating the missing entry
$(i,t)=(500,500)$. Under the sub-block
condition~\eqref{EqnSubmatrix-general}, the theoretical scaling
$\gamma_{i,t}^\star$ defined from equation~\eqref{eq:variance-defn} is
of the order
\begin{align*}
\Varstar_{i, t} \asymp \frac{\noise^{2}}{N_1 T} \Vert \bm{U}_{i,
  \cdot}^{\star} \Vert_2^2 + \frac{\noise^{2}}{N T_1} \Vert \bm{V}_{t,
  \cdot}^{\star} \Vert_2^2.
\end{align*}
Although the rank $r$ does not appear explicitly, it arises implicitly
via the $\ell_2$-norm of two vectors $\bm{U}_{i,\cdot}^\star$ and
$\bm{V}_{t,\cdot}^\star \in \mathbb{R}^r$. Recall the definition of
$\bm{U}^\star (\cdot) \in \mathbb{R}^{N\times r}$ and $\bm{V}^\star
(\cdot) \in \mathbb{R}^{T \times r}$ from
equation~\eqref{eq:ensemble}, we consider the following two cases:
\begin{itemize}
\item {\bf Case 1:} We generate $\bm{M}^{\star}=\bm{U}^{\star} (1)
  \bm{V}^{\star\top}(1)$ with $r$ varies between $1$ and $30$. This is
  the ideal incoherent setting, and our theory predicts that the MSE
  should scale as $\gamma_{i,t}^\star \asymp r$, so linearly in the
  rank $r$.
\item {\bf Case 2:} We generate $\bm{M}^{\star}=\bm{U}^{\star}
  (r^{1/4}) \bm{V}^{\star}(r^{1/4})^\top$ where $\bm{U}^\star (\cdot)
  \in \mathbb{R}^{N\times r}$ and $\bm{V}^\star (\cdot) \in
  \mathbb{R}^{T \times r}$ are defined in equations~\eqref{eq:U-mu}
  and~\eqref{eq:V-nu}, with the rank $r$ varying between $1$ and
  $30$. Under this setting, our theory predicts that the MSE should
  scale as $\gamma_{i,t}^\star \asymp r^{3/2}$, so super-linearly in
  the rank $r$.
\end{itemize}
\Cref{fig:rank} compares the mean-squared estimation
$\mathbb{E}[(\widehat{M}_{i,t}-M_{i,t}^\star)^2]$ and the theoretical
scaling $\gamma^\star_{i,t}$ for these two cases.  Panel (a) shows
results for Case 1, where we expect the scaling with rank to be
roughly linear; here for the given dimensions $N = T = 500$, we see
good agreement between empirical and theory up to around rank $r
\approx 25$.  On the other hand, panel (b) shows results for the more
challenging Case 2, where the theoretical effect of the rank is more
severe (growing as $r^{3/2}$). In this case, we see good agreement
between empirical and theoretical prediction up until around rank $r
\approx 20$.  Given that the matrix dimension $N=T=500$ and $N_1 = T_1
= 200$, we can see that our algorithm and theory remain valid for over
a reasonably large range of ranks $r$.

\subsection{Computational costs}

Finally, we did some simple experiments to explore the computational
efficiency of our methods.  Both of our procedures
(cf.~Algorithms~\ref{alg:4-blocks-Md}~and~\ref{alg:general}) involve
only the computation of partial SVDs, and solving least-squares
problems.  Viewing the rank $r$ as an order one quantity, under the
four-block setting, the computational complexity of
Algorithm~\ref{alg:4-blocks-Md} is dominated by the cost of computing
the top-$r$ SVD of a $N$ by $T_1$ matrix and a $N_1$ by $T$ matrix;
the cost of doing so scales $O(N T_1+N_1 T)$, apart from log factors.
For the general staggered adoption setting with $k$ groups, the
computational complexity is
\begin{align*}
\sum_{i_0 = 1}^k \sum_{j_0 = k+1-i_0}^{k} \Bigg[ \bigg(
  \sum_{i=1}^{k+1-j_0} N_i \bigg) \bigg( \sum_{j=1}^{j_0}T_j \bigg) +
  \bigg( \sum_{i=1}^{i_0} N_i \bigg) \bigg( \sum_{j=1}^{k+1-i_0}T_j
  \bigg) \Bigg].
\end{align*}
In rough terms, this upper bound scales as $O(k^2 NT)$, so that when
the number of groups $k$ is small, the computational cost is of the
same order as cost of reading the data (i.e., the $N T$ entries of the
matrix).
\begin{table}[h]
  \centering
  \begin{tabular}{c|c|c|c}
    Dimension $N$ & Computation time (seconds) & Dimension $N$ &
    Computation time (seconds) \\ \hline
    512 & 0.10 & 4096 & 5.3 \\
    1024 & 0.14 & 8192 & 22.3 \\
    2048 & 0.94 & 16384 & 97.0 \\ \hline 
  \end{tabular}
\smallskip
\caption{The runtime of Algorithm~\ref{alg:general} vs.~the dimension
  $N=T$. Reported runtimes are computed using a 2023 MacBook Pro with an Apple M2 Pro chip, and are averaged over 10
  independent trials.
  \label{table:runtime}}
\end{table}

To give some sense of the typical runtime of our algorithm, we
conducted some experiments on square matrices ($N = T$) with the side
length $N$ varying over the set $\{2^9,2^{10},\ldots,2^{15}\}$, and
with fixed rank $r = 3$.  In all cases, we ran trials on a staggered
adoption design with $k = 3$ groups with $N_{1} = T_{1} = \lfloor 0.4N
\rfloor$ and $N_2 = T_2 = \lfloor 0.4N \rfloor$.  Our code is
Python-based, making use of the standard \texttt{numpy} and
\texttt{scipy} routines for linear algebra.\footnote{One could likely
speed up our code substantially by incorporating various speed-ups,
e.g., from randomized numerical linear algebra; the current results
show runtimes without any such code optimization.}  Based on our
discussion in the previous paragraph, we expect that the runtime of
Algorithm~\ref{alg:general} scales quadratically as $O(N^2)$; this is
confirmed by the simulation results shown in~\Cref{table:runtime}. We
note that this quadratic scaling is much faster than SDP-based
algorithms like convex relaxation considered in previous literature
(e.g.,~\cite{athey2021matrix}), which become computationally
burdensome even for moderate matrix dimensions (e.g., when $N >
3000$).



\section{Proofs}
\label{sec:Proof-outline}

In this section, we provide the proofs of certain results, including
our achievable guarantee in the four-block case
(\Cref{thm:distribution}) in~\Cref{subsec:proof-thm-distribution}; our
data-dependent procedure for computing confidence intervals
(\Cref{prop:CI}) in~\Cref{sec:proof-CI}; and our lower bounds
(\Cref{thm:CRLB}) in~\Cref{sec:proof-thm-crlb}.

In all cases, we defer the proofs of various intermediate results, of
a more technical nature, to the appendices.  In addition, we defer the
proof of~\Cref{thm:distribution-general}
to~\Cref{sec:proof-general-distribution}.


\subsection{Proof of~\Cref{thm:distribution}}
\label{subsec:proof-thm-distribution}

Throughout this and other proofs, we use the shorthand notation
$\spicy \coloneqq \log(N + T)$.

\subsubsection{A master decomposition}

Our proof is based on a key decomposition, which we state as a
separate lemma in its own right.  It involves the matrix
\begin{align}
\bm{Z} & \coloneqq \bm{U}_2^{\star} (\bm{U}_{1}^{\star\top}
\bm{U}_{1}^{\star})^{-1} \bm{U}_{1}^{\star\top} \bm{E}_{b} +
\bm{E}_{c} \bm{V}_{1}^{\star} (\bm{V}_{1}^{\star\top}
\bm{V}_{1}^{\star})^{-1} \bm{V}_2^{\star\top} + \bm{E}_{c}
\bm{V}_{1}^{\star}(\bm{V}_{1}^{\star\top} \bm{V}_{1}^{\star})^{-1}
(\bm{\Sigma}^{\star})^{-1} (\bm{U}_{1}^{\star\top}
\bm{U}_{1}^{\star})^{-1} \bm{U}_{1}^{\star\top} \bm{E}_{b}, \label{eq:Z-defn}  
\end{align}
as well as the error terms
\begin{align*}
\delbound_1 (i,t) & \mydefn \frac{1}{\sqrt{N_1 T}}\bigg
(\frac{\noise^{3} }{\sigma_{r}^{\star2} N_1 } +
\frac{\noise^{2}}{\sigma_{r}^{\star} \sqrt{T_{1}}} \bigg ) \left \Vert
\bm{U}_{i,\cdot}^{\star} \right \Vert_{2} \sqrt{r + \spicy}, \\
\delbound_2 (i,t) & \mydefn \frac{1}{\sqrt{N T_1} }\bigg
(\frac{\noise^{3} }{\sigma_{r}^{\star2} N_1 } +
\frac{\noise^{2}}{\sigma_{r}^{\star} \sqrt{N} } +
\frac{\noise^{2}}{\sigma_{r}^{\star} \sqrt{T_1} } \bigg ) \left \Vert
\bm{V}_{t,\cdot}^{\star} \right \Vert_{2} \sqrt{r + \spicy}, \\
\delbound_3 (i,t) & \mydefn \bigg
(\Enoise + \noise \sqrt{\frac{r}{ N_{1} T } } + \noise
\sqrt{\frac{1}{\Nunit_1 \Time_1} \spicy } \bigg ) \left \Vert
\bm{U}_{i,\cdot}^{\star} \right \Vert _{2} \left \Vert
\bm{V}_{t,\cdot}^{\star} \right \Vert _{2}, \\
\delbound_4 (i,t) & \mydefn \frac{1}{\sqrt{NT}} \frac{1}{\sqrt{N_1 T_1}} \, \bigg (\frac{\noise^{3} }{\sigma_{r}^{\star 2} \sqrt{N} }
+ \frac{\noise^{3}}{\sigma_{r}^{\star2} \sqrt{T_1} } + \frac{\noise^{4}}{\sigma_{r}^{\star3} N_1} \bigg )
\left(r + \spicy \right ).
\end{align*}
\begin{lemma}[\textsf{Master decomposition}]
\label{lem:master}
Suppose that the sub-block condition~\eqref{EqnSubmatrix} and the
noise condition~\eqref{eq:noise-condition-est} hold. Then we can write
\begin{subequations}
\begin{align}
\label{eq:Md-decomposition}  
\widehat{\bm{M}}_{d} - \bm{M}_{d}^{\star} = \bm{Z} + \bm{\Delta},
\end{align}
for a perturbation matrix $\bm{\Delta}$ whose entries satisfy the
entrywise bound
\begin{align}
\big |\Delta_{i,t} \big | \leq C_4 \sum_{k=1}^4 \delbound_k (i,t)
\qquad \mbox{for all entries $(i,t)$}
\end{align}
\end{subequations}
with probability at least $1 - O((N + T)^{-9})$.
\end{lemma}

\noindent See~\Cref{subsec:proof-lemma-master} for the proof. \\


\subsubsection{Main argument for~\Cref{thm:distribution}}

Armed with this result, we are now equipped to
prove~\Cref{thm:distribution}.  We begin with the additive
decomposition $Z_{i,t} = \alpha_{i,t} + \beta_{i,t} + \hackgam_{i,t}$,
where
\begin{subequations}
  \label{eq:Zij-decom}
  \begin{align}
    \alpha_{i, t} & \mydefn (\bm{E}_{c})_{i,\cdot}
    \bm{V}_{1}^{\star}( \bm{V}_{1}^{\star\top}
    \bm{V}_{1}^{\star})^{-1} \bm{V}_{t,\cdot}^{\star\top},
    \\
    \beta_{i, t} & \mydefn \bm{U}_{i,\cdot}^{\star} (
    \bm{U}_{1}^{\star\top} \bm{U}_{1}^{\star})^{-1} \bm{U}_{1}^{\star\top}
    ( \bm{E}_{b})_{\cdot,t}, \quad \mbox{and} \\
    \hackgam_{i, t} & \mydefn ( \bm{E}_{c})_{i,\cdot} \bm{V}_{1}^{\star}(
    \bm{V}_{1}^{\star\top} \bm{V}_{1}^{\star})^{-1}(
    \bm{\Sigma}^{\star})^{-1}( \bm{U}_{1}^{\star\top}
    \bm{U}_{1}^{\star})^{-1} \bm{U}_{1}^{\star\top}
    (\bm{E}_{b})_{\cdot,t}.
  \end{align}
\end{subequations}
Observe that
$\alpha_{i,t}$ and $\beta_{i,t}$ are independent zero-mean Gaussian
random variables with variances
\begin{subequations}
  \label{eq:alpha-beta-var}
  \begin{align}
    \mathsf{var} \left (\alpha_{i,t} \right ) & = \frac{\noise^{2}}{NT} \big\Vert
    \bm{U}_{i,\cdot}^{\star}( \bm{U}_{1}^{\star\top}
    \bm{U}_{1}^{\star}){}^{-1} \bm{U}_{1}^{\star\top} \big
    \Vert_{2}^{2} \overset{\text{(i)}}{\in} \bigg[ \frac{1}{\cupper}
      \frac{\noise^{2}}{N_{1} T} \Vert \bm{U}_{i,\cdot}^{\star}
      \Vert_{2}^{2}, \frac{1}{\clow} \frac{\noise^{2}}{N_{1} T} \Vert
      \bm{U}_{i,\cdot}^{\star} \Vert_{2}^{2} \bigg], \quad \mbox{and}
    \\
    \mathsf{var} \left (\beta_{i,t} \right ) & = \frac{\noise^{2}}{NT}
    \big\Vert \bm{V}_{t,\cdot}^{\star}( \bm{V}_{1}^{\star\top}
    \bm{V}_{1}^{\star}){}^{-1} \bm{V}_{1}^{\star\top}
    \big\Vert_{2}^{2} \overset{\text{(ii)}}{\in} \bigg[
      \frac{1}{\cupper} \frac{\noise^{2}}{N \Time_1}\Vert
      \bm{V}_{t,\cdot}^{\star}\Vert_{2}^{2}, \frac{1}{\clow}
      \frac{\noise^{2}}{ N\Time_1}\Vert
      \bm{V}_{t,\cdot}^{\star}\Vert_{2}^{2} \bigg],
  \end{align}
\end{subequations}
where (i) and (ii) both follow from \eqref{EqnSubmatrix}. Then we can
check that
\begin{align}
  \frac{b_1(i,t)}{\mathsf{var}^{1/2} (\alpha_{i,t} )} +
  \frac{b_2(i,t)}{\mathsf{var}^{1/2} (\beta_{i,t} )} & \leq 3
  \sqrt{\cupper} ( \Enoise^2 + \Enoise ) \sqrt{r + \spicy}
  \overset{\text{(i)}}{\leq} \frac{\delta}{4}, \nonumber
  \\ \frac{b_3(i,t)}{\mathsf{var}^{1/2} (\alpha_{i,t} +
    \beta_{i,t} )} & \leq \sqrt{\cupper} \bigg ( \sqrt{r} (\EN
  \nu_t + \ET \mu_i) + \min\left\{ \frac{\mu_{i}}{\sqrt{N_1}} ,
  \frac{\nu_{t} }{\sqrt{T_1}} \right\} \sqrt{ r \spicy} \bigg )
  \overset{\text{(ii)}}{\leq} \frac{\delta}{4}, \quad \text{and}
  \nonumber \\ \frac{b_4(i,t)}{\mathsf{var}^{1/2} (\alpha_{i,t}
    + \beta_{i,t} )} & \leq 2 \sqrt{\cupper} ( \Enoise^2 +
  \Enoise ) \min \left\{ \frac{\ET}{\mu_i} , \frac{\EN}{\nu_t}
  \right\} \left( \sqrt{r} + \frac{\spicy}{\sqrt{r}} \right )
  \overset{\text{(iii)}}{\leq}
  \frac{\delta}{4}. \label{eq:b-ratio}
\end{align}
Here step (i) follows from \eqref{eq:noise-condition-est}, step (ii)
utilizes the bound~\eqref{eq:incoherence-est}, whereas step (iii)
follows from the bounds~\eqref{eq:noise-condition-est}
and~\eqref{eq:signal-lb}.  Conditioned on the random matrix
$\bm{E}_{b}$, we have
\begin{align*}
\hackgam_{i,t} \sim \mathcal{N} \big(0, \frac{\noise^2}{NT} \big\Vert
\bm{V}_{1}^{\star}( \bm{V}_{1}^{\star\top} \bm{V}_{1}^{\star})^{-1}
(\bm{\Sigma}^{\star})^{-1} ( \bm{U}_{1}^{\star \top}
\bm{U}_{1}^{\star})^{-1} \bm{U}_{1}^{\star \top} (
\bm{E}_{b})_{\cdot,t} \big\Vert_{2}^{2} \big).
\end{align*}
Therefore, for any fixed $(i,t)$, with probability at least $1 - O((N
+ T)^{-10})$, we have
\begin{align*}
  \left |\hackgam_{i,t} \right | & \overset{\text{(i)}}{\leq} \frac{5 \noise}{\sqrt{NT}}
  \big\Vert \bm{V}_{1}^{\star}( \bm{V}_{1}^{\star\top}
  \bm{V}_{1}^{\star})^{-1}( \bm{\Sigma}^{\star})^{-1}(
  \bm{U}_{1}^{\star\top} \bm{U}_{1}^{\star})^{-1}
  \bm{U}_{1}^{\star\top}( \bm{E}_{b})_{\cdot,t} \big\Vert_{2}
  \sqrt{\spicy} \nonumber \\
	& \overset{\text{(ii)}}{\leq} \frac{5 C_g \noise^{2}}{NT} \big\Vert
	\bm{V}_{1}^{\star}( \bm{V}_{1}^{\star\top} \bm{V}_{1}^{\star})^{-1}(
	\bm{\Sigma}^{\star})^{-1}( \bm{U}_{1}^{\star\top}
	\bm{U}_{1}^{\star})^{-1} \bm{U}_{1}^{\star\top} \big\Vert \sqrt{
		\left (r + \spicy \right )\spicy}\nonumber \\
	& \overset{\text{(iii)}}{\leq} \frac{5C_g}{\clow^2}
        \frac{\noise^{2}}{\sigma_{r}^{\star}} \sqrt{\frac{ 1 }{ N T \Nunit_1 \Time_1} \left (r + \spicy \right )
          \spicy}.
\end{align*}
Here the constant $5$ in step (i) comes from the fact that
$\mathbb{P}(X\geq 5\sqrt{\spicy} ) = O((N+T)^{-10})$ for
$X\sim\mathcal{N}(0,1)$, step (ii) uses the concentration bound
\[
\big\Vert \bm{\zeta} ( \bm{E}_{b})_{\cdot,t} \big\Vert_{2} \leq C_g \sigma \big\Vert
\bm{\zeta} \big\Vert \sqrt{
	r + \spicy   }, \quad \text{where} \quad \bm{\zeta} = \bm{V}_{1}^{\star}( \bm{V}_{1}^{\star\top}
\bm{V}_{1}^{\star})^{-1}( \bm{\Sigma}^{\star})^{-1}(
\bm{U}_{1}^{\star\top} \bm{U}_{1}^{\star})^{-1}
\bm{U}_{1}^{\star\top},
\]
which is a consequence of \Cref{lemma:gaussian-spectral} in \Cref{appendix:technical_lemmas}, whereas
step (iii) follows from \eqref{EqnSubmatrix}. 
Taking the bound on $|\lambda_{i,t}|$ and \eqref{eq:signal-lb} together yields
\begin{align}
	\frac{ | \lambda_{i,t} | }{ \mathsf{var}^{1/2} (\alpha_{i,t} + \beta_{i,t} ) } \leq \frac{5C_g \sqrt{\cupper}}{\clow^2} \min \left\{ 
	\frac{\ET}{\mu_i} , \frac{\EN}{\nu_t}  \right\} \sqrt{\spicy + \spicy^2 / r} \leq \frac{\delta}{4}. \label{eq:lambda-ratio}
\end{align}
Let $G_{i,t} = \alpha_{i,t} + \beta_{i,t}$ and $\gamma_{i,t}^\star = \mathsf{var} (\alpha_{i,t}+\beta_{i,t})$, we can check that
\begin{align*}
\Biggr| \frac{1}{ (\Varstar_{i,t})^{1/2} }
\big(\widehat{\bm{M}}_{d} - \bm{M}_{d}^{\star}\big)_{i, t} - G_{i, t}
\Biggr| \overset{\text{(i)}}{\leq} \sum_{j=1}^4
\frac{b_j(i,t)}{\mathsf{var}^{1/2} (\alpha_{i,t} + \beta_{i,t} )} +
\frac{ | \lambda_{i,t} | }{ \mathsf{var}^{1/2} (\alpha_{i,t} +
  \beta_{i,t} ) } \overset{\text{(ii)}}{\leq} \delta
\end{align*}
where step (i) follows from~\Cref{lem:master} and step (ii) follows
from equations~\eqref{eq:b-ratio} and~\eqref{eq:lambda-ratio}.


\subsubsection{Proof of~\Cref{lem:master}}
\label{subsec:proof-lemma-master}

Our proof is based on decomposing the estimation error as
\begin{align} 
 \bm{M}_{d}- \bm{M}_{d}^{\star} & = \bm{U}_{2}( \bm{U}_{1}^{\top}
 \bm{U}_{1})^{-1} \bm{U}_{1}^{\top} \bm{M}_{b}- \bm{U}_{2}^{\star}(
 \bm{U}_{1}^{\star\top} \bm{U}_{1}^{\star})^{-1}
 \bm{U}_{1}^{\star\top} \bm{M}_{b}^{\star}\nonumber \\
 & = \bm{U}_{2}( \bm{U}_{1}^{\top} \bm{U}_{1})^{-1} \bm{U}_{1}^{\top}(
 \bm{M}_{b}- \bm{M}_{b}^{\star}) + \big[ \bm{U}_{2}( \bm{U}_{1}^{\top}
   \bm{U}_{1})^{-1} \bm{U}_{1}^{\top}- \bm{U}_{2}^{\star}(
   \bm{U}_{1}^{\star\top} \bm{U}_{1}^{\star})^{-1}
   \bm{U}_{1}^{\star\top} \big] \bm{M}_{b}^{\star} \nonumber \\
%
%
 & = \bm{A} + \bm{B},  \label{eq:Md-error-decom}   
\end{align}
where
\begin{align*}
\bm{A} \mydefn \bm{U}_{2}( \bm{U}_{1}^{\top} \bm{U}_{1})^{-1}
\bm{U}_{1}^{\top}( \bm{M}_{b}- \bm{M}_{b}^{\star}), \quad \mbox{and}
\quad \bm{B} \mydefn \big[ \bm{U}_{2}( \bm{U}_{1}^{\top}
  \bm{U}_{1})^{-1} \bm{U}_{1}^{\top}- \bm{U}_{2}^{\star}(
  \bm{U}_{1}^{\star\top} \bm{U}_{1}^{\star})^{-1}
  \bm{U}_{1}^{\star\top} \big] \bm{M}_{b}^{\star}.
\end{align*}
With these definitions in hand, we now state two auxiliary lemmas that
bound these two terms.

\begin{lemma}
\label{lemma:decom-A}
Under the conditions of \Cref{lem:master}, we have the decomposition
\begin{align}
\label{eq:decom-A}  
\bm{A} = \bm{U}_{2}^{\star}( \bm{U}_{1}^{\star\top}
\bm{U}_{1}^{\star})^{-1} \bm{U}_{1}^{\star\top} \bm{E}_{b} +
\bm{E}_{c} \bm{V}_{1}^{\star}( \bm{V}_{1}^{\star\top}
\bm{V}_{1}^{\star})^{-1}( \bm{\Sigma}^{\star})^{-1}(
\bm{U}_{1}^{\star\top} \bm{U}_{1}^{\star})^{-1} \bm{U}_{1}^{\star\top}
\bm{E}_{b} + \bm{\xi}_{ \bm{A}},
\end{align}
where $| ( \bm{\xi}_{ \bm{A}})_{i,t} | \leq  C_A \sum_{k=1}^4 \delbound_k (i,t) $ holds for each $(i,t)$ with probability at least $1-O((N + T)^{-9})$.
\end{lemma}
\noindent See~\Cref{subsec:proof-decom-A} for the proof.

\begin{lemma}
  \label{lemma:decom-B}
Under the conditions of \Cref{lem:master}, there exists some universal constant $C_B>0$ such that we can decompose
\begin{align}
	\label{eq:decom-B}  
\bm{B} = \bm{E}_{c} \bm{V}_{1}^{\star}( \bm{V}_{1}^{\star\top}
\bm{V}_{1}^{\star})^{-1} \bm{V}_{2}^{\star\top} + \bm{\xi}_{ \bm{B}},
\end{align}
where $| ( \bm{\xi}_{ \bm{B}})_{i,t} | \leq  C_B \sum_{k=2}^3 \delbound_k (i,t) $ holds for each $(i,t)$ with probability at least $1-O((N + T)^{-9})$.
\end{lemma}
\noindent See~\Cref{subsec:proof-decom-C} for the proof. 

Combining equation~\eqref{eq:Md-error-decom} with these two lemmas
yields $ \bm{M}_{d} - \bm{M}_{d}^{\star} = \bm{Z} + \bm{\Delta}$,
where the random matrix $\bm{Z}$ was previously
defined~\eqref{eq:Z-defn}, and $\bm{\Delta} \mydefn \bm{\xi}_{\bm{A}}
+ \bm{\xi}_{\bm{B}}$. The residual matrix $\bm{\Delta}$ satisfies,
with probability exceeding $1-O((N+T)^{-9})$, that $ | \Delta_{ i , t}
| \leq C_4 \sum_{k=1}^4 b_k (i,t) $ holds for each $(i,t)$ with some
universal constant $C_4 = C_A + C_B$.


\subsection{Proof of~\Cref{prop:CI}}
\label{sec:proof-CI}

%
  Naturally, a central step in the proof is to bound the
error incurred by using $\widehat{\gamma}_{i,t}$ as an estimate of
$\gamma_{i,t}^\star$.  Let us summarize this auxiliary claim here:
with probability at least $1 - O((N+T)^{-10})$, there exists some constant $\cgamma>0$ such that
\begin{align}
\label{eq:var-est-err}  
|\widehat{\gamma}_{i,t} - \gamma_{i,t}^\star| & \leq \cgamma \bigg( \tfrac{\delta}{\sqrt{\spicy}} + \sqrt{\tfrac{( N_{1} + T ) r
  }{N_{1} T}} \bigg) \gamma_{i,t}^\star.
\end{align}
Here $\delta \in (0,1)$ is the tolerance parameter in the statement
of~\Cref{prop:CI}.  We return to prove this result
in~\Cref{SecProofAuxClaim}.


\subsubsection{Main argument}

\newcommand{\SpecVar}{\ensuremath{V_{i,t}}}


Taking the auxiliary claim~\eqref{eq:var-est-err} as given, let us
prove the coverage guarantee stated in~\Cref{prop:CI}.  Define the
random variable
\begin{align*}
\SpecVar \coloneqq (\widehat{M}_{i,t}-M_{i,t}^{\star}) /
\widehat{\gamma}_{i,t}^{1/2} - G_{i,t} ,
\end{align*}
where $G_{i,t}$ is the standard Gaussian variable defined
in~\Cref{thm:distribution}.

By applying~\Cref{thm:distribution}, we find that with probability at
least $1-O((N+T)^{-10})$,
\begin{align}
\label{EqnInterBound}             
|\SpecVar| & = \left | \tfrac{\widehat{M}_{i,t} - M_{i,t}^{\star}}
{(\gamma_{i,t}^\star)^{1/2}} - G_{i,t} \right | + \left |
\tfrac{\widehat{M}_{i,t} - M_{i,t}^{\star}}
      {(\gamma_{i,t}^\star)^{1/2}} -
      \tfrac{\widehat{M}_{i,t}-M_{i,t}^{\star}}
            {(\widehat{\gamma}_{i,t})^{1/2}} \right | \; \leq \delta +
            \left | \tfrac{\widehat{M}_{i,t} - M_{i,t}^{\star}}
                  {(\gamma_{i,t}^\star)^{1/2}} \right | \; \; \left |
                  1 -
                  \tfrac{(\gamma_{i,t}^\star)^{1/2}}{(\widehat{\gamma}_{i,t})^{1/2}}
                  \right |.
\end{align}
Next we observe that
\begin{align*}
\left | \tfrac{\widehat{M}_{i,t}-M_{i,t}^{\star}}
      {(\gamma_{i,t}^\star)^{1/2}} \right | \leq \left |
      \tfrac{\widehat{M}_{i,t}-M_{i,t}^{\star}}
           {(\gamma_{i,t}^\star)^{1/2}} -G_{i,t} \right | + | G_{i,t}
           | \overset{\text{(i)}}{\leq} \delta + 5\sqrt{\spicy}
\end{align*}
and
\begin{align*}
\left | 1 -
\tfrac{(\gamma_{i,t}^\star)^{1/2}}{(\widehat{\gamma}_{i,t})^{1/2}}
\right | = \tfrac{| \gamma_{i,t}^\star - \widehat{\gamma}_{i,t} |
}{(\widehat{\gamma}_{i,t})^{1/2} [ (\gamma_{i,t}^\star)^{1/2} +
    (\widehat{\gamma}_{i,t})^{1/2} ] } \overset{\text{(ii)}}{\leq} 2 \cgamma \tfrac{\delta}{\sqrt{\spicy}} + 2 \cgamma \sqrt{\tfrac{(
    N_{1} + T ) r }{N_{1} T}};
\end{align*}
where claim (i) follows from~\Cref{thm:distribution}, and step (ii)
follows from the auxiliary claim~\eqref{eq:var-est-err} and its consequence $\widehat{\gamma}_{i,t} \geq \gamma_{i,t}^\star  / 2$, 
provided that $\delta$ is sufficiently small and $\min\{N_1,T\}$ is sufficiently
large.

Combining these facts with our earlier bound~\eqref{EqnInterBound} and
defining $C_v = 12 \cgamma$, we find that
\begin{align}
\label{eq:epsilon-bound}  
|\SpecVar| & \leq \delta + (\delta + 5\sqrt{\spicy} ) \left( 2 \cgamma \tfrac{\delta}{\sqrt{\spicy}} + 2 \cgamma \sqrt{\tfrac{( N_{1} + T
    ) r }{N_{1} T}} \right)  \; \overset{\text{(iii)}}{\leq} C_v
\delta,
\end{align}
where step (iii) holds provided that $\min\{ N_1, T \} \geq
\delta^{-2} r \spicy$.


We now use this high probability bound on $\SpecVar$ to establish our
coverage guarantee. For any $x \in \mathbb{R}$, from the definition of
$\SpecVar$, we have
\begin{subequations}
  \begin{align}
\label{EqnStar}    
\mathbb{P} \bigg( \tfrac{\widehat{M}_{i,t} -
  M_{i,t}^{\star}}{(\widehat{\gamma}_{i,t})^{1/2}} \leq x \bigg) &
\stackrel{(\star)}{=} \mathbb{P} \big( G_{i,t} + \SpecVar \leq x, |
\SpecVar | \leq C_v \delta \big) + \mathbb{P} \big( G_{i,t} + \SpecVar
\leq x, | \SpecVar | > C_v \delta \big)  \\
& \leq \mathbb{P} \big( G_{i,t} \leq x + C_v \delta \big) + \mathbb{P}
\left( | \SpecVar | > C_v \delta \right) \notag \\
& \overset{\text{(a)}}{=} \Phi \left ( x + C_v \delta \right )+ O
\left((N+T)^{-10}\right) \notag \\
\label{EqnBoundOne}
& \overset{\text{(b)}}{\leq} \Phi(x) + O\big( \delta + (N+T)^{-10} \big)
\end{align}
where step (a) follows from the bound~\eqref{eq:epsilon-bound}; and
step (b) follows since the normal CDF $\Phi$ is a
$1/\sqrt{2\pi}$-Lipschitz function.
Similarly, we have the lower bound
\begin{align}
\mathbb{P} \bigg( \tfrac{\widehat{M}_{i,t} -
  M_{i,t}^{\star}}{(\widehat{\gamma}_{i,t})^{1/2}} \leq x \bigg) &
\overset{\text{(c)}}{\geq} \mathbb{P} \big( G_{i,t} + \SpecVar \leq x,
| \SpecVar | \leq C_v \delta \big) \notag
\\
& \geq \mathbb{P} \big( G_{i,t} \leq x - C_v \delta \big) \notag \\
& = \Phi
\left( x - C_v \delta \right) \notag \\
\label{EqnBoundTwo}
& \overset{\text{(d)}}{\geq} \Phi(x) -O (\delta),
\end{align}
\end{subequations}
where step (c) follows from equality $(\star)$ in
equation~\eqref{EqnStar}; and step (d) follows from the Lipschitz
property of $\Phi$.

Finally, by taking $x=\pm \Phi^{-1} (1-\alpha/2)$ and combining the
upper and lower bounds~\eqref{EqnBoundOne} and~\eqref{EqnBoundTwo},
we find that
\begin{align*}
  \mathbb{P} \big( \mathsf{CI}_{i,t}^{1-\alpha} \ni M_{i,t}^\star
  \big) = \mathbb{P} \bigg(
  \tfrac{\widehat{M}_{i,t}-M_{i,t}^{\star}}{(\widehat{\gamma}_{i,t})^{1/2}}
  \in \big[ \pm \Phi^{-1} (1-\alpha/2)\big] \bigg) = 1 - \alpha + O
  \big( \delta + (N+T)^{-10} \big),
\end{align*}
as claimed in~\Cref{prop:CI}.


\subsubsection{Proof of the auxiliary claim~\eqref{eq:var-est-err}}
\label{SecProofAuxClaim}

To simplify presentation, we introduce the shorthand $\sigma =
\omega/\sqrt{NT}$, along with the associated estimate
$\widehat{\sigma}^2 = \widehat{\omega}^2 / NT$.  Using this notation,
we write
\begin{align*}
& \widehat{\gamma}_{i,t} - \gamma_{i,t}^{\star} = \underbrace{(
    \widehat{\sigma}^2 - \sigma^{2} )
    \bm{U}_{i,\cdot}^{\star}(\bm{U}_{1}^{\star\top}\bm{U}_{1}^{\star})^{-1}
    \bm{U}_{i,\cdot}^{\star\top} + (\widehat{\sigma}^{2} - \sigma^2 )
    \bm{V}_{t,\cdot}^{\star}(\bm{V}_{1}^{\star\top}
    \bm{V}_{1}^{\star})^{-1}\bm{V}_{t,\cdot}^{\top}}_{\eqqcolon
    \beta_1} \\
& \quad + \underbrace{\widehat{\sigma}^{2} \big( \bm{U}_{i,\cdot}
    (\bm{U}_{1}^{\top}\bm{U}_{1})^{-1}\bm{U}_{i,\cdot}^{\top} -
    \bm{U}_{i,\cdot}^{\star}(\bm{U}_{1}^{\star\top}\bm{U}_{1}^{\star})^{-1}
    \bm{U}_{i,\cdot}^{\star\top} \big) + \widehat{\sigma}^{2} \big(
    \bm{V}_{t,\cdot}^{\star}(\bm{V}_{1}^{\star\top}
    \bm{V}_{1}^{\star})^{-1}\bm{V}_{t,\cdot}^{\top} -
    \bm{V}_{t,\cdot}^{\star}(\bm{V}_{1}^{\star\top}
    \bm{V}_{1}^{\star})^{-1}\bm{V}_{t,\cdot}^{\star\top} \big)
  }_{\eqqcolon\beta_2}
\end{align*}
Next we bound each of the terms $\beta_1$ and $\beta_2$ in turn.

\paragraph{Bounding $\beta_1$:}
In order to bound $\beta_1$, we need to control the estimation error
associated with $\widehat{\sigma}^2$. The following result provides
such a bound:
\begin{lemma}
\label{LemSigBound}
Defining the constant $\csig = 2 \sqrt{2\cb + 1} + \cu^2 + 2\cb$, we
have
\begin{align}
\label{eq:noise-est}  
\vert \widehat{\sigma}^{2}-\sigma^{2}\vert \leq \alpha_1 + \alpha_2 +
\alpha_3 \leq \csig \sigma^{2} \sqrt{\tfrac{( N_{1} + T ) r }{N_{1}
    T}},
\end{align}
with probability at least $1 - O((N + T)^{-10})$.
\end{lemma}
\noindent See~\Cref{SecProofLemSigBound} for the proof. \\

We now apply the relation~\eqref{eq:noise-est} to bound $\beta_1$,
thereby obtaining
\begin{subequations}
\begin{align}
| \beta_1 | & \leq \csig \sigma^{2} \sqrt{\tfrac{( N_{1} + T ) r
  }{N_{1} T}} \big(
\bm{U}_{i,\cdot}^{\star}(\bm{U}_{1}^{\star\top}\bm{U}_{1}^{\star})^{-1}
\bm{U}_{i,\cdot}^{\star\top} +
\bm{V}_{t,\cdot}^{\star}(\bm{V}_{1}^{\star\top}
\bm{V}_{1}^{\star})^{-1}\bm{V}_{t,\cdot}^{\top} \big) \nonumber \\
\label{EqnBetaOne}
& = \csig \sqrt{\tfrac{( N_{1} + T ) r }{N_{1} T}} \gamma_{i,t}^\star.
\end{align}


\paragraph{Bounding $\beta_2$:}
In order to bound the term $\beta_2$, we need the following lemma
characterizing the error of the plug-in estimate.
\begin{lemma}
\label{lemma:CI-2}
Under the conditions of~\Cref{prop:CI}, with probability at least $1 -
O((N + T)^{-10})$,
\begin{align*}
\sigma^2 \big | \bm{U}_{i,\cdot}
(\bm{U}_{1}^{\top}\bm{U}_{1})^{-1}\bm{U}_{i,\cdot}^{\top} -
\bm{U}_{i,\cdot}^{\star}(\bm{U}_{1}^{\star\top}\bm{U}_{1}^{\star})^{-1}
\bm{U}_{i,\cdot}^{\star\top} \big | + \sigma^2 \big | \bm{V}_{t,\cdot}
(\bm{V}_{1}^{\top} \bm{V}_{1})^{-1}\bm{V}_{t,\cdot}^{\top} -
\bm{V}_{t,\cdot}^{\star}(\bm{V}_{1}^{\star\top}
\bm{V}_{1}^{\star})^{-1}\bm{V}_{t,\cdot}^{\star\top} \big | & \leq
\frac{\delta \, \gamma_{i,t}^\star }{\sqrt{\spicy}} .
\end{align*}
\end{lemma}
\noindent See~\Cref{sec:proof-lemma-CI-2} for the proof.

When $\min\{N_1,T\} \geq r$, a direct consequence of \eqref{eq:noise-est} is that $\widehat{\sigma}^2 \leq (\sqrt{2}\csig+1) \sigma^2$. Combining this bound with \Cref{lemma:CI-2}  yields
\begin{align}
  \label{EqnBetaTwo}
| \beta_2 | & \leq ( \sqrt{2} \csig + 1) \tfrac{\delta}{\sqrt{\spicy}}
\gamma_{i,t}^\star,
\end{align}
\end{subequations}

\vspace*{0.1in}
Collecting together the two bounds~\eqref{EqnBetaOne}
and~\eqref{EqnBetaTwo} on $\beta_1$ and $\beta_2$, respectively,
yields the upper bound
\begin{align*}
|\widehat{\gamma}_{i,t} - \gamma_{i,t}^\star| & \leq |\beta_1| +
|\beta_2| \leq \bigg( ( \sqrt{2} \csig + 1 ) \tfrac{\delta}{\sqrt{\spicy}} +
\csig \sqrt{\tfrac{( N_{1} + T ) r }{N_{1} T}} \bigg)
\gamma_{i,t}^\star.
\end{align*}
Let $\cgamma = \sqrt{2} \csig +1$ to achieve the claimed bound~\eqref{eq:var-est-err}.


\subsubsection{Proof of~\Cref{LemSigBound}}
\label{SecProofLemSigBound}

We first decompose $\widehat{\sigma}^{2}$ into a sum of three terms as
\begin{align*}
\widehat{\sigma}^{2} & = \frac{1}{N_{1} T} \big\Vert
\bm{M}_{\mathsf{upper}}^{\star} + \bm{E}_{\mathsf{upper}} -
\widehat{\bm{M}}_{\mathsf{upper}} \big\Vert_{\mathrm{F}}^{2} \\ &=
\frac{1}{N_{1}T} \Vert \bm{E}_{\mathsf{upper}} \Vert_{\mathrm{F}}^{2}
+ \frac{1}{N_{1} T} \big\Vert \bm{M}_{\mathsf{upper}}^{\star} -
\widehat{\bm{M}}_{\mathsf{upper}} \big\Vert_{\mathrm{F}}^{2} +
\frac{2}{N_{1} T} \big\langle \bm{M}_{\mathsf{upper}}^{\star} -
\widehat{\bm{M}}_{\mathsf{upper}} , \bm{E}_{\mathsf{upper}}
\big\rangle.
\end{align*}
As our analysis will show, the first quantity concentrates around
$\sigma^2$.  Thus, it is natural to bound the estimation error
associated with $\widehat{\sigma}^2$ in the following way:
\begin{align*}
 \vert \widehat{\sigma}^{2} - \sigma^{2} \vert & \leq
 \underbrace{\Big\vert \frac{1}{N_{1}T} \Vert\bm{E}_{\mathsf{upper}}
   \Vert_{\mathrm{F}}^{2} - \sigma^{2} \Big\vert }_{\eqqcolon
   \alpha_1} + \underbrace{\frac{1}{N_{1}T} \big\Vert
   \bm{M}_{\mathsf{upper}}^{\star} - \widehat{\bm{M}}_{\mathsf{upper}}
   \big\Vert_{\mathrm{F}}^{2} }_{\eqqcolon \alpha_2} +
 \underbrace{\frac{2}{N_1 T} \big\Vert \bm{M}_{\mathsf{upper}}^{\star}
   - \bm{M}_{\mathsf{upper}} \big\Vert_{\mathrm{F}} \Vert
   \bm{E}_{\mathsf{upper}} \Vert_{\mathrm{F}}}_{\eqqcolon \alpha_3}.
\end{align*}
By applying Bernstein's inequality \citep[Theorem
  2.8.1]{vershynin2016high}, we can conclude that for some constant
$\cb>0$, as long as $N_1 T \geq \spicy$, we find that
\begin{subequations}
\label{eq:E-upper-fro}
\begin{align}
\alpha_1 = \frac{1}{N_1 T} \Big | \sum_{i=1}^{N_1} \sum_{j=1}^T (
E_{i,j}^2 - \sigma^2 ) \Big | \leq \cb \sigma^{2}\sqrt{\frac{ \spicy
  }{N_1 T}} + \cb \sigma^2 \frac{\spicy}{N_1 T} \leq 2 \cb
\sigma^{2}\sqrt{\frac{ \spicy }{N_1 T}}
\end{align}
with probability at least $1-O((N+T)^{-10})$.  As a consequence, we
have
\begin{align} 
\Vert \bm{E}_{\mathsf{upper}} \Vert_{\mathrm{F}}^2 \leq N_1 T \sigma^2
+ N_1 T \alpha_1 \leq (2 \cb + 1 ) N_1 T \sigma^2,
\end{align} 
\end{subequations}
provided that $N_1 T \geq \spicy$.

\noindent In order to bound the terms $\alpha_2$ and $\alpha_3$, we
make use of the following lemma.
\begin{lemma}
\label{lemma:CI-1}
Under the conditions of~\Cref{prop:CI}, we have
\begin{align*}
  \big \Vert \bm{M}_{\mathsf{upper}}^{\star} -
  \widehat{\bm{M}}_{\mathsf{upper}}\big\Vert_{\mathrm{F}} & \leq \cu
  \sigma \sqrt{ ( T + N_{1} ) r}
\end{align*}
with probability at least $1 - O((N + T)^{-10})$.
\end{lemma}
\noindent See~\Cref{sec:proof-lemma-CI-1} for the proof.\\

Combining~\Cref{lemma:CI-1} with the bounds~\eqref{eq:E-upper-fro}
yields
\begin{align}
  \alpha_2 \leq \cu^2 \frac{\sigma^{2} \left( T + N_{1} \right)
    r}{N_{1} T} \qquad \text{and} \qquad \alpha_3 \leq 2 \sqrt{2\cb +
    1} \cu \sigma^{2} \sqrt{\frac{ ( N_{1} + T ) r }{N_{1} T}}.
\end{align}
Putting together the upper bounds for $\alpha_1$, $\alpha_2$ and
$\alpha_3$ yields the claimed bound~\eqref{eq:noise-est}.


\subsection{Proof of lower bounds}
\label{sec:proof-thm-crlb}

In this section, we prove the Cram\'{e}r--Rao and local minimax bounds
stated in the main text.  While our formal statement (\Cref{thm:CRLB})
only covers the four-block case, here we actually provide a proof for
the multi-block case, so that we match our general set of achievable
results (\Cref{thm:distribution-general}).

Let us recall the set-up for our genie-aided problem associated with
estimating entry $\bm{M}_{i,t}^{\star}$ of the matrix.  The only
unknown quantities are the vectors $\xstar \coloneqq
\bm{X}_{i,\cdot}^{\star}$ and $\ystar \coloneqq
\bm{Y}_{t,\cdot}^{\star}$, and our goal is to estimate the inner
product $\bm{M}_{i,t}^{\star} = \inprod{\xstar}{\ystar}$ based on
linear measurements of the form
\begin{align*}
\big\{ \bm{M}_{i, s} = \inprod{\xstar}{\bm{Y}_{s, \cdot}^{\star}} +
E_{i, s} \big\}_{s=1}^{\bar{T}_1} \qquad \text{and} \quad \big\{
\bm{M}_{k, t} = \inprod{\bm{X}_{k, \cdot}^{\star}}{\ystar} + E_{k, t}
\big\}_{k=1}^{\bar{\Nunit}_1}.
\end{align*}

\paragraph{Cram\'{e}r--Rao lower bound:}
We first compute the Cram\'{e}r--Rao lower bound (CRLB) for the
problem.  Using $\bftheta = (\xstar, \ystar)$ to denote the unknown
parameters, our goal is to estimate the functional $\Psi(\bftheta)
\coloneqq \inprod{\xstar}{\ystar}$.  The CRLB is given by $\nabla
\Psi(\bftheta)^T \FishMat^{-1}(\bftheta) \nabla \Psi(\bftheta)$, where
$\FishMat(\bftheta)$ is the Fisher information matrix, and the
gradient takes the form $\nabla \Psi(\bftheta) = \begin{bmatrix}
  \ystar & \xstar
\end{bmatrix}^T$.
Given that our measurements are linear with Gaussian observation noise
with variance $\noise^2/(N T)$, the Fisher information matrix takes
the form
\begin{align*}
  \FishMat(\bftheta) = \frac{NT}{\noise^{2}} \mbox{blkdiag} \Big(
  \sum_{s = 1}^{\bar{T}_{1}} \bm{Y}_{s,\cdot}^{\star \top}
  \bm{Y}_{s,\cdot}^{\star}, \quad \sum_{k = 1}^{\bar{N}_{1}}
  \bm{X}_{k,\cdot}^{\star \top} \bm{X}_{k, \cdot}^\star \Big).
\end{align*}
Putting together the pieces, we find that the CRLB is given by
\begin{align*}
\mathsf{CRLB} \left (M_{i,t}^{\star} \right) & =
\bm{Y}_{t,\cdot}^{\star} \big ( \frac{NT}{\noise^{2}} \sum_{s =
  1}^{\bar{T}_{1}} \bm{Y}_{s,\cdot}^{\star \top}
\bm{Y}_{s,\cdot}^{\star} \big )^{-1} \bm{Y}_{t,\cdot}^{\star \top} +
\bm{X}_{i,\cdot}^{\star} \big ( \frac{NT}{\noise^{2}} \sum_{k =
  1}^{\bar{N}_{1}} \bm{X}_{k,\cdot}^{\star \top}
\bm{X}_{k,\cdot}^{\star} \big )^{-1} \bm{X}_{i,\cdot}^{\star \top} \\
& \overset{\text{(i)}}{=} \frac{\noise^{2}}{NT} \Big \{
\bm{V}_{t,\cdot}^{\star} \big (\sum_{s = 1}^{\bar{T}_{1}}
\bm{V}_{s,\cdot}^{\star \top} \bm{V}_{s,\cdot}^{\star} \big )^{-1}
\bm{V}_{t,\cdot}^{\star \top} + \bm{U}_{k,\cdot}^{\star} \big (\sum_{k
  = 1}^{\bar{N}_{1}} \bm{U}_{k,\cdot}^{\star \top}
\bm{U}_{k,\cdot}^{\star} \big )^{-1} \bm{U}_{i,\cdot}^{\star \top}
\Big \} \\
& = \underbrace{\frac{\noise^2}{NT} \Big \{ \bm{V}_{t,\cdot}^{\star}(
  \bm{V}_{1}^{\star \top} \bm{V}_{1}^{\star})^{-1}
  \bm{V}_{t,\cdot}^{\star \top} + \bm{U}_{i,\cdot}^{\star}(
  \bm{U}_{1}^{\star \top} \bm{U}_{1}^{\star})^{-1}
  \bm{U}_{i,\cdot}^{\star \top} \Big \}}_{ \equiv \gamma_{i,t}^\star}
\end{align*}
Here step (i) uses the fact that $\bm{X}^\star = \bm{U}^\star
(\bm{\Sigma}^\star)^{1/2}$ and $\bm{Y}^\star = \bm{V}^\star
(\bm{\Sigma}^\star)^{1/2}$, along with some algebra.

\paragraph{Computation of the local minimax bound:}
We now turn to computation of the local minimax lower bound, where we
make use of a Bayesian form of the CRLB.  We begin by observing that
for any $\varepsilon > 0$, the local minimax risk for estimating
$\bm{M}^\star_{i,t}$ can be lower bounded as
\begin{align*}
  R^\star(\varepsilon) & \geq \inf_{\widehat{M}_{i,t}}
  \mathop{\mathbb{E}}_{ (\bm{x} , \bm{y}) \sim \pi } \mathbb{E} \left[
    \big (\widehat{M}_{i,t} - \inprod{\bm{x}}{\bm{y}} \big)^2 \right],
\end{align*}
where $\pi$ is any prior supported on the Cartesian product space.
$\ball_\infty (\bm{x}^\star , \varepsilon ) \times \ball_\infty
(\bm{y}^\star , \varepsilon )$.

Following a standard avenue for computing Bayesian
CRLBs~\cite{polyanskiy2023information}, we first define the squared
cosine density function $g(u) = \cos^2 (\pi u / 2)$ for $u \in [-1 ,
  1]$.  Using this building block, we then define the product prior
distribution $\pi(\theta_1, \ldots, \theta_{2r}) = \prod_{j=1}^{2 r}
\pi_j(\theta_j)$ with marginals
\begin{align*}
  \pi_j(\theta) & = \begin{cases} \frac{1}{\varepsilon} g \left(
    \frac{\theta - x_{j}^\star }{\varepsilon} \right) & \mbox{for
      $\theta_j \in [ x_j^\star - \varepsilon, x_j^\star +
        \varepsilon]$ and $j = 1, \ldots, r$, and} \\
    \frac{1}{\varepsilon} g \left( \frac{\theta -
      y_{j-r}^\star}{\varepsilon} \right) & \mbox{for $\theta_{j} \in
      [y_{j-r}^\star - \varepsilon, y_{j-r}^\star + \varepsilon ]$ and
      $j = r+1, \ldots, 2r$.}
  \end{cases}
\end{align*}
With this choice of prior, we can compute the lower bound
\begin{align*}
  R^\star(\varepsilon) & \stackrel{\text{(i)}}{\geq}
  \bm{Y}_{t,\cdot}^{\star} \big ( \frac{NT}{\noise^{2}} \sum_{s =
    1}^{\bar{T}_{1}} \bm{Y}_{s,\cdot}^{\star \top}
  \bm{Y}_{s,\cdot}^{\star} + \frac{\pi^2 }{\varepsilon^2} \bm{I}_r
  \big )^{-1} \bm{Y}_{t,\cdot}^{\star \top} + \bm{X}_{i,\cdot}^{\star}
  \big ( \frac{NT}{\noise^{2}} \sum_{k = 1}^{\bar{N}_{1}}
  \bm{X}_{k,\cdot}^{\star \top} \bm{X}_{k,\cdot}^{\star} + \frac{\pi^2
  }{\varepsilon^2} \bm{I}_r \big )^{-1} \bm{X}_{i,\cdot}^{\star \top}
  \\
& \stackrel{\text{(ii)}}{=} \underbrace{\frac{\noise^{2}}{NT}
    \bm{V}_{t,\cdot}^{\star} \big ( \bm{V}_{1}^{\star \top}
    \bm{V}_{1}^{\star} + \frac{\pi^2 \noise^{2} }{\varepsilon^2 NT}
    (\bm{\Sigma}^\star)^{-1} \big )^{-1} \bm{V}_{t,\cdot}^{\star \top}
    + \frac{\noise^{2}}{NT} \bm{U}_{i,\cdot}^{\star} \big (
    \bm{U}_{1}^{\star \top} \bm{U}_{1}^{\star} + \frac{\pi^2
      \noise^{2} }{\varepsilon^2 NT} (\bm{\Sigma}^\star)^{-1} \big
    )^{-1} \bm{U}_{i,\cdot}^{\star \top}}_{ \eqqcolon
    \widetilde{\gamma}_{i,t}(\epsilon)}
\end{align*}
where step (i) makes use of some standard computations with squared
cosine priors (e.g., \S 29.2 in the
book~\cite{polyanskiy2023information}); and step (ii) follows from the
spectral decomposition of $\bm{M}^\star$, along with some algebra.

Thus far, we proven that $R^\star(\varepsilon) \geq
\widetilde{\gamma}_{i,t} (\varepsilon) $.  It remains to show that
with the choice $\varepsilon = \varepsilon_{N,T} \equiv
\sqrt{\sigma_r^\star} \max \{1/\sqrt{N}, 1/\sqrt{T} \}$ we have
\begin{align}
\label{EqnAux}  
\Delta_{i,t}^\star \coloneqq \gamma_{i,t}^\star -
\widetilde{\gamma}_{i,t} (\varepsilon) \leq \gamma_{i,t}^\star \;
\frac{\cupper \pi^2}{c_\ell} \Enoise^2,
\end{align}
from which it will follow that $R^\star(\varepsilon) \geq
\widetilde{\gamma}_{i,t} (\varepsilon) = \gamma^\star_{i,t} -
\Delta^\star_{i,t} \geq \big(1 - \frac{\cupper \pi^2}{c_\ell}
\Enoise^2 \big) \gamma^\star_{i,t}$, as claimed.

From the definitions of $\widetilde{\gamma}_{i,t}(\varepsilon)$ and
$\gamma_{i,t}^\star$, we have
\begin{align*}
\Delta^*_{i,t} = \gamma_{i,t}^\star - \widetilde{\gamma}_{i,t} (\varepsilon) & =
\frac{\pi^2 \noise^{4} }{\varepsilon^2 (NT)^2}
\bm{V}_{t,\cdot}^{\star}( \bm{V}_{1}^{\star \top}
\bm{V}_{1}^{\star})^{-1} (\bm{\Sigma}^\star)^{-1} \big (
\bm{V}_{1}^{\star \top} \bm{V}_{1}^{\star} + \frac{\pi^2 \noise^{2}
}{\varepsilon^2 NT} (\bm{\Sigma}^\star)^{-1} \big )^{-1}
\bm{V}_{t,\cdot}^{\star \top} \\
& \qquad + \frac{\pi^2 \noise^{4} }{\varepsilon^2 (NT)^2}
\bm{U}_{i,\cdot}^{\star}( \bm{U}_{1}^{\star \top}
\bm{U}_{1}^{\star})^{-1} (\bm{\Sigma}^\star)^{-1} \big (
\bm{U}_{1}^{\star \top} \bm{U}_{1}^{\star} + \frac{\pi^2 \noise^{2}
}{\varepsilon^2 NT} (\bm{\Sigma}^\star)^{-1} \big)^{-1}
\bm{U}_{i,\cdot}^{\star \top} \\
& \leq \frac{\pi^2 \noise^{4} }{\varepsilon^2 (NT)^2} \Big(
\Vert (
\bm{V}_{1}^{\star \top} \bm{V}_{1}^{\star})^{-1}
\Vert^2 \Vert
(\bm{\Sigma}^\star)^{-1} \Vert \Vert \bm{V}_{t,\cdot}^{\star}
\Vert_2^2 + \Vert ( \bm{U}_{1}^{\star \top} \bm{U}_{1}^{\star})^{-1}
\Vert \Vert^2 (\bm{\Sigma}^\star)^{-1} \Vert \Vert
\bm{U}_{i,\cdot}^{\star} \Vert_2^2 \Big).
\end{align*}
Now observe that $\Vert (\bm{\Sigma}^\star)^{-1} \Vert \leq
\frac{1}{\sigma_r^\star}$, and that the sub-block condition
condition~\eqref{EqnSubmatrix-general} ensures that
\begin{align*}
\Vert (\bm{V}_{1}^{\star \top} \bm{V}_{1}^{\star})^{-1} \Vert^2 \leq
\tfrac{1}{\clow^2} \tfrac{T^2}{T_1^2} \quad \mbox{and} \quad \Vert
(\bm{U}_{1}^{\star \top} \bm{U}_{1}^{\star})^{-1} \Vert^2 \leq
\tfrac{1}{\clow^2} \tfrac{N^2}{N_1^2}.
\end{align*}
Combining these facts, we find that
\begin{subequations}
  \begin{align}
    \label{EqnSierraOne}
\Delta^*_{i,t} & \leq \frac{\pi^2 \noise^{4} }{ \clow^2 \varepsilon^2
  \sigma_r^\star (NT)^2} \Big( \frac{T^2}{T_1^2} \Vert
\bm{V}_{t,\cdot}^{\star} \Vert_2^2 + \frac{N^2}{N_1^2} \Vert
\bm{U}_{i,\cdot}^{\star} \Vert_2^2 \Big).
\end{align}
On the other hand, using the $\cupper$-bound in the sub-block
condition~\eqref{EqnSubmatrix-general}, we have
\begin{align}
\label{EqnSierraTwo}  
  \frac{1}{\cupper} \noise^{2} \frac{1}{N_{1} T} \Vert
  \bm{U}_{i,\cdot}^{\star} \Vert_{2}^{2} + \frac{1}{\cupper}
  \noise^{2} \frac{1}{N T_{1}} \Vert \bm{V}_{t,\cdot}^{\star}
  \Vert_{2}^{2} \leq \gamma_{i,t}^\star.
\end{align}
\end{subequations}
Combining inequalities~\eqref{EqnSierraOne} and~\eqref{EqnSierraTwo}
yields the bound
\begin{align*}
\Delta^*_{i,t} & \leq \frac{\cupper \pi^2 \noise^{2}}{ \clow^2
  \varepsilon^2 \sigma_r^\star} \max \Big\{ \frac{1}{N T_1},
\frac{1}{N_1 T} \Big\} \gamma_{i,t}^\star \leq \frac{\cupper \pi^2 \noise^{2} }{ \clow^2 \sigma_r^{\star, 2}}
\frac{1}{\min \{T_1, N_1 \}} 
\underbrace{\frac{\sigma_r^\star}{\epsilon^2} \max \{ \frac{1}{N},
  \frac{1}{T} \}}_{=1} \, \gamma_{i,t}^\star \overset{\text{(i)}}{\leq} \frac{\cupper \pi^2 }{ \clow^2 }
\Enoise^2 \gamma_{i,t}^\star,
\end{align*}
where step (i) follows from our choice of $\varepsilon$, and the
definition of $\Enoise$.  This establishes the auxiliary
claim~\eqref{EqnAux}, and thus the claimed lower bound.


%

\section{Discussion}

In this paper, we proposed and analyzed a simple algorithm for
estimation and inference for panel data with missingness induced by
staggered adoption design.  It is appealing from a computational point
of view, since it is non-iterative, requiring only elementary matrix
operations and singular value decomposition.  At the same time, we
demonstrated that its statistical properties are also attractive, in
that it achieves estimation accuracy, in an elementwise sense, that
matches non-asymptotic lower bounds applicable to any estimator.
Moreover, we showed how this theory enables data-driven construction
of confidence intervals for the missing potential outcomes.  We note
that our development relies on a more general inferential toolbox for
the SVD algorithm in the matrix denoising model.  It introduces the
idea of ``leave-one-block-out'', an analysis technique that we suspect
will be useful for other problems.

Moving forward, let us outline some open questions, along with broader
directions for future study.  Our current theory relies on a noise
condition~\eqref{eq:noise-condition-est} and an incoherence
condition~\eqref{eq:incoherence}; it remains unclear whether or not
these conditions are improvable, for instance in terms of dependence
on the rank $r$.  Moreover, it would be interesting to extend our
guarantees to accommodate matrices that are only approximately
low-rank, along with more general noise models.  Finally, while the
the staggered adoption design arises frequently, it would be
interesting to understand optimal procedures for more general
mechanisms of missing data.


\subsection*{Acknowledgements}  
Y.~Yan was supported by the Norbert Wiener Postdoctoral Fellowship
from MIT.  M.~J.~Wainwright was partially supported by ONR grant
N00014-21-1-2842 and NSF grant DMS-2311072. We thank Eric Xia for
helpful discussions.


\printbibliography

\appendix


%

\section{Distributional theory without condition~\eqref{eq:signal-lb}}

In this section, we discuss how to characterize the output
distribution of~\Cref{alg:4-blocks-Md} without imposing the lower
bounds~\eqref{eq:signal-lb} on the local incoherence parameters.

\subsection{More careful control}

Recall the decomposition from~\Cref{lem:master}. Some computation
shows that the entrywise variance of $\bm{Z}$ takes the form
\begin{align}
  \mathsf{var} \left (Z_{i,t} \right ) & = \frac{\noise^{2}}{N T}
  \big\Vert \bm{V}_{t,\cdot}^{\star}( \bm{V}_{1}^{\star\top}
  \bm{V}_{1}^{\star})^{-1} \bm{V}_{1}^{\star\top} \big\Vert_{2}^{2} +
  \frac{\noise^{2}}{NT} \big\Vert \bm{U}_{i,\cdot}^{\star}(
  \bm{U}_{1}^{\star\top} \bm{U}_{1}^{\star})^{-1}
  \bm{U}_{1}^{\star\top} \big\Vert_{2}^{2}\nonumber \\
  & \quad + \frac{\noise^{4}}{N^2 T^2 } \mathsf{tr} \big[
    \bm{V}_{1}^{\star}( \bm{V}_{1}^{\star\top}
    \bm{V}_{1}^{\star})^{-1}( \bm{\Sigma}^{\star})^{-1}(
    \bm{U}_{1}^{\star\top} \bm{U}_{1}^{\star})^{-1}(
    \bm{\Sigma}^{\star})^{-1}( \bm{V}_{1}^{\star\top}
    \bm{V}_{1}^{\star})^{-1} \bm{V}_{1}^{\star\top}
    \big]. \label{eq:defn-z-bar}
\end{align}
The following lemma provides a lower bound on the variance of
$Z_{i,t}$.
\begin{lemma}
\label{lemma:variance-lb}
Under the sub-block condition~\eqref{EqnSubmatrix}, the variance of
$Z_{i,t}$ is lower bounded by
\begin{align*}
  \mathsf{var} \left (Z_{i,t} \right ) & \geq \frac{1}{\cupper}
  \frac{\noise^2}{N \Time_1} \left \Vert \bm{V}_{t,\cdot}^{\star}
  \right \Vert _{2}^{2} + \frac{1}{\cupper} \frac{\noise^2}{N_{1} T}
  \left \Vert \bm{U}_{i,\cdot}^{\star} \right \Vert _{2}^{2} +
  \frac{1}{\cupper^{2}}
  \frac{\noise^4}{\sigma_{r}^{\star2}}\frac{1}{NT \Nunit_1 \Time_1}.
\end{align*}
\end{lemma}
\noindent See~\Cref{subsec:proof-variance-lb} for the proof.

Combining the upper bound on $|\Delta_{i,t}|$ from~\Cref{lem:master}
and the lower bound on $\mathsf{var}(Z_{i,t})$
from~\Cref{lemma:variance-lb}, we can check that the size of
$|\Delta_{i,t}|$ in the decomposition in~\Cref{lem:master} is
dominated by the standard deviation of $Z_{i,j}$.  We summarize our
conclusion in the following:
\begin{lemma}
\label{lem:inference}
Under all conditions of~\Cref{thm:distribution} with the exception of
the lower bound~\eqref{eq:signal-lb}), the
decomposition~\eqref{eq:Md-decomposition} satisfies $|\Delta_{i,t}|
\leq \delta \sqrt{\mathsf{var}(Z_{i,t})}$ with probability at least
\mbox{$1 - O((N + T)^{-10})$.}
\end{lemma}
The above lemma shows that the error $(\bm{M}_{d}-
\bm{M}_{d}^{\star})_{i,t}$ is well-approximated by the variable
$Z_{i,t}$, and its distribution is fully determined by
equation~\eqref{eq:Z-defn}. Comparing with~\Cref{thm:distribution},
this lemma does not impose the lower bound~\eqref{eq:signal-lb} on
either of $\Vert \bm{U}_{i,\cdot}^{\star}\Vert_{2}$ and $\Vert
\bm{V}_{t,\cdot}^{\star}\Vert_{2}$, hence is more general and can be
useful especially when $\Vert \bm{U}_{i,\cdot}^{\star}\Vert_{2}$ and
$\Vert \bm{V}_{t,\cdot}^{\star}\Vert_{2}$ are vanishingly small.


\subsection{Proof of~\Cref{lemma:variance-lb}}
\label{subsec:proof-variance-lb}

Recall from equation~\eqref{eq:Zij-decom} the decomposition $Z_{i,t} =
\alpha_{i,t} + \beta_{i,t} + \hackgam_{i,t}$, along with the variances
$\alpha_{i,t}$ and $\beta_{i,t}$ given in
equation~\eqref{eq:alpha-beta-var}. Given the sub-block condition~\eqref{EqnSubmatrix}, they are lower bounded by
\begin{align*}
	\mathsf{var} \left (\alpha_{i,t} \right) \geq
        \frac{1}{\cupper} \frac{\noise^2}{ N T_{1}} \left \Vert
        \bm{V}_{t,\cdot}^{\star} \right\Vert _{2}^{2}, \qquad
        \mathsf{var} \left (\beta_{i,t} \right) \geq \frac{1}{\cupper}
        \frac{\noise^2}{N_{1} T} \left \Vert \bm{U}_{i,\cdot}^{\star}
        \right\Vert _{2}^{2}.
\end{align*}
Then we lower bound the variance of $\hackgam_{i,t}$. Recall
that
\begin{align*}
	\hackgam_{i,t} = ( \bm{E}_{c})_{i,\cdot} \bm{V}_{1}^{\star}(
	\bm{V}_{1}^{\star \top} \bm{V}_{1}^{\star})^{-1}(
	\bm{\Sigma}^{\star})^{-1}( \bm{U}_{1}^{\star \top}
	\bm{U}_{1}^{\star})^{-1} \bm{U}_{1}^{\star \top}(
	\bm{E}_{b})_{\cdot,t}.
\end{align*}
It is straightforward to check that $\hackgam_{i,t}$ is mean zero.
Since $\bm{E}_{c}$ and $ \bm{E}_{b}$ are independent, we have
\begin{align*}
	\mathsf{var}(\hackgam_{i,t}) & \overset{\text{(i)}}{ = }
	\frac{\noise^2}{NT} \mathbb{E} \big [( \bm{E}_{c})_{i,\cdot} \bm{V}_{1}^{\star}(
	\bm{V}_{1}^{\star \top} \bm{V}_{1}^{\star})^{-1}(
	\bm{\Sigma}^{\star})^{-1}( \bm{U}_{1}^{\star \top}
	\bm{U}_{1}^{\star})^{-1}( \bm{\Sigma}^{\star})^{-1}(
	\bm{V}_{1}^{\star \top} \bm{V}_{1}^{\star})^{-1}
	\bm{V}_{1}^{\star \top}( \bm{E}_{c})_{i,\cdot}^{\top} \big ]\\ &
	\overset{\text{(ii)}}{ = } \frac{\noise^4}{N^2 T^2} \mathsf{tr} \big [
	\bm{V}_{1}^{\star}( \bm{V}_{1}^{\star \top} \bm{V}_{1}^{\star})^{-1}(
	\bm{\Sigma}^{\star})^{-1}( \bm{U}_{1}^{\star \top}
	\bm{U}_{1}^{\star})^{-1}( \bm{\Sigma}^{\star})^{-1}(
	\bm{V}_{1}^{\star \top} \bm{V}_{1}^{\star})^{-1}
	\bm{V}_{1}^{\star \top} \big ] \\
	& \overset{\text{(iii)}}{ \geq } \frac{1}{\cupper}
	\frac{N}{N_{1}} \frac{\noise^4}{N^2 T^2} \mathsf{tr} \big [ \bm{V}_{1}^{\star}(
	\bm{V}_{1}^{\star \top} \bm{V}_{1}^{\star})^{-1}(
	\bm{\Sigma}^{\star})^{-2}( \bm{V}_{1}^{\star \top}
	\bm{V}_{1}^{\star})^{-1} \bm{V}_{1}^{\star \top} \big ] \\
	&  =  \frac{1}{\cupper} \frac{N}{N_{1}} \frac{\noise^4}{N^2 T^2} \mathsf{tr} \big [(
	\bm{\Sigma}^{\star})^{-2}( \bm{V}_{1}^{\star \top}
	\bm{V}_{1}^{\star})^{-1} \big ]
	\overset{\text{(iv)}}{ \geq }
	\frac{1}{\cupper^{2}} \frac{\noise^4}{N^2 T^2} \frac{NT}{N_{1}T_{1}}\mathsf{tr} \big [(
	\bm{\Sigma}^{\star})^{-2} \big ] 
	\geq  \frac{1}{\cupper^{2}}
	\frac{\noise^{4}}{\sigma_{r}^{\star2}} \frac{1}{NT N_{1}T_{1}}.
\end{align*}
Here (i) uses the fact that $\mathbb{E}[( \bm{E}_{b})_{\cdot,t}(
\bm{E}_{b})_{\cdot,t}^{\top}] = \sigma^{2} \bm{I}_{N_{1}}$, (ii)
utilizes $\mathbb{E}[( \bm{E}_{c})_{i,\cdot}^{\top}(
\bm{E}_{c})_{i,\cdot}] = \sigma^{2} \bm{I}_{T_{1}}$, while (iii) and
(iv) both follow from \eqref{EqnSubmatrix}.  Since $\alpha_{i,t}$, $\beta_{i,t}$ and $\hackgam_{i,t}$ are mutually
uncorrelated, we conclude that
\begin{align*}
	\mathsf{var} \left (Z_{i,t} \right) &  = \mathsf{var}(\alpha_{i,t}) +
	\mathsf{var}(\beta_{i,t}) + \mathsf{var}(\hackgam_{i,t}) \geq 
	\frac{1}{\cupper} \frac{\noise^2}{ N T_{1}} \left \Vert \bm{V}_{t,\cdot}^{\star} \right\Vert
	_{2}^{2} + \frac{1}{\cupper} \frac{\noise^2}{N_{1} T} \left \Vert
	\bm{U}_{i,\cdot}^{\star} \right\Vert _{2}^{2} + \frac{1}{\cupper^{2}} \frac{\noise^{4}}{\sigma_{r}^{\star2}} \frac{1}{NT N_{1}T_{1}}
\end{align*}
as claimed.

\section{Proofs of Lemmas from~\Cref{subsec:proof-lemma-master}}

This section is devoted to the proofs of the lemmas stated
in~\Cref{subsec:proof-lemma-master}.
\Cref{subsec:perturbation-bounds} provides some useful matrix/subspace
perturbation results, obtained by applying general results for matrix
denoising (cf.~\Cref{appendix:proof-thm-denoising}) to the causal
panel data model. We use these results to prove the lemmas stated
in~\Cref{subsec:proof-lemma-master}. 

Throughout this section, in addition to the rescaled noise level $\noise$ defined in \Cref{sec:model}, we will also use the standard deviation of the noise added on each observed entry of $\bm{M}^\star$, given by
\begin{equation}
\sigma \coloneqq \noise / \sqrt{NT}. \label{eq:def-sigma}
\end{equation}


\subsection{A few perturbation bounds}
\label{subsec:perturbation-bounds}

In order to apply our general theory for matrix denoising model to
$\bm{M}_{\mathsf{left}}$ and $\bm{M}_{\mathsf{upper}}$, we first need
to understand the spectra of the corresponding ground truth matrices
$\bm{M}_{\mathsf{left}}^{\star}$ and $\bm{M}_{\mathsf{upper}}^{\star}$.

\begin{lemma}
\label{lemma:submatrix-properties}
Under the sub-block condition \eqref{EqnSubmatrix}, we have
\begin{subequations}
\label{subeq:submatrix-spectrum}
\begin{align}
 \sqrt{\clow \frac{T_{1}}{T}}\sigma_{r}^{\star} \leq \sigma_{r} \left (
 \bm{M}_{\mathsf{left}}^{\star} \right) & \leq \sqrt{\cupper
   \frac{T_{1}}{T}}\sigma_{r}^{\star},\label{eq:submatrix-left-spectrum}\\
\label{eq:submatrix-upper-spectrum} 
\sqrt{\clow \frac{N_{1}}{N}}\sigma_{r}^{\star} \leq \sigma_{r} \left (
\bm{M}_{\mathsf{upper}}^{\star} \right) & \leq \sqrt{\cupper
  \frac{N_{1}}{N}}\sigma_{r}^{\star}.
\end{align}
\end{subequations}
\end{lemma}
\noindent See~\Cref{subsec:proof-lemma-submatrix} for the proof. \\

Define the matrices $\bm{E}_{\mathsf{left}} = \bm{M}_{\mathsf{left}}-
\bm{M}_{\mathsf{left}}^{\star}$ and $ \bm{E}_{\mathsf{upper}} =
\bm{M}_{\mathsf{upper}}- \bm{M}_{\mathsf{upper}}^{\star}$.  Recall that the
truncated rank-$r$ SVD of $ \bm{M}_{\mathsf{left}}$ was assumed to be
$ \bm{U}_{\mathsf{left}} \bm{\Sigma}_{\mathsf{left}}
\bm{V}_{\mathsf{left}}^{\top}$ in the main text. However, we simplify
the notation in the proof by dropping the subscript: in particular, we
let $\bm{U} \bm{\Sigma} \bm{V}^{\top}$ be the truncated rank-$r$ SVD
of $\bm{M}_{\mathsf{left}}$.  We also let $ \bm{H} = \mathsf{sgn}(
\bm{U}^{\top} \bm{U}^{\star})$ be the optimal rotation matrix that
aligns $\bm{U}$ and $ \bm{U}^{\star}$.  We have the following results
for the subspace estimation error $ \bm{U} \bm{H}- \bm{U}^{\star}$.

\begin{lemma}
\label{lemma:subspace-error-1}
Under the conditions of \Cref{lem:master}, we have the decomposition
\begin{align*}
 \bm{U} \bm{H}- \bm{U}^{\star} &  = \bm{Z} + \bm{\Psi}, \quad 
   \text{where} \quad \bm{Z} \coloneqq \bm{E}_{\mathsf{left}}
   \bm{V}_{1}^{\star}  ( \bm{V}_{1}^{\star \top}
   \bm{V}_{1}^{\star} )^{-1} \left ( \bm{\Sigma}^{\star}
   \right)^{-1}, 
\end{align*}
and there exists universal constant $C_6>0$ such that the matrix $\bm{\Psi}$ satisfies the bounds
\begin{align}
\label{eq:Psi-i-bound}   \hspace{-1ex}
 \left \Vert \bm{\Psi}_{i,\cdot} \right\Vert _{2} \hspace{-0.2ex} \leq \hspace{-0.2ex} C_6 \bigg[
 \frac{\sigma^{2} \sqrt{T_{1} \left (r + \spicy
     \right)}}{\sigma_{r}^{\star2}T_{1}/T} \hspace{-0.1ex} + \hspace{-0.1ex} \frac{\sigma^{3}N \sqrt{r
     + \spicy}}{(\sigma_{r}^{\star} \sqrt{T_{1}/T})^{3}} \hspace{-0.1ex} + \hspace{-0.1ex} \bigg(
 \frac{\sigma^{2} \left (N + T_{1}
   \right)}{\sigma_{r}^{\star2}T_{1}/T} + \frac{\sigma \sqrt{r +
     \spicy}}{\sigma_{r}^{\star} \sqrt{T_{1}/T}}\bigg) \left \Vert
 \bm{U}_{i,\cdot}^{\star} \right\Vert _{2} \hspace{-0.5ex} \bigg]
\end{align}
with probability at least $1-O((N + T)^{-9})$, for all $i \in [N]$.
As a consequence, with probability at least $1-O((N + T)^{-9})$, the
row-wise error bound
\begin{align}
\label{eq:UH-U-star-i}  
 \left \Vert \bm{U}_{i,\cdot} \bm{H}- \bm{U}_{i,\cdot}^{\star}
 \right\Vert _{2} \leq C_6 \frac{\sigma \sqrt{r +
     \spicy}}{\sigma_{r}^{\star} \sqrt{T_{1}/T}} + C_6 \frac{\sigma^{2}
   \left (N + T_{1} \right)}{\ensuremath{\sigma_{r}^{\star2}}T_{1}/T}
 \left \Vert \bm{U}_{i,\cdot}^{\star} \right\Vert _{2} 
\end{align}
holds for all $i \in [N]$, and the submatrix $ \bm{U}_{1}$ satisfies
the following spectral norm error bound 
\begin{align}
\label{eq:U1-H-U1-star-spectral}  
 \left \Vert \bm{U}_{1} \bm{H}- \bm{U}_{1}^{\star} \right\Vert &
 \leq C_6 \frac{\sigma^{2}T}{\ensuremath{\sigma_{r}^{\star2}}} \sqrt{
   \frac{N_{1}}{N}} + C_6 \frac{\sigma \sqrt{N_{1}}}{\sigma_{r}^{\star}
   \sqrt{T_{1}/T}}.
\end{align}
As a byproduct, the spectrum of $ \bm{U}_{1}$ is upper and lower
bounded by
\begin{align}
\label{eq:U1-spectrum}  
 \frac{\clow}{2} \frac{N_{1}}{N} \bm{I}_{r} \preceq \bm{U}_{1}^{\top}
 \bm{U}_{1} \preceq 2\cupper \frac{N_{1}}{N} \bm{I}_{r}.
\end{align}
\end{lemma}
\noindent See \Cref{subsec:proof-lemma-subspace-1} for the proof.

We also need the following lemma that studies the proximity of $\bm{U}
\bm{\Sigma}$ and $ \bm{U}^{\star} \bm{\Sigma}^{\star}$. While it looks
similar to the previous one, it is crucial to establish our result
that is condition number free (i.e.~does not depend on the condition
number of $ \bm{\Sigma}^{\star}$).

\begin{lemma}
  \label{lemma:subspace-error-2}
Under the conditions of \Cref{lem:master}, we have the decomposition
\begin{align*}
 \bm{U} \bm{\Sigma} \bm{R}- \bm{U}^{\star} \bm{\Sigma}^{\star} & =
 \bm{E}_{\mathsf{left}} \bm{V}_{1}^{\star}( \bm{V}_{1}^{\star \top}
 \bm{V}_{1}^{\star})^{-1} + \bm{\Delta} = \bm{Z} \bm{\Sigma}^{\star} +
 \bm{\Delta},
\end{align*}
where $\bm{R} \in \mathbb{R}^{r\times r}$ is some invertible matrix
and $\bm{\Delta}$ is some residual matrix satisfying: there exists some universal constant $C_7>0$ such that with probability
at least $1-O((N + T)^{-9})$, for all $i \in [N]$,
\begin{align}
 \left \Vert \bm{\Delta}_{i,\cdot} \right\Vert _{2} & \leq C_7 \bigg[
 \frac{\sigma^{2} \sqrt{T_{1} \left (r + \spicy
     \right)}}{\sigma_{r}^{\star}T_{1}/T} + \frac{\sigma^{3}N \sqrt{r
     + \spicy}}{\sigma_{r}^{\star2} \left (T_{1}/T \right)^{3/2}} +
 \frac{\sigma^{2} \left (N + T_{1} \right)}{\sigma_{r}^{\star}T_{1}/T}
 \left \Vert \bm{U}_{i,\cdot}^{\star} \right\Vert
 _{2} \bigg].\label{eq:Delta-i-bound}
\end{align}
As a consequence, with probability at least $1-O((N + T)^{-9})$, the
row-wise error bound
\begin{align}
 \left \Vert \bm{U}_{i,\cdot} \bm{\Sigma} \bm{R}-
 \bm{U}_{i,\cdot}^{\star} \bm{\Sigma}^{\star} \right\Vert _{2} &
 \leq C_7 \bigg[ \frac{\sigma \sqrt{r + \spicy}}{ \sqrt{T_{1}/T}} +
 \frac{\sigma^{2} \left (N + T_{1} \right)}{\sigma_{r}^{\star}T_{1}/T}
 \left \Vert \bm{U}_{i,\cdot}^{\star} \right\Vert
 _{2} \bigg] \label{eq:USigmaH-i}
\end{align}
holds for all $i \in [N]$, and the submatrix $ \bm{U}_{1} \bm{\Sigma}$
satisfies the following spectral norm error bound
\begin{align}
 \left \Vert \bm{U}_{1} \bm{\Sigma} \bm{R}- \bm{U}_{1}^{\star}
 \bm{\Sigma}^{\star} \right\Vert & \leq C_7 \bigg[ \frac{\sigma
   \sqrt{N_{1}}}{ \sqrt{T_{1}/T}} +
 \frac{\sigma^{2}T}{\text{\ensuremath{\sigma_{r}^{\star}}}} \sqrt{
   \frac{N_{1}}{N}} \bigg] .\label{eq:U1Sigma-spectral}
\end{align}
We also have the following spectral norm bound on the submatrix $
\bm{\Delta}_{1}\coloneqq \bm{\Delta}_{1 : N_1, \cdot}$
\begin{align}
 \left \Vert \bm{\Delta}_{1} \right\Vert \leq C_7 \frac{\sigma^{2}
   \left (N + T_{1}
   \right)}{\text{\ensuremath{\sigma_{r}^{\star}T_{1}/T}}} \sqrt{
   \frac{N_{1}}{N}}.\label{eq:Delta-1-bound}
\end{align}
\end{lemma}
\noindent See \Cref{subsec:proof-lemma-subspace-2} for the proof.


The next lemma characterizes the error associated with estimating
$\bm{M}_{b}^{\star}$.

\begin{lemma}\label{lemma:denoising-error}
  Under the conditions of \Cref{lem:master},  we can decompose
\begin{align*}
 \bm{M}_{b}- \bm{M}_{b}^{\star} & = \bm{W} + \bm{\Phi}, \quad \text{where} \quad  \bm{W} \coloneqq \bm{E}_{a}
 	\bm{V}_{1}^{\star} \bm{V}_{2}^{\star \top} + \bm{E}_{b}
 	\bm{V}_{2}^{\star} \bm{V}_{2}^{\star \top} + \bm{U}_{1}^{\star}(
 	\bm{U}_{1}^{\star \top} \bm{U}_{1}^{\star})^{-1} \bm{U}_{1}^{\star
 		\top} \bm{E}_{b}.
\end{align*}
Here $ \bm{\Phi}$ is some residual matrix satisfying: there exists some universal constant $C_8>0$ such that with
probability at least $1-O((N + T)^{-9})$, for all $t \in [T]$,
\begin{align*}
 \left \Vert \bm{\Phi}_{\cdot,t} \right\Vert _{2} & \leq C_8 \bigg[
 \frac{\sigma^{2} \sqrt{N \left (r + \spicy
     \right)}}{\sigma_{r}^{\star}} + \frac{\sigma^{3}T \sqrt{r +
     \spicy}}{\sigma_{r}^{\star2}N_{1}/N} + \bigg( \frac{\sigma^{2}
     \left (N_{1} + T \right)}{\sigma_{r}^{\star} \sqrt{N_{1}/N}} +
   \sigma \sqrt{r + \spicy}\bigg) \left \Vert \bm{V}_{t,\cdot}^{\star}
 \right\Vert _{2} \bigg] .
\end{align*}
As a consequence, with probability at least $1-O((N + T)^{-9})$, the
column-wise error bound
\begin{align*}
\big\Vert( \bm{M}_{b}- \bm{M}_{b}^{\star})_{\cdot,t}\big\Vert_{2}
\leq C_8 \sigma \sqrt{r + \spicy} + C_8 \bigg(
  \frac{\sigma^{2}T}{\sigma_{r}^{\star} \sqrt{N_{1}/N}} + \sigma
  \sqrt{N_{1}}\bigg) \left \Vert \bm{V}_{t,\cdot}^{\star} \right\Vert
_{2} 
\end{align*}
holds for all $t \in [T]$.
\end{lemma}
\noindent See \Cref{subsec:proof-lemma-denoising} for the proof.

Finally, we collect some results in the next lemma that will be useful
in later proof.

\begin{lemma}
\label{lemma:U1-properties}
Under the conditions of \Cref{lem:master},  there exists universal constant $C_9>0$ such that with probability at least $1-O((N + T)^{-10})$, we have
\begin{align}
\big \Vert \bm{U}_{1}^{\star \top} \left ( \bm{U}_{1} \bm{H}-
\bm{U}_{1}^{\star} \right) \big \Vert \leq  C_9  \frac{\sigma \sqrt{r +
    \spicy}}{\sigma_{r}^{\star} \sqrt{T_{1}/T}} \sqrt{\frac{N_{1}}{N}}
+ C_9 \frac{\sigma^{2} \left( N + T_{1} \right)}{\sigma_{r}^{\star 2}
  T_{1}/T} \frac{N_{1}}{N} \coloneqq \delta_{\mathsf{prod}}.
\label{eq:U1-star-U1-err-inner}
\end{align}
As a direct consequence, the following two bounds also hold:
\begin{align}
\label{eq:U1topU1-inv-spectral-error}   
 \big \Vert \big [( \bm{U}_{1} \bm{H})^{\top} \bm{U}_{1} \bm{H} \big
 ]^{-1} - \big ( \bm{U}_{1}^{\star \top} \bm{U}_{1}^{\star} \big
 )^{-1} \big \Vert & \leq \delta_{\mathsf{prod}}
 \frac{N^{2}}{N_{1}^{2}}, \\
\label{eq:U1topU1-inv-U1top-spectral} 
\big \Vert \big [( \bm{U}_{1} \bm{H})^{\top} \bm{U}_{1} \bm{H} \big
]^{-1}( \bm{U}_{1} \bm{H})^{\top}-( \bm{U}_{1}^{\star \top}
\bm{U}_{1}^{\star})^{-1} \bm{U}_{1}^{\star \top} \big \Vert & \leq C_9
\frac{\sigma^{2} T}{\sigma_{r}^{\star 2}} \sqrt{ \frac{N}{N_{1}}} + C_9
\frac{\sigma \sqrt{N_{1}}}{\sigma_{r}^{\star} \sqrt{T_{1}/T}}
\frac{N}{N_{1}}.
\end{align}
\end{lemma}
\noindent
See~\Cref{subsec:proof-lemma-U1-properties} for the proof.



\subsection{Proof of~\Cref{lemma:decom-A}}
\label{subsec:proof-decom-A}

It is convenient to define the matrix
\begin{align}
\bm{L} \coloneqq \bm{U}_{2} \bm{H} \big [ \left ( \bm{U}_{1} \bm{H}
  \right)^{\top} \bm{U}_{1} \bm{H} \big ]^{-1} \left ( \bm{U}_{1}
\bm{H} \right)^{\top}- \bm{U}_{2}^{\star} ( \bm{U}_{1}^{\star
  \top} \bm{U}_{1}^{\star} )^{-1} \bm{U}_{1}^{\star
  \top}. \label{eq:defn-L}
\end{align}
Recall the decomposition $ \bm{U}_{2} \bm{H} =  \bm{U}_{2}^{\star} +
\bm{Z}_{2} + \bm{\Psi}_{2}$ from Lemma~\ref{lemma:subspace-error-1},
we can write
\begin{align}
 \bm{L} &  =  \bm{Z}_{2} ( \bm{U}_{1}^{\star \top}
 \bm{U}_{1}^{\star} )^{-1} \bm{U}_{1}^{\star \top} +
 \bm{\Delta}_{ \bm{L}}\label{eq:L-decom}
\end{align}
where we define 
\begin{align*}
 \bm{\Delta}_{ \bm{L}}\coloneqq \bm{\Psi}_{2}( \bm{U}_{1}^{\star \top}
 \bm{U}_{1}^{\star})^{-1} \bm{U}_{1}^{\star \top} + \bm{U}_{2} \bm{H}
 \big \{ \big [( \bm{U}_{1} \bm{H})^{\top} \bm{U}_{1} \bm{H} \big
 ]^{-1}( \bm{U}_{1} \bm{H})^{\top}-( \bm{U}_{1}^{\star \top}
 \bm{U}_{1}^{\star})^{-1} \bm{U}_{1}^{\star} \big \}.
\end{align*}
The matrix $ \bm{\Delta}_{ \bm{L}}$ can be bounded as follows: for
each $i \in [N_{2}]$, we have
\begin{align}
 & \big \Vert( \bm{\Delta}_{ \bm{L}})_{i,\cdot} \big \Vert_{2}
  \leq \big \Vert( \bm{\Psi}_{2})_{i,\cdot} \big
  \Vert_{2} \sigma_{r}^{-1} \left ( \bm{U}_{1}^{\star} \right) + \left
  \Vert \bm{U}_{i,\cdot} \right\Vert _{2} \big \Vert \big [(
    \bm{U}_{1} \bm{H})^{\top} \bm{U}_{1} \bm{H} \big ]^{-1}(
  \bm{U}_{1} \bm{H})^{\top}-( \bm{U}_{1}^{\star \top}
  \bm{U}_{1}^{\star})^{-1} \bm{U}_{1}^{\star \top} \big \Vert
  \nonumber \\
%
& \quad \overset{\text{(i)}}{ \leq } \frac{C_6}{\clow} \bigg[ \frac{\sigma^{2}
      \sqrt{T_{1} \left (r + \spicy
        \right)}}{\sigma_{r}^{\star2}T_{1}/T} + \frac{\sigma^{3}N
      \sqrt{r + \spicy}}{(\sigma_{r}^{\star} \sqrt{T_{1}/T})^{3}} +
    \bigg( \frac{\sigma^{2} \left (N + T_{1}
      \right)}{\sigma_{r}^{\star2}T_{1}/T} + \frac{\sigma \sqrt{r +
        \spicy}}{\sigma_{r}^{\star} \sqrt{T_{1}/T}} \bigg) \left \Vert
    \bm{U}_{i,\cdot}^{\star} \right\Vert _{2} \bigg] \sqrt{
    \frac{N}{N_{1}}}\nonumber \\ & \quad \qquad + \big ( 2 \left \Vert
  \bm{U}_{i,\cdot}^{\star} \right\Vert _{2} + C_6 \frac{\sigma \sqrt{r +
      \spicy}}{\sigma_{r}^{\star} \sqrt{T_{1}/T}} \big ) \cdot C_9 \big (
  \frac{\sigma^{2}T}{\sigma_{r}^{\star2}} \sqrt{ \frac{N}{N_{1}}} +
  \frac{\sigma \sqrt{N_{1}}}{\sigma_{r}^{\star} \sqrt{T_{1}/T}}
  \frac{N}{N_{1}} \big ) \label{eq:Delta-L-bound} \\
& \quad \overset{\text{(ii)}}{ \leq}   C_L \bigg ( \frac{\sigma^{2} \sqrt{r + \spicy}
  }{\sigma_{r}^{\star2} \sqrt{T_{1}/T}} \sqrt{
  \frac{NT}{N_{1}}}  + \frac{ \sigma^{2} }{\sigma_{r}^{\star2}}\frac{NT}{T_1} \sqrt{\frac{r + \spicy}{N_1}} \bigg) +  C_L \left \Vert
\bm{U}_{i,\cdot}^{\star} \right\Vert _{2} \bigg
(\frac{\sigma^{2}T}{\sigma_{r}^{\star2}}  + 
\frac{\sigma}{\sigma_{r}^{\star}} \sqrt{\frac{NT}{T_1}} \bigg)
\sqrt{\frac{N}{N_{1}}}, \nonumber
\end{align}
where we take $C_L = 2C_6 / \clow + 2C_6 C_9 +4 C_9$. Here (i) follows from \eqref{EqnSubmatrix}, relations (\ref{eq:Psi-i-bound}), 
(\ref{eq:U1topU1-inv-U1top-spectral}), and a direct consequence of (\ref{eq:UH-U-star-i}):
\begin{align}
\label{eq:U-i-bound}  
 \left \Vert \bm{U}_{i,\cdot} \right\Vert _{2} \leq \ \left \Vert
 \bm{U}_{i,\cdot}^{\star} \right\Vert _{2} + \left \Vert
 \bm{U}_{i,\cdot} \bm{H}- \bm{U}_{i,\cdot}^{\star} \right\Vert
\overset{\text{(iv)}}{\leq} 2 \left \Vert \bm{U}_{i,\cdot}^{\star} \right\Vert _{2} + C_6
 \frac{\sigma \sqrt{r + \spicy}}{\sigma_{r}^{\star} \sqrt{T_{1}/T}};
\end{align}
steps (ii) and (iii) holds provided that
$\sigma \sqrt{NT/T_{1}} \leq \Cnoise \sigma_{r}^{\star}$ and $\sigma \sqrt{T}
\leq \Cnoise \sigma_{r}^{\star}$ for some sufficiently small $\Cnoise>0$.  We decompose
\begin{align*}
 \bm{A} & = \bm{L} \left ( \bm{M}_{b}- \bm{M}_{b}^{\star} \right) +
 \bm{U}_{2}^{\star}( \bm{U}_{1}^{\star \top} \bm{U}_{1}^{\star})^{-1}
 \bm{U}_{1}^{\star \top} \left ( \bm{M}_{b}- \bm{M}_{b}^{\star}
 \right)\\ & = \bm{Z}_{2}( \bm{U}_{1}^{\star \top}
 \bm{U}_{1}^{\star})^{-1} \bm{U}_{1}^{\star \top} \left ( \bm{W} +
 \bm{\Phi} \right) + \bm{\Delta}_{ \bm{L}} \left ( \bm{M}_{b}-
 \bm{M}_{b}^{\star} \right) + \bm{U}_{2}^{\star}( \bm{U}_{1}^{\star
   \top} \bm{U}_{1}^{\star})^{-1} \bm{U}_{1}^{\star \top} \left (
 \bm{W} + \bm{\Phi} \right) \\
& = \underbrace{ \bm{U}_{2}^{\star}( \bm{U}_{1}^{\star \top}
   \bm{U}_{1}^{\star})^{-1} \bm{U}_{1}^{\star \top} \bm{E}_{b} +
   \bm{E}_{c} \bm{V}_{1}^{\star}( \bm{V}_{1}^{\star \top}
   \bm{V}_{1}^{\star})^{-1}( \bm{\Sigma}^{\star})^{-1}(
   \bm{U}_{1}^{\star \top} \bm{U}_{1}^{\star})^{-1} \bm{U}_{1}^{\star
     \top} \bm{E}_{b}}_{\eqqcolon \bm{A}_{1}} \\
& \qquad + \underbrace{ \bm{E}_{c} \bm{V}_{1}^{\star}(
   \bm{V}_{1}^{\star \top} \bm{V}_{1}^{\star})^{-1}(
   \bm{\Sigma}^{\star})^{-1}( \bm{U}_{1}^{\star \top}
   \bm{U}_{1}^{\star})^{-1} \bm{U}_{1}^{\star \top}
   \bm{E}_{\mathsf{upper}} \bm{V}^{\star} \bm{V}_{2}^{\star
     \top}}_{\eqqcolon \bm{A}_{2}} + \underbrace{ \bm{U}_{2}^{\star}(
   \bm{U}_{1}^{\star \top} \bm{U}_{1}^{\star})^{-1} \bm{U}_{1}^{\star
     \top} \bm{E}_{\mathsf{upper}} \bm{V}^{\star} \bm{V}_{2}^{\star
     \top}}_{\eqqcolon \bm{A}_{3}} \\
 & \qquad + \underbrace{ \bm{\Delta}_{ \bm{L}} \left ( \bm{M}_{b}-
   \bm{M}_{b}^{\star} \right)}_{\eqqcolon \bm{A}_{4}} + \underbrace{
   \left ( \bm{U}_{2}^{\star} + \bm{Z}_{2} \right)( \bm{U}_{1}^{\star
     \top} \bm{U}_{1}^{\star})^{-1} \bm{U}_{1}^{\star \top}
   \bm{\Phi}}_{\eqqcolon \bm{A}_{5}},
\end{align*}
where the second relation follows from~\Cref{lemma:denoising-error}
and~\eqref{eq:L-decom}. The matrix $ \bm{A}_{1}$ appears in our
desired decomposition~\eqref{eq:decom-A}. Then we bound the entrywise
magnitude of the other four matrices.

We start with $\bm{A}_{2}$. Notice that $\bm{E}_{c}$ and
$\bm{E}_{\mathsf{upper}}$ are independent, then conditional on $
\bm{E}_{\mathsf{upper}}$, 
\begin{align*}
( \bm{A}_{2})_{i,t}\,|\, \bm{E}_{\mathsf{upper}}\sim\mathcal{N} \big
  (0,\sigma^{2} \big \Vert \bm{V}_{1}^{\star}( \bm{V}_{1}^{\star \top}
  \bm{V}_{1}^{\star})^{-1}( \bm{\Sigma}^{\star})^{-1}(
  \bm{U}_{1}^{\star \top} \bm{U}_{1}^{\star})^{-1} \bm{U}_{1}^{\star
    \top} \bm{E}_{\mathsf{upper}} \bm{V}^{\star} \bm{V}_{t,\cdot}^{\star
    \top} \big \Vert_{2}^{2} \big ).
\end{align*}
Therefore, with probability at least $1-O((N + T)^{-10})$, we have
\begin{align*}
 \big |( \bm{A}_{2})_{i,t} \big | & \leq 5 \sigma \big \Vert
 \bm{V}_{1}^{\star}( \bm{V}_{1}^{\star \top} \bm{V}_{1}^{\star})^{-1}(
 \bm{\Sigma}^{\star})^{-1}( \bm{U}_{1}^{\star \top}
 \bm{U}_{1}^{\star})^{-1} \bm{U}_{1}^{\star \top} \bm{E}_{\mathsf{upper}}
 \bm{V}^{\star} \bm{V}_{t,\cdot}^{\star \top} \big \Vert_{2}
 \sqrt{\spicy} \\
 & \overset{\text{(i)}}{\leq} 5 C_g \sigma^{2} \sqrt{r + \spicy} \big \Vert
 \bm{V}_{1}^{\star}( \bm{V}_{1}^{\star \top} \bm{V}_{1}^{\star})^{-1}(
 \bm{\Sigma}^{\star})^{-1}( \bm{U}_{1}^{\star \top}
 \bm{U}_{1}^{\star})^{-1} \bm{U}_{1}^{\star \top} \big \Vert \big
 \Vert \bm{V}^{\star} \bm{V}_{t,\cdot}^{\star \top} \big \Vert
 \sqrt{\spicy} \\
 & \overset{\text{(ii)}}{\leq} \widetilde{C}_2 \frac{\sigma^{2}}{\sigma_{r}^{\star}} \sqrt{ \frac{N
     T}{N_{1}T_{1}}\spicy \left (r + \spicy \right)} \left \Vert
 \bm{V}_{j, \cdot}^{\star} \right\Vert _{2},
\end{align*}
where $\widetilde{C}_2 = 5C_g / \clow $. Here (i) follows from~\Cref{lemma:gaussian-spectral}, and (ii) holds due to \eqref{EqnSubmatrix}.
Regarding $\bm{A}_{3}$, one can check that
\begin{align*}
( \bm{A}_{3})_{i,t}\sim\mathcal{N} \big (0,\sigma^{2} \big \Vert
  \bm{U}_{i,\cdot}^{\star}( \bm{U}_{1}^{\star \top}
  \bm{U}_{1}^{\star})^{-1} \bm{U}_{1}^{\star \top} \big \Vert_{2}^{2}
  \big \Vert \bm{V}_{t,\cdot}^{\star} \bm{V}^{\star \top} \big
  \Vert_{2}^{2} \big ),
\end{align*}
hence with probability at least $1-O((N + T)^{-10})$, 
\begin{align*}
 \big |( \bm{A}_{3})_{i,t} \big | & \leq 5 \sigma \big \Vert
 \bm{U}_{i,\cdot}^{\star}( \bm{U}_{1}^{\star \top}
 \bm{U}_{1}^{\star})^{-1} \bm{U}_{1}^{\star \top} \big \Vert_{2} \big
 \Vert \bm{V}_{t,\cdot}^{\star} \bm{V}^{\star \top} \big \Vert_{2}
 \sqrt{\spicy} \leq \widetilde{C}_3\sigma \sqrt{ \frac{N}{N_{1}}\spicy} \left \Vert
 \bm{U}_{i,\cdot}^{\star} \right\Vert _{2} \left \Vert
 \bm{V}_{t,\cdot}^{\star} \right\Vert _{2},
\end{align*}
where $\widetilde{C}_3 = 5/\sqrt{\clow}$, and we use \eqref{EqnSubmatrix} in the last step.
Then we bound $ \bm{A}_{4}$, where we use the Cauchy--Schwarz
inequality to achieve
\begin{align*}
 & \big |( \bm{A}_{4})_{i,t} \big | \leq \big \Vert( \bm{\Delta}_{
    \bm{L}})_{i,\cdot} \big \Vert_{2} \big \Vert( \bm{M}_{b}-
  \bm{M}_{b}^{\star})_{\cdot,t} \big \Vert_{2}\\ 
 & \quad \leq
  \widetilde{C}_4 \bigg ( \frac{\sigma^{2} \sqrt{r + \spicy}
  }{\sigma_{r}^{\star2} \sqrt{T_{1}/T}} \sqrt{
  	\frac{NT}{N_{1}}}  + \frac{ \sigma^{2} }{\sigma_{r}^{\star2}}\frac{NT}{T_1} \sqrt{\frac{r + \spicy}{N_1}} \bigg) \bigg[\sigma \sqrt{r + \spicy} +
    \bigg ( \frac{\sigma^{2}T}{\sigma_{r}^{\star} \sqrt{N_{1}/N}} +
    \sigma \sqrt{N_{1}} \bigg) \left \Vert \bm{V}_{t,\cdot}^{\star}
    \right\Vert _{2} \bigg ] \\
  & \quad \quad + \widetilde{C}_4 \left \Vert
  \bm{U}_{i,\cdot}^{\star} \right\Vert _{2} \bigg
  (\frac{\sigma^{2}T}{\sigma_{r}^{\star2}}  + 
  \frac{\sigma}{\sigma_{r}^{\star}} \sqrt{\frac{NT}{T_1}} \bigg)
  \sqrt{\frac{N}{N_{1}}} \bigg[\sigma \sqrt{r +
      \spicy} + \bigg( \frac{\sigma^{2}T}{\sigma_{r}^{\star}
      \sqrt{N_{1}/N}} + \sigma \sqrt{N_{1}} \bigg) \left \Vert
    \bm{V}_{t,\cdot}^{\star} \right\Vert _{2} \bigg] \\
  & \quad \leq \widetilde{C}_4 \Big( \frac{\sigma^{3}T}{\sigma_{r}^{\star2}} \sqrt{
    \frac{N}{N_{1}T_{1}}} + \frac{\sigma^{3}
    \sqrt{N_{1}}}{\sigma_{r}^{\star2}} \frac{NT}{N_{1}T_{1}} \Big) \left (r
  + \spicy \right) + \widetilde{C}_4 \bigg[ \frac{\sigma^{2} \sqrt{r +
        \spicy }}{\sigma_{r}^{\star} \sqrt{N_{1} / N}}
    \sqrt{\frac{N T}{T_{1}} } + \frac{\sigma^{3} T \sqrt{r +
        \spicy}}{\sigma_{r}^{\star2} \sqrt{N_{1}/N}} \bigg] \left
  \Vert \bm{U}_{i,\cdot}^{\star} \right\Vert _{2}\\ & \quad \quad + \widetilde{C}_4
  \bigg( \frac{\sigma^{2}}{\sigma_{r}^{\star}} \frac{NT}{N_{1}} +
  \sigma \sqrt{N} \bigg) \bigg[ \frac{\sigma^{2} \sqrt{ \left (N +
        T_{1} \right) \left (r + \spicy
        \right)}}{\sigma_{r}^{\star2}T_{1}/T} \left \Vert
    \bm{V}_{t,\cdot}^{\star} \right\Vert _{2} + \bigg(
    \frac{\sigma^{2}T}{\sigma_{r}^{\star2}} +
    \frac{\sigma}{\sigma_{r}^{\star}} \sqrt{ \frac{NT}{T_{1}}} \bigg)
    \left \Vert \bm{U}_{i,\cdot}^{\star} \right\Vert _{2} \left \Vert
    \bm{V}_{t,\cdot}^{\star} \right\Vert _{2} \bigg].
\end{align*}
where $\widetilde{C}_4 = C_L C_8$. Here we use the bounds (\ref{eq:Delta-L-bound})
and~\Cref{lemma:denoising-error}. Finally we
use~\Cref{lemma:subspace-error-1,lemma:denoising-error} to reach
\begin{align}
 & \big |( \bm{A}_{5})_{i,t} \big | \leq \big (\Vert
  \bm{U}_{i,\cdot}^{\star}\Vert_{2} + \big \Vert \bm{E}_{i,\cdot}
  \bm{V}_{1}^{\star}( \bm{V}_{1}^{\star \top}
  \bm{V}_{1}^{\star})^{-1}( \bm{\Sigma}^{\star})^{-1} \big \Vert_{2}
  \big )\sigma_{r}^{-1} \left ( \bm{U}_{1}^{\star}
  \right) \Vert \bm{\Phi}_{ \cdot , t} \Vert_2 \nonumber \\
\label{eq:A5-ij-inter}   
& \quad \overset{ \text{(i)} }{\leq} \Big ( \left \Vert \bm{U}_{i,\cdot}^{\star}
\right\Vert _{2} + \frac{C_g}{\sqrt{\clow}} \frac{\sigma \sqrt{r + \spicy}}{\sigma_{r}^{\star}}
\sqrt{ \frac{T}{T_{1}}} \Big ) \frac{1}{\sqrt{\clow}} \sqrt{ \frac{N}{N_{1}}} \Vert \bm{\Phi}_{ \cdot , t} \Vert_2 \\
& \quad \overset{ \text{(ii)} }{\leq} \widetilde{C}_5 \Big ( \frac{\sigma^{3}
	\sqrt{N}}{\sigma_{r}^{\star2}} +
\frac{\sigma^{4}}{\sigma_{r}^{\star3}} \frac{NT}{N_{1}} \Big ) \sqrt{
	\frac{N T}{N_{1}T_{1}}} \left (r + \spicy \right) + \widetilde{C}_5 \bigg[ \frac{\sigma^{2} \sqrt{N \left (r + \spicy
      \right)}}{\sigma_{r}^{\star} \sqrt{N_{1}/N}} + \frac{\sigma^{3}T
    \sqrt{r + \spicy}}{\sigma_{r}^{\star2}(N_{1}/N)^{3/2}} \bigg]
\left \Vert \bm{U}_{i,\cdot}^{\star} \right\Vert _{2} \nonumber \\
& \quad \qquad + \widetilde{C}_5 \bigg[
  \frac{\sigma^{3} \left (N_{1} + T \right) \sqrt{r +
      \spicy}}{\sigma_{r}^{\star2} \sqrt{N_{1}/N}} + \frac{\sigma^{2}
    \left (r + \spicy \right)}{\sigma_{r}^{\star}} \bigg] \sqrt{
  \frac{NT}{N_{1}T_{1}}} \left \Vert \bm{V}_{t,\cdot}^{\star} \right
\Vert _{2} \nonumber \\
& \quad \qquad  + \widetilde{C}_5 \bigg[
  \frac{\sigma^{2} \left (N_{1} + T
    \right)}{\sigma_{r}^{\star}N_{1}/N} + \frac{\sigma \sqrt{r +
      \spicy}}{ \sqrt{N_{1}/N}} \bigg] \left \Vert
\bm{U}_{i,\cdot}^{\star} \right\Vert _{2} \left \Vert
\bm{V}_{t,\cdot}^{\star} \right\Vert _{2}, \nonumber
\end{align}
where $\widetilde{C}_5 = 4 / \sqrt{\clow} + 4 C_8 C_g / \clow$. Here  (i) holds due to an application of
\Cref{lemma:gaussian-spectral} and \eqref{EqnSubmatrix}:
\begin{align*}
 \big \Vert \bm{E}_{i,\cdot} \bm{V}_{1}^{\star}( \bm{V}_{1}^{\star
   \top} \bm{V}_{1}^{\star})^{-1}( \bm{\Sigma}^{\star})^{-1} \big
 \Vert_{2} \leq C_g \sigma \big \Vert \bm{V}_{1}^{\star}(
 \bm{V}_{1}^{\star \top} \bm{V}_{1}^{\star})^{-1}(
 \bm{\Sigma}^{\star})^{-1} \big \Vert \sqrt{r + \spicy} \leq \frac{C_g}{\sqrt{\clow}}
 \frac{\sigma \sqrt{r + \spicy}}{\sigma_{r}^{\star}} \sqrt{
   \frac{T}{T_{1}}},
\end{align*}
while (ii) follows from \Cref{lemma:denoising-error} and holds true when $\Cnoise>0$ is sufficiently small. 

Taking the above bounds on $ \bm{A}_{2}$ to $ \bm{A}_{5}$
collectively, we conclude that $ \bm{A} = \bm{A}_{1} + \bm{\xi}_{\bm{A}}$, where
\begin{align*}
 & \big |( \bm{\xi}_{ \bm{A}})_{i,t} \big | \leq \big |(
 \bm{A}_{2})_{i,t} \big | + \big |( \bm{A}_{3})_{i,t} \big | + \big |(
 \bm{A}_{4})_{i,t} \big | + \big |( \bm{A}_{5})_{i,t} \big | \\
 & \quad \leq C_A
	\frac{\sigma^{3}T}{\sigma_{r}^{\star2}} \sqrt{\frac{N}{\Nunit_1
			\Time_1}} \left (r + \spicy \right ) + C_A \frac{\sigma^{3}
		\sqrt{N_{1}}}{\sigma_{r}^{\star2}}\frac{\Nunit \Time}{\Nunit_1
		\Time_1} \left (r + \spicy \right ) + C_A
	\frac{\sigma^{4}}{\sigma_{r}^{\star3}} \sqrt{\frac{\Nunit
			\Time}{\Nunit_1 \Time_1}} \frac{\Nunit \Time}{N_{1}} \left (r +
	\spicy \right ) \nonumber \\
	& \quad \quad +C_A  \bigg [\frac{\sigma^{2} \sqrt{N_{1} \left (r + \spicy
			\right )}}{\sigma_{r}^{\star} \sqrt{\Time_1/T}}\frac{N}{N_{1}}
	+ \frac{\sigma^{3}T \sqrt{r +
			\spicy}}{\sigma_{r}^{\star2}(N_{1}/N)^{3/2}} \bigg ] \left
	\Vert \bm{U}_{i,\cdot}^{\star} \right \Vert _{2}\nonumber \\
	& \quad \quad + C_A \bigg [\frac{\sigma^{2}}{\sigma_{r}^{\star}}
	\sqrt{\frac{\Nunit \Time \left (r + \spicy \right )}{\Nunit_1
			\Time_1}} + \frac{\sigma^{2} \sqrt{ \left (\Nunit + \Time_1
			\right )}}{\sigma_{r}^{\star2}\Time_1/T}
	\big(\frac{\sigma^{2}}{\sigma_{r}^{\star}}\frac{\Nunit
		\Time}{N_{1}} + \sigma \sqrt{N} \big) + \frac{\sigma^{3}\Nunit
		\Time}{\sigma_{r}^{\star2}N_{1}} \sqrt{\frac{T}{\Time_1}} \bigg ]
	\left \Vert \bm{V}_{t,\cdot}^{\star} \right \Vert _{2} \sqrt{r +
		\spicy}\nonumber \\
	& \quad \quad + C_A \bigg [\sigma \sqrt{\frac{N}{N_{1}} \left (r + \spicy \right
		)} + \frac{\sigma^{2}\Nunit \Time}{\sigma_{r}^{\star}N_{1}} +
	\big(\frac{\sigma^{2}T}{\sigma_{r}^{\star2}} +
	\frac{\sigma}{\sigma_{r}^{\star}} \sqrt{\frac{\Nunit
			\Time}{\Time_1}} \big) \sigma \sqrt{N} \bigg ] \left \Vert
	\bm{U}_{i,\cdot}^{\star} \right \Vert _{2} \left \Vert
	\bm{V}_{t,\cdot}^{\star} \right \Vert _{2}.
\end{align*}
Here we take $C_A = 2\sum_{k=2}^5 \widetilde{C}_k$ and the above relation holds as long as $\sigma\sqrt{NT/T_1} \leq \Cnoise \sigma_r^\star$ and $\sigma\sqrt{T}\leq \Cnoise \sigma_r^\star$ for some sufficiently small constant $\Cnoise>0$. 

\subsection{Proof of \Cref{lemma:decom-B}}
\label{subsec:proof-decom-C}
We begin with the decomposition
\begin{align*}
\bm{B} & = \bm{U}_{2}( \bm{U}_{1}^{\top} \bm{U}_{1})^{-1}
\bm{U}_{1}^{\top}( \bm{U}_{1} \bm{\Sigma} \bm{R}- \bm{U}_{1}^{\star}
\bm{\Sigma}^{\star}) \bm{V}_{2}^{\star\top} + (\bm{U}_{2} \bm{\Sigma}
\bm{R}- \bm{U}_{2}^{\star} \bm{\Sigma}^{\star}) \bm{V}_{2}^{\star\top}
\\
& \overset{\text{(i)}}{=} \bm{U}_{2}^{\star}( \bm{U}_{1}^{\star \top}
\bm{U}_{1}^{\star})^{-1} \bm{U}_{1}^{\star \top}( \bm{U}_{1}
\bm{\Sigma} \bm{R}- \bm{U}_{1}^{\star} \bm{\Sigma}^{\star})
\bm{V}_{2}^{\star \top} + \bm{L}( \bm{U}_{1} \bm{\Sigma} \bm{R}-
\bm{U}_{1}^{\star} \bm{\Sigma}^{\star}) \bm{V}_{2}^{\star \top} +
(\bm{U}_{2} \bm{\Sigma} \bm{R}- \bm{U}_{2}^{\star}
\bm{\Sigma}^{\star}) \bm{V}_{2}^{\star\top} \\
& \overset{\text{(ii)}}{=} \underbrace{ \bm{U}_{2}^{\star}(
  \bm{U}_{1}^{\star \top} \bm{U}_{1}^{\star})^{-1} \bm{U}_{1}^{\star
    \top} \bm{E}_{a} \bm{V}_{1}^{\star}( \bm{V}_{1}^{\star \top}
  \bm{V}_{1}^{\star})^{-1} \bm{V}_{2}^{\star \top}}_{\eqqcolon
  \bm{B}_{1}} + \underbrace{ \bm{L}( \bm{U}_{1} \bm{\Sigma} \bm{R}-
  \bm{U}_{1}^{\star} \bm{\Sigma}^{\star}) \bm{V}_{2}^{\star
    \top}}_{\eqqcolon \bm{B}_{2}} \\
& \qquad + \underbrace{ \left ( \bm{U}_{2}^{\star} + \bm{Z}_{2}
  \right)( \bm{U}_{1}^{\star \top} \bm{U}_{1}^{\star})^{-1}
  \bm{U}_{1}^{\star \top} \bm{\Delta}_{1} \bm{V}_{2}^{\star
    \top}}_{\eqqcolon \bm{B}_{3}} + + \underbrace{\bm{E}_c
  \bm{V}_1^\star (\bm{V}_1^{\star \top} \bm{V}_1^\star)^{-1}
  \bm{V}_2^\star }_{\eqqcolon \bm{B}_4} + \underbrace{ \bm{\Delta}_2
  \bm{V}_2^\star}_{\eqqcolon \bm{B}_5}.
\end{align*}
Here step (i) follows from the definition~\eqref{eq:defn-L} of
$\bm{L}$, while step (ii) follows from~\Cref{lemma:subspace-error-2}.

We first bound~$ \bm{B}_{1}$, whose entries follow the Gaussian
distribution
\begin{align*}
( \bm{B}_{1})_{i,t} \sim \mathcal{N} \left (0,\sigma^{2} \big \Vert
  \bm{U}_{i,\cdot}^{\star}( \bm{U}_{1}^{\star \top}
  \bm{U}_{1}^{\star})^{-1} \bm{U}_{1}^{\star \top} \big \Vert_{2}^{2}
  \big \Vert \bm{V}_{t,\cdot}^{\star}( \bm{V}_{1}^{\star \top}
  \bm{V}_{1}^{\star})^{-1} \bm{V}_{1}^{\star \top} \big \Vert_{2}^{2}
  \right).
\end{align*}
Therefore we can use \eqref{EqnSubmatrix} to show that with probability at least $1-O((N + T)^{-10})$, 
\begin{align*}
 \big |( \bm{B}_{1})_{i,t} \big | & \leq 5 \sigma \big \Vert
 \bm{U}_{i,\cdot}^{\star} (\bm{U}_{1}^{\star \top}
 \bm{U}_{1}^{\star})^{-1} \bm{U}_{1}^{\star \top} \big \Vert_{2} \big
 \Vert \bm{V}_{t,\cdot}^{\star} (\bm{V}_{1}^{\star \top}
 \bm{V}_{1}^{\star})^{-1} \bm{V}_{1}^{\star \top} \big \Vert_{2}
 \sqrt{\spicy} \leq \widetilde{C}_1 \sigma \sqrt{ \frac{NT}{N_{1}T_{1}}\spicy}
 \Vert \bm{U}_{i,\cdot}^{\star} \Vert_{2}\Vert
 \bm{V}_{t,\cdot}^{\star} \Vert_{2}, 
\end{align*}
where $\widetilde{C}_1=5/\clow$. Regarding $ \bm{C}_{2}$, we first derive a
bound for $\Vert \bm{L}_{i,\cdot}\Vert_{2}$. Starting from
(\ref{eq:L-decom}), we have
\begin{align}
  \Vert \bm{L}_{i,\cdot}\Vert_{2} & \leq \Vert \bm{U}_{i,\cdot} \bm{H} - \bm{U}_{i,\cdot}^\star \Vert_2 \sigma_{r}^{-1}
  \left ( \bm{U}_{1}^{\star} \right) + \left \Vert \bm{U}_{i,\cdot}
  \right\Vert _{2} \big \Vert \big [( \bm{U}_{1} \bm{H})^{\top}
    \bm{U}_{1} \bm{H} \big ]^{-1}( \bm{U}_{1} \bm{H})^{\top}-(
  \bm{U}_{1}^{\star \top} \bm{U}_{1}^{\star})^{-1} \bm{U}_{1}^{\star}
  \big \Vert\nonumber \\
  & \overset{\text{(i)}}{\leq} C_6 \Big ( \frac{\sigma \sqrt{r +
      \spicy}}{\sigma_{r}^{\star} \sqrt{T_{1}/T}} + \frac{\sigma^{2}
    \left (N + T_{1}
    \right)}{\ensuremath{\sigma_{r}^{\star2}}T_{1}/T}\Vert
  \bm{U}_{i,\cdot}^{\star}\Vert_{2} \Big ) \cdot \clow^{-1/2} \sqrt{ \frac{N}{N_{1}}} \nonumber \\
  & \qquad + 
  \Big ( 2 \Vert \bm{U}_{i,\cdot}^{\star}\Vert_{2} + C_6 \frac{\sigma \sqrt{r
      + \spicy}}{\sigma_{r}^{\star} \sqrt{T_{1}/T}} \Big ) \cdot C_9 \Big (
  \frac{\sigma^{2}T}{\sigma_{r}^{\star2}} +
  \frac{\sigma}{\sigma_{r}^{\star}} \sqrt{ \frac{NT}{T_{1}}} \Big )
  \sqrt{ \frac{N}{N_{1}}}\nonumber \\
\label{eq:L-bound}  
  & \overset{\text{(ii)}}{\leq} C_L' \left \Vert \bm{U}_{i,\cdot}^{\star} \right\Vert _{2}
\big ( \frac{\sigma^{2}T}{\sigma_{r}^{\star2}} +
\frac{\sigma}{\sigma_{r}^{\star}} \sqrt{ \frac{NT}{T_{1}}} \big )
\sqrt{ \frac{N}{N_{1}}} + C_L' \frac{\sigma \sqrt{r +
    \spicy}}{\sigma_{r}^{\star} \sqrt{T_{1}/T}} \sqrt{
  \frac{N}{N_{1}}},
\end{align}
where we take $C_L' = 2C_6 / \sqrt{\clow} + 4C_9 + 2C_6 C_9$. Here (i) follows from (\ref{eq:Psi-i-bound}), (\ref{eq:U-i-bound}), 
\eqref{EqnSubmatrix} and (\ref{eq:U1topU1-inv-U1top-spectral}),
while (ii) holds provided that $\sigma \sqrt{NT/T_{1}}
\leq \Cnoise \sigma_{r}^{\star}$ and $\sigma \sqrt{T} \leq \Cnoise
\sigma_{r}^{\star}$ for some sufficiently small $\Cnoise>0$. Apply the bounds (\ref{eq:L-bound}) and (\ref{eq:U1Sigma-spectral}) to achieve
\begin{align*}
 \big |( \bm{B}_{2})_{i,t} \big | & \leq \big \Vert \bm{L}_{i,\cdot}
 \big \Vert_{2} \left \Vert \bm{U}_{1} \bm{\Sigma} \bm{R}-
 \bm{U}_{1}^{\star} \bm{\Sigma}^{\star} \right\Vert \left \Vert
 \bm{V}_{t,\cdot}^{\star} \right\Vert _{2} \\
 & \leq \widetilde{C}_2 \big (
 \frac{\sigma^{2}T}{\ensuremath{\sigma_{r}^{\star2}}} +
 \frac{\sigma}{\sigma_{r}^{\star}} \sqrt{ \frac{NT}{T_{1}}} \big )^{2}
 \sigma_{r}^{\star} \left \Vert \bm{U}_{i,\cdot}^{\star} \right \Vert
 _{2} \left \Vert \bm{V}_{t,\cdot}^{\star} \right\Vert _{2} + \widetilde{C}_2
 \frac{\sigma \sqrt{r + \spicy}}{ \sqrt{T_{1}/T}} \big (
 \frac{\sigma^{2}T}{\ensuremath{\sigma_{r}^{\star2}}} +
 \frac{\sigma}{\sigma_{r}^{\star}} \sqrt{ \frac{NT}{T_{1}}} \big )
 \left \Vert \bm{V}_{t,\cdot}^{\star} \right\Vert _{2},
\end{align*}
where $\widetilde{C}_2=C_L' C_7$.
Then we move on to bounding $ \bm{B}_{3}$, where we have
\begin{align*}
 \big |( \bm{B}_{3})_{i,t} \big | & \leq  \left \Vert
 \bm{U}_{i,\cdot}^{\star} +\bm{Z}_{i,\cdot} \right\Vert _{2} 
 \sigma_{r}^{-1} \left ( \bm{U}_{1}^{\star} \right) \left \Vert
 \bm{\Delta}_{1} \right\Vert \left \Vert \bm{V}_{t,\cdot}^{\star}
 \right\Vert _{2}\\ 
 &\overset{\text{(i)}}{\leq} \big ( \left \Vert
 \bm{U}_{i,\cdot}^{\star} \right\Vert _{2} + \frac{C_g}{\sqrt{\clow}} \frac{\sigma \sqrt{r +
     \spicy}}{\sigma_{r}^{\star}} \sqrt{ \frac{T}{T_{1}}} \big
 ) \cdot \clow^{-1/2} \sqrt{\frac{N}{N_1}} \cdot C_7 \frac{\sigma^{2}
 	\left (N + T_{1}
 	\right)}{\text{\ensuremath{\sigma_{r}^{\star}T_{1}/T}}} \sqrt{
 	\frac{N_{1}}{N}} \left \Vert \bm{V}_{t,\cdot}^{\star}
 \right\Vert _{2}\\ 
 & \overset{\text{(ii)}}{\leq} \widetilde{C}_3 \frac{\sigma^{2} \left (N + T_{1}
   \right)}{\text{\ensuremath{\sigma_{r}^{\star}T_{1}/T}}} \left \Vert
 \bm{U}_{i,\cdot}^{\star} \right\Vert _{2} \left \Vert
 \bm{V}_{t,\cdot}^{\star} \right\Vert _{2} + \widetilde{C}_3 \frac{\sigma^{3} \left (N
   + T_{1}
   \right)}{\text{\ensuremath{\sigma_{r}^{\star2}(T_{1}/T)^{3/2}}}}
 \sqrt{r + \spicy} \left \Vert \bm{V}_{t,\cdot}^{\star} \right\Vert
 _{2}.
\end{align*}
Here (i) follows from
(\ref{eq:Delta-1-bound}), \eqref{EqnSubmatrix}, and also uses a similar upper bound for $\Vert(
\bm{U}_{2}^{\star} + \bm{Z}_{2})_{i,\cdot}\Vert_{2}$ as in
(\ref{eq:A5-ij-inter}); while (ii) holds true if we take $\widetilde{C}_3 = C_g C_7 / \clow + C_7 / \sqrt{\clow}$. 
The matrix  $\bm{B}_4$ appears in our desired decomposition \eqref{eq:decom-B}. Finally we invoke \Cref{lemma:subspace-error-2} to show that
\begin{align*}
	\big | ( \bm{B}_5 )_{i,t} \big | & \leq C_7  \bigg
	(\frac{\sigma^{2} \sqrt{\Time_1 \left (r + \spicy \right
			)}}{\sigma_{r}^{\star}\Time_1/T} + \frac{\sigma^{3}N \sqrt{r +
			\spicy}}{\sigma_{r}^{\star2} \left (\Time_1/T \right )^{3/2}}
	\bigg ) \left \Vert \bm{V}_{t,\cdot}^{\star} \right \Vert _{2} + C_7
	\frac{\sigma^{2} \left (\Nunit + \Time_1 \right)}{\sigma_{r}^{\star}
		\Time_1/T} \left \Vert \bm{U}_{i,\cdot}^{\star} \right \Vert _{2}
	\left \Vert \bm{V}_{t,\cdot}^{\star} \right \Vert _{2}.
\end{align*}
Taking the bounds on $ \bm{B}_{1}$, $ \bm{B}_{2}$, $ \bm{B}_{3}$ and $\bm{B}_5$ 
collectively yields $\bm{B}=\bm{B}_4+\bm{\xi}_{\bm{B}}$ where $\bm{\xi}_{\bm{B}}=\bm{B}_1+\bm{B}_2+\bm{B}_3+\bm{B}_5$. The residual matrix $\bm{\xi}_{\bm{B}}$ satisfies
\begin{align*}
 &\big | (\bm{\xi}_{\bm{B}})_{i,t} \big |  \leq \big |( \bm{B}_{1})_{i,t} \big | + \big
 |( \bm{B}_{2})_{i,t} \big | + \big |( \bm{B}_{3})_{i,t} \big | + \big |( \bm{B}_{5})_{i,t} \big | \\
 & \quad \leq C_B \frac{\sigma \sqrt{r +
 		\spicy}}{\sigma_{r}^{\star} \sqrt{\Time_1/T}} \bigg
 (\sigma \sqrt{T} + \sigma
 \sqrt{\frac{\Nunit \Time}{\Time_1}} \bigg ) \left \Vert
 \bm{V}_{t,\cdot}^{\star} \right \Vert _{2} + C_B \bigg (\frac{\sigma^{2}
 	\left (\Nunit + \Time_1 \right
 	)}{\text{\ensuremath{\sigma_{r}^{\star}\Time_1/T}}} + \sigma
 \sqrt{\frac{ \Nunit \Time \spicy}{\Nunit_1 \Time_1}} \bigg ) \left
 \Vert \bm{U}_{i,\cdot}^{\star} \right \Vert _{2} \left \Vert
 \bm{V}_{t,\cdot}^{\star} \right \Vert _{2},
\end{align*}
where we take $C_B=2\widetilde{C}_1+2\widetilde{C}_2+2\widetilde{C}_3+2C_7$. The above relation holds provided that $\sigma \sqrt{NT/T_{1}} \leq \Cnoise
\sigma_{r}^{\star}$ and $\sigma\sqrt{T} \leq \Cnoise \sigma_r^\star$ for some sufficiently small $\Cnoise>0$.

\subsection{Proof of \Cref{lemma:submatrix-properties}}
\label{subsec:proof-lemma-submatrix}

It suffices to upper and lower bound the $r$-th eigenvalue of $
\bm{M}_{\mathsf{left}}^{\star} \bm{M}_{\mathsf{left}}^{\star \top} =
\bm{U}^{\star} \bm{\Sigma}^{\star} \bm{V}_{1}^{\star \top}
\bm{V}_{1}^{\star} \bm{\Sigma}^{\star} \bm{U}^{\star \top}$.On the one
hand, by the max-min characterization for eigenvalues, we know that
\begin{align}
\lambda_{r}( \bm{M}_{\mathsf{left}}^{\star}
\bm{M}_{\mathsf{left}}^{\star \top}) & = \max_{S \subset
  \mathbb{R}^{N}:\mathrm{dim} \left (S \right) = r}\min_{ \bm{x} \in
  S, \left \Vert \bm{x} \right\Vert _{2} = 1} \bm{x}^{\top}
\bm{M}_{\mathsf{left}}^{\star} \bm{M}_{\mathsf{left}}^{\star \top}
\bm{x} \overset{\text{(i)}}{\geq} \min_{ \bm{\beta} \in
  \mathbb{R}^{r}: \left \Vert \bm{\beta} \right\Vert _{2} = 1}
\bm{\beta}^{\top} \bm{U}^{\star \top} \bm{M}_{\mathsf{left}}^{\star}
\bm{M}_{\mathsf{left}}^{\star \top} \bm{U}^{\star} \bm{\beta}\nonumber
\\
\label{eq:submatrix-interm-1}
& = \min_{ \bm{\beta} \in \mathbb{R}^{r}: \left \Vert \bm{\beta}
  \right \Vert _{2} = 1} \bm{\beta}^{\top} \bm{\Sigma}^{\star}
\bm{V}_{1}^{\star \top} \bm{V}_{1}^{\star} \bm{\Sigma}^{\star}
\bm{\beta} \overset{\text{(ii)}}{ \geq} \clow \frac{T_{1}}{T}\min_{
  \bm{\beta} \in \mathbb{R}^{r}: \left \Vert \bm{\beta} \right \Vert
  _{2} = 1} \left \Vert \bm{\Sigma}^{\star} \bm{\beta} \right\Vert
_{2}^{2} = \clow \frac{T_{1}}{T} \sigma_{r}^{\star 2}.
\end{align}
Here step (i) holds by taking the subspace $S$ in the first line to be
the subspace spanned by the columns of $ \bm{U}^{\star}$, which allows
one to parameterize any unit vector $ \bm{x} \in \mathcal{S}$ with $
\bm{U}^{\star} \bm{\beta}$ for some unit vector $ \bm{\beta} \in
\mathbb{R}^{r}$; and step (ii) follows again from
\eqref{EqnSubmatrix}. On the other hand, by the
min-max characterization for eigenvalues, we know that
\begin{align*}
\lambda_{r}( \bm{M}_{\mathsf{left}}^{\star}
\bm{M}_{\mathsf{left}}^{\star \top}) & = \min_{S \subset
  \mathbb{R}^{N}:\mathrm{dim} \left (S \right) = N-r + 1}\max_{ \bm{x}
  \in S, \left \Vert \bm{x} \right\Vert _{2} = 1} \bm{x}^{\top}
\bm{M}_{\mathsf{left}}^{\star} \bm{M}_{\mathsf{left}}^{\star \top}
\bm{x}.
\end{align*}
By taking $S$ to be the subspace spanned by the orthogonal complement
of $ \bm{U}^{\star}$ and the singular vector $
\bm{U}_{\cdot,r}^{\star}$ associated with the minimum singular value
$\sigma_{r}^{\star}$, namely
\begin{align*}
S = \mathsf{span} \left \{ \bm{U}_{\perp}^{\star},
\bm{U}_{\cdot,r}^{\star} \right\} ,
\end{align*}
we can express any unit vector $ \bm{x}$ in $S$ as
\begin{align*}
 \bm{x} = ( \bm{I}- \bm{U}^{\star} \bm{U}^{\star \top}) \bm{\beta} +
 \alpha \bm{U}_{\cdot,r}^{\star}
\end{align*}
for some vector $\bm{\beta} \in \mathbb{R}^{N}$ and scalar $\alpha \in
\mathbb{R}$ that satisfy $\Vert( \bm{I}- \bm{U}^{\star} \bm{U}^{\star
  \top}) \bm{\beta}\Vert_{2}^{2} + \alpha^{2} = 1$.  We can also
compute
\begin{align*}
 \bm{x}^{\top} \bm{M}_{\mathsf{left}}^{\star}
 \bm{M}_{\mathsf{left}}^{\star \top} \bm{x} & = \big [( \bm{I}-
   \bm{U}^{\star} \bm{U}^{\star \top}) \bm{\beta} + \alpha
   \bm{U}_{\cdot,r}^{\star} \big ]^{\top} \bm{U}^{\star}
 \bm{\Sigma}^{\star} \bm{V}_{1}^{\star \top} \bm{V}_{1}^{\star}
 \bm{\Sigma}^{\star} \bm{U}^{\star \top} \big [( \bm{I}-
   \bm{U}^{\star} \bm{U}^{\star \top}) \bm{\beta} + \alpha
   \bm{U}_{\cdot,r}^{\star} \big ]^{\top}\\ & = \alpha^{2}
 \bm{U}_{\cdot,r}^{\star \top} \bm{U}^{\star} \bm{\Sigma}^{\star}
 \bm{V}_{1}^{\star \top} \bm{V}_{1}^{\star} \bm{\Sigma}^{\star}
 \bm{U}^{\star \top} \bm{U}_{\cdot,r}^{\star} =
 \alpha^{2}\sigma_{r}^{\star2} \bm{e}_{r}^{\top} \bm{V}_{1}^{\star
   \top} \bm{V}_{1}^{\star} \bm{e}_{r} \leq \alpha^{2} \cupper
 \frac{T_{1}}{T}\sigma_{r}^{\star2}.
\end{align*}
Here $ \bm{e}_{r}$ is the $r$-th canonical basis vector in
$\mathbb{R}^{r}$, and we use \eqref{EqnSubmatrix} in
the last relation.  Then we can bound
\begin{align}
\lambda_{r}( \bm{M}_{\mathsf{left}}^{\star}
\bm{M}_{\mathsf{left}}^{\star \top}) \leq \max_{\alpha \in
  [-1,1]}\alpha^{2} \cupper \frac{T_{1}}{T}\sigma_{r}^{\star2} = \cupper
\frac{T_{1}}{T}\sigma_{r}^{\star2}.\label{eq:submatrix-interm-2}
\end{align}
Taking equations~(\ref{eq:submatrix-interm-1})
and~(\ref{eq:submatrix-interm-2}) together gives the desired
result~(\ref{eq:submatrix-left-spectrum}).  The remaining
relation~\eqref{eq:submatrix-upper-spectrum} can be proved in a
similar way.


\subsection{Proof of~\Cref{lemma:subspace-error-1}}
\label{subsec:proof-lemma-subspace-1}


We apply~\Cref{prop:denoising_main_1} to the observation
$\bm{M}_{\mathsf{left}} = \bm{M}_{\mathsf{left}}^{\star} +
\bm{E}_{\mathsf{left}}$, where we assume that the SVD of
$\bm{M}_{\mathsf{left}}^{\star}$ is given by
$\bm{U}_{\mathsf{left}}^{\star} \bm{\Sigma}_{\mathsf{left}}^{\star}
\bm{V}_{\mathsf{left}}^{\star \top}$.  In order to ensure that the
condition of~\Cref{prop:denoising_main_1} is satisfied, we check that
\begin{align*}
\sigma \sqrt{\max\{ N, T_{1} \}} \leq  \Cnoise \sqrt{T_1 / T}  \overset{\text{(i)}}{\leq} \cnoise \sqrt{\clow T_1 / T} \overset{\text{(ii)}}{\leq} \cnoise \sigma_{r}(
\bm{M}_{\mathsf{left}}^{\star}).
\end{align*}
Here (i) holds as long as $\Cnoise \leq \cnoise \sqrt{c}$, while (ii) follows from
\Cref{lemma:submatrix-properties}.  We take $\mathcal{I} = \{i\}$
for any $i \in [N]$. Then (\ref{eq:denoising-Psi-U}) tells us that,
with probability at least $1-O \big ((N + T)^{-9} \big )$,
\begin{align}
 \bm{U}\mathsf{sgn} \big ( \bm{U}^{\top}
 \bm{U}_{\mathsf{left}}^{\star} \big )- \bm{U}_{\mathsf{left}}^{\star}
 = \bm{E}_{\mathsf{left}} \bm{V}_{\mathsf{left}}^{\star}(
 \bm{\Sigma}_{\mathsf{left}}^{\star})^{-1} +
 \bm{\Psi}_{\mathsf{left}}\label{eq:proof-lemma-subspace-interm-1}
\end{align}
where for each $i \in [N]$, 
\begin{align*}
 & \big \Vert( \bm{\Psi}_{\mathsf{left}})_{i,\cdot} \big \Vert_{2}   \leq c_1 \frac{\sigma \sqrt{r + \spicy}}{\sigma_{r}( \bm{M}_{\mathsf{left}}^{\star})} \bigg( \frac{\sigma \sqrt{T_{1}}}{\sigma_{r}( \bm{M}_{\mathsf{left}}^{\star})} +  \frac{\sigma^{2}N}{\sigma_{r}^{2}( \bm{M}_{\mathsf{left}}^{\star})} \bigg) + c_1  \bigg( \frac{\sigma^{2} \left (N + T_{1} \right)}{\text{\ensuremath{\sigma_{r}^{2}( \bm{M}_{\mathsf{left}}^{\star})}}} +  \frac{\sigma \sqrt{r + \spicy}}{\sigma_{r}( \bm{M}_{\mathsf{left}}^{\star})} \bigg) \big \Vert( \bm{U}_{\mathsf{left}}^{\star})_{i,\cdot} \big \Vert_{2}\\
 & \quad   \leq c_1 \clow^{-3/2} \bigg[ \frac{\sigma^{2} \sqrt{T_{1} \left (r + \spicy \right)}}{\sigma_{r}^{\star2}T_{1}/T} +  \frac{\sigma^{3}N \sqrt{r + \spicy}}{ \big (\sigma_{r}^{\star} \sqrt{T_{1}/T} \big )^{3}} +  \left ( \frac{\sigma^{2} \left (N + T_{1} \right)}{\sigma_{r}^{\star2}T_{1}/T} +  \frac{\sigma \sqrt{r + \spicy}}{\sigma_{r}^{\star} \sqrt{T_{1}/T}} \right) \left \Vert  \bm{U}_{i,\cdot}^{\star} \right\Vert _{2} \bigg].
\end{align*}
Here we makes use of \Cref{lemma:submatrix-properties}.
Since $ \bm{U}_{\mathsf{left}}^{\star}$ and $ \bm{U}^{\star}$
have the same column subspace, there exists a rotation matrix $ \bm{O}_{\mathsf{left}} \in \mathcal{O}^{r\times r}$
such that $ \bm{U}_{\mathsf{left}}^{\star} \bm{O}_{\mathsf{left}} =  \bm{U}^{\star}$.
In addition, since $ \bm{V}_{\mathsf{left}}^{\star}$ and $ \bm{V}_{1}^{\star}$
have the same column subspace, there exists an invertible matrix $ \bm{R}_{\mathsf{left}} \in \mathbb{R}^{r\times r}$
such that $ \bm{V}_{\mathsf{left}}^{\star} \bm{R}_{\mathsf{left}} =  \bm{V}_{1}^{\star}$.
Note that $ \bm{R}_{\mathsf{left}}$ is in general not a rotation
matrix because $ \bm{V}_{1}^{\star}$ does not necessarily have orthonormal
columns. We can express $ \bm{M}_{\mathsf{left}}^{\star}$ in two ways,
either $ \bm{U}^{\star} \bm{\Sigma}^{\star} \bm{V}_{1}^{\star \top}$
or $ \bm{U}_{\mathsf{left}}^{\star} \bm{\Sigma}_{\mathsf{left}}^{\star} \bm{V}_{\mathsf{left}}^{\star \top}$,
which further gives 
\begin{align}
 \bm{\Sigma}_{\mathsf{left}}^{\star} =  \bm{U}_{\mathsf{left}}^{\star \top} \bm{M}_{\mathsf{left}}^{\star} \bm{V}_{\mathsf{left}}^{\star} =  \bm{U}_{\mathsf{left}}^{\star \top} \bm{U}^{\star} \bm{\Sigma}^{\star} \bm{V}_{1}^{\star \top} \bm{V}_{\mathsf{left}}^{\star} =  \bm{O}_{\mathsf{left}} \bm{\Sigma}^{\star} \bm{V}_{1}^{\star \top} \bm{V}_{\mathsf{left}}^{\star}.\label{eq:proof-lemma-subspace-interm-2}
\end{align}
Notice that $ \bm{H} = \mathsf{sgn}( \bm{U}^{\top} \bm{U}^{\star}) = \mathsf{sgn}( \bm{U}^{\top} \bm{U}_{\mathsf{left}}^{\star}) \bm{O}_{\mathsf{left}}$.
Then we have 
\begin{align*}
 \bm{U} \bm{H}- \bm{U}^{\star} &  =  \bm{U}\mathsf{sgn}( \bm{U}^{\top} \bm{U}_{\mathsf{left}}^{\star}) \bm{O}_{\mathsf{left}}- \bm{U}^{\star} =  \bm{E}_{\mathsf{left}} \bm{V}_{\mathsf{left}}^{\star}( \bm{\Sigma}_{\mathsf{left}}^{\star})^{-1} \bm{O}_{\mathsf{left}} +  \bm{\Psi}_{\mathsf{left}} \bm{O}_{\mathsf{left}}\\
 &  \overset{\text{(i)}}{=}  \bm{E}_{\mathsf{left}} \bm{V}_{\mathsf{left}}^{\star}( \bm{O}_{\mathsf{left}} \bm{\Sigma}^{\star} \bm{V}_{1}^{\star \top} \bm{V}_{\mathsf{left}}^{\star})^{-1} \bm{O}_{\mathsf{left}} +  \bm{\Psi}_{\mathsf{left}} \bm{O}_{\mathsf{left}}\\
 &  \overset{\text{(ii)}}{=}  \bm{E}_{\mathsf{left}} \bm{V}_{\mathsf{left}}^{\star}( \bm{V}_{1}^{\star \top} \bm{V}_{\mathsf{left}}^{\star})^{-1}( \bm{\Sigma}^{\star})^{-1} +  \bm{\Psi}_{\mathsf{left}} \bm{O}_{\mathsf{left}}\\
 &  =  \bm{E}_{\mathsf{left}} \bm{V}_{1}^{\star} \bm{R}_{\mathsf{left}}^{-1}( \bm{V}_{1}^{\star \top} \bm{V}_{1}^{\star} \bm{R}_{\mathsf{left}}^{-1})^{-1}( \bm{\Sigma}^{\star})^{-1} +  \bm{\Psi}_{\mathsf{left}} \bm{O}_{\mathsf{left}}\\
 &  =  \bm{E}_{\mathsf{left}} \bm{V}_{1}^{\star}( \bm{V}_{1}^{\star \top} \bm{V}_{1}^{\star})^{-1}( \bm{\Sigma}^{\star})^{-1} +  \bm{\Psi}_{\mathsf{left}} \bm{O}_{\mathsf{left}}.
\end{align*}
Here (i) follows from (\ref{eq:proof-lemma-subspace-interm-1}),
and (ii) follows from (\ref{eq:proof-lemma-subspace-interm-2}).
Note that for each $i \in [N]$, we have $\Vert( \bm{\Psi}_{\mathsf{left}})_{i,\cdot} \bm{O}_{\mathsf{left}}\Vert_{2} = \Vert( \bm{\Psi}_{\mathsf{left}})_{i,\cdot}\Vert_{2}$,
therefore we have established the desired bound \eqref{eq:Psi-i-bound} by letting 
\begin{align*}
 \bm{Z}\coloneqq \bm{E}_{\mathsf{left}} \bm{V}_{1}^{\star}( \bm{V}_{1}^{\star \top} \bm{V}_{1}^{\star})^{-1}( \bm{\Sigma}^{\star})^{-1} \qquad \text{and} \qquad  \bm{\Psi}\coloneqq \bm{\Psi}_{\mathsf{left}} \bm{O}_{\mathsf{left}},
\end{align*}
and by taking $C_6 \geq c_1 \clow^{-3/2}$. Next, (\ref{eq:denoising-U-I}) tells us that with probability at least
$1-O \big ((N + T)^{-9} \big )$, for any $i \in [N]$,
\begin{align*}
 \left \Vert  \bm{U}_{i,\cdot} \bm{H}- \bm{U}_{i,\cdot}^{\star} \right\Vert _{2} &  =  \big \Vert \big [ \bm{U}\mathsf{sgn}( \bm{U}^{\top} \bm{U}_{\mathsf{left}}^{\star})- \bm{U}_{\mathsf{left}}^{\star} \big ]_{i,\cdot} \big \Vert_{2} \leq c_1   \frac{\sigma}{\sigma_{r}( \bm{M}_{\mathsf{left}}^{\star})} \sqrt{r + \spicy} + c_1 \frac{\sigma^{2} \left (N + T_{1} \right)}{\text{\ensuremath{\sigma_{r}^{2}( \bm{M}_{\mathsf{left}}^{\star})}}} \left \Vert  \bm{U}_{i,\cdot}^{\star} \right\Vert _{2}\\
 &  \leq c_1 \clow^{-1} \bigg[ \frac{\sigma \sqrt{r + \spicy}}{\sigma_{r}^{\star} \sqrt{T_{1}/T}} +  \frac{\sigma^{2} \left (N + T_{1} \right)}{\ensuremath{\sigma_{r}^{\star2}}T_{1}/T} \left \Vert  \bm{U}_{i,\cdot}^{\star} \right\Vert _{2} \bigg].
\end{align*}
In addition, we may apply (\ref{eq:denoising-U-I}) with $\mathcal{I} = [N_{1}]$
to obtain that 
\begin{align*}
 \left \Vert  \bm{U}_{1} \bm{H}- \bm{U}_{1}^{\star} \right\Vert  &  \leq c_1 \bigg[ \frac{\sigma^{2} \left (N + T_{1} \right)}{\sigma_{r}^{2}( \bm{M}_{\mathsf{left}}^{\star})} \left \Vert  \bm{U}_{1}^{\star} \right\Vert  +  \frac{\sigma \sqrt{N_{1} + r + \spicy}}{\sigma_{r}( \bm{M}_{\mathsf{left}}^{\star})} \bigg]
 \overset{\text{(i)}}{ \leq} c_1 \clow^{-1} \bigg[ \frac{\sigma^{2} \left (N + T_{1} \right)}{\ensuremath{\sigma_{r}^{\star2}}T_{1}/T} \left \Vert  \bm{U}_{1}^{\star} \right\Vert  +  \frac{\sigma \sqrt{N_{1}}}{\sigma_{r}^{\star} \sqrt{T_{1}/T}} \bigg] \\
 & \overset{\text{(ii)}}{ \leq } c_1 \clow^{-1} \bigg[  \frac{\sigma^{2}T}{\ensuremath{\sigma_{r}^{\star2}}} \sqrt{ \frac{N_{1}}{N}} +  \frac{\sigma \sqrt{N_{1}} }{\sigma_{r}^{\star} \sqrt{T_{1}/T}} \bigg]
 \overset{\text{(iii)}}{ \leq } \frac{\sqrt{\clow}}{5} \sqrt{ \frac{N_{1}}{N}}
 \overset{\text{(iv)}}{ \leq } \frac{\sqrt{\clow}}{5} \sigma_{r} \left ( \bm{U}_{1}^{\star} \right).
\end{align*}
Here (i) makes use of Lemma~\ref{lemma:submatrix-properties}, (ii)
and (iii) follow from \eqref{EqnSubmatrix} and
hold as long as $\sigma \sqrt{NT/T_{1}} \leq \Cnoise \sigma_{r}^{\star}$ and
$\sigma \sqrt{T} \leq \Cnoise \sigma_{r}^{\star}$ for some sufficiently small constant  $\Cnoise>0$, while (iv) follows from \eqref{EqnSubmatrix}.
This establishes \eqref{eq:U1-H-U1-star-spectral} by setting $C_6 \geq c_1 c^{-1}$. By Weyl's inequality, we can deduce that 
\begin{align*}
 \sqrt{ \frac{\clow}{2} \frac{N_{1}}{N}}  \leq  \sigma_{r} \left (
 \bm{U}_{1}^{\star} \right)- \left \Vert \bm{U}_{1} \bm{H}-
 \bm{U}_{1}^{\star} \right\Vert  \leq  \sigma_{r} \left (
 \bm{U}_{1} \right)  \leq  \sigma_{1} \left ( \bm{U}_{1} \right)  \leq 
  \sigma_{1} \left ( \bm{U}_{1}^{\star} \right) + \left \Vert
 \bm{U}_{1} \bm{H}- \bm{U}_{1}^{\star} \right\Vert  \leq  \sqrt{2 \cupper
   \frac{N_{1}}{N}}.
\end{align*}

\subsection{Proof of \Cref{lemma:subspace-error-2} \label{subsec:proof-lemma-subspace-2}}

We apply~\Cref{prop:denoising_main_2} to the observation $
\bm{M}_{\mathsf{left}} = \bm{M}_{\mathsf{left}}^{\star} +
\bm{E}_{\mathsf{left}}$ and use the notation introduced in the last
section for the proof of~\Cref{lemma:subspace-error-1}. Similar to the
proof of~\Cref{lemma:subspace-error-1}, we can check that the
conditions of~\Cref{prop:denoising_main_2} are satisfied. We take
$\mathcal{I} = \{i\}$ for any $i \in [N]$. Then the
bound~\eqref{eq:denoising-Delta-U} guarantees us that, with
probability at least $1-O \big ((N + T)^{-9} \big )$,
\begin{align*}
 \bm{U} \bm{\Sigma}\mathsf{sgn}( \bm{U}^{\top}
 \bm{U}_{\mathsf{left}}^{\star})- \bm{U}_{\mathsf{left}}^{\star}
 \bm{\Sigma}_{\mathsf{left}}^{\star} = \bm{E}_{\mathsf{left}}
 \bm{V}_{\mathsf{left}}^{\star} + \bm{\Delta}_{\mathsf{left}}
\end{align*}
where for each $i \in [N]$
\begin{align}
 \left \Vert ( \bm{\Delta}_{\mathsf{left}})_{i,\cdot} \right\Vert _{2}
 & \leq c_2 \frac{\sigma^{2} \left (N + T_{1}
   \right)}{\text{\ensuremath{\sigma_{r}(
       \bm{M}_{\mathsf{left}}^{\star})}}} \big \Vert(
 \bm{U}_{\mathsf{left}}^{\star})_{i,\cdot} \big \Vert_{2} + c_2 \sigma
 \sqrt{r + \spicy} \bigg[ \frac{\sigma \sqrt{T_{1}}}{\sigma_{r}(
     \bm{M}_{\mathsf{left}}^{\star})} +
   \frac{\sigma^{2}N}{\sigma_{r}^{2}( \bm{M}_{\mathsf{left}}^{\star})}
   \bigg]\nonumber \\
\label{eq:proof-lemma-subspace-2-interm-1} 
 & \leq c_2 \clow^{-1} \bigg[ \frac{\sigma^{2} \sqrt{T_{1} \left (r +
      \spicy \right)}}{\sigma_{r}^{\star} \sqrt{T_{1}/T}} +
  \frac{\sigma^{3}N \sqrt{r + \spicy}}{\sigma_{r}^{\star2}T_{1}/T} +
  \frac{\sigma^{2} \left (N + T_{1} \right)}{\sigma_{r}^{\star}
    \sqrt{T_{1}/T}} \left \Vert \bm{U}_{i,\cdot}^{\star} \right\Vert
  _{2} \bigg].
\end{align}
The matrix $ \bm{U}_{\mathsf{left}}^{\star}
\bm{\Sigma}_{\mathsf{left}}^{\star}$ can be expressed as
\begin{align*}
 \bm{U}_{\mathsf{left}}^{\star} \bm{\Sigma}_{\mathsf{left}}^{\star} =
 \bm{U}^{\star} \bm{O}_{\mathsf{left}}^{\top} \bm{O}_{\mathsf{left}}
 \bm{\Sigma}^{\star} \bm{V}_{1}^{\star \top}
 \bm{V}_{\mathsf{left}}^{\star} = \bm{U}^{\star} \bm{\Sigma}^{\star}
 \bm{V}_{1}^{\star \top} \bm{V}_{\mathsf{left}}^{\star},
\end{align*}
where $ \bm{O}_{\mathsf{left}}$ was defined
in~\Cref{subsec:proof-lemma-subspace-1}.  If we let $ \bm{R} =
\mathsf{sgn}( \bm{U}^{\top} \bm{U}_{\mathsf{left}}^{\star})(
\bm{V}_{1}^{\star \top} \bm{V}_{\mathsf{left}}^{\star})^{-1}$, then
\begin{align}
 \bm{U} \bm{\Sigma} \bm{R}- \bm{U}^{\star} \bm{\Sigma}^{\star} & =
 \bm{E}_{\mathsf{left}} \bm{V}_{\mathsf{left}}^{\star}(
 \bm{V}_{1}^{\star \top} \bm{V}_{\mathsf{left}}^{\star})^{-1} +
 \bm{\Delta}_{\mathsf{left}}( \bm{V}_{1}^{\star \top}
 \bm{V}_{\mathsf{left}}^{\star})^{-1}\nonumber \\
 & = \bm{E}_{\mathsf{left}} \bm{V}_{1}^{\star}
 \bm{R}_{\mathsf{left}}^{-1}( \bm{V}_{1}^{\star \top}
 \bm{V}_{1}^{\star} \bm{R}_{\mathsf{left}}^{-1})^{-1} +
 \bm{\Delta}_{\mathsf{left}}( \bm{V}_{1}^{\star \top}
 \bm{V}_{\mathsf{left}}^{\star})^{-1}\nonumber \\
\label{eq:proof-lemma-subspace-2-interm-4} 
 & = \bm{E}_{\mathsf{left}} \bm{V}_{1}^{\star}( \bm{V}_{1}^{\star
  \top} \bm{V}_{1}^{\star})^{-1} + \bm{\Delta}_{\mathsf{left}}(
\bm{V}_{1}^{\star \top} \bm{V}_{\mathsf{left}}^{\star})^{-1}.
\end{align}
Let $ \bm{\Delta} = \bm{\Delta}_{\mathsf{left}}( \bm{V}_{1}^{\star
  \top} \bm{V}_{\mathsf{left}}^{\star})^{-1}$, and we have
\begin{align}
\Vert \bm{\Delta}_{i,\cdot}\Vert_{2} & \leq \big \Vert(
\bm{\Delta}_{\mathsf{left}})_{i,\cdot} \big \Vert_{2} \big \Vert(
\bm{V}_{1}^{\star \top} \bm{V}_{\mathsf{left}}^{\star})^{-1} \big
\Vert \leq \big \Vert( \bm{\Delta}_{\mathsf{left}})_{i,\cdot} \big
\Vert_{2}\sigma_{r}^{-1} \left ( \bm{V}_{1}^{\star} \right)\nonumber
\\
& \overset{\text{(i)}}{\leq} c_2 \clow^{-3/2} \bigg[ \frac{\sigma^{2}
    \sqrt{T_{1} \left (r + \spicy \right)}}{\sigma_{r}^{\star}
    \sqrt{T_{1}/T}} + \frac{\sigma^{3}N \sqrt{r +
      \spicy}}{\sigma_{r}^{\star2}T_{1}/T} + \frac{\sigma^{2} \left (N
    + T_{1} \right)}{\sigma_{r}^{\star} \sqrt{T_{1}/T}} \left \Vert
  \bm{U}_{i,\cdot}^{\star} \right\Vert _{2} \bigg] \sqrt{
  \frac{T}{T_{1}}}\nonumber \\
\label{eq:proof-lemma-subspace-2-interm-2}
& \leq c_2 \clow^{-3/2} \bigg[ \frac{\sigma^{2} \sqrt{T_{1} \left (r +
      \spicy \right)}}{\sigma_{r}^{\star}T_{1}/T} + \frac{\sigma^{3}N
    \sqrt{r + \spicy}}{\sigma_{r}^{\star2}(T_{1}/T)^{3/2}} +
  \frac{\sigma^{2} \left (N + T_{1}
    \right)}{\sigma_{r}^{\star}T_{1}/T} \left \Vert
  \bm{U}_{i,\cdot}^{\star} \right\Vert _{2} \bigg],
\end{align}
where step (i) follows from
equation~\eqref{eq:proof-lemma-subspace-2-interm-1} and the sub-block
condition~\eqref{EqnSubmatrix}. In addition, with probability at least
$1-O ((N + T)^{-10} )$, we have
\begin{align}
 \big \Vert \big [ \bm{E}_{\mathsf{left}} \bm{V}_{1}^{\star}(
   \bm{V}_{1}^{\star \top} \bm{V}_{1}^{\star})^{-1} \big ]_{i,\cdot}
 \big \Vert_{2} \overset{\text{(ii)}}{\leq} C_g \sigma \sqrt{r +
   \spicy} \big \Vert \bm{V}_{1}^{\star}( \bm{V}_{1}^{\star \top}
 \bm{V}_{1}^{\star})^{-1} \big \Vert \overset{\text{(iii)}}{\leq}
 \clow^{-1/2} C_g \frac{\sigma \sqrt{r + \spicy}}{
   \sqrt{T_{1}/T}}, \label{eq:proof-lemma-subspace-2-interm-3}
\end{align}
where step (ii) utilizes~\Cref{lemma:gaussian-spectral}; and step
(iii) follows from the sub-block
condition~\eqref{EqnSubmatrix}. Putting together the
relations~\eqref{eq:proof-lemma-subspace-2-interm-2}
and~\eqref{eq:proof-lemma-subspace-2-interm-3} yields
\begin{align*}
 & \big \Vert \bm{U}_{i,\cdot} \bm{\Sigma} \bm{R}-
  \bm{U}_{i,\cdot}^{\star} \bm{\Sigma}^{\star} \big \Vert_{2} \leq
  \big \Vert \big [ \bm{E}_{\mathsf{left}} \bm{V}_{1}^{\star}(
    \bm{V}_{1}^{\star \top} \bm{V}_{1}^{\star})^{-1} \big ]_{i,\cdot}
  \big \Vert_{2} + \Vert \bm{\Delta}_{i,\cdot}\Vert_{2} \\
& \quad \leq \clow^{-1/2} C_g \frac{\sigma \sqrt{r + \spicy}}{
    \sqrt{T_{1}/T}} + c_2 \clow^{-3/2} \bigg[ \frac{\sigma^{2}
      \sqrt{T_{1} \left (r + \spicy
        \right)}}{\sigma_{r}^{\star}T_{1}/T} + \frac{\sigma^{3}N
      \sqrt{r + \spicy}}{\sigma_{r}^{\star2}(T_{1}/T)^{3/2}} +
    \frac{\sigma^{2} \left (N + T_{1}
      \right)}{\sigma_{r}^{\star}T_{1}/T}\Vert
    \bm{U}_{i,\cdot}^{\star}\Vert_{2} \bigg] \\
  & \quad \leq C_7 \frac{\sigma \sqrt{r + \spicy}}{ \sqrt{T_{1}/T}} +
  C_7 \frac{\sigma^{2} \left (N + T_{1} \right)}{\sigma_{r}^{\star}
    T_{1} /T} \Vert \bm{U}_{i,\cdot}^{\star} \Vert_{2},
\end{align*}
as long as $\sigma \sqrt{NT/T_{1}} \leq \Cnoise \sigma_{r}^{\star}$
and $\sigma \sqrt{T} \leq \Cnoise \sigma_{r}^{\star}$ for some
sufficiently small constant $\Cnoise>0$, and $C_7 \geq 2 \clow^{-1/2}
C_g + 2 c_2 \clow^{-3/2}$. In addition, we can apply the
bound~\eqref{eq:denoising-Delta-U} with $\mathcal{I} = [N_{1}]$ to
obtain
\begin{align}
\Vert( \bm{\Delta}_{\mathsf{left}})_{\mathcal{I},\cdot}\Vert & \leq
c_2 \sigma \sqrt{N_{1}} \bigg[ \frac{\sigma \sqrt{N_{1} +
      T_{1}}}{\sigma_{r}( \bm{M}_{\mathsf{left}}^{\star})} +
  \frac{\sigma^{2} N}{\sigma_{r}^{2}( \bm{M}_{\mathsf{left}}^{\star})}
  \bigg] + c_2 \frac{\sigma^{2} \left (N + T_{1}
  \right)}{\text{\ensuremath{\sigma_{r}(
      \bm{M}_{\mathsf{left}}^{\star})}}} \left \Vert (
\bm{U}_{\mathsf{left}}^{\star})_{\mathcal{I},\cdot} \right\Vert _{2}
\nonumber \\ & \leq c_2 \clow^{-1} \bigg[ \sigma \sqrt{N_{1}} \bigg(
  \frac{\sigma \sqrt{N_{1} + T_{1}}}{\sigma_{r}^{\star}
    \sqrt{T_{1}/T}} + \frac{\sigma^{2}N}{\sigma_{r}^{\star2}T_{1}/T}
  \bigg) + \frac{\sigma^{2} \left (N + T_{1}
    \right)}{\sigma_{r}^{\star} \sqrt{T_{1}/T}} \left \Vert
  \bm{U}_{1}^{\star} \right\Vert \bigg].
 \label{eq:proof-lemma-subspace-2-interm-5}
\end{align}
Therefore, we have
\begin{align}
\Vert \bm{\Delta}_{1}\Vert & \leq \big \Vert(
\bm{\Delta}_{\mathsf{left}})_{\mathcal{I},\cdot} \big \Vert \big
\Vert( \bm{V}_{1}^{\star \top} \bm{V}_{\mathsf{left}}^{\star})^{-1}
\big \Vert\nonumber \\
& \overset{\text{(i)}}{\leq} c_2 \clow^{-3/2} \bigg[ \sigma
  \sqrt{N_{1}} \left ( \frac{\sigma \sqrt{N_{1} +
      T_{1}}}{\sigma_{r}^{\star} \sqrt{T_{1}/T}} +
  \frac{\sigma^{2}N}{\sigma_{r}^{\star2}T_{1}/T} \right) \sqrt{
    \frac{T}{T_{1}}} + \frac{\sigma^{2} \left (N + T_{1}
    \right)}{\sigma_{r}^{\star} \sqrt{T_{1}/T}} \left \Vert
  \bm{U}_{1}^{\star} \right\Vert \bigg] \sqrt{
  \frac{T}{T_{1}}}\nonumber \\
\label{eq:proof-lemma-subspace-2-interm-7}
& \overset{\text{(ii)}}{\leq} 2 c_2 \clow^{-3/2} \bigg[
  \frac{\sigma^{2} \sqrt{N_{1} \left (N_{1} + T_{1}
      \right)}}{\ensuremath{\sigma_{r}^{\star} \sqrt{T_{1} / T}}} +
  \frac{\sigma^{2} \left (N + T_{1}
    \right)}{\text{\ensuremath{\sigma_{r}^{\star} T_{1} / T}}} \sqrt{
    \frac{N_{1}}{N}} \bigg] \overset{\text{(iii)}}{\leq} C_7
\frac{\sigma^{2} \left (N + T_{1}
  \right)}{\text{\ensuremath{\sigma_{r}^{\star}T_{1}/T}}} \sqrt{
  \frac{N_{1}}{N}}.
\end{align}
Here step (i) follows from the
relations~\eqref{eq:proof-lemma-subspace-2-interm-5}
and~\eqref{EqnSubmatrix}; step (ii) follows from the sub-block
condition~\eqref{EqnSubmatrix} and holds provided that $\sigma
\sqrt{NT/T_{1}} \leq \Cnoise \sigma_{r}^{\star}$ for some sufficiently
small $\Cnoise>0$, and (iii) follows from the AM-GM inequality and
holds by taking $C_7 \geq 4c_2 \clow^{-3/2}$. We also
invoke~\Cref{lemma:gaussian-spectral} to guarantee that with
probability at least $1 -O \big ((N + T)^{-10} \big )$,
\begin{align}
 \big \Vert \bm{E}_{a} \bm{V}_{1}^{\star}( \bm{V}_{1}^{\star \top}
 \bm{V}_{1}^{\star})^{-1} \big \Vert \leq \frac{1}{2} C_g \sigma
 \sqrt{N_{1} + r + \spicy} \big \Vert \bm{V}_{1}^{\star}(
 \bm{V}_{1}^{\star \top} \bm{V}_{1}^{\star})^{-1} \big \Vert \leq
 \clow^{-1} C_g \frac{\sigma \sqrt{N_{1}}}{
   \sqrt{T_{1}/T}}. \label{eq:proof-lemma-subspace-2-interm-6}
\end{align}
Collecting together claims~(\ref{eq:proof-lemma-subspace-2-interm-4}),
(\ref{eq:proof-lemma-subspace-2-interm-7})
and~(\ref{eq:proof-lemma-subspace-2-interm-6}),
we conclude that
\begin{align*}
 \big \Vert \bm{U}_{1} \bm{\Sigma} \bm{R}- \bm{U}_{1}^{\star}
 \bm{\Sigma}^{\star} \big \Vert & \leq \big \Vert \bm{E}_{a}
 \bm{V}_{1}^{\star}( \bm{V}_{1}^{\star \top} \bm{V}_{1}^{\star})^{-1}
 \big \Vert + \left \Vert \bm{\Delta}_{1} \right\Vert \leq C_7
 \frac{\sigma \sqrt{N_{1}}}{ \sqrt{T_{1}/T}} + C_7
 \frac{\sigma^{2}T}{\text{\ensuremath{\sigma_{r}^{\star}}}} \sqrt{
   \frac{N_{1}}{N}}
\end{align*}
provided that $\sigma \sqrt{NT/T_{1}} \leq \Cnoise \sigma_{r}^{\star}$
for some sufficiently small constant $\Cnoise > 0$ (Here we take $C_7
\geq 2 \clow^{-1} C_g +4 c_2 \clow^{-3/2}$.)


\subsection{Proof of~\Cref{lemma:denoising-error}}
\label{subsec:proof-lemma-denoising}

We apply~\Cref{corollary:full-matrix} to the observation
$\bm{M}_{\mathsf{upper}} = \bm{M}_{\mathsf{upper}}^{\star} +
\bm{E}_{\mathsf{upper}}$, where we assume that the SVD of $
\bm{M}_{\mathsf{upper}}^{\star}$ is given by $
\bm{U}_{\mathsf{upper}}^{\star} \bm{\Sigma}_{\mathsf{upper}}^{\star}
\bm{V}_{\mathsf{upper}}^{\star \top}$.  In order to ensure that the
condition of~\Cref{corollary:full-matrix} is satisfied, we check that
\begin{align*}
\sigma \sqrt{\max\{ N_1, T \}} \leq \Cnoise \sqrt{N_1 / N}
\overset{\text{(i)}}{\leq} \cnoise \sqrt{\clow N_1 / N}
\overset{\text{(ii)}}{\leq} \cnoise \sigma_{r}(
\bm{M}_{\mathsf{upper}}^{\star}).
\end{align*}
Here step (i) holds as long as $\Cnoise \leq \cnoise \sqrt{\clow}$,
while step (ii) follows from~\Cref{lemma:submatrix-properties}.  We
take $\mathcal{I} = [N_1]$ and $\mathcal{J} = \{ t \}$ for any $t \in
[T]$.  Then the relation~\eqref{eq:denoising-Phi-J} guarantees that
with probability at least $1-O \big ((N + T)^{-9} \big )$,
\begin{align*}
 \bm{U}_{\mathsf{upper}} \bm{\Sigma}_{\mathsf{upper}}
 \bm{V}_{\mathsf{upper}}^{\top}- \bm{M}_{\mathsf{upper}}^{\star} =
 \bm{E}_{\mathsf{upper}} \bm{V}_{\mathsf{upper}}^{\star}
 \bm{V}_{\mathsf{upper}}^{\star \top} +
 \bm{U}_{\mathsf{upper}}^{\star} \bm{U}_{\mathsf{upper}}^{\star \top}
 \bm{E}_{\mathsf{upper}} + \bm{\Phi}_{\mathsf{upper}}
\end{align*}
where for each $t \in [T]$ 
\begin{align*}
 \left \Vert ( \bm{\Phi}_{\mathsf{upper}})_{\cdot,t} \right\Vert _{2}
 & \overset{\text{(i)}}{\leq} 4 c_3 \bigg[ \frac{\sigma^{2} \left
     (N_{1} + T \right)}{\sigma_{r}( \bm{M}_{\mathsf{upper}}^{\star})}
   + \sigma \sqrt{r + \spicy} \bigg] \left \Vert (
 \bm{V}_{\mathsf{upper}}^{\star})_{t,\cdot} \right\Vert _{2} + 4 c_3
 \sigma \sqrt{r + \spicy} \bigg[ \frac{\sigma \sqrt{N_{1}
   }}{\sigma_{r}( \bm{M}_{\mathsf{upper}}^{\star})} +
   \frac{\sigma^{2}T}{\sigma_{r}^{2}
     (\bm{M}_{\mathsf{upper}}^{\star})} \bigg] \\
& \overset{\text{(ii)}}{\leq} 4 c_3 \clow^{-1} \bigg[ \frac{\sigma^{2}
     \left (N_{1} + T \right)}{\sigma_{r}^{\star} \sqrt{N_{1}/N}} +
   \sigma \sqrt{r + \spicy} \bigg] \left \Vert
 \bm{V}_{t,\cdot}^{\star} \right\Vert _{2} + 4 c_3 \clow^{-1} \sigma
 \sqrt{r + \spicy} \bigg[ \frac{\sigma \sqrt{N}}{\sigma_{r}^{\star}} +
   \frac{\sigma^{2}T}{\sigma_{r}^{\star2} N_{1}/N} \bigg] \\
& \leq C_8 \bigg[ \frac{\sigma^{2} \sqrt{N \left (r + \spicy
       \right)}}{\sigma_{r}^{\star}} + \frac{\sigma^{3}T \sqrt{r +
       \spicy}}{\sigma_{r}^{\star2}N_{1}/N} + \bigg( \frac{\sigma^{2}
     \left (N_{1} + T \right)}{\sigma_{r}^{\star} \sqrt{N_{1}/N}} +
   \sigma \sqrt{r + \spicy} \bigg) \left \Vert
   \bm{V}_{t,\cdot}^{\star} \right \Vert _{2} \bigg].
\end{align*}
Here step (i) holds provided that $\Cnoise$ is sufficiently small;
step (ii) follows from~\Cref{lemma:submatrix-properties}; and step
(iii) holds by taking $C_8 \geq 4 c_3 \clow^{-1}$.  Note that
$\bm{U}_{\mathsf{upper}}^{\star} \bm{U}_{\mathsf{upper}}^{\star \top}
= \bm{U}_{1}^{\star}( \bm{U}_{1}^{\star \top} \bm{U}_{1}^{\star})^{-1}
\bm{U}_{1}^{\star \top}$ and $ \bm{V}_{\mathsf{upper}}^{\star}
\bm{V}_{\mathsf{upper}}^{\star \top} = \bm{V}^{\star} \bm{V}^{\star
  \top}$, therefore
\begin{align*}
 \bm{U}_{\mathsf{upper}} \bm{\Sigma}_{\mathsf{upper}}
 \bm{V}_{\mathsf{upper}}^{\top}- \bm{M}_{\mathsf{upper}}^{\star} =
 \bm{E}_{\mathsf{upper}} \bm{V}^{\star} \bm{V}^{\star \top} +
 \bm{U}_{1}^{\star} ( \bm{U}_{1}^{\star \top} \bm{U}_{1}^{\star}
 )^{-1} \bm{U}_{1}^{\star \top} \bm{E}_{\mathsf{upper}} +
 \bm{\Phi}_{\mathsf{upper}}.
\end{align*}
Finally, since we are only interested the error on the subblock $
\bm{M}_{b}^{\star}$, we can let
\begin{align*}
 \bm{W}\coloneqq \bm{E}_{a} \bm{V}_{1}^{\star} \bm{V}_{2}^{\star \top}
 + \bm{E}_{b} \bm{V}_{2}^{\star} \bm{V}_{2}^{\star \top} +
 \bm{U}_{1}^{\star}( \bm{U}_{1}^{\star \top} \bm{U}_{1}^{\star})^{-1}
 \bm{U}_{1}^{\star \top} \bm{E}_{b} \qquad \text{and} \qquad
 \bm{\Phi}\coloneqq( \bm{\Phi}_{\mathsf{upper}})_{\cdot,(T_{1} + 1):T}
\end{align*}
to obtain the desired bound. In addition, we can invoke
equation~\eqref{eq:denoising-1st} to show that with probability at
least $1-O \big ((N + T)^{-9} \big )$, for each $t \in [T]$
\begin{align*}
 \big \Vert( \bm{M}_{b}- \bm{M}_{b}^{\star})_{\cdot,t} \big \Vert_{2}
 & \leq 2 c_3 \bigg[ \sigma \sqrt{r + \spicy} + \sigma \sqrt{N_{1}}
   \left \Vert \bm{V}_{t,\cdot}^{\star} \right\Vert _{2} +
   \frac{\sigma^{2} \left (N_{1} + T \right)}{\sigma_{r}(
     \bm{M}_{\mathsf{upper}}^{\star})} \left \Vert
   \bm{V}_{t,\cdot}^{\star} \right\Vert _{2} + \frac{\sigma^{2}
     \sqrt{N_{1} \left (r + \spicy \right)}}{\sigma_{r}(
     \bm{M}_{\mathsf{upper}}^{\star})} \bigg] \\
& \overset{\text{(a)}}{\geq} C_8 \sigma \sqrt{r + \spicy} + C_8 \bigg[
   \frac{\sigma^{2}T}{\sigma_{r}^{\star} \sqrt{N_{1}/N}} + \sigma
   \sqrt{N_{1}} \bigg]\Vert \bm{V}_{t,\cdot}^{\star}\Vert_{2}.
\end{align*}
Here step (a) follows from~\Cref{lemma:submatrix-properties}, and
holds as long as $C_8 \geq 4 c_3 \clow^{-1/2}$.


\subsection{Proof of~\Cref{lemma:U1-properties}}
\label{subsec:proof-lemma-U1-properties}

We apply~\Cref{corollary:inner-product-sub} with $\mathcal{I} =
[N_{1}]$ to show that with probability at least $1-O((N + T)^{-10})$
\begin{align}
 & \big \Vert \bm{U}_{1}^{\star \top}( \bm{U}_{1} \bm{H}-
  \bm{U}_{1}^{\star}) \big \Vert \leq c_4 \bigg[ \frac{\sigma \sqrt{r
        + \spicy}}{\sigma_{r}( \bm{M}_{\mathsf{left}}^{\star})} +
    \frac{\sigma^{2} \sqrt{ \left (N_{1} + T_{1} \right) \left (T_{1}
        \right)}}{\sigma_{r}^{2}( \bm{M}_{\mathsf{left}}^{\star})} +
    \frac{\sigma^{3}N \sqrt{N_{1}}}{\sigma_{r}^{3}(
      \bm{M}_{\mathsf{left}}^{\star})} \bigg]\Vert
  \bm{U}_{1}^{\star}\Vert + c_4 \frac{\sigma^{2} \left (N + T_{1}
    \right)}{\sigma_{r}^{2}( \bm{M}_{\mathsf{left}}^{\star})}\Vert
  \bm{U}_{1}^{\star}\Vert^{2} \nonumber \\
  & \qquad \overset{\text{(i)}}{ \leq} c_4 \clow^{-1} \bigg[
    \frac{\sigma \sqrt{r + \spicy}}{\sigma_{r}^{\star} \sqrt{T_{1}/T}}
    + \frac{\sigma^{2} \sqrt{ \left (N_{1} + T_{1}
        \right)T_{1}}}{\sigma_{r}^{\star 2}T_{1}/T} +
    \frac{\sigma^{3}N \sqrt{N_{1}}}{ \big (\sigma_{r}^{\star}
      \sqrt{T_{1}/T} \big )^{3}} \bigg] \sqrt{ \frac{N_{1}}{N}} + c_4
  c^{-1} \frac{\sigma^{2} \left (N + T_{1} \right)}{\sigma_{r}^{\star
      2}T_{1}/T} \frac{N_{1}}{N}\nonumber \\ & \qquad
  \overset{\text{(ii)}}{ \leq } 4c_4 \clow^{-1} \frac{\sigma \sqrt{r +
      \spicy}}{\sigma_{r}^{\star} \sqrt{T_{1}/T}} \sqrt{
    \frac{N_{1}}{N}} + 4c_4 \clow^{-1} \frac{\sigma^{2} \left (N +
    T_{1} \right)}{\sigma_{r}^{\star 2} T_{1}/T}
  \frac{N_{1}}{N}, \label{eq:lemma-U1-1}
\end{align}
Here step (i) follows from~\Cref{lemma:submatrix-properties}, whereas
step (ii) holds provided that $\sigma \sqrt{NT/T_{1}} \leq \Cnoise
\sigma_{r}^{\star}$ and $\sigma \sqrt{T} \leq \Cnoise
\sigma_{r}^{\star}$ for some sufficiently small constant
$\Cnoise>0$. It is worth mentioning that, we also use the AM-GM
inequality in the relation~\eqref{eq:lemma-U1-1}, which guarantees
that
\begin{align*}
 \frac{\sigma^{2} \sqrt{N_{1}T_{1}}}{\sigma_{r}^{\star 2}T_{1}/T}
 \sqrt{ \frac{N_{1}}{N}} = \frac{\sigma^{2}
   \sqrt{NT_{1}}}{\sigma_{r}^{\star 2}T_{1}/T} \frac{N_{1}}{N} \leq
 \frac{1}{2} \frac{\sigma^{2} \left (N + T_{1}
   \right)}{\sigma_{r}^{\star 2}T_{1}/T} \frac{N_{1}}{N}.
\end{align*}
The inequality~\eqref{eq:lemma-U1-1} allows us to control some other
quantities. For example, we can decompose
\begin{align*}
 \left ( \bm{U}_{1} \bm{H} \right)^{\top} \bm{U}_{1} \bm{H}-
 \bm{U}_{1}^{\star \top} \bm{U}_{1}^{\star} = \left ( \bm{U}_{1}
 \bm{H}- \bm{U}_{1}^{\star} \right)^{\top} \bm{U}_{1}^{\star} +
 \bm{U}_{1}^{\star \top} \left ( \bm{U}_{1} \bm{H}- \bm{U}_{1}^{\star}
 \right) + \left ( \bm{U}_{1} \bm{H}- \bm{U}_{1}^{\star}
 \right)^{\top} \left ( \bm{U}_{1} \bm{H}- \bm{U}_{1}^{\star} \right),
\end{align*}
which leads to the bound
\begin{align}
 & \big \Vert( \bm{U}_{1} \bm{H})^{\top} \bm{U}_{1} \bm{H}-
  \bm{U}_{1}^{\star \top} \bm{U}_{1}^{\star} \big \Vert \leq 2 \big
  \Vert \bm{U}_{1}^{\star \top}( \bm{U}_{1} \bm{H}-
  \bm{U}_{1}^{\star}) \big \Vert + \Vert \bm{U}_{1} \bm{H}-
  \bm{U}_{1}^{\star}\Vert^{2}\nonumber \\
& \qquad \overset{\text{(a)}}{ \leq } 8c_4 \clow^{-1} \frac{\sigma
    \sqrt{r + \spicy}}{\sigma_{r}^{\star} \sqrt{T_{1}/T}} \sqrt{
    \frac{N_{1}}{N}} + 8c_4 \clow^{-1} \frac{\sigma^{2} \left (N +
    T_{1} \right)}{\sigma_{r}^{\star 2}T_{1}/T} \frac{N_{1}}{N} +
  C_6^2 \bigg[ \frac{\sigma^{2}T}{\ensuremath{\sigma_{r}^{\star 2}}}
    \sqrt{ \frac{N_{1}}{N}} + \frac{\sigma
      \sqrt{N_{1}}}{\sigma_{r}^{\star} \sqrt{T_{1}/T}}
    \bigg]^{2} \nonumber \\
& \qquad \overset{\text{(b)}}{ \leq } ( 16c_4 \clow^{-1} + 2 C_6^2 )
  \bigg[ \frac{\sigma \sqrt{r + \spicy}}{\sigma_{r}^{\star}
      \sqrt{T_{1}/T}} \sqrt{ \frac{N_{1}}{N}} + \frac{\sigma^{2} \left
      (N + T_{1} \right)}{\sigma_{r}^{\star 2}T_{1}/T} \frac{N_{1}}{N}
    \bigg]. \label{eq:lemma-U1-2}
\end{align}
Here step (a) follows from the relations~\eqref{eq:lemma-U1-1}
and~\eqref{eq:U1-H-U1-star-spectral}, while step (b) holds provided
that $\Cnoise$ is sufficiently small. In addition, in view of the
decomposition
\begin{align*}
 \big [ \left ( \bm{U}_{1} \bm{H} \right)^{\top} \bm{U}_{1} \bm{H}
   \big ]^{-1}-( \bm{U}_{1}^{\star \top} \bm{U}_{1}^{\star})^{-1} = (
 \bm{U}_{1}^{\top} \bm{U}_{1})^{-1} \big [( \bm{U}_{1} \bm{H})^{\top}
   \bm{U}_{1} \bm{H}- \bm{U}_{1}^{\star \top} \bm{U}_{1}^{\star} \big
 ]( \bm{U}_{1}^{\star \top} \bm{U}_{1}^{\star})^{-1},
\end{align*}
we can show that
\begin{align}
& \big \Vert \big [( \bm{U}_{1} \bm{H})^{\top} \bm{U}_{1} \bm{H} \big
  ]^{-1}-( \bm{U}_{1}^{\star \top} \bm{U}_{1}^{\star})^{-1} \big \Vert
  \leq \sigma_{r}^{-2} \left ( \bm{U}_{1} \right)\sigma_{r}^{-2} \left
  ( \bm{U}_{1}^{\star} \right) \left \Vert ( \bm{U}_{1} \bm{H})^{\top}
  \bm{U}_{1} \bm{H}- \bm{U}_{1}^{\star \top} \bm{U}_{1}^{\star}
  \right\Vert \nonumber \\
 \label{eq:Lemma-U1-3} 
 & \qquad \leq 2 \clow^{-2} ( 16c_4 \clow^{-1} + 2 C_6^2 ) \bigg(
 \frac{N}{N_{1}} \bigg)^{2} \bigg[ \frac{\sigma \sqrt{r +
       \spicy}}{\sigma_{r}^{\star} \sqrt{T_{1}/T}} \sqrt{
     \frac{N_{1}}{N}} + \frac{\sigma^{2} \left (N + T_{1}
     \right)}{\sigma_{r}^{\star 2}T_{1}/T} \frac{N_{1}}{N} \bigg].
\end{align}
Here the last step follows from
equations~\eqref{eq:U1-spectrum},~\eqref{eq:U1-spectrum}
and~\eqref{eq:lemma-U1-2}.  Starting from
\begin{align*}
 & \big [ \left ( \bm{U}_{1} \bm{H} \right)^{\top} \bm{U}_{1} \bm{H}
    \big ]^{-1} \left ( \bm{U}_{1} \bm{H} \right)^{\top}-(
  \bm{U}_{1}^{\star \top} \bm{U}_{1}^{\star})^{-1} \bm{U}_{1}^{\star
    \top} \\
& \quad = \big \{ \big [( \bm{U}_{1} \bm{H})^{\top} \bm{U}_{1} \bm{H}
    \big ]^{-1}-( \bm{U}_{1}^{\star \top} \bm{U}_{1}^{\star})^{-1}
  \big \} \left ( \bm{U}_{1} \bm{H} \right)^{\top} + (
  \bm{U}_{1}^{\star \top} \bm{U}_{1}^{\star})^{-1}( \bm{U}_{1} \bm{H}-
  \bm{U}_{1}^{\star})^{\top},
\end{align*}
we can also deduce that
\begin{align*}
& \big \Vert \big [ \left ( \bm{U}_{1} \bm{H} \right)^{\top}
    \bm{U}_{1} \bm{H} \big ]^{-1} \left ( \bm{U}_{1} \bm{H}
  \right)^{\top}-( \bm{U}_{1}^{\star \top} \bm{U}_{1}^{\star})^{-1}
  \bm{U}_{1}^{\star \top} \big \Vert \\
& \quad \leq \big \Vert \big [ \left ( \bm{U}_{1} \bm{H}
    \right)^{\top} \bm{U}_{1} \bm{H} \big ]^{-1}-( \bm{U}_{1}^{\star
    \top} \bm{U}_{1}^{\star})^{-1} \big \Vert \left \Vert \bm{U}_{1}
  \right\Vert + \sigma_{r}^{-2} \left ( \bm{U}_{1}^{\star} \right)
  \left \Vert \bm{U}_{1} \bm{H}- \bm{U}_{1}^{\star} \right\Vert \\
& \quad \overset{\text{(i)}}{\leq} 2 \clow^{-2} ( 16c_4 \clow^{-1} + 2
  C_6^2 ) \bigg( \frac{N}{N_{1}} \bigg)^{2} \bigg[ \frac{\sigma
      \sqrt{r + \spicy}}{\sigma_{r}^{\star} \sqrt{T_{1} / T}} \sqrt{
      \frac{N_{1}}{N}} + \frac{\sigma^{2} \left (N + T_{1}
      \right)}{\sigma_{r}^{\star 2}T_{1} / T} \frac{N_{1}}{N} \bigg]
  \sqrt{ \cupper \frac{N_{1}}{N}} \\
  & \quad \qquad + \clow^{-1} C_6 \frac{N}{N_{1}} \bigg[
    \frac{\sigma^{2}T}{\ensuremath{\sigma_{r}^{\star 2}}} \sqrt{
      \frac{N_{1}}{N}} + \frac{\sigma \sqrt{N_{1}}}{\sigma_{r}^{\star}
      \sqrt{T_{1}/T}} \bigg] \\
& \quad \overset{\text{(ii)}}{\leq} 2 \clow^{-1} C_6
  \frac{\sigma^{2}T}{\sigma_{r}^{\star 2}} \sqrt{ \frac{N}{N_{1}}} + 2
  \clow^{-1} C_6 \frac{\sigma \sqrt{N_{1}}}{\sigma_{r}^{\star}
    \sqrt{T_{1}/T}} \frac{N}{N_{1}}.
\end{align*}
Here step (i) follows from the bounds~\eqref{eq:U1-spectrum},
\eqref{eq:Lemma-U1-3} and~\eqref{eq:U1-H-U1-star-spectral}, while step
(ii) holds provided that $\sigma \sqrt{NT/T_{1}} \leq \Cnoise
\sigma_{r}^{\star}$ for some sufficiently small constant $\Cnoise >
0$. Taking $C_9 =2 \clow^{-2} ( 16c_4 \clow^{-1} + 2 C_6^2 )$
completes the proof.


\section{Proof of Lemmas from \Cref{sec:proof-CI}}

\subsection{Proof of \Cref{lemma:CI-1}} \label{sec:proof-lemma-CI-1}

Similar to the proof of \Cref{lemma:denoising-error}, we apply~\Cref{corollary:full-matrix} to the observation
$\bm{M}_{\mathsf{upper}} = \bm{M}_{\mathsf{upper}}^{\star} +
\bm{E}_{\mathsf{upper}}$.  The
condition of~\Cref{corollary:full-matrix} was checked in the proof of \Cref{lemma:denoising-error}. Then with probability exceeding $1-O((N+T)^{-10})$, we have
\begin{align*}
	\big \Vert \bm{M}_{\mathsf{upper}}^{\star} -
        \widehat{\bm{M}}_{\mathsf{upper}}\big\Vert &
        \overset{\text{(i)}}{\leq} 2 c_3 \sigma \sqrt{T} + 2 c_3
        \sigma \sqrt{N_1} + c_3 \frac{\sigma^2 (N_1 +
          T)}{\sigma_{r}^{\star} (\bm{M}_{\mathsf{upper}}^\star) } +
        c_3 \frac{\sigma^2 }{\sigma_{r}^{\star}} \sqrt{ N_1 T} \\ &
        \overset{\text{(ii)}}{\leq} 4 c_3 \sigma \sqrt{N_1 + T} + c_3
        \frac{\sigma^2 (N_1 + T)}{\sigma_{r}^{\star} \sqrt{N_1 / N} }
        + c_3 \frac{\sigma^2 }{\sigma_{r}^{\star}} \sqrt{ N_1 T}
        \overset{\text{(iii)}}{\leq} 5 c_3 \sigma \sqrt{N_1 + T} .
\end{align*}
Here step (i) is obtained by taking $\mathcal{I}=[N_1]$ and $\mathcal{J}=[T]$ in \eqref{eq:denoising-1st} in \Cref{corollary:full-matrix}; step (ii) uses the sub-block condition \eqref{EqnSubmatrix}; step (iii) follows from the noise condition \eqref{eq:noise-condition-est}. Finally, we have
\begin{align*}
	\big \Vert \bm{M}_{\mathsf{upper}}^{\star} -
	\widehat{\bm{M}}_{\mathsf{upper}}\big\Vert_{\mathrm{F}} \leq \sqrt{r} \big \Vert \bm{M}_{\mathsf{upper}}^{\star} -
	\widehat{\bm{M}}_{\mathsf{upper}}\big\Vert \leq 5c_3 \sigma \sqrt{(N_1+T) r}.
\end{align*}
The desired bound can be obtained by taking $\cu = 5c_3$.

\subsection{Proof of \Cref{lemma:CI-2}} \label{sec:proof-lemma-CI-2}
We first decompose $\Delta_U \coloneqq  \sigma^2 | \bm{U}_{i,\cdot}
(\bm{U}_{1}^{\top}\bm{U}_{1})^{-1}\bm{U}_{i,\cdot}^{\top} -
\bm{U}_{i,\cdot}^{\star}(\bm{U}_{1}^{\star\top}\bm{U}_{1}^{\star})^{-1}
\bm{U}_{i,\cdot}^{\star\top} | $ into three terms
\begin{align*}
	\Delta _U & =  \sigma^2 \big| \bm{U}_{i,\cdot} \bm{H} \big[ (
	\bm{U}_{1} \bm{H})^{\top} \bm{U}_{1} \bm{H} \big]^{-1}
	(\bm{U}_{i,\cdot} \bm{H})^{\top} -  \bm{U}_{i,\cdot}^{\star}
	(\bm{U}_{1}^{\star\top} \bm{U}_{1}^{\star} )^{-1}
	\bm{U}_{i,\cdot}^{\star \top} \big |\\
	& \leq \underbrace{2 \sigma^2 \big|  ( \bm{U}_{i,\cdot} \bm{H} - \bm{U}_{i,\cdot}^{\star} )
		( \bm{U}_{1}^{\star\top} \bm{U}_{1}^{\star} )^{-1}
		\bm{U}_{i,\cdot}^{\star \top} \big|}_{\eqqcolon \alpha_1}
	+ \underbrace{ \sigma^2 \big| \bm{U}_{i,\cdot} \bm{H}
		\big\{ \big[ (\bm{U}_{1} \bm{H})^{\top} \bm{U}_{1} \bm{H} \big]^{-1} -
		\big(\bm{U}_{1}^{\star\top} \bm{U}_{1}^{\star} \big)^{-1} \big\}
		\bm{U}_{i,\cdot}^{\star\top} \big|}_{\eqqcolon \alpha_2} \\
	& \qquad + \underbrace{ \sigma^2 \big| (\bm{U}_{i,\cdot} \bm{H}
		- \bm{U}_{i,\cdot}^{\star}) \big[ (\bm{U}_{1}
		\bm{H})^{\top} \bm{U}_{1} \bm{H} \big]^{-1} (\bm{U}_{i,\cdot} \bm{H}
		- \bm{U}_{i,\cdot}^{\star})^{\top} \big|}_{\eqqcolon \alpha_3 } .
\end{align*}
On the other hand, we can lower bound $\gamma_{i,t}^\star$ by
\begin{subequations}
	\label{eq:gamma-lb}
	\begin{align}
		\label{eq:gamma-lb-1}
		\gamma_{i,t}^\star & \overset{\text{(i)}}{\geq} \frac{1}{\clow} \sigma^2 \frac{N}{N_1} \Vert \bm{U}_{i,\cdot}^\star \Vert_2^2 + \frac{1}{\clow} \sigma^2 \frac{T}{T_1} \Vert \bm{V}_{t,\cdot}^\star \Vert_2^2  = \frac{1}{\clow} \sigma^2 \frac{\mu_i^2 r}{N_1}  + \frac{1}{\clow} \sigma^2 \frac{\nu_t^2 r}{T_1}  \overset{\text{(ii)}}{\geq} \frac{1}{\cinc^2 \clow \delta^2} \frac{\sigma^4}{\sigma_r^{\star2}} \frac{NT}{N_1 T_1} \spicy^2.
	\end{align}
	Here step (i) follows from we the sub-block condition \eqref{EqnSubmatrix}, step (ii) follows from the incoherence condition \eqref{eq:signal-lb}. By the AM-GM inequality, the above bounds further give
	\begin{align}
		\label{eq:gamma-lb-2}
		\gamma_{i,t}^\star \geq \frac{\spicy}{2\sqrt{2} \clow \cinc \delta} \frac{\sigma^3}{\sigma_r^\star} \frac{N}{N_1} \sqrt{\frac{T}{T_1}} \Vert \bm{U}_{i,\cdot}^\star \Vert_2 + \frac{\spicy}{2\sqrt{2} \clow \cinc \delta} \frac{\sigma^3}{\sigma_r^\star} \frac{T}{T_1} \sqrt{\frac{N}{N_1}} \Vert \bm{V}_{t,\cdot}^\star \Vert_2.
	\end{align}
\end{subequations}
Then we bound $\alpha_1$, $\alpha_2$ and $\alpha_3$ separately. The first term $\alpha_1$ can be bounded by
\begin{align*}
	\alpha_1 & \overset{\text{(a)}}{\leq} \frac{1}{\clow} \sigma^2 \frac{N}{N_{1}}
	\Vert \bm{U}_{i,\cdot} \bm{H} - \bm{U}_{i,\cdot}^{\star}
	\Vert_{2} \Vert \bm{U}_{i,\cdot}^{\star} \Vert_{2} \overset{\text{(b)}}{\leq} 
	\frac{C_6}{\clow} \sigma^2 \frac{N}{N_{1}}  \Big(
	\frac{\sigma \sqrt{r+\spicy}}{\sigma_{r}^{\star} \sqrt{T_{1}/T}} +
	\frac{\sigma^{2}\left(N+T_{1}\right)}{\ensuremath{\sigma_{r}^{\star
				2}}T_{1}/T} \left\Vert \bm{U}_{i,\cdot}^{\star} \right\Vert _{2}
	\Big) \Vert\bm{U}_{i,\cdot}^{\star}\Vert_{2} \\
	& \overset{\text{(c)}}{\leq} \frac{\delta}{12 \sqrt{\spicy}} \Big( \frac{\spicy }{2\sqrt{2} \clow \cinc \delta} \frac{\sigma^3}{\sigma_r^\star} \frac{N}{N_1} \sqrt{\frac{T}{T_1}} \Vert \bm{U}_{i,\cdot}^\star \Vert_2 + \frac{1}{\clow} \sigma^2 \frac{N}{N_1} \Vert \bm{U}_{i,\cdot}^\star \Vert_2^2  \Big) 
	\overset{\text{(d)}}{\leq} \frac{\delta}{6 \sqrt{\spicy}} \gamma_{i,t}^\star
\end{align*}
Here step (a) follows from the sub-block condition \eqref{EqnSubmatrix}, step (b) follows from the bound \eqref{eq:UH-U-star-i} in \Cref{lemma:subspace-error-1}, step (c) follows from the noise condition \eqref{eq:noise-condition-est} and holds provided that $\cinc$ is sufficiently small, and step (d) uses \eqref{eq:gamma-lb-1} and \eqref{eq:gamma-lb-2}. For the second term $\alpha_2$, we have
\begin{align*}
	\alpha_2 & \leq \sigma^{2}  \big \Vert \big [( \bm{U}_{1} \bm{H})^{\top} \bm{U}_{1} \bm{H} \big
	]^{-1} - \big ( \bm{U}_{1}^{\star \top} \bm{U}_{1}^{\star} \big
	)^{-1} \big \Vert \Vert
	\bm{U}_{i,\cdot} \Vert_{2} \Vert \bm{U}_{i,\cdot}^{\star} \Vert_{2} \\
	& \overset{\text{(i)}}{\leq} 
	C_{9} \Big( \frac{\sigma
		\sqrt{r+\spicy}}{\sigma_{r}^{\star} \sqrt{T_{1}/T}}
	\sqrt{\frac{N_{1}}{N}} + \frac{\sigma^{2} \left(N+T_{1}
		\right)}{\sigma_{r}^{\star 2} T_{1}/T} \frac{N_{1}}{N} \Big)
	\frac{N^{2}}{N_{1}^{2}} \sigma^{2} \Vert \bm{U}_{i,\cdot}^{\star}
	\Vert_{2} \Big( 2\Vert\bm{U}_{i,\cdot}^{\star}\Vert_{2} + C_{6}
	\frac{\sigma \sqrt{r+\spicy}}{\sigma_{r}^{\star}
		\sqrt{T_{1}/T}}\Big) \\
	& \overset{\text{(ii)}}{\leq} \frac{\delta}{12 \sqrt{\spicy}} \Big( \frac{\spicy }{2\sqrt{2} \clow \cinc \delta} \frac{\sigma^3}{\sigma_r^\star} \frac{N}{N_1} \sqrt{\frac{T}{T_1}} \Vert \bm{U}_{i,\cdot}^\star \Vert_2 + \frac{1}{\clow} \sigma^2 \frac{N}{N_1} \Vert \bm{U}_{i,\cdot}^\star \Vert_2^2  \Big) 
	\overset{\text{(iii)}}{\leq} \frac{\delta}{6 \sqrt{\spicy}} \gamma_{i,t}^\star
\end{align*}
Here step (i) follows from \eqref{eq:U1topU1-inv-spectral-error} in \Cref{lemma:U1-properties} and \eqref{eq:U-i-bound}, step (ii) holds under the noise condition \eqref{eq:noise-condition-est}, and step (iii) follows from uses \eqref{eq:gamma-lb-1} and \eqref{eq:gamma-lb-2}. The last term $\alpha_3$ is bounded by
\begin{align*}
	\alpha_3 & \overset{\text{(a)}}{\leq} \frac{1}{\clow} \sigma^2 \frac{N}{N_{1}}
	\Vert \bm{U}_{i,\cdot} \bm{H} - \bm{U}_{i,\cdot}
	\Vert_{2}^2 \overset{\text{(b)}}{\leq} 
	\frac{C_6^2}{\clow} \sigma^2 \frac{N}{N_{1}}  \Big(
	\frac{\sigma \sqrt{r+\spicy}}{\sigma_{r}^{\star} \sqrt{T_{1}/T}} +
	\frac{\sigma^{2}\left(N+T_{1}\right)}{\ensuremath{\sigma_{r}^{\star
				2}}T_{1}/T} \left\Vert \bm{U}_{i,\cdot}^{\star} \right\Vert _{2}
	\Big)^2 \\
	& \overset{\text{(c)}}{\leq} \frac{\delta}{12 \sqrt{\spicy}} \Big( \frac{1}{\cinc^2 \clow \delta^2} \frac{\sigma^4}{\sigma_r^{\star2}} \frac{NT}{N_1 T_1} \spicy^2 + \frac{1}{\clow} \sigma^2 \frac{N}{N_1} \Vert \bm{U}_{i,\cdot}^\star \Vert_2^2  \Big) 
	\overset{\text{(d)}}{\leq} \frac{\delta}{6 \sqrt{\spicy}} \gamma_{i,t}^\star
\end{align*}
Here step (a) follows from the sub-block condition \eqref{EqnSubmatrix}, step (b) follows from the bound \eqref{eq:UH-U-star-i} in \Cref{lemma:subspace-error-1}, step (c) follows from the noise condition \eqref{eq:noise-condition-est} and holds provided that $\cinc$ is sufficiently small, and step (d) uses \eqref{eq:gamma-lb-1}. Taking above bounds on $\alpha_1$, $\alpha_2$ and $\alpha_3$ collectively yields
\begin{align*}
	\Delta_U  \leq \alpha_1 + \alpha_2 + \alpha_3 \leq \frac{\delta}{2\sqrt{\spicy}} \gamma_{i,t}^\star.
\end{align*}
A similar argument can be used to show that $\Delta_V \coloneqq \sigma^2 | \bm{V}_{t,\cdot}^{\star}(\bm{V}_{1}^{\star\top}
\bm{V}_{1}^{\star})^{-1}\bm{V}_{t,\cdot}^{\top} - 
\bm{V}_{t,\cdot}^{\star}(\bm{V}_{1}^{\star\top}
\bm{V}_{1}^{\star})^{-1}\bm{V}_{t,\cdot}^{\star\top} |$ admits the same upper bound. Putting the bounds for $\Delta_U$ and $\Delta_V$ together gives the claimed bound.



\section{Proof of~\Cref{thm:distribution-general}}
\label{sec:proof-general-distribution}

So as to simplify notation, we introduce the shorthand
$\bar{\bm{M}}^{\star} \coloneqq \bm{M}_{1:\bar{N},1:\bar{T}}^{\star}$.
Note that $\bar{ \bm{M}}^{\star}$ can be viewed as the ground truth of
the four-block data matrix $ \bm{M}^{(i_{0},j_{0})}$. The following
lemma characterizes its spectrum.

\begin{lemma}
  \label{lemma:submatrix-spectrum}
Under the sub-block condition \eqref{EqnSubmatrix-general}, we have
\begin{align*}
  \clow \sqrt{ \frac{\bar{N}\bar{T}}{NT}}\sigma_{r}^{\star} \leq
  \sigma_{r}(\bar{ \bm{M}}^{\star}) \leq \sigma_{r}(\bar{
    \bm{M}}^{\star}) \leq \cupper \sqrt{
    \frac{\bar{N}\bar{T}}{NT}}\sigma_{r}^{\star}.
\end{align*}
\end{lemma}
\noindent See~\Cref{subsec:proof-lemma-submatrix-spectrum} for the
proof. \\

Beginning with the singular value decomposition $\bar{ \bm{M}}^{\star}
= \bar{ \bm{U}}^{\star}\bar{ \bm{\Sigma}}^{\star}\bar{ \bm{V}}^{\star
  \top}$, we can partition the matrices $\bar{ \bm{U}}^{\star}$ and $
\bm{V}^{\star}$ into two blocks
\begin{align*}
\bar{ \bm{U}}^{\star} = \left [\begin{array}{c} \bar{
      \bm{U}}_{1}^{\star}\\ \bar{ \bm{U}}_{2}^{\star}
\end{array} \right] \qquad \text{and} \qquad \bar{ \bm{V}}^{\star} =  \left [\begin{array}{c}
	\bar{ \bm{V}}_{1}^{\star}\\ \bar{ \bm{V}}_{2}^{\star}
\end{array} \right],
\end{align*}
where $\bar{ \bm{U}}_{1}^{\star} \in \mathbb{R}^{\bar{N}_{1}\times
  r}$, $\bar{ \bm{U}}_{2}^{\star} \in \mathbb{R}^{\bar{N}_{2}\times
  r}$, $ \bm{\bar{V}}_{1}^{\star} \in \mathbb{R}^{\bar{T}_{1}\times
  r}$ and $\bar{ \bm{V}}_{2}^{\star} \in \mathbb{R}^{\bar{T}_{2}\times
  r}$.  Note that $\bar{ \bm{M}}^{\star}$ can also be written as $
\bm{U}_{1:\bar{N},\cdot}^{\star} \bm{\Sigma}^\star(
\bm{V}_{1:\bar{T},\cdot}^{\star})^\top$.  We know that there exists
two invertible matrices $\bm{R}_{ \bm{U}}, \bm{R}_{ \bm{V}} \in
\mathbb{R}^{r\times r}$ such that $\bm{U}_{1:\bar{N},\cdot}^{\star} =
\bar{ \bm{U}}^{\star} \bm{R}_{\bm{U}}$ and $
\bm{V}_{1:\bar{T},\cdot}^{\star} = \bar{ \bm{V}}^{\star} \bm{R}_{
  \bm{V}}$.  Since $\bar{ \bm{U}}^{\star}$ has orthonormal columns, we
have $ \bm{R}_{ \bm{U}} = \bar{ \bm{U}}^{\star \top}
\bm{U}_{1:\bar{N},\cdot}^{\star}$ and therefore
\begin{align*}
 \bm{R}_{ \bm{U}} \bm{R}_{ \bm{U}}^{\top} = \bar{ \bm{U}}^{\star \top}
 \bm{U}_{1:\bar{N},\cdot}^{\star} \bm{U}_{1:\bar{N},\cdot}^{\star
   \top}\bar{ \bm{U}}^{\star} = \bar{ \bm{U}}^{\star \top}(
 \bm{U}_{1}^{\star \top} \bm{U}_{1}^{\star} + \bm{U}_{2}^{\star \top}
 \bm{U}_{2}^{\star})\bar{ \bm{U}}^{\star}.
\end{align*}
In view of the sub-block condition~\eqref{EqnSubmatrix-general}, we
know that
\begin{align*}
\clow \frac{\bar{N}}{N} \bm{I}_{r} = \clow \frac{\bar{N}}{N}\bar{
  \bm{U}}^{\star \top}\bar{ \bm{U}}^{\star} \preceq \bm{R}_{ \bm{U}}
\bm{R}_{ \bm{U}}^{\top} \preceq \cupper \frac{\bar{N}}{N}\bar{
  \bm{U}}^{\star \top}\bar{ \bm{U}}^{\star} = \cupper
\frac{\bar{N}}{N} \bm{I}_{r}.
\end{align*}
From equation~\eqref{EqnSubmatrix-general} and the fact that
$\bm{U}_{1}^{\star} = \bar{ \bm{U}}_{1}^{\star} \bm{R}_{ \bm{U}}$, we
can check that
\begin{align*}
  \bar{ \bm{U}}_{1}^{\star \top}\bar{ \bm{U}}_{1}^{\star}& = (
  \bm{R}_{ \bm{U}})^{-\top} \bm{U}_{1}^{\star \top}
  \bm{U}_{1}^{\star}( \bm{R}_{ \bm{U}})^{-1}\succeq \clow
  \frac{\bar{N}_{1}}{N}( \bm{R}_{ \bm{U}})^{-\top}( \bm{R}_{
    \bm{U}})^{-1} = \clow \frac{\bar{N}_{1}}{N}( \bm{R}_{ \bm{U}}
  \bm{R}_{ \bm{U}}^{\top})^{-1}\succeq \frac{\clow}{\cupper}
  \frac{\bar{N}_{1}}{\bar{N}} \bm{I}_{r}, \\ \bar{ \bm{U}}_{1}^{\star
    \top}\bar{ \bm{U}}_{1}^{\star}& = ( \bm{R}_{ \bm{U}})^{-\top}
  \bm{U}_{1}^{\star \top} \bm{U}_{1}^{\star}( \bm{R}_{ \bm{U}})^{-1}
  \preceq \cupper \frac{\bar{N}_{1}}{N}( \bm{R}_{ \bm{U}})^{-\top}(
  \bm{R}_{ \bm{U}})^{-1} = \cupper \frac{\bar{N}_{1}}{N}( \bm{R}_{
    \bm{U}} \bm{R}_{ \bm{U}}^{\top})^{-1} \preceq
  \frac{\cupper}{\clow} \frac{\bar{N}_{1}}{\bar{N}} \bm{I}_{r}
\end{align*}
hold. A similar argument can be used to verify that
\begin{align*}
\clow \frac{\bar{T}}{T} \bm{I}_{r} \preceq \bm{R}_{ \bm{V}} \bm{R}_{
  \bm{V}}^{\top} \preceq \cupper \frac{\bar{T}}{T}
\bm{I}_{r}\quad\text{and}\quad \frac{\clow}{\cupper}
\frac{\bar{T}_{1}}{\bar{T}} \bm{I}_{r} \preceq \bar{
  \bm{V}}_{1}^{\star \top}\bar{ \bm{V}}_{1}^{\star} \preceq
\frac{\cupper}{\clow} \frac{\bar{T}_{1}}{\bar{T}} \bm{I}_{r}.
\end{align*}
Therefore, the sub-block condition~\eqref{EqnSubmatrix} holds for
$\bar{ \bm{M}}^{\star}$. We can also check that the assumptions
of~\Cref{thm:distribution-general}, together
with~\Cref{lemma:submatrix-spectrum} guarantees that all conditions
of~\Cref{thm:distribution} hold. Then we can
apply~\Cref{thm:distribution} to the four-block data matrix
$\bm{M}^{(i_{0},j_{0})}$ to show that we can decompose
\begin{align*}
  \big( \widehat{ \bm{M}}_{i_0,j_0}- \bm{M}_{i_0,j_0}^{\star}
  \big)_{i,t} & = g_{i,t} + \Delta_{i,t},
\end{align*}
where $g_{i,t}$ is a mean-zero Gaussian random variable with variance
\begin{align*}
  &\mathsf{var} \left (g_{i,t} \right) = \frac{\noise^{2}}{NT} \big
  \Vert\bar{ \bm{U}}_{i,\cdot}^{\star}(\bar{ \bm{U}}_{1}^{\star
    \top}\bar{ \bm{U}}_{1}^{\star})^{-1}\bar{ \bm{U}}_{1}^{\star \top}
  \big \Vert_{2}^{2} + \frac{\noise^{2}}{NT} \big \Vert\bar{
    \bm{V}}_{1}^{\star}(\bar{ \bm{V}}_{1}^{\star \top}\bar{
    \bm{V}}_{1}^{\star})^{-1}\bar{ \bm{V}}_{t,\cdot}^{\star \top} \big
  \Vert_{2}^{2} \\
& \quad = \frac{\noise^{2}}{NT} \big \Vert \bm{U}_{i,\cdot}^{\star}
  \bm{R}_{ \bm{U}}^{-1}( \bm{R}_{ \bm{U}}^{-\top} \bm{U}_{1}^{\star
    \top} \bm{U}_{1}^{\star} \bm{R}_{ \bm{U}}^{-1})^{-1} \bm{R}_{
    \bm{U}}^{-\top} \bm{U}_{1}^{\star \top} \big \Vert_{2}^{2} +
  \frac{\noise^{2}}{NT} \big \Vert \bm{V}_{1}^{\star} \bm{R}_{
    \bm{V}}^{-1}( \bm{R}_{ \bm{V}}^{-\top} \bm{V}_{1}^{\star \top}
  \bm{V}_{1}^{\star} \bm{R}_{ \bm{V}}^{-1})^{-1} \bm{R}_{
    \bm{V}}^{-\top} \bm{V}_{t,\cdot}^{\star \top} \big
  \Vert_{2}^{2} \\
& \quad = \frac{\noise^{2}}{NT} \big \Vert \bm{U}_{i,\cdot}^{\star}(
  \bm{U}_{1}^{\star \top} \bm{U}_{1}^{\star})^{-1} \bm{U}_{1}^{\star
    \top} \big \Vert_{2}^{2} + \frac{\noise^{2}}{NT} \big \Vert
  \bm{V}_{1}^{\star}( \bm{V}_{1}^{\star \top} \bm{V}_{1}^{\star})^{-1}
  \bm{V}_{t,\cdot}^{\star \top} \big \Vert_{2}^{2} =
  \gamma_{i,t}^{\star},
\end{align*}
and $|\Delta_{i,t}| \leq \delta \gamma_{i,t}^{\star1/2}$ holds with
probability at least $1 - O((N + T)^{-10})$.


\subsection{Proof of~\Cref{lemma:submatrix-spectrum}}
\label{subsec:proof-lemma-submatrix-spectrum}

We can use~\Cref{lemma:submatrix-properties} to
establish~\Cref{lemma:submatrix-spectrum}. In view of
the sub-block condition \eqref{EqnSubmatrix-general}, we have
\begin{align*}
\clow \frac{\bar{T}}{T} \bm{I}_{r} \preceq
\bm{V}_{1:\bar{T},\cdot}^{\star \top} \bm{V}_{1:\bar{T},\cdot}^\star =
\bm{V}_{1}^{\star \top} \bm{V}_{1}^{\star} + \bm{V}_{2}^{\star \top}
\bm{V}_{2}^{\star} \preceq \cupper \frac{\bar{T}}{T} \bm{I}_{r}.
\end{align*}
Consequently, we can apply the
bound~\eqref{eq:submatrix-left-spectrum} to $\bm{M}^\star$ so as to
obtain
\begin{align*}
\sqrt{\clow \frac{\bar{T}}{T}}\sigma_{r}^{\star} \leq \sigma_{r}(
\bm{M}_{\cdot, 1:\bar{T}}^{\star}) \leq \sqrt{\cupper
  \frac{\bar{T}}{T}}\sigma_{r}^{\star}.
\end{align*}
We form the SVD $\bm{M}_{\cdot,1:\bar{T}}^{\star} =
\bm{U}_{\mathsf{left}}^{\star} \bm{\Sigma}_{\mathsf{left}}^{\star}
\bm{V}_{\mathsf{left}}^{\star \top}$.  Since the matrix
$\bm{M}_{\cdot,1:\bar{T}}^{\star}$ can also be written as
$\bm{U}^{\star} \bm{\Sigma}^{\star} \bm{V}_{1:\bar{T},\cdot}^{\star
  \top}$, there exists some rotation matrix $ \bm{H}$ such that
$\bm{U}_{\mathsf{left}}^{\star} = \bm{U}^{\star} \bm{H}$.  If we write
\begin{align*}
\bm{U}_{\mathsf{left}}^{\star} = \left [\begin{array}{c}
    \bm{U}_{\mathsf{left},1}^{\star}\\ \bm{U}_{\mathsf{left},2}^{\star}
\end{array} \right ]
\end{align*}
where $\bm{U}_{\mathsf{left},1}^{\star} \in \mathbb{R}^{\bar{N} \times
  r}$ and $\bm{U}_{\mathsf{left},2}^{\star} \in
\mathbb{R}^{(N-\bar{N})\times r}$.  Thus, we can write
\begin{align*}
\bm{U}_{\mathsf{left},1}^{\star \top} \bm{U}_{\mathsf{left},1}^{\star}
= \bm{H}^{\top} \bm{U}_{1:\bar{N},\cdot}^{\star \top}
\bm{U}_{1:\bar{N},\cdot}^{\star} \bm{H} = \bm{H}^{\top}(
\bm{U}_{1}^{\star \top} \bm{U}_{1}^{\star} + \bm{U}_{2}^{\star \top}
\bm{U}_{2}^{\star}) \bm{H}.
\end{align*}
In view of the sub-block condition~\eqref{EqnSubmatrix-general}, we
know that
\begin{align*}
\clow \frac{\bar{N}}{N} \bm{I}_{r} \preceq
\bm{U}_{\mathsf{left},1}^{\star \top} \bm{U}_{\mathsf{left},1}^{\star}
\preceq \cupper \frac{\bar{N}}{N} \bm{I}_{r}.
\end{align*}
Therefore, we can apply the bound~\eqref{eq:submatrix-upper-spectrum}
on $\bm{M}_{\cdot,1:\bar{T}}^{\star}$ to obtain
\begin{align*}
  \clow \sqrt{ \frac{\bar{N}\bar{T}}{NT}}\sigma_{r}^{\star} \leq
  \sqrt{\clow \frac{\bar{N}}{N}}\sigma_{r}(
  \bm{M}_{\cdot,1:\bar{T}}^{\star}) \leq \sigma_{r}(
  \bm{M}_{1:\bar{N},1:\bar{T}}^{\star}) \leq \sqrt{\cupper
    \frac{\bar{N}}{N}}\sigma_{r}( \bm{M}_{\cdot,1:\bar{T}}^{\star})
  \leq \cupper \sqrt{ \frac{\bar{N}\bar{T}}{NT}} \sigma_{r}^{\star}.
\end{align*}


\section{Auxiliary results on matrix denoising}
\label{appendix:proof-thm-denoising}

In this appendix, we collect together various results on matrix
denoising that are used in proving~\Cref{lem:master}.  Let
$\bm{M}^{\star} \in \mathbb{R}^{n_{1}\times n_{2}}$ be a rank $r$
matrix with SVD $\bm{M}^{\star} = \bm{U}^{\star} \bm{\Sigma}^{\star}
\bm{V}^{\star\top}$, where $\bm{U}^{\star} \in
\mathbb{R}^{n_{1}\times r}$ and $\bm{V}^{\star} \in
\mathbb{R}^{n_{2}\times r}$ have orthonormal columns, and $
\bm{\Sigma}^{\star} = \mathsf{diag}(\sigma_{1}^{\star}, \ldots,
\sigma_{r}^{\star})$ is a diagonal matrix with
$\sigma_{1}^{\star}\geq\sigma_{2}^{\star}\geq\cdots\geq\sigma_{r}^{\star}$.
For simplicity, let $n\coloneqq\max\{n_{1}, n_{2}\}$. Let $\bm{E}$ be
a noise matrix whose entries are independent sub-Gaussian random
variables with sub-Gaussian norm at most $\sigma$, i.e.~$ \Vert E_{i,
  j} \Vert_{\psi_{2}} \leq \sigma$ (see \cite[Definition
  2.5.6]{vershynin2016high}). Suppose that we observe
\begin{align*}
 \bm{M} =  \bm{M}^{\star} +  \bm{E}
\end{align*}
and our goal is to estimate the unknown low-rank matrix
$\bm{M}^{\star}$ and its singular supspaces $\bm{U}^{\star}$ and
$\bm{V}^{\star}$.  Perhaps the most natural approach is to compute the
(truncated) rank-$r$ SVD $\bm{U} \bm{\Sigma} \bm{V}^{\top}$ of $
\bm{M}$, and use $\bm{U}$, $\bm{V}$ and $\bm{U} \bm{\Sigma}
\bm{V}^{\top}$ to estimate $\bm{U}^{\star}$, $\bm{V}^{\star}$ and $
\bm{M}^{\star}$ respectively.

Our first result concerns subspace estimation error. Due to rotational
ambiguity of linear subspaces, namely for any rotation matrix $\bm{O}
\in \mathcal{O}^{r\times r}$, $\bm{U} \bm{O}$ (resp.~$\bm{V}
\bm{O}$) represents the same subspace as $\bm{U}$ (resp.~$\bm{V}$),
we define the optimal rotation matrices $\bm{R}_{ \bm{U}}$ that best
aligns $\bm{U}$ with $\bm{U}^{\star}$ in the Euclidean sense:
\begin{align*}
 \bm{R}_{ \bm{U}}\coloneqq\mathop{\arg\min}_{ \bm{R} \in
   \mathcal{O}^{r\times r}} \left \Vert \bm{U} \bm{R}- \bm{U}^{\star}
 \right \Vert _{\mathrm{F}}.
\end{align*}
We know that $\bm{R}_{ \bm{U}}$ is the matrix sign of $\bm{U}^{\top}
\bm{U}^{\star}$, which is given by $\bm{X} \bm{Y}^{\top}$ where $
\bm{X} \bm{Y} \bm{Z}^{\top}$ is the SVD of $\bm{U}^{\top}
\bm{U}^{\star}$. Similarly we define
\begin{align*}
 \bm{R}_{ \bm{V}} \coloneqq \mathop{\arg\min}_{ \bm{R} \in
   \mathcal{O}^{r\times r}} \left \Vert \bm{V} \bm{R}- \bm{V}^{\star}
 \right \Vert _{\mathrm{F}}.
\end{align*}
The following proposition gives first-order expansions of $\bm{U}
\bm{R}_{ \bm{U}}- \bm{U}^{\star}$ and $\bm{V} \bm{R}_{ \bm{V}}-
\bm{V}^{\star}$ with respect to the random perturbation matrix $
\bm{E}$.

\begin{proposition}
  \label{prop:denoising_main_1}
Suppose that $\sigma \sqrt{n} \leq \cnoise \sigma_{r}^{\star}$ for some sufficiently small constant $\cnoise>0$.  Then we can
decompose
\begin{subequations}
\label{eq:UV-decompose-denoising}
\begin{align}
\label{eq:UR-decompose-denoising}     
 \bm{U} \bm{R}_{ \bm{U}}- \bm{U}^{\star} & = \bm{E} \bm{V}^{\star}
 \left( \bm{\Sigma}^{\star} \right )^{-1} + \bm{\Psi}_{\bm{U}}, \quad \mbox{and} \\
\label{eq:VR-decompose-denoising} 
\bm{V} \bm{R}_{ \bm{V}}- \bm{V}^{\star} & = \bm{E}^{\top}
\bm{U}^{\star} \left( \bm{\Sigma}^{\star} \right )^{-1} +
\bm{\Psi}_{\bm{V}},
\end{align}
\end{subequations}
where $\bm{\Psi}_{ \bm{U}}$ and $\bm{\Psi}_{ \bm{V}}$ are two
matrices that satisfy: there exists some universal constant $c_1>0$ such that for any given $\mathcal{I}\subseteq[n_{1}]$ and
$\mathcal{J}\subseteq[n_{2}]$, with probability exceeding
$1-O(n^{-10})$,
\begin{subequations}
\label{eq:denoising-Psi-bounds}
\begin{align}
 \Vert( \bm{\Psi}_{ \bm{U}})_{\mathcal{I}, \cdot} \Vert & \leq c_1
 \frac{\sigma \sqrt{I + r + \log n}}{\sigma_{r}^{\star}} \bigg
      [\frac{\sigma \sqrt{I + n_{2} + \log
            n}}{\sigma_{r}^{\star}} +
        \frac{\sigma^2 n_{1}}{\sigma_{r}^{\star2}}
        \bigg] + c_1 \bigg (\frac{\sigma^2
        n}{\sigma_{r}^{\star2}} + \frac{\sigma \sqrt{r + \log
          n}}{\sigma_{r}^{\star}} \bigg ) \Vert
      \bm{U}_{\mathcal{I}, \cdot}^{\star}
      \Vert, \label{eq:denoising-Psi-U} \\
 \Vert( \bm{\Psi}_{ \bm{V}})_{\mathcal{J}, \cdot} \Vert
& \leq c_1 \frac{\sigma \sqrt{J + r + \log n}}{\sigma_{r}^{\star}}
\bigg [\frac{\sigma \sqrt{J + n_{1} + \log
      n}}{\sigma_{r}^{\star}} + \frac{\sigma^2 n_{2}}{\sigma_{r}^{\star2}} \bigg
] + c_1 \bigg (\frac{\sigma^2
  n}{\sigma_{r}^{\star2}} + \frac{\sigma
  \sqrt{r + \log n}}{\sigma_{r}^{\star}} \bigg ) \Vert
\bm{V}_{\mathcal{J}, \cdot}^{\star} \Vert, \label{eq:denoising-Psi-V}
\end{align}
\end{subequations}
where $I = |\mathcal{I}|$ and $J = |\mathcal{J}|$.  In addition, with probability exceeding $1-O( n^{-10} )$,
\begin{subequations}
\label{eq:denoising-UV-I-bounds}
\begin{align}
 \left \Vert \bm{U}_{\mathcal{I}, \cdot} \bm{R}_{ \bm{U}}-
 \bm{U}_{\mathcal{I}, \cdot}^{\star} \right \Vert & \leq c_1
 \frac{\sigma^2 n}{\sigma_{r}^{\star2}} \left \Vert
 \bm{U}_{\mathcal{I}, \cdot}^{\star} \right \Vert + c_1
 \frac{\sigma}{\sigma_{r}^{\star}} \sqrt{I + r + \log
   n}, \label{eq:denoising-U-I} \\
 \left \Vert \bm{V}_{\mathcal{J}, \cdot} \bm{R}_{ \bm{V}}-
 \bm{V}_{\mathcal{J}, \cdot}^{\star} \right \Vert & \leq c_1
 \frac{\sigma^2 n}{\sigma_{r}^{\star2}} \left \Vert
 \bm{V}_{\mathcal{J}, \cdot}^{\star} \right \Vert + c_1
 \frac{\sigma}{\sigma_{r}^{\star}} \sqrt{J + r + \log
   n}.\label{eq:denoising-V-I}
\end{align}
\end{subequations}
\end{proposition}

In fact, the proof of~\Cref{prop:denoising_main_1} also produces
another expansion for $\bm{U} \bm{\Sigma} \bm{R}_{ \bm{U}}-
\bm{U}^{\star} \bm{\Sigma}^{\star}$ and $\bm{V} \bm{\Sigma} \bm{R}_{
  \bm{V}}- \bm{V}^{\star} \bm{\Sigma}^{\star}$, which is given below.

%

\begin{proposition}
\label{prop:denoising_main_2}
Under the condition of \Cref{prop:denoising_main_1}, we can
decompose
\begin{subequations}
\begin{align}
\label{eq:USigmaR-decompose-denoising}   
\bm{U} \bm{\Sigma} \bm{R}_{ \bm{U}}- \bm{U}^{\star}
\bm{\Sigma}^{\star} & = \bm{E} \bm{V}^{\star} + \bm{\Delta}_{ \bm{U}},
\quad \mbox{and} \\
\label{eq:VSigmaR-decompose-denoising}
\bm{V} \bm{\Sigma} \bm{R}_{ \bm{V}}- \bm{V}^{\star}
\bm{\Sigma}^{\star} & = \bm{E}^{\top} \bm{U}^{\star} + \bm{\Delta}_{
  \bm{V}},
\end{align}
\end{subequations}
where $\bm{\Delta}_{ \bm{U}}$ and $\bm{\Delta}_{ \bm{V}}$ are two
matrices that satisfy: there exists some universal constant $c_2>0$ such that for any given $\mathcal{I}\subseteq[n_{1}]$ and
$\mathcal{J}\subseteq[n_{2}]$, with probability exceeding
$1-O(n^{-10})$,
\begin{subequations}
\label{eq:denoising-Delta-bounds}
\begin{align}
 \Vert( \bm{\Delta}_{ \bm{U}})_{\mathcal{I}, \cdot} \Vert & \leq c_2
 \sigma \sqrt{I + r + \log n} \bigg [\frac{\sigma}{\sigma_{r}^{\star}}
   \sqrt{I + n_{2} + \log n} + \frac{\sigma^2
     n_{1}}{\sigma_{r}^{\star2}} \bigg ] + c_2 \frac{\sigma^2
   n}{\sigma_{r}^{\star}} \Vert \bm{U}_{\mathcal{I}, \cdot}^{\star}
 \Vert, \label{eq:denoising-Delta-U} \\
 \Vert(\bm{\Delta}_{ \bm{V}})_{\mathcal{J}, \cdot} \Vert & \leq c_2
 \sigma \sqrt{J + r + \log n} \bigg [\frac{\sigma}{\sigma_{r}^{\star}}
   \sqrt{J + n_{1} + \log n} + \frac{\sigma^2
     n_{2}}{\sigma_{r}^{\star2}} \bigg ] + c_2 \frac{\sigma^2
   n}{\sigma_{r}^{\star}} \Vert \bm{V}_{\mathcal{J}, \cdot}^{\star}
 \Vert, \label{eq:denoising-Delta-V}
\end{align}
\end{subequations}
where $I = |\mathcal{I}|$ and $J = |\mathcal{J}|$.  In addition, with
probability exceeding
$1-O(n^{-10})$,
\begin{subequations}
\label{eq:denoising-UVSigma-I-bounds}
\begin{align}
 \left \Vert \bm{U} \bm{\Sigma} \bm{R}_{ \bm{U}}- \bm{U}^{\star}
 \bm{\Sigma}^{\star} \right \Vert & \leq c_2 \frac{\sigma^2
   n}{\sigma_{r}^{\star}} \Vert \bm{U}_{\mathcal{I}, \cdot}^{\star}
 \Vert + c_2 \sigma \sqrt{I + r + \log
   n}, \label{eq:denoising-USigma-I} \\
 \left \Vert \bm{V} \bm{\Sigma} \bm{R}_{ \bm{V}}- \bm{V}^{\star}
 \bm{\Sigma}^{\star} \right \Vert & \leq c_2 \frac{\sigma^2
   n}{\sigma_{r}^{\star}} \Vert \bm{V}_{\mathcal{J}, \cdot}^{\star}
 \Vert + c_2 \sigma \sqrt{J + r + \log n}.\label{eq:denoising-VSigma-I}
\end{align}
\end{subequations}
\end{proposition}

\Cref{prop:denoising_main_1} mainly concerns characterizing the
subspace perturbation error. In fact, we can readily characterize the
estimation error for the underlying low-rank matrix $\bm{M}^{\star}$,
which is presented as follows.

\begin{corollary}
  \label{corollary:full-matrix}
Under the condition of \Cref{prop:denoising_main_1}, we can
decompose
\begin{align*}
 \bm{U} \bm{\Sigma} \bm{V}^{\top}- \bm{M}^{\star} & = \bm{U}^{\star}
 \bm{U}^{\star\top} \bm{E} + \bm{E} \bm{V}^{\star} \bm{V}^{\star\top}
 + \bm{\Phi}
\end{align*}
where $\bm{\Phi}$ is some matrix that satisfies: there exists some universal constant $c_3>0$ such that for any given
$\mathcal{I}\subseteq[n_{1}]$ and $\mathcal{J}\subseteq[n_{2}]$, with
probability exceeding $1-O(n^{-10})$, we have
\begin{align}
 & \big \Vert \bm{\Phi}_{\mathcal{I}, \mathcal{J}} \big \Vert \leq c_3
  \frac{\sigma^{2}}{\sigma_{r}^{\star}} \sqrt{ \left(I + r + \log n
    \right ) \left(J + r + \log n \right )} + c_3 \sigma \sqrt{J + r +
    \log n} \big(\frac{\sigma \sqrt{J + n_{1}} }{\sigma_{r}^{\star}} +
  \frac{\sigma^2 n_{2}}{\sigma_{r}^{\star2}} \big) \big \Vert
  \bm{U}_{\mathcal{I}, \cdot}^{\star} \big \Vert\nonumber \\
\label{eq:denoising-Phi-J}  
& \quad + c_3 \sigma \sqrt{I + r + \log n}
\big(\frac{\sigma \sqrt{I + n_{2}} }{\sigma_{r}^{\star}} +
\frac{\sigma^2 n_{1}}{\sigma_{r}^{\star2}} \big) \big \Vert
\bm{V}_{\mathcal{J}, \cdot}^{\star} \big \Vert + c_3 \big(\frac{\sigma^2
  n}{\sigma_{r}^{\star}} + \sigma \sqrt{r + \log n} \big) \big \Vert
\bm{U}_{\mathcal{I}, \cdot}^{\star} \big \Vert \big \Vert
\bm{V}_{\mathcal{J}, \cdot}^{\star} \big \Vert
\end{align}
where $I = |\mathcal{I}|$ and $J = |\mathcal{J}|$. As a direct
consequence, with probability exceeding $1-O(n^{-10})$,
\begin{align}
 \big \Vert( \bm{U} \bm{\Sigma} \bm{V}^{\top}-
 \bm{M}^{\star})_{\mathcal{I}, \mathcal{J}} \big \Vert & \leq c_3
 \sigma \sqrt{J + r + \log n} \big \Vert \bm{U}_{\mathcal{I},
   \cdot}^{\star} \big \Vert + c_3 \sigma \sqrt{I + r + \log n} \big \Vert
 \bm{V}_{\mathcal{J}, \cdot}^{\star} \big \Vert\nonumber \\
\label{eq:denoising-1st}
 & \qquad + c_3 \frac{\sigma^2 n}{\sigma_{r}^{\star}} \big \Vert
\bm{U}_{\mathcal{I}, \cdot}^{\star} \big \Vert \big \Vert
\bm{V}_{\mathcal{J}, \cdot}^{\star} \big \Vert + c_3 \frac{\sigma^2
}{\sigma_{r}^{\star}} \sqrt{ \left(I + r + \log n \right ) \left(J + r
  + \log n \right )}.
\end{align}
\end{corollary}
\noindent See~\Cref{subsec:proof-full-matrix} for the proof.

Last but not least, the following result will be useful when being
applied to the causal panel data model.

\begin{corollary}
\label{corollary:inner-product-sub}
Under the condition of \Cref{prop:denoising_main_1}, there exists some universal constant $c_4>0$ such that with
probability exceeding $1-O(n^{-10})$, 
\begin{align*}
 \big \Vert \bm{U}_{\mathcal{I}, \cdot}^{\star\top}(
 \bm{U}_{\mathcal{I}, \cdot} \bm{R}_{ \bm{U}}- \bm{U}_{\mathcal{I},
   \cdot}^{\star}) \big \Vert & \leq c_4 \bigg [\frac{\sigma \sqrt{r +
       \log n}}{\sigma_{r}^{\star}} + \frac{\sigma^2 \sqrt{n_{2}
       \left(n_{2} + I \right )}}{\sigma_{r}^{\star2}} +
   \frac{\sigma^{3}n_{1}}{\sigma_{r}^{\star3}} \sqrt{I} \bigg ] \Vert
 \bm{U}_{\mathcal{I}, \cdot}^{\star} \Vert + c_4 \frac{\sigma^2
   n}{\sigma_{r}^{\star2}} \Vert \bm{U}_{\mathcal{I}, \cdot}^{\star}
 \Vert^{2}.
\end{align*}
\end{corollary}
\noindent See~\Cref{subsec:proof-inner-product-sum} for the proof.


\subsection{Proofs of \Cref{prop:denoising_main_1,prop:denoising_main_2}}
\label{sec:proof-denoising-main}

\subsubsection{Step 1: Establishing preliminary lemmas}

We first define two matrices
\begin{align*}
 \bm{H}_{ \bm{U}}\coloneqq \bm{U}^{\top} \bm{U}^{\star}\qquad\text{and}\qquad \bm{H}_{ \bm{V}}\coloneqq \bm{V}^{\top} \bm{V}^{\star}.
\end{align*}
Lemma \ref{lemma:denoising-basic-facts-crude} below shows that $
\bm{H}_{ \bm{U}}$ (resp.~$\bm{H}_{ \bm{V}}$) is a good proxy of the
rotation matrix $\bm{R}_{ \bm{U}}$ (resp.~$\bm{R}_{ \bm{V}}$), which
depends on $\bm{U}^{\top} \bm{U}^{\star}$ (resp.~$\bm{V}^{\top}
\bm{V}^{\star}$) in a highly nonlinear way.  As we will see from
the proof, $\bm{H}_{ \bm{U}}$ and $\bm{H}_{ \bm{V}}$ are much easier
to deal with. We start with applying the celebrated Wedin's $\sin
\Theta$ Theorem \citep{wedin1972perturbation,chen2020spectral}, which gives us the following perturbation bounds.

\begin{lemma}
  \label{lemma:denoising-basic-facts-crude}
Under the condition of \Cref{prop:denoising_main_1}, there exists some universal constant $c_5>0$ such that with
probability exceeding $1-O(n^{-10})$,
\begin{align*}
\max \left\{ \left \Vert \bm{U} \bm{R}_{ \bm{U}}- \bm{U}^{\star}
\right \Vert , \left \Vert \bm{V} \bm{R}_{ \bm{V}}- \bm{V}^{\star}
\right \Vert \right\} & \leq c_5 \frac{\sigma}{\sigma_{r}^{\star}}
\sqrt{n}, \quad \text{and} \\
\max \left\{ \left \Vert \bm{\Sigma} \left( \bm{R}_{ \bm{U}}- \bm{H}_{
  \bm{U}} \right ) \right \Vert , \left \Vert \bm{\Sigma} \left(
\bm{R}_{ \bm{V}}- \bm{H}_{ \bm{V}} \right ) \right \Vert \right\}
& \leq c_5 \frac{\sigma^2 n}{\sigma_{r}^{\star}}.
\end{align*}
As an immediate consequence,  we have 
\begin{align*}
\frac{1}{2} \leq \sigma_{i} \left( \bm{H}_{ \bm{U}} \right ) \leq
2\qquad\text{and}\qquad & \frac{1}{2} \leq \sigma_{i} \left( \bm{H}_{
  \bm{V}} \right ) \leq 2
\end{align*}
for all $1 \leq i \leq r$.
\end{lemma}
\noindent See~\Cref{appendix:proof-denoising-basic-facts} for the proof.

From Weyl's inequality and~\Cref{lemma:gaussian-spectral}, we know
that $ \left \Vert \bm{\Sigma}- \bm{\Sigma}^{\star} \right \Vert \leq
\left \Vert \bm{E} \right \Vert \lesssim \sigma \sqrt{n}$.  The next
lemma shows that if we properly rotate $\bm{\Sigma}$ using $\bm{R}_{
  \bm{U}}$ and $\bm{R}_{ \bm{V}}$ (or $\bm{H}_{ \bm{U}}$ and $
\bm{H}_{ \bm{V}}$), it can become even closer to $
\bm{\Sigma}^{\star}$.

\begin{lemma}
\label{lemma:denoising-approx-2-crude}
Under the condition of \Cref{prop:denoising_main_1}, there exists some universal constant $c_6>0$ such that with
probability exceeding $1-O(n^{-10})$,
\begin{align*}
 \big \Vert \bm{R}_{ \bm{U}}^{\top} \bm{\Sigma} \bm{R}_{ \bm{V}}-
 \bm{\Sigma}^{\star} \big \Vert & \leq c_6 \frac{\sigma^2
   n}{\sigma_{r}^{\star}} + c_6 \sigma \sqrt{r + \log
   n} \quad \text{and} \quad \big \Vert \bm{H}_{ \bm{U}}^{\top}
 \bm{\Sigma} \bm{H}_{ \bm{V}}- \bm{\Sigma}^{\star} \big \Vert \leq c_6
 \frac{ \left(\sigma \sqrt{n} \right )^{3}}{\sigma_{r}^{\star2}} + c_6
 \sigma \sqrt{r + \log n}.
\end{align*}
\end{lemma}
\noindent See~\Cref{appendix:proof-lemma-denoising-approx-2} for the
proof.

The next lemma characterizes the spectral norm bound for submatrices
that are formed by a subset $\mathcal{I}$ (resp.~$\mathcal{J}$) of
rows of $\bm{U} \bm{\Sigma} \bm{H}_{ \bm{V}}- \bm{M} \bm{V}^{\star}$
(resp.~$\bm{V} \bm{\Sigma} \bm{H}_{ \bm{U}}- \bm{M}^{\top}
\bm{U}^{\star}$), which will play a crucial crucial in our analysis.

\begin{lemma}
  \label{lemma:denoising-approx-1-crude}
Under the condition of \Cref{prop:denoising_main_1}, there exists some universal constant $c_7>0$ such that for any given
$\mathcal{I}\subseteq[n_{1}]$ and $\mathcal{J}\subseteq[n_{2}]$, with
probability exceeding $1-O(n^{-10})$, we have
\begin{align*}
 \big \Vert \bm{U}_{\mathcal{I}, \cdot} \bm{\Sigma} \bm{H}_{ \bm{V}}-
 \bm{M}_{\mathcal{I}, \cdot} \bm{V}^{\star} \big \Vert & \leq c_7
 \frac{\sigma^2 n}{\sigma_{r}^{\star}} \Vert \bm{U}_{\mathcal{I},
   \cdot}^{\star} \Vert + c_7 \frac{\sigma^2 }{\sigma_{r}^{\star}} \sqrt{n
   \left(I + r + \log n \right )} + c_7 \frac{\sigma^2
   n_{2}}{\sigma_{r}^{\star}} \big \Vert \bm{U}_{\mathcal{I}, \cdot}
 \bm{H}_{ \bm{U}}- \bm{U}_{\mathcal{I}, \cdot}^{\star} \big \Vert,
 \\
 \big \Vert \bm{V}_{\mathcal{J}, \cdot} \bm{\Sigma} \bm{H}_{ \bm{U}}-
 \bm{M}_{\cdot, \mathcal{J}}^{\top} \bm{U}^{\star} \big \Vert &
 \leq c_7 \frac{\sigma^2 n}{\sigma_{r}^{\star}} \Vert
 \bm{V}_{\mathcal{J}, \cdot}^{\star} \Vert + c_7 \frac{\sigma^2
 }{\sigma_{r}^{\star}} \sqrt{n \left(J + r + \log n \right )} + c_7
 \frac{\sigma^2 n_{1}}{\sigma_{r}^{\star}} \big \Vert
 \bm{V}_{\mathcal{J}, \cdot} \bm{H}_{ \bm{V}}- \bm{V}_{\mathcal{J},
   \cdot}^{\star} \big \Vert,
\end{align*}
where $I = |\mathcal{I}|$ and $J = |\mathcal{J}|$.
\end{lemma}
\noindent 
See~\Cref{appendix:proof-denoising-approx-1} for the proof.

Lemma \ref{lemma:denoising-approx-2-crude} and Lemma
\ref{lemma:denoising-approx-1-crude} allow us to derive the following
estimation error bounds.

\begin{lemma}
  \label{lemma:denoising-est-crude}
Under the condition of \Cref{prop:denoising_main_1}, there exists some universal constant $c_8>0$ such that for any given 
$\mathcal{I}\subseteq[n_{1}]$ and $\mathcal{J}\subseteq[n_{2}]$, with
probability exceeding $1-O(n^{-10})$
\begin{align*}
\big \Vert \bm{U}_{\mathcal{I}, \cdot} \bm{H}_{ \bm{U}}-
\bm{U}_{\mathcal{I}, \cdot}^{\star} \big \Vert & \leq c_8
\frac{\sigma^2 n}{\sigma_{r}^{\star2}} \Vert \bm{U}_{\mathcal{I},
  \cdot}^{\star} \Vert + c_8 \frac{\sigma}{\sigma_{r}^{\star}} \sqrt{I + r
  + \log n}, \\ \big \Vert \bm{V}_{\mathcal{J}, \cdot} \bm{H}_{
  \bm{V}}- \bm{V}_{\mathcal{J}, \cdot}^{\star} \big \Vert & \leq c_8
\frac{\sigma^2 n}{\sigma_{r}^{\star2}} \Vert \bm{V}_{\mathcal{J},
  \cdot}^{\star} \Vert + c_8 \frac{\sigma}{\sigma_{r}^{\star}} \sqrt{J + r
  + \log n}.
\end{align*}
\end{lemma}
\noindent See \Cref{appendix:proof-denoising-est} for the proof.

\subsubsection{Step 2: Deriving a crude bound}
\label{subsec:crude-bound-proof}

In this section, we will prove a crude bound using the lemmas
in the last section. More specifically, instead of showing
(\ref{eq:UV-decompose-denoising}) with guarantees in
(\ref{eq:denoising-Psi-bounds}), we first establish weaker guarantees
as follows:
\begin{subequations}
\label{eq:denoising-Psi-bounds-crude}
\begin{align}
\label{eq:denoising-Psi-U-crude} \left \Vert ( \bm{\Psi}_{
    \bm{U}})_{\mathcal{I}, \cdot} \right \Vert & \leq c_9
\frac{\sigma^2 }{\sigma_{r}^{\star2}} \sqrt{n \left(I + r + \log n
  \right )} + c_9 \bigg (\frac{\sigma^2 n}{\sigma_{r}^{\star2}} +
\frac{\sigma \sqrt{r + \log n}}{\sigma_{r}^{\star}} \bigg ) \left
\Vert \bm{U}_{\mathcal{I}, \cdot}^{\star} \right \Vert, \quad
\mbox{and} \\
\label{eq:denoising-Psi-V-crude}
 \left \Vert ( \bm{\Psi}_{ \bm{V}})_{\mathcal{J}, \cdot} \right \Vert
 & \leq c_9 \frac{\sigma^2 }{\sigma_{r}^{\star2}} \sqrt{n \left(J + r
   + \log n \right )} + c_9 \bigg (\frac{\sigma^2 n}{\sigma_{r}^{\star2}}
 + \frac{\sigma \sqrt{r + \log n}}{\sigma_{r}^{\star}} \bigg ) \left
 \Vert \bm{V}_{\mathcal{J}, \cdot}^{\star} \right \Vert
\end{align}
\end{subequations}
for some universal constant $c_9>0$.
The rest of this section is devoted to proving the
bounds~\eqref{eq:denoising-Psi-bounds-crude}.  

We first use the
triangle inequality to achieve
\begin{align*}
 \left \Vert \bm{U}_{\mathcal{I}, \cdot} \bm{\Sigma} \bm{R}_{ \bm{V}}-
 \bm{M}_{\mathcal{I}, \cdot} \bm{V}^{\star} \right \Vert & \leq
 \underbrace{ \left \Vert \bm{U}_{\mathcal{I}, \cdot} \bm{\Sigma}
   \bm{H}_{ \bm{V}}- \bm{M}_{\mathcal{I}, \cdot} \bm{V}^{\star} \right
   \Vert }_{\eqqcolon\alpha_{1}} + \underbrace{ \left \Vert
   \bm{U}_{\mathcal{I}, \cdot} \bm{\Sigma} \left( \bm{H}_{ \bm{V}}-
   \bm{R}_{ \bm{V}} \right ) \right \Vert }_{\eqqcolon\alpha_{2}}.
\end{align*}
We start by controlling $\alpha_{1}$. Taking
\Cref{lemma:denoising-approx-1-crude,lemma:denoising-est-crude} collectively, we know that with
probability exceeding $1-O(n^{-10})$,
\begin{align}
\alpha_{1} & \leq c_7 \frac{\sigma^2 n}{\sigma_{r}^{\star}} \left
\Vert \bm{U}_{\mathcal{I}, \cdot}^{\star} \right \Vert + c_7
\frac{\sigma^2 }{\sigma_{r}^{\star}} \sqrt{n \left(I + r + \log n
  \right )} + c_7 \frac{\sigma^2 n_{2}}{\sigma_{r}^{\star}} \bigg
( c_8 \frac{\sigma^2 n}{\sigma_{r}^{\star2}} \left \Vert
\bm{U}_{\mathcal{I}, \cdot}^{\star} \right \Vert + c_8
\frac{\sigma}{\sigma_{r}^{\star}} \sqrt{I + r + \log n} \bigg
)\nonumber \\ 
& \leq (c_7 + c_7 c_8 \cnoise^2)\frac{\sigma^2 n}{\sigma_{r}^{\star}} \left
\Vert \bm{U}_{\mathcal{I}, \cdot}^{\star} \right \Vert + (c_7 + c_7 \cnoise)
\frac{\sigma^2 }{\sigma_{r}^{\star}} \sqrt{n \left(I + r + \log n
  \right ) } \nonumber \\
& \leq 2 c_7 \frac{\sigma^2 n}{\sigma_{r}^{\star}} \left
\Vert \bm{U}_{\mathcal{I}, \cdot}^{\star} \right \Vert + 2 c_7 
\frac{\sigma^2 }{\sigma_{r}^{\star}} \sqrt{n \left(I + r + \log n
	\right ) } \label{eq:crude-interm-1}
\end{align}
provided that $\cnoise>0$ is sufficiently small.  Next, we control $\alpha_{2}$,
towards which we have
\begin{align}
\alpha_{2} \leq \left \Vert \bm{U}_{\mathcal{I}, \cdot} \right \Vert
\left \Vert \bm{\Sigma} \left( \bm{H}_{ \bm{V}}- \bm{R}_{ \bm{V}}
\right ) \right \Vert \leq c_5 \frac{\sigma^2 n}{\sigma_{r}^{\star}}
\big( 3 \left \Vert \bm{U}_{\mathcal{I}, \cdot}^{\star} \right \Vert + 2 c_8 
\frac{\sigma}{\sigma_{r}^{\star}} \sqrt{I + r + \log n}
\big).\label{eq:crude-interm-2}
\end{align}
Here the last relation follows from \Cref{lemma:denoising-basic-facts-crude} and the following bound
on $ \Vert \bm{U}_{\mathcal{I}, \cdot} \Vert$:
\begin{align}
 \left \Vert \bm{U}_{\mathcal{I}, \cdot} \right \Vert &
 \overset{ \text{(i)} }{ \leq }2 \left \Vert \bm{U}_{\mathcal{I}, \cdot}
 \bm{H}_{ \bm{U}} \right \Vert \leq 2 \left \Vert \bm{U}_{\mathcal{I},
   \cdot} \bm{H}_{ \bm{U}}- \bm{U}_{\mathcal{I}, \cdot}^{\star} \right
 \Vert + 2 \left \Vert \bm{U}_{\mathcal{I}, \cdot}^{\star} \right
 \Vert \nonumber \\ 
 & \overset{ \text{(ii)} }{ \leq } 2 c_8 \frac{\sigma^2
   n}{\sigma_{r}^{\star2}} \left \Vert \bm{U}_{\mathcal{I},
   \cdot}^{\star} \right \Vert + 2 c_8 \frac{\sigma}{\sigma_{r}^{\star}}
 \sqrt{I + r + \log n} + 2 \left \Vert \bm{U}_{\mathcal{I},
   \cdot}^{\star} \right \Vert \nonumber \\
  & \overset{ \text{(iii)} }{ \leq } 3
 \left \Vert \bm{U}_{\mathcal{I}, \cdot}^{\star} \right \Vert + 2 c_8
 \frac{\sigma}{\sigma_{r}^{\star}} \sqrt{I + r + \log
   n}, \label{eq:U-I-spectral-norm}
\end{align}
where (i) follows from \Cref{lemma:denoising-basic-facts-crude},
(ii) makes use of \Cref{lemma:denoising-est-crude}, and (iii)
holds true provided that $\cnoise>0$ is sufficiently small. Putting (\ref{eq:crude-interm-1}) and
(\ref{eq:crude-interm-2}) together yields
\begin{align}
 \left \Vert \bm{U}_{\mathcal{I}, \cdot} \bm{\Sigma} \bm{R}_{ \bm{V}}-
 \bm{M}_{\mathcal{I}, \cdot} \bm{V}^{\star} \right \Vert & \leq
 \alpha_{1} + \alpha_{2} \leq (2 c_7 + 3 c_5) \frac{\sigma^2
   n}{\sigma_{r}^{\star}} \left \Vert \bm{U}_{\mathcal{I},
   \cdot}^{\star} \right \Vert + 3 c_7 \frac{\sigma^2 }{\sigma_{r}^{\star}}
 \sqrt{n \left(I + r + \log n \right )} \label{eq:crude-interm-3}
\end{align}
provided that $\cnoise>0$ is sufficiently small. Follow the same
analysis, we can also prove a similar bound on $\Vert
\bm{V}_{\mathcal{J},\cdot} \bm{\Sigma} \bm{R}_{ \bm{U}}-
\bm{M}_{\cdot, \mathcal{J}}^{\top} \bm{U}^{\star} \Vert$.  These allow
us to write
\begin{align}
 \bm{U} \bm{\Sigma} \bm{R}_{ \bm{V}} = \bm{M} \bm{V}^{\star} +
 \bm{\Delta}_{ \bm{U}}\qquad\text{and}\qquad \bm{V} \bm{\Sigma}
 \bm{R}_{ \bm{U}} = \bm{M}^{\top} \bm{U}^{\star} + \bm{\Delta}_{
   \bm{V}}, \label{eq:UVSigma-expansion}
\end{align}
where for any $\mathcal{I}\subseteq[n_{1}]$ and
$\mathcal{J}\subseteq[n_{2}]$, with probability exceeding
$1-O(n^{-10})$ we have
\begin{subequations}
\label{eq:Delta-bound-crude}
\begin{align}
 \Vert( \bm{\Delta}_{ \bm{U}})_{\mathcal{I}, \cdot} \Vert & \leq (2 c_7 + 3 c_5)
 \frac{\sigma^2 n}{\sigma_{r}^{\star}} \Vert \bm{U}_{\mathcal{I},
   \cdot}^{\star} \Vert + 3 c_7 \frac{\sigma^2 }{\sigma_{r}^{\star}} \sqrt{n
   \left(I + r + \log n \right
   )}, \label{eq:Delta-U-bound-crude} \\
\Vert( \bm{\Delta}_{ \bm{V}})_{\mathcal{J}, \cdot} \Vert & \leq (2 c_7 + 3 c_5)
\frac{\sigma^2 n}{\sigma_{r}^{\star}} \Vert \bm{V}_{\mathcal{J},
  \cdot}^{\star} \Vert + 3 c_7 \frac{\sigma^2 }{\sigma_{r}^{\star}} \sqrt{n
  \left(J + r + \log n \right )}.\label{eq:Delta-V-bound-crude}
\end{align}
\end{subequations}
Besides, in light of \Cref{lemma:denoising-approx-2-crude} we can
write
\begin{align}
\label{eq:Delta-Sigma-bound}  
 \bm{\Sigma}^{\star} = \bm{R}_{ \bm{U}}^{\top} \bm{\Sigma} \bm{R}_{
   \bm{V}} + \bm{\Delta}_{ \bm{\Sigma}} \qquad \text{where} \qquad \left
 \Vert \bm{\Delta}_{ \bm{\Sigma}} \right \Vert \leq c_6 \frac{\sigma^2
   n}{\sigma_{r}^{\star}} + c_6 \sigma \sqrt{r + \log n}.
\end{align}

The expansions (\ref{eq:UVSigma-expansion}) and
(\ref{eq:Delta-Sigma-bound}) yield
\begin{align*}
 \bm{U} \bm{R}_{ \bm{U}} \bm{\Sigma}^{\star} & = \bm{U} \bm{\Sigma}
 \bm{R}_{ \bm{V}} + \bm{U} \left( \bm{R}_{ \bm{U}}
 \bm{\Sigma}^{\star}- \bm{\Sigma} \bm{R}_{ \bm{V}} \right ) = \bm{M}
 \bm{V}^{\star} + \bm{\Delta}_{ \bm{U}} + \bm{U} \bm{R}_{ \bm{U}}
 \big( \bm{\Sigma}^{\star}- \bm{R}_{ \bm{U}}^{\top} \bm{\Sigma}
 \bm{R}_{ \bm{V}} \big) \\
 & = \left( \bm{M}^{\star} + \bm{E} \right ) \bm{V}^{\star} +
 \bm{\Delta}_{ \bm{U}} + \bm{U} \bm{R}_{ \bm{U}} \bm{\Delta}_{
   \bm{\Sigma}} = \bm{U}^{\star} \bm{\Sigma}^{\star} + \bm{E}
 \bm{V}^{\star} + \bm{\Delta}_{ \bm{U}} + \bm{U} \bm{R}_{ \bm{U}}
 \bm{\Delta}_{ \bm{\Sigma}},
\end{align*}
which further leads to
\begin{align}
 \bm{U} \bm{R}_{ \bm{U}} & = \bm{U}^{\star} + \bm{E} \bm{V}^{\star}
 \left( \bm{\Sigma}^{\star} \right )^{-1} + \underbrace{ \bm{\Delta}_{
     \bm{U}} \left( \bm{\Sigma}^{\star} \right )^{-1} + \bm{U}
   \bm{R}_{ \bm{U}} \bm{\Delta}_{ \bm{\Sigma}} \left(
   \bm{\Sigma}^{\star} \right )^{-1}}_{\eqqcolon\, \bm{\Psi}_{
     \bm{U}}}.\label{eq:URU-decomposition-Psi-U}
\end{align}
For any $\mathcal{I}\subseteq[n_{1}]$,  we can bound 
\begin{align}
  \big \Vert( \bm{\Psi}_{ \bm{U}})_{\mathcal{I}, \cdot} \big \Vert
  & \leq \frac{1}{\sigma_{r}^{\star}} \Vert( \bm{\Delta}_{
    \bm{U}})_{\mathcal{I}, \cdot} \Vert + \frac{1}{\sigma_{r}^{\star}}
  \big \Vert \bm{U}_{\mathcal{I}, \cdot} \bm{R}_{ \bm{U}} \big \Vert
  \left \Vert \bm{\Delta}_{ \bm{\Sigma}} \right \Vert \nonumber \\ 
  &  \overset{ \text{(i)} }{\leq} (2 c_7 + 3 c_5) \frac{\sigma^2 n}{\sigma_{r}^{\star2}} \Vert
  \bm{U}_{\mathcal{I}, \cdot}^{\star} \Vert + 3 c_7 \frac{\sigma^2 \sqrt{n
      \left(I + r + \log n \right )}}{\sigma_{r}^{\star2}} \nonumber \\
  & \qquad + \big( 3
  \Vert \bm{U}_{\mathcal{I}, \cdot}^{\star} \Vert + 2 c_8 \frac{\sigma
    \sqrt{I + r + \log n}}{\sigma_{r}^{\star}} \big)
  \big( c_6 \frac{\sigma^2 n}{\sigma_{r}^{\star2}} + c_6 \frac{\sigma \sqrt{r +
      \log n}}{\sigma_{r}^{\star}} \big)\nonumber \\ 
  & \overset{ \text{(ii)} }{\leq} c_9
  \frac{\sigma^2 }{\sigma_{r}^{\star2}} \sqrt{n \left(I + r + \log n
    \right )} + c_9 \bigg (\frac{\sigma^2 n}{\sigma_{r}^{\star2}} +
  \frac{\sigma \sqrt{r + \log n}}{\sigma_{r}^{\star}} \bigg ) \Vert
  \bm{U}_{\mathcal{I}, \cdot}^{\star} \Vert.\label{eq:Psi-U-crude}
\end{align}
Here (i) follows from (\ref{eq:Delta-U-bound-crude}),
(\ref{eq:U-I-spectral-norm}) and (\ref{eq:Delta-Sigma-bound}), while
(ii) holds true if we take $c_9 = 3 c_7 + 3 c_5 + 3 c_6$ and $\cnoise>0$ is sufficiently small. The other bound
(\ref{eq:denoising-Psi-V-crude}) can be verified in the similar way.


\subsubsection{Step 3: Refining the bounds \label{subsec:refine}}

In this section we will improve the bounds
(\ref{eq:denoising-Psi-bounds-crude}) to
(\ref{eq:denoising-Psi-bounds}). Note that the suboptimality of
(\ref{eq:denoising-Psi-bounds-crude}) is partly because we start our
analysis with an application of Wedin's $\sin\Theta$ Theorem, which
gives the following result in \Cref{lemma:denoising-basic-facts-crude}:
\begin{align}
\max \left\{  \left \Vert  \bm{U} \bm{R}_{ \bm{U}}- \bm{U}^{\star} \right \Vert ,  \left \Vert  \bm{V} \bm{R}_{ \bm{V}}- \bm{V}^{\star} \right \Vert  \right\}  &  \leq c_5 \frac{\sigma}{\sigma_{r}^{\star}} \sqrt{n}.\label{eq:Wedin-original}
\end{align}
This bound does not capture any potential imbalances of the data
matrix, namely when $n_{1}$ and $n_{2}$ are not on the same order.
However, in light of
\Cref{lemma:denoising-basic-facts-crude,lemma:denoising-est-crude}, we
can sharpen the above estimate to the following: there exists some
universal constant $c_5'>0$ such that for any given $\mathcal{I}
\subseteq [n_1]$ and $\mathcal{J} \subseteq [n_2]$, with probability
exceeding $1-O(n^{-10})$
\begin{subequations}
\label{eq:denoising-UV-I-spectral}
\begin{align}
\label{eq:denoising-U-I-spectral}   
 \left \Vert \bm{U}_{\mathcal{I}, \cdot} \bm{R}_{ \bm{U}}-
 \bm{U}_{\mathcal{I}, \cdot}^{\star} \right \Vert & \leq c_5'
 \frac{\sigma^2 n}{\sigma_{r}^{\star2}} \left \Vert
 \bm{U}_{\mathcal{I}, \cdot}^{\star} \right \Vert + c_5'
 \frac{\sigma}{\sigma_{r}^{\star}} \sqrt{I + r + \log n}, \\
\label{eq:denoising-V-I-spectral} 
 \left \Vert \bm{V}_{\mathcal{J}, \cdot} \bm{R}_{ \bm{V}}-
 \bm{V}_{\mathcal{J}, \cdot}^{\star} \right \Vert & \leq c_5'
 \frac{\sigma^2 n}{\sigma_{r}^{\star2}} \left \Vert
 \bm{V}_{\mathcal{J}, \cdot}^{\star} \right \Vert + c_5'
 \frac{\sigma}{\sigma_{r}^{\star}} \sqrt{J + r + \log n}.
\end{align}
\end{subequations}
See~\Cref{subsec:proof-sharpen-eqs} for the proof of the
claim~\eqref{eq:denoising-UV-I-spectral}. By taking $\mathcal{I} = [n_{1}]$ and $\mathcal{J}
= [n_{2}]$, we
have
\begin{subequations}
\label{eq:denoising-UV-spectral}
\begin{align}
 \left \Vert \bm{U} \bm{R}_{ \bm{U}}- \bm{U}^{\star} \right \Vert &
 \leq c_5' \frac{\sigma^2 n}{\sigma_{r}^{\star2}} + 2 c_5'
 \frac{\sigma}{\sigma_{r}^{\star}} \sqrt{n_{1} + \log
   n}, \label{eq:denoising-U-spectral} \\
\label{eq:denoising-V-spectral} 
 \left \Vert \bm{V} \bm{R}_{ \bm{V}}- \bm{V}^{\star} \right \Vert &
 \leq c_5' \frac{\sigma^2 n}{\sigma_{r}^{\star2}} + 2c_5'
 \frac{\sigma}{\sigma_{r}^{\star}} \sqrt{n_{2} + \log n},
\end{align}
\end{subequations}
which improves the original
bound~\eqref{eq:Wedin-original}. We can also obtain the following
sharper bounds:
\begin{subequations}
\label{eq:Sigma-H-R-refined}
\begin{align}
\label{eq:Sigma-RU-HU-refined}   
 \left \Vert \bm{\Sigma} \left( \bm{R}_{ \bm{U}}- \bm{H}_{ \bm{U}}
 \right ) \right \Vert & \leq c_5' \big(\sigma \sqrt{n_{1} + \log n} +
 \frac{\sigma^2 n_{2}}{\sigma_{r}^{\star}}
 \big)\frac{\sigma}{\sigma_{r}^{\star}} \sqrt{n_{1} + \log n}, \quad
 \mbox{and} \\
\label{eq:Sigma-RV-HV-refined} 
 \left \Vert \bm{\Sigma} \left( \bm{R}_{ \bm{V}}- \bm{H}_{ \bm{V}}
 \right ) \right \Vert & \leq c_5' \big(\sigma \sqrt{n_{2} + \log n} +
 \frac{\sigma^2 n_{1}}{\sigma_{r}^{\star}}
 \big)\frac{\sigma}{\sigma_{r}^{\star}} \sqrt{n_{2} + \log n}.
\end{align}
\end{subequations}
See \Cref{subsec:proof-sharpen-eqs} for the proof of the claim~\eqref{eq:Sigma-H-R-refined}.

Based on these refined perturbation bounds, we can sharpen the bounds in  \Cref{lemma:denoising-approx-1-crude} to 
\begin{subequations}
\label{eq:denoising-UVSigma}
\begin{align}
 \big \Vert \bm{U}_{\mathcal{I}, \cdot} \bm{\Sigma} \bm{H}_{ \bm{V}}-
 \bm{M}_{\mathcal{I}, \cdot} \bm{V}^{\star} \big \Vert & \leq c_7'
 \bigg (\frac{\sigma^2 n_{1}}{\sigma_{r}^{\star2}} +
 \frac{\sigma}{\sigma_{r}^{\star}} \sqrt{n_{2} + \log n} \bigg
 )^{2}\sigma_{r}^{\star} \left \Vert \bm{U}_{\mathcal{I},
   \cdot}^{\star} \right \Vert \nonumber \\ 
   & \qquad + c_7' \sigma \sqrt{I
   + r + \log n} \bigg (\frac{\sigma^2 n_{1}}{\sigma_{r}^{\star2}} +
 \frac{\sigma}{\sigma_{r}^{\star}} \sqrt{I + n_{2} + \log n} \bigg
 ), \label{eq:denoising-U-Sigma} \\
 \big \Vert \bm{V}_{\mathcal{J}, \cdot} \bm{\Sigma} \bm{H}_{ \bm{U}}-
 \bm{M}_{\cdot, \mathcal{J}}^{\top} \bm{U}^{\star} \big \Vert &
 \leq c_7' \bigg (\frac{\sigma^2 n_{2}}{\sigma_{r}^{\star2}} +
 \frac{\sigma}{\sigma_{r}^{\star}} \sqrt{n_{1} + \log n} \bigg
 )^{2}\sigma_{r}^{\star} \left \Vert \bm{V}_{\mathcal{J},
   \cdot}^{\star} \right \Vert _{2}\nonumber \\
 & \qquad + c_7' \sigma \sqrt{J + r + \log n} \bigg (\frac{\sigma^2
   n_{2}}{\sigma_{r}^{\star2}} + \frac{\sigma}{\sigma_{r}^{\star}}
 \sqrt{J + n_{1} + \log n} \bigg ),\label{eq:denoising-V-Sigma}
\end{align}
\end{subequations}
where $c_7'>0$ is some universal constant. See~\Cref{subsec:proof-sharpen-eqs}.  for the proof of the
claim~\eqref{eq:denoising-UVSigma}.

With this set of refined bounds in place, we are ready to finish the
proof of~\Cref{prop:denoising_main_1}
and~\Cref{prop:denoising_main_2}.  Recall the definitions of
$\bm{\Delta}_{\bm{U}}$ and $\bm{\Psi}_{ \bm{U}}$ from
equations~\eqref{eq:UVSigma-expansion}
and~\eqref{eq:URU-decomposition-Psi-U} respectively.  We have
\begin{align}
 \left \Vert ( \bm{\Delta}_{ \bm{U}})_{\mathcal{I}.\cdot} \right \Vert
 & = \big \Vert \bm{U}_{\mathcal{I}, \cdot} \bm{\Sigma} \bm{R}_{
   \bm{V}}- \bm{M}_{\mathcal{I}, \cdot} \bm{V}^{\star} \big \Vert \leq
 \big \Vert \bm{U}_{\mathcal{I}, \cdot} \bm{\Sigma} \bm{H}_{ \bm{V}}-
 \bm{M}_{\mathcal{I}, \cdot} \bm{V}^{\star} \big \Vert + \big \Vert
 \bm{U}_{\mathcal{I}, \cdot} \bm{\Sigma}( \bm{H}_{ \bm{V}}- \bm{R}_{
   \bm{V}}) \big \Vert\nonumber \\
& \leq \big \Vert \bm{U}_{\mathcal{I}, \cdot} \bm{\Sigma} \bm{H}_{
   \bm{V}}- \bm{M}_{\mathcal{I}, \cdot} \bm{V}^{\star} \big \Vert +
 \Vert \bm{U}_{\mathcal{I}, \cdot} \Vert \Vert \bm{\Sigma}( \bm{H}_{
   \bm{V}}- \bm{R}_{ \bm{V}}) \Vert\nonumber \\
& \overset{ \text{(i)} }{\leq} c_7' \big(\frac{\sigma^2 n_{1}}{\sigma_{r}^{\star2}} +
 \frac{\sigma}{\sigma_{r}^{\star}} \sqrt{n_{2} + \log n}
 \big)^{2}\sigma_{r}^{\star} \Vert \bm{U}_{\mathcal{I}, \cdot}^{\star}
 \Vert + c_7' \sigma \sqrt{I + r + \log n} \big(\frac{\sigma^2
   n_{1}}{\sigma_{r}^{\star2}} + \frac{\sigma}{\sigma_{r}^{\star}}
 \sqrt{I + n_{2} + \log n} \big)\nonumber \\
 & \qquad + \big( 3 \Vert \bm{U}_{\mathcal{I}, \cdot}^{\star} \Vert + 2 c_8
 \frac{\sigma}{\sigma_{r}^{\star}} \sqrt{I + r + \log n} \big) \cdot c_5'
 \big(\sigma \sqrt{n_{2} + \log n} + \frac{\sigma^2
   n_{1}}{\sigma_{r}^{\star}} \big)\frac{\sigma}{\sigma_{r}^{\star}}
 \sqrt{n_{2} + \log n}\nonumber \\ 
 & \overset{ \text{(ii)} }{\leq} 2 c_7' \sigma \sqrt{I + r +
   \log n} \Big(\frac{\sigma^2 n_{1}}{\sigma_{r}^{\star2}} +
 \frac{\sigma}{\sigma_{r}^{\star}} \sqrt{I + n_{2} + \log n} \Big) + (4 c_7' + 7 c_5')
 \frac{\sigma^2 n}{\sigma_{r}^{\star}} \Vert \bm{U}_{\mathcal{I},
   \cdot}^{\star} \Vert. \label{eq:Delta-U-bound}
\end{align}
Here (i) utilizes
from (\ref{eq:denoising-U-Sigma}), (\ref{eq:U-I-spectral-norm}) and
(\ref{eq:Sigma-RV-HV-refined}), while (ii) holds provided that $\sigma\sqrt{n} \leq \cnoise \sigma_r^\star$ for some sufficiently small $\cnoise>0$. By taking $c_2 = 4c_7'+7c_5'$, we have proved
(\ref{eq:denoising-Delta-U}), and (\ref{eq:denoising-Delta-V}) can be
verified similarly. This finishes the proof
of~\Cref{prop:denoising_main_2}.  We can also improve
(\ref{eq:Psi-U-crude}) to
\begin{align*}
 \Vert( \bm{\Psi}_{ \bm{U}})_{\mathcal{I}, \cdot} \Vert & \leq
 \frac{1}{\sigma_{r}^{\star}} \Vert( \bm{\Delta}_{
   \bm{U}})_{\mathcal{I}, \cdot} \Vert + \frac{1}{\sigma_{r}^{\star}}
 \big( \Vert \bm{U}_{\mathcal{I}, \cdot}^{\star} \Vert +
 \frac{\sigma}{\sigma_{r}^{\star}} \sqrt{I + r + \log n} \big) \Vert
 \bm{\Delta}_{ \bm{\Sigma}} \Vert\\ 
 & \overset{ \text{(a)} }{\leq} c_2 \frac{\sigma^2 n}{\sigma_{r}^{\star}} \Vert \bm{U}_{\mathcal{I},
   \cdot}^{\star} \Vert + c_2 \frac{\sigma}{\sigma_{r}^{\star}} \sqrt{I +
   r + \log n} \Big(\frac{\sigma^2 n_{1}}{\sigma_{r}^{\star2}} +
 \frac{\sigma}{\sigma_{r}^{\star}} \sqrt{I + n_{2} + \log n} \Big)\\ &
 \quad + \frac{1}{\sigma_{r}^{\star}} \big( \left \Vert
 \bm{U}_{\mathcal{I}, \cdot}^{\star} \right \Vert +
 \frac{\sigma}{\sigma_{r}^{\star}} \sqrt{I + r + \log n} \big) \cdot c_6
 \Big(\frac{\sigma^2 n}{\sigma_{r}^{\star}} + \sigma \sqrt{r + \log n}
 \Big)\\ 
 & \overset{ \text{(b)} }{\leq} c_1\frac{\sigma}{\sigma_{r}^{\star}} \sqrt{I + r +
   \log n} \big(\frac{\sigma}{\sigma_{r}^{\star}} \sqrt{I + n_{2} +
   \log n} + \frac{\sigma^2 n_{1}}{\sigma_{r}^{\star2}} \big) + c_1
 \big(\frac{\sigma^2 n}{\sigma_{r}^{\star2}} + \frac{\sigma \sqrt{r +
     \log n}}{\sigma_{r}^{\star}} \big) \Vert \bm{U}_{\mathcal{I},
   \cdot}^{\star} \Vert
\end{align*}
where we take $c_1 = c_2 + 2 c_6$. Here step (a) follows from (\ref{eq:denoising-Delta-U}) and (\ref{eq:Delta-Sigma-bound}), while step (b) holds provided that $\sigma\sqrt{n} \leq \cnoise \sigma_r^\star$ for some sufficiently small $\cnoise>0$. This has established (\ref{eq:denoising-Psi-U}). Similarly we can
verify (\ref{eq:denoising-Psi-V}), which completes the proof
of~\Cref{prop:denoising_main_1}.

\subsection{Proof of auxiliary lemmas for matrix denoising}

This section provides omitted proofs in \Cref{sec:proof-denoising-main}.

\subsubsection{Proof of~\Cref{lemma:denoising-basic-facts-crude}}
\label{appendix:proof-denoising-basic-facts}

To begin with, we find it convenient to write the full SVD of $
\bm{M}$ as
\begin{align}
 \bm{M} =  \left[\begin{array}{cc}
 \bm{U} &  \bm{U}_{\perp}\end{array} \right] \left[\begin{array}{cc}
 \bm{\Sigma} &  \bm{0}\\
 \bm{0} &  \bm{\Sigma}_{\perp}
\end{array} \right] \left[\begin{array}{cc}
 \bm{V} &  \bm{V}_{\perp}\end{array} \right]^{\top} =  \bm{U} \bm{\Sigma} \bm{V}^{\top} +  \bm{U}_{\perp} \bm{\Sigma}_{\perp} \bm{V}_{\perp}^{\top}, \label{eq:full-SVD}
\end{align}
where $\bm{U}_{\perp} \in \mathbb{R}^{n_{1}\times(n_{1}-r)}$ and $\bm{V}_{\perp} \in \mathbb{R}^{n_{2}\times(n_{2}-r)}$
are the orthogonal complements of $\bm{U}$ and $\bm{V}$,  and $\bm{\Sigma}_{\perp} \in \mathbb{R}^{(n_{1}-r)\times(n_{2}-r)}$
contains the smaller singular values (which is not necessarily a diagonal
matrix). We also define where $\bm{U}_{\perp}^{\star} \in \mathbb{R}^{n_{1}\times(n_{1}-r)}$
and $\bm{V}_{\perp}^{\star} \in \mathbb{R}^{n_{2}\times(n_{2}-r)}$
to be the orthogonal complements of $\bm{U}^{\star}$ and $\bm{V}^{\star}$.

\Cref{lemma:gaussian-spectral} shows that with probability
exceeding $1-O(n^{-10})$,  
\begin{align}
 \left \Vert  \bm{E} \right \Vert  \leq C_g \sigma \sqrt{n}.\label{eq:denoising-basic-0}
\end{align}
Then Weyl's inequality immediately gives 
\begin{align}
\frac{1}{2}\sigma_{r}^{\star} \leq \sigma_{r}^{\star}- \left \Vert  \bm{E} \right \Vert  \leq \sigma_{r} \left( \bm{M} \right ) \leq \sigma_{1}^{\star} +  \left \Vert  \bm{E} \right \Vert  \leq 2\sigma_{1}^{\star}\label{eq:denoising-basic-1}
\end{align}
since $\Vert \bm{E} \Vert \leq  C_g \cnoise \sigma_r^\star \leq \sigma_r^\star / 2$ with the proviso that $\cnoise$ is sufficiently small. In view
of Wedin's $\sin \bm{\Theta}$ Theorem (see \cite{wedin1972perturbation} or \cite[Theorem 2.3.1]{chen2020spectral}), we have
\begin{align}
 \frac{ 1 }{ \sqrt{2} } \Vert \bm{U} \bm{R}_{ \bm{U}}- \bm{U}^{\star} \Vert \leq \Vert \bm{U} \bm{U}^{\top}- \bm{U}^{\star} \bm{U}^{\star\top}
 \Vert =  \Vert \bm{U}^{\star\top} \bm{U}_{\perp} \Vert  \leq \frac{ \left
   \Vert \bm{E} \right \Vert }{\sigma_{r}^\star - \Vert \bm{E} \Vert } \leq 2 C_g
 \frac{\sigma}{\sigma_{r}^{\star}}
 \sqrt{n}, \label{eq:denoising-basic-2}
\end{align}
where the last step follows from $\sigma_{r + 1}( \bm{M}^{\star}) = 0$
and (\ref{eq:denoising-basic-1}). The first three quantities in the
above formula are all metrics between subspace $\bm{U}$ and $
\bm{U}^{\star}$ and are equivalent up to constant multiple; see
\cite[Section 2.2.3]{chen2020spectral} for detailed
discussions. Similarly we can show that
\begin{align}
\frac{ 1 }{ \sqrt{2} } \Vert
 \bm{V} \bm{R}_{ \bm{V}}- \bm{V}^{\star} \Vert \leq  \Vert \bm{V} \bm{V}^{\top}- \bm{V}^{\star} \bm{V}^{\star\top}
 \Vert = \Vert \bm{V}^{\star\top} \bm{V}_{\perp} \Vert  \leq  2C_g
 \frac{\sigma}{\sigma_{r}^{\star}}
 \sqrt{n}.\label{eq:denoising-basic-3}
\end{align}

Now we move on to establishing the proximity between $\bm{H}_{
  \bm{U}}$ and $\bm{R}_{ \bm{U}}$. Since $\bm{U}$ and $
\bm{U}^{\star}$ both have orthonormal columns, the SVD of $\bm{H}_{
  \bm{U}} = \bm{U}^{\top} \bm{U}^{\star}$ can be written as
\begin{align*}
 \bm{H}_{ \bm{U}} =  \bm{X}(\cos \bm{\Theta}) \bm{Y}^{\top}, 
\end{align*}
where $\bm{X}, \bm{Y} \in \mathbb{R}^{r\times r}$ are orthogonal
matrices and $\bm{\Theta} = \mathsf{diag}\{\theta_{1}, \ldots,
\theta_{r}\}$ is a diagonal matrix where $0 \leq \theta_{1} \leq
\cdots \leq \theta_{r} \leq \pi/2$ are the principal (or canonical)
angles between the two subspaces $\bm{U}$ and $\bm{U}^{\star}$ (see
\cite[Section 2.1]{chen2020spectral} for a systematic
introduction). Here $\cos \bm{\Theta} = \mathsf{diag}\{\cos\theta_{1},
\ldots, \cos\theta_{r}\}$ applies the cosine function to the diagonal
entries, and $\sin\Theta = \mathsf{diag}\{\sin\theta_{1}, \ldots,
\sin\theta_{r}\}$ is defined analogously. Then by definition, we have
$\bm{R}_{ \bm{U}} = \mathsf{sgn}( \bm{H}_{ \bm{U}}) = \bm{X}
\bm{Y}^{\top}$.  This immediately gives
\begin{align*}
 \bm{R}_{ \bm{U}}- \bm{H}_{ \bm{U}} &  =  \bm{X} \bm{Y}^{\top}- \bm{X}(\cos \bm{\Theta}) \bm{Y}^{\top} =  \bm{X} \left( \bm{I}_{r}-\cos \bm{\Theta} \right ) \bm{Y}^{\top} = 2 \bm{X}\sin^{2} \left( \bm{\Theta}/2 \right ) \bm{Y}^{\top}.
\end{align*}
On the other hand,  we have 
\begin{align*}
 \bm{U}^{\top} \bm{U}_{\perp}^{\star}( \bm{U}^{\top} \bm{U}_{\perp}^{\star}) =  \bm{U}^{\top}( \bm{I}_{n_{1}}- \bm{U}^{\star} \bm{U}^{\star\top}) \bm{U} =  \bm{I}_{r}- \bm{H}_{ \bm{U}} \bm{H}_{ \bm{U}}^{\top} =  \bm{I}_{r}- \bm{X} \left(\cos \bm{\Theta} \right )^{2} \bm{X}^{\top} =  \bm{X} \left(\sin \bm{\Theta} \right )^{2} \bm{X}^{\top}.
\end{align*}
This suggest that,  there exists a matrix $\bm{Z} \in \mathbb{R}^{(n_{1}-r)\times r}$
with orthonormal columns such that 
\begin{align*}
 \bm{U}^{\top} \bm{U}_{\perp}^{\star} =  \bm{X} \left(\sin \bm{\Theta} \right ) \bm{Z}^{\top}.
\end{align*}
Taking the above two equations collectively yields 
\begin{align}
 \bm{\Sigma} \left( \bm{R}_{ \bm{U}}- \bm{H}_{ \bm{U}} \right ) &  = 2 \bm{\Sigma} \bm{X}\sin^{2} \left( \bm{\Theta}/2 \right ) \bm{Y}^{\top} = 2 \bm{\Sigma} \bm{U}^{\top} \bm{U}_{\perp}^{\star} \bm{Z} \left(\sin \bm{\Theta} \right )^{-1}\sin \left( \bm{\Theta}/2 \right )\sin \left( \bm{\Theta}/2 \right ) \bm{Y}^{\top}.\label{eq:denoising-basic-4}
\end{align}
Here we define $0/0 = 0$ such that in the above equation,  
\begin{align*}
 \left(\sin \bm{\Theta} \right )^{-1}\sin \left( \bm{\Theta}/2 \right ) = \mathsf{diag} \left\{ \frac{\sin(\theta_{1}/2)}{\sin\theta_{1}}, \ldots, \frac{\sin(\theta_{r}/2)}{\sin\theta_{r}} \right\} 
\end{align*}
is always well defined. In view of the relation between the principal
angles and subspace distances (see e.g.~\cite[Lemma 2.5]{chen2020spectral})
and (\ref{eq:denoising-basic-2}),  we have 
\begin{align}
 \left \Vert \sin \bm{\Theta} \right \Vert  =  \Vert \bm{U} \bm{U}^{\top}- \bm{U}^{\star} \bm{U}^{\star\top} \Vert \leq 2C_g \frac{\sigma}{\sigma_{r}^{\star}} \sqrt{n}.\label{eq:denoising-basic-5}
\end{align}
Since $\sin(\theta/2) \leq \sin(\theta)$ for any $0 \leq \theta \leq \pi/2$, 
we know that 
\begin{align}
 \left \Vert \sin \left( \bm{\Theta}/2 \right ) \right \Vert  \leq  \left \Vert \sin \bm{\Theta} \right \Vert \qquad\text{and}\qquad \big  \Vert \left(\sin \bm{\Theta} \right )^{-1}\sin \left( \bm{\Theta}/2 \right ) \big  \Vert \leq 1.\label{eq:denoising-basic-6}
\end{align}
These relations allow us to obtain 
\begin{align}
 \left \Vert  \bm{\Sigma} \left( \bm{R}_{ \bm{U}}- \bm{H}_{ \bm{U}} \right ) \right \Vert  & \overset{ \text{(i)} }{ \leq }2 \big  \Vert \bm{\Sigma} \bm{U}^{\top} \bm{U}_{\perp}^{\star} \big  \Vert \left \Vert  \bm{Z} \right \Vert  \big  \Vert \left(\sin \bm{\Theta} \right )^{-1}\sin \left( \bm{\Theta}/2 \right ) \big  \Vert \left \Vert \sin \left( \bm{\Theta}/2 \right ) \right \Vert \nonumber \\
 & \overset{ \text{(ii)} }{ \leq }2 \big  \Vert \bm{\Sigma} \bm{U}^{\top} \bm{U}_{\perp}^{\star} \big  \Vert \left \Vert \sin \bm{\Theta} \right \Vert \overset{ \text{(iii)} }{ \leq } 4 C_g \frac{\sigma}{\sigma_{r}^{\star}} \sqrt{n} \big  \Vert \bm{\Sigma} \bm{U}^{\top} \bm{U}_{\perp}^{\star} \big  \Vert, \label{eq:denoising-basic-6.5}
\end{align}
where (i) follows from (\ref{eq:denoising-basic-4}),  (ii) utilizes
(\ref{eq:denoising-basic-6}),  and (iii) makes use of (\ref{eq:denoising-basic-5}).
Then we write 
\begin{align*}
 \bm{\Sigma} \bm{U}^{\top} \bm{U}_{\perp}^{\star} &  =  \bm{V}^{\top} ( \bm{U} \bm{\Sigma} \bm{V}^{\top} )^{\top} \bm{U}_{\perp}^{\star} =  \bm{V}^{\top} ( \bm{M}- \bm{U}_{\perp} \bm{\Sigma}_{\perp} \bm{V}_{\perp}^{\top} )^{\top} \bm{U}_{\perp}^{\star}\\
 &  =  \bm{V}^{\top} ( \bm{U}^{\star} \bm{\Sigma}^{\star} \bm{V}^{\star\top} +  \bm{E}- \bm{U}_{\perp} \bm{\Sigma}_{\perp} \bm{V}_{\perp}^{\top} )^{\top} \bm{U}_{\perp}^{\star} =  \bm{V}^{\top} \bm{E}^{\top} \bm{U}_{\perp}^{\star\top}.
\end{align*}
Here the last step holds since $\bm{U}^{\star\top} \bm{U}_{\perp}^{\star} =  \bm{0}$
and $\bm{V}^{\top} \bm{V}_{\perp} =  \bm{0}$. Therefore we have 
\begin{align}
\Vert  \bm{\Sigma} \bm{U}^{\top} \bm{U}_{\perp}^{\star} \Vert  &  \leq  \left \Vert  \bm{E} \bm{V} \right \Vert  \leq  \left \Vert  \bm{E} \bm{V}^{\star} \right \Vert  +  \left \Vert  \bm{E} \left( \bm{V} \bm{R}_{ \bm{V}}- \bm{V}^{\star} \right ) \right \Vert  \leq  \left \Vert  \bm{E} \bm{V}^{\star} \right \Vert  +  \left \Vert  \bm{E} \right \Vert  \left \Vert  \bm{V} \bm{R}_{ \bm{V}}- \bm{V}^{\star} \right \Vert \nonumber \\
 &  \overset{\text{(i)}}{\leq} C_g \sigma \sqrt{n_{1} + \log n} + 2 \sqrt{2} C_g \frac{\sigma}{\sigma_{r}^{\star}} \sqrt{n} \cdot C_g \sigma \sqrt{n} \leq C_g \sigma \sqrt{n_{1} + \log n} + 3 C_g^2 \frac{\sigma^2 n}{\sigma_{r}^{\star}}.\label{eq:denoising-basic-8}
\end{align}
Here (i) follows from \Cref{lemma:gaussian-spectral}, 
(\ref{eq:denoising-basic-0}) and (\ref{eq:denoising-basic-3}). Taking
(\ref{eq:denoising-basic-6.5}) and (\ref{eq:denoising-basic-8})
collectively yields 
\begin{align*}
 \left \Vert  \bm{\Sigma} \left( \bm{R}_{ \bm{U}}- \bm{H}_{ \bm{U}} \right ) \right \Vert  \leq 4 C_g \frac{\sigma}{\sigma_{r}^{\star}} \sqrt{n} \big( C_g \sigma \sqrt{n_{1} + \log n} + 3 C_g^2 \frac{\sigma^2 n}{\sigma_{r}^{\star}} \big) \leq 8 C_g^2 \frac{\sigma^2 n}{\sigma_{r}^{\star}}
\end{align*}
provided that $\cnoise$ is sufficiently small. This taken collectively with (\ref{eq:denoising-basic-1}) yields
\begin{align}
 \left \Vert  \bm{H}_{ \bm{U}}- \bm{R}_{ \bm{U}} \right \Vert  &  \leq \frac{ \left \Vert  \bm{\Sigma} \left( \bm{R}_{ \bm{U}}- \bm{H}_{ \bm{U}} \right ) \right \Vert }{\sigma_{r} \left( \bm{M} \right )} \leq 16 C_g^2 \frac{\sigma^2 n}{\sigma_{r}^{\star2}}.\label{eq:denoising-basic-9}
\end{align}
Given that $\bm{R}_{ \bm{U}}$ is an orthogonal matrix,  all its $r$
singular values are one. Thus by Weyl's inequality,  
\begin{align*}
\frac{1}{2} \leq \sigma_{r}( \bm{R}_{ \bm{U}})- \left \Vert  \bm{H}_{ \bm{U}}- \bm{R}_{ \bm{U}} \right \Vert  \leq \sigma_{r}( \bm{H}_{ \bm{U}}) \leq \sigma_{1}( \bm{H}_{ \bm{U}}) &  \leq \sigma_{1}( \bm{R}_{ \bm{U}}) +  \left \Vert  \bm{H}_{ \bm{U}}- \bm{R}_{ \bm{U}} \right \Vert  \leq 2, 
\end{align*}
provided that $\sigma \sqrt{n} \leq \cnoise \sigma_{r}^{\star}$ with a sufficiently small $\cnoise>0$.  Similarly,  we
can show that 
\begin{align*}
 \left \Vert \bm{\Sigma} \left( \bm{R}_{ \bm{V}}- \bm{H}_{ \bm{V}}
 \right ) \right \Vert \leq 8 C_g^2 \frac{\sigma^2 n}{\sigma_{r}^{\star}},
 \quad \left \Vert \bm{H}_{ \bm{V}}- \bm{R}_{ \bm{V}} \right \Vert
 \leq 16 C_g^2 \frac{\sigma^2 n}{\sigma_{r}^{\star2}},
 \quad\text{and}\quad\frac{1}{2} \leq \sigma_{r}( \bm{H}_{ \bm{U}})
 \leq \sigma_{1}( \bm{H}_{ \bm{U}}) \leq 2.
\end{align*}
Let $c_5 = 16 C_g^2$ to complete the proof. 

\subsubsection{Proof of~\Cref{lemma:denoising-approx-2-crude}}
\label{appendix:proof-lemma-denoising-approx-2}

Let us begin by upper bounding the error of interest using the
triangle inequality
\begin{align*}
 \big  \Vert \bm{R}_{ \bm{U}}^{\top} \bm{\Sigma} \bm{R}_{ \bm{V}}- \bm{\Sigma}^{\star} \big  \Vert \leq \underbrace{ \big  \Vert \bm{R}_{ \bm{U}}^{\top} \bm{\Sigma} \bm{R}_{ \bm{V}}- \bm{H}_{ \bm{U}}^{\top} \bm{\Sigma} \bm{H}_{ \bm{V}} \big  \Vert}_{\eqqcolon\, \alpha_{1}} + \underbrace{ \big  \Vert \bm{H}_{ \bm{U}}^{\top} \bm{\Sigma} \bm{H}_{ \bm{V}}- \bm{U}^{\star\top} \bm{M} \bm{V}^{\star} \big  \Vert}_{\eqqcolon\, \alpha_{2}} + \underbrace{ \big  \Vert \bm{U}^{\star\top} \bm{M} \bm{V}^{\star}- \bm{\Sigma}^{\star} \big  \Vert}_{\eqqcolon\, \alpha_{3}}.
\end{align*}
It then suffices to bound the three terms $\alpha_{1}$,  $\alpha_{2}$
and $\alpha_{3}$. 

We first invoke \Cref{lemma:denoising-basic-facts-crude} to
bound $\alpha_{1}$,  which gives 
\begin{align*}
\alpha_{1} &  \leq  \big  \Vert( \bm{H}_{ \bm{U}}- \bm{R}_{ \bm{U}}){}^{\top} \bm{\Sigma} \bm{H}_{ \bm{V}} \big  \Vert +  \big  \Vert \bm{R}_{ \bm{U}}^{\top} \bm{\Sigma}( \bm{H}_{ \bm{V}}- \bm{R}_{ \bm{V}}) \big  \Vert\\
 &  \leq 2 \big  \Vert \bm{\Sigma}( \bm{H}_{ \bm{U}}- \bm{R}_{ \bm{U}}) \big  \Vert +  \big  \Vert \bm{\Sigma}( \bm{H}_{ \bm{V}}- \bm{R}_{ \bm{V}}) \big  \Vert \leq 3 c_5 \frac{\sigma^2 n}{\sigma_{r}^{\star}}.
\end{align*}
Regarding $\alpha_{2}$,  recall the full SVD of $\bm{M}$ in (\ref{eq:full-SVD}), 
which allows us to write 
\begin{align*}
 \bm{H}_{ \bm{U}}^{\top} \bm{\Sigma} \bm{H}_{ \bm{V}}- \bm{U}^{\star\top} \bm{M} \bm{V}^{\star} =  \bm{U}^{\star\top} ( \bm{U} \bm{\Sigma} \bm{V}^{\top}- \bm{M} ) \bm{V}^{\star} =  \bm{U}^{\star\top} \bm{U}_{\perp} \bm{\Sigma}_{\perp} \bm{V}_{\perp}^{\top} \bm{V}^{\star}.
\end{align*}
In view of Weyl's inequality and (\ref{eq:denoising-basic-0}),  we
have 
\begin{align}
 \left \Vert  \bm{\Sigma}_{\perp} \right \Vert  \leq  \left \Vert  \bm{E} \right \Vert  \leq C_g \sigma \sqrt{n}.\label{eq:denoising-approx-2-1}
\end{align}
Taking together (\ref{eq:denoising-approx-2-1}),  (\ref{eq:denoising-basic-2})
and (\ref{eq:denoising-basic-3}) gives 
\begin{align*}
\alpha_{2} =  \big  \Vert \bm{U}^{\star\top} \bm{U}_{\perp} \bm{\Sigma}_{\perp} \bm{V}_{\perp}^{\top} \bm{V}^{\star} \big  \Vert \leq  \big  \Vert \bm{U}^{\star\top} \bm{U}_{\perp} \big  \Vert \Vert \bm{\Sigma}_{\perp} \Vert \big  \Vert \bm{V}^{\star\top} \bm{V}_{\perp} \big  \Vert \leq 4 C_g^3 \frac{ \left(\sigma \sqrt{n} \right )^{3}}{\sigma_{r}^{\star2}}.
\end{align*}
To control $\alpha_{3}$,  we first observe that 
\begin{align*}
 \bm{U}^{\star\top} \bm{M} \bm{V}^{\star}- \bm{\Sigma}^{\star} =  \bm{U}^{\star\top} \bm{M}^{\star} \bm{V}^{\star} +  \bm{U}^{\star\top} \bm{E} \bm{V}^{\star}- \bm{\Sigma}^{\star} =  \bm{\Sigma}^{\star} +  \bm{U}^{\star\top} \bm{E} \bm{V}^{\star}- \bm{\Sigma}^{\star} =  \bm{U}^{\star\top} \bm{E} \bm{V}^{\star}.
\end{align*}
It suffices to bound the spectral norm of $\bm{U}^{\star\top} \bm{E} \bm{V}^{\star}$.
Invoke \Cref{lemma:gaussian-spectral} shows that  with
probability exceeding $1-O(n^{-10})$,  
\begin{align}
\alpha_{3} =  \big  \Vert \bm{U}^{\star\top} \bm{E} \bm{V}^{\star} \big  \Vert \leq C_g \sigma \sqrt{r + \log n}.\label{eq:U-top-E-V-spectral}
\end{align}

Taking the bounds on $\alpha_{1}$,  $\alpha_{2}$ and $\alpha_{3}$
collectively yields
\begin{align*}
 \big  \Vert \bm{R}_{ \bm{U}}^{\top} \bm{\Sigma} \bm{R}_{ \bm{V}}- \bm{\Sigma}^{\star} \big \Vert \leq 3 c_5 \frac{\sigma^2 n}{\sigma_{r}^{\star}} + 4C_g^3 \frac{(\sigma \sqrt{n})^{3}}{\sigma_{r}^{\star2}} + C_g \sigma \sqrt{r + \log n} 
  \leq c_6 \frac{\sigma^2 n}{\sigma_{r}^{\star}} + c_6 \sigma \sqrt{r + \log n}
\end{align*}
as long as $c_6 \geq 3 c_5 + 4C_g^3 \cnoise + C_g$.
In addition,  we also have the following by-product: 
\begin{align*}
 \big  \Vert \bm{H}_{ \bm{U}}^{\top} \bm{\Sigma} \bm{H}_{ \bm{V}}- \bm{\Sigma}^{\star} \big  \Vert &  \leq  \big  \Vert \bm{H}_{ \bm{U}}^{\top} \bm{\Sigma} \bm{H}_{ \bm{V}}- \bm{U}^{\star\top} \bm{M} \bm{V}^{\star} \big  \Vert +  \big  \Vert \bm{U}^{\star\top} \bm{M} \bm{V}^{\star}- \bm{\Sigma}^{\star} \big  \Vert = \alpha_{2} + \alpha_{3} \\
 &   \leq 4C_g^3 \frac{(\sigma \sqrt{n})^{3}}{\sigma_{r}^{\star2}} + C_g \sigma \sqrt{r + \log n}  \leq  c_6 \frac{(\sigma \sqrt{n})^{3}}{\sigma_{r}^{\star2}} + c_6 \sigma \sqrt{r + \log n},
\end{align*}
which holds provided that $c_6 \geq 4 C_g^3 + C_g$. 

\subsubsection{Proof of~\Cref{lemma:denoising-approx-1-crude}}
\label{appendix:proof-denoising-approx-1}

Let $I = |\mathcal{I}|$ be the cardinality of $\mathcal{I}$. We start
by the following decomposition
\begin{align*}
 \big  \Vert \bm{U}_{\mathcal{I}, \cdot} \bm{\Sigma} \bm{H}_{ \bm{V}}- \bm{M}_{\mathcal{I}, \cdot} \bm{V}^{\star} \big  \Vert =  \big  \Vert \bm{M}_{\mathcal{I}, \cdot}( \bm{V} \bm{H}_{ \bm{V}}- \bm{V}^{\star}) \big  \Vert &  \leq \underbrace{ \big  \Vert \bm{M}_{\mathcal{I}, \cdot}^{\star}( \bm{V} \bm{H}_{ \bm{V}}- \bm{V}^{\star}) \big  \Vert}_{\eqqcolon\alpha} + \underbrace{ \big  \Vert \bm{E}_{\mathcal{I}, \cdot}( \bm{V} \bm{H}_{ \bm{V}}- \bm{V}^{\star}) \big  \Vert}_{\eqqcolon\beta}, 
\end{align*}
where the first relation follows from $\bm{U} \bm{\Sigma} =  \bm{M} \bm{V}$.
The rest of the proof is devoted to bounding $\alpha$ and $\beta$.

\paragraph{Step 1: bounding $\alpha$.}

To begin with,  we first notice that 
\begin{align*}
\alpha &  =  \big  \Vert \bm{U}_{\mathcal{I}, \cdot}^{\star} \bm{\Sigma}^{\star} \bm{V}^{\star\top}( \bm{V} \bm{H}_{ \bm{V}}- \bm{V}^{\star}) \big  \Vert \leq  \big  \Vert \bm{U}_{\mathcal{I}, \cdot}^{\star} \big  \Vert \big  \Vert \bm{\Sigma}^{\star} \bm{V}^{\star\top}( \bm{V} \bm{H}_{ \bm{V}}- \bm{V}^{\star}) \big  \Vert.
\end{align*}
Using the full SVD of $\bm{M}$ in (\ref{eq:full-SVD}),  we can write 
\begin{align*}
 \bm{\Sigma}^{\star} \bm{V}^{\star\top}( \bm{V} \bm{H}_{ \bm{V}}- \bm{V}^{\star}) &  =  \bm{\Sigma}^{\star} \bm{V}^{\star\top}( \bm{V} \bm{V}^{\top} \bm{V}^{\star}- \bm{V}^{\star}) =  \bm{U}^{\star\top} \bm{M}^{\star}( \bm{V} \bm{V}^{\top}- \bm{I}_{n_{2}}) \bm{V}^{\star}\\
 &  =  \bm{U}^{\star\top} \left( \bm{M}- \bm{E} \right )( \bm{V} \bm{V}^{\top}- \bm{I}_{n_{2}}) \bm{V}^{\star}\\
 &  =  \bm{U}^{\star\top}( \bm{U} \bm{\Sigma} \bm{V}^{\top} +  \bm{U}_{\perp} \bm{\Sigma}_{\perp} \bm{V}_{\perp}^{\top}- \bm{E})( \bm{V} \bm{V}^{\top}- \bm{I}_{n_{2}}) \bm{V}^{\star}\\
 &  =  \bm{U}^{\star\top} \bm{U}_{\perp} \bm{\Sigma}_{\perp} \bm{V}_{\perp}^{\top} \bm{V}^{\star}- \bm{U}^{\star\top} \bm{E}( \bm{V} \bm{V}^{\top}- \bm{I}_{n_{2}}) \bm{V}^{\star}.
\end{align*}
Invoke (\ref{eq:denoising-basic-2}), 
(\ref{eq:denoising-basic-3}) and (\ref{eq:denoising-approx-2-1}) to bound the spectral norm of the first matrix
\begin{align}
 \big  \Vert \bm{U}^{\star\top} \bm{U}_{\perp} \bm{\Sigma}_{\perp} \bm{V}_{\perp}^{\top} \bm{V}^{\star} \big  \Vert \leq  \big  \Vert \bm{U}^{\star\top} \bm{U}_{\perp} \big  \Vert \big  \Vert \bm{\Sigma}_{\perp} \big  \Vert \big  \Vert \bm{V}^{\star\top} \bm{V}_{\perp} \big  \Vert \leq 4 C_g^3 \frac{ \left(\sigma \sqrt{n} \right )^{3}}{\sigma_{r}^{\star2}}. \label{eq:denoising-approx-1-1}
\end{align}
The spectral norm of the second matrix is upper bounded by 
\begin{align}
 \big  \Vert \bm{U}^{\star\top} \bm{E}( \bm{V} \bm{V}^{\top}- \bm{I}_{n_{2}}) \bm{V}^{\star} \big  \Vert &  \leq  \big  \Vert \bm{U}^{\star\top} \bm{E} \big  \Vert \big  \Vert ( \bm{V} \bm{V}^{\top}- \bm{I}_{n_{2}} ) \bm{V}^\star \big  \Vert =  \big  \Vert \bm{U}^{\star\top} \bm{E} \big  \Vert \big  \Vert \bm{V} \bm{H}_{ \bm{V}}- \bm{V}^{\star} \big  \Vert\nonumber \\
 &  \leq C_g \sigma \sqrt{n_{2} + \log n} \cdot 2C_g \frac{\sigma}{\sigma_{r}^{\star}} \sqrt{n} =  2 C_g^2 \frac{\sigma^2 }{\sigma_{r}^{\star}} \sqrt{n \left(n_{2} + \log n \right )}, \label{eq:denoising-approx-1-2}
\end{align}
where the penultimate relation follows from \Cref{lemma:gaussian-spectral}
and a consequence of \eqref{eq:denoising-basic-3}:
\begin{align}
	\label{eq:denoising-approx-1-5}  
	\left \Vert \bm{V} \bm{H}_{ \bm{V}}- \bm{V}^{\star} \right \Vert & =
	\big \Vert( \bm{V} \bm{V}^{\top}- \bm{V}^{\star} \bm{V}^{\star\top})
	\bm{V}^{\star} \big \Vert \leq \big \Vert \bm{V} \bm{V}^{\top}-
	\bm{V}^{\star} \bm{V}^{\star\top} \big \Vert \leq 2C_g
	\frac{\sigma}{\sigma_{r}^{\star}} \sqrt{n}.
\end{align}
Therefore by the triangle inequality, 
\begin{align}
 \big  \Vert \bm{\Sigma}^{\star} \bm{V}^{\star\top}( \bm{V} \bm{H}_{ \bm{V}}- \bm{V}^{\star}) \big  \Vert &  \leq  \big  \Vert \bm{U}^{\star\top} \bm{U}_{\perp} \bm{\Sigma}_{\perp} \bm{V}_{\perp}^{\top} \bm{V}^{\star} \big  \Vert +  \big  \Vert \bm{U}^{\star\top} \bm{E}( \bm{V} \bm{V}^{\top}- \bm{I}_{n_{2}}) \bm{V}^{\star} \big  \Vert\nonumber \\
 &  \leq 4 C_g^3 \frac{ \left(\sigma \sqrt{n} \right )^{3}}{\sigma_{r}^{\star2}} + 2 C_g^2 \frac{\sigma^2 }{\sigma_{r}^{\star}} \sqrt{n \left(n_{2} + \log n \right )} \leq 3 C_g^2 \frac{\sigma^2 n}{\sigma_{r}^{\star}}\label{eq:denoising-approx-1-3}
\end{align}
provided that $\sigma \sqrt{n} \leq \cnoise \sigma_{r}^{\star}$ with a sufficiently small $\cnoise>0$. Hence
we obtain 
\begin{align}
\alpha \leq 3 C_g^2 \frac{\sigma^2 n}{\sigma_{r}^{\star}} \big  \Vert \bm{U}_{\mathcal{I}, \cdot}^{\star} \big  \Vert.\label{eq:denoising-approx-1-4}
\end{align}

\paragraph{Step 2: introducing leave-$I$-out estimates.}

The key technical difficulty for bounding $\beta$ is the statistical
dependency between $\bm{E}_{\mathcal{I}, \cdot}$ and $\bm{V} \bm{H}_{ \bm{V}}- \bm{V}^{\star}$.
Towards this,  we define an auxiliary data matrix 
\begin{align*}
 \bm{M}^{(\mathcal{I})}\coloneqq \bm{M}^{\star} + \mathcal{P}_{-\mathcal{I}, \cdot} \left( \bm{E} \right ), 
\end{align*}
where the operator $\mathcal{P}_{-\mathcal{I}, \cdot}:\mathbb{R}^{n_{1}\times n_{2}}\to\mathbb{R}^{n_{1}\times n_{2}}$
is defined as 
\begin{align*}
 \left[\mathcal{P}_{-\mathcal{I}, \cdot} \left( \bm{A} \right ) \right]_{i, j} = \begin{cases}
A_{i, j},  & \text{if }i \in \mathcal{I}, \\ 0,  & \text{otherwise}.
\end{cases}
\end{align*}
We also define the auxiliary estimates $\bm{U}^{(\mathcal{I})}$, 
$\bm{V}^{(\mathcal{I})}$,  $\bm{H}_{ \bm{U}}^{(\mathcal{I})}$ and
$\bm{H}_{ \bm{V}}^{(\mathcal{I})}$ w.r.t.~the data matrix $\bm{M}^{(\mathcal{I})}$, 
in the same way as how we define $\bm{U}$,  $\bm{V}$,  $\bm{H}_{ \bm{U}}$
and $\bm{H}_{ \bm{V}}$ w.r.t.~the data matrix $\bm{M}$. Introducing
these auxiliary estimates have the following benefits: 
\begin{itemize}
\item By construction $\bm{M}^{(\mathcal{I})}$ is independent of
  $\bm{E}_{\mathcal{I}, \cdot}$, therefore all these auxiliary
  estimates are independent of $\bm{E}_{\mathcal{I}, \cdot}$.  It is
  easier to bound $ \Vert \bm{E}_{\mathcal{I}, \cdot}(
  \bm{V}^{(\mathcal{I})} \bm{H}_{ \bm{V}}^{(\mathcal{I})}-
  \bm{V}^{\star}) \Vert$ than directly working on $\beta$.
\item Since $\bm{M}^{(\mathcal{I})}$ is constructed by leaving a small
amount of noise information from $\bm{M}$,  we can expect that these
auxiliary estimates are close to the original estimates,  at least
when $I$ is small. 
\end{itemize}
In what follows, we will exploit these statistical benefits to bound
$\beta$.

\paragraph{Step 3: bounding $\beta$.}

We start by applying the triangle inequality to obtain
\begin{align*}
\beta \leq \underbrace{ \big \Vert \bm{E}_{\mathcal{I}, \cdot} \big(
  \bm{V}^{(\mathcal{I})} \bm{H}_{ \bm{V}}^{(\mathcal{I})}-
  \bm{V}^{\star} \big) \big \Vert}_{\eqqcolon\, \beta_{1}} +
\underbrace{ \big \Vert \bm{E}_{\mathcal{I}, \cdot} \big(
  \bm{V}^{(\mathcal{I})} \bm{H}_{ \bm{V}}^{(\mathcal{I})}- \bm{V}
  \bm{H}_{ \bm{V}} \big) \big \Vert}_{\eqqcolon\, \beta_{2}}.
\end{align*}
Regarding $\beta_{1}$, since $\bm{V}^{(\mathcal{I})}
\bm{H}^{(\mathcal{I})}- \bm{V}^{\star}$ and $\bm{E}_{\mathcal{I},
  \cdot}$ are independent, we can condition on $\bm{V}^{(\mathcal{I})}
\bm{H}^{(\mathcal{I})}- \bm{V}^{\star}$ and
invoke~\Cref{lemma:gaussian-spectral}, thereby finding that
\begin{align}
\beta_{1} & \overset{ \text{(i)} }{\leq} C_g \sigma \sqrt{I + r + \log n} \big \Vert
\bm{V}^{(\mathcal{I})} \bm{H}_{ \bm{V}}^{(\mathcal{I})}-
\bm{V}^{\star} \big \Vert\nonumber \\
& \leq C_g \sigma \sqrt{I + r + \log n} \big \Vert \bm{V} \bm{H}_{
  \bm{V}}- \bm{V}^{\star} \big \Vert + C_g \sigma \sqrt{I + r + \log n}
\big \Vert \bm{V}^{(\mathcal{I})} \bm{H}_{ \bm{V}}^{(\mathcal{I})}-
\bm{V} \bm{H}_{ \bm{V}} \big \Vert \nonumber \\
& \overset{ \text{(ii)} }{\leq} 2 C_g^2 \frac{\sigma^2 }{\sigma_{r}^{\star}} \sqrt{n
	\left(I + r + \log n \right )} + C_g \sigma \sqrt{I + r + \log n} \big
\Vert \bm{V}^{(\mathcal{I})} \bm{H}_{ \bm{V}}^{(\mathcal{I})}- \bm{V}
\bm{H}_{ \bm{V}} \big \Vert. \label{eq:denoising-approx-1-6}  
\end{align}
Here (i) holds with probability exceeding $1-O(n^{-10})$ without
conditioning, while (ii) follows from \eqref{eq:denoising-approx-1-5}.
Regarding $\beta_{2}$,  we have
\begin{align}
 \label{eq:denoising-approx-1-7}  
\beta_{2} \leq \left \Vert \bm{E}_{\mathcal{I}, \cdot} \right \Vert
\big \Vert \bm{V}^{(\mathcal{I})} \bm{H}_{ \bm{V}}^{(\mathcal{I})}-
\bm{V} \bm{H}_{ \bm{V}} \big \Vert \leq C_g \sigma \sqrt{I + n_{2} +
  \log n} \big \Vert \bm{V}^{(\mathcal{I})} \bm{H}_{
  \bm{V}}^{(\mathcal{I})}- \bm{V} \bm{H}_{ \bm{V}} \big \Vert,
\end{align}
where the last relation follows from~\Cref{lemma:gaussian-spectral}.
It suffices to bound the spectral norm of $
\bm{V}^{(\mathcal{I})} \bm{H}_{ \bm{V}}^{(\mathcal{I})}- \bm{V}
\bm{H}_{ \bm{V}}$, which appears in both upper
bounds~\eqref{eq:denoising-approx-1-6}
and~\eqref{eq:denoising-approx-1-7} for $\beta_{1}$ and $\beta_{2}$.


We first observe that
\begin{subequations}
\label{eq:denoising-approx-1-8}
\begin{align}
 \big \Vert \bm{U}^{(\mathcal{I})} \bm{H}_{ \bm{U}}^{(\mathcal{I})}-
 \bm{U} \bm{H}_{ \bm{U}} \big \Vert & = \big \Vert \big(
 \bm{U}^{(\mathcal{I})} \bm{U}^{(\mathcal{I})\top}- \bm{U}
 \bm{U}^{\top} \big) \bm{U}^{\star} \big \Vert \leq \big \Vert
 \bm{U}^{(\mathcal{I})} \bm{U}^{(\mathcal{I})\top}- \bm{U}
 \bm{U}^{\top} \big \Vert, \quad \mbox{and} \\
 \big \Vert \bm{V}^{(\mathcal{I})} \bm{H}_{ \bm{V}}^{(\mathcal{I})}-
 \bm{V} \bm{H}_{ \bm{V}} \big \Vert & = \big \Vert \big(
 \bm{V}^{(\mathcal{I})} \bm{V}^{(\mathcal{I})\top}- \bm{V}
 \bm{V}^{\top} \big) \bm{V}^{\star} \big \Vert \leq \big \Vert
 \bm{V}^{(\mathcal{I})} \bm{V}^{(\mathcal{I})\top}- \bm{V}
 \bm{V}^{\top} \big \Vert.
\end{align}
\end{subequations}
The right hand side of the above two inequalities are the subspace
distance between $\bm{V}$ and $\bm{V}^{(\mathcal{I})}$, as well as $
\bm{U}$ and $\bm{U}^{(\mathcal{I})}$. This motivates us to apply
Wedin's $\sin \bm{\Theta}$ Theorem to achieve
\begin{align}
& \max \big \{ \big \Vert \bm{U}^{(\mathcal{I})}
  \bm{U}^{(\mathcal{I})\top}- \bm{U} \bm{U}^{\top} \big \Vert, \big
  \Vert \bm{V}^{(\mathcal{I})} \bm{V}^{(\mathcal{I})\top}- \bm{V}
  \bm{V}^{\top} \big \Vert \big \}\nonumber \\
& \qquad \leq \frac{ \sqrt{2}\max \big \{ \big \Vert \big(
    \bm{M}^{(\mathcal{I})}- \bm{M} \big) \bm{V}^{(\mathcal{I})} \big
    \Vert, \big \Vert \big( \bm{M}^{(\mathcal{I})}- \bm{M}
    \big)^{\top} \bm{U}^{(\mathcal{I})} \big \Vert \big \}}{\sigma_{r}
    \left( \bm{M}^{(\mathcal{I})} \right )-\sigma_{r + 1} \left(
    \bm{M} \right )- \big \Vert \bm{M}^{(\mathcal{I})}- \bm{M} \big
    \Vert} \nonumber \\
& \qquad \overset{\text{(i)}}{\leq} \frac{ \sqrt{2}\max \big \{ \big \Vert \big(
    \bm{M}^{(\mathcal{I})}- \bm{M} \big) \bm{V}^{(\mathcal{I})} \big
    \Vert, \big \Vert \big( \bm{M}^{(\mathcal{I})}- \bm{M}
    \big)^{\top} \bm{U}^{(\mathcal{I})} \big \Vert \big \}}{\sigma_{r}
    \left( \bm{M}^{\star} \right )- \big
    \Vert\mathcal{P}_{-\mathcal{I}, \cdot} \left( \bm{E} \right ) \big
    \Vert-\sigma_{r + 1} \left( \bm{M}^{\star} \right )- \left \Vert
    \bm{E} \right \Vert - \Vert \bm{E}_{\mathcal{I}, \cdot} \Vert}
  \nonumber \\
\label{eq:denoising-approx-1-9}  
& \qquad \overset{\text{(ii)}}{\leq} \frac{2 \sqrt{2}\max \big \{ \big \Vert \big(
  \bm{M}^{(\mathcal{I})}- \bm{M} \big) \bm{V}^{(\mathcal{I})} \big
  \Vert, \big \Vert \big( \bm{M}^{(\mathcal{I})}- \bm{M} \big)^{\top}
  \bm{U}^{(\mathcal{I})} \big \Vert \big \}}{\sigma_{r}^{\star}}.
\end{align}
Here (i) follows from Weyl's inequality, while
(ii) follows from the following consequence of \eqref{eq:denoising-basic-0}:
\begin{align*}
\max \big \{ \big \Vert\mathcal{P}_{-\mathcal{I}, \cdot} \left(
\bm{E} \right ) \big \Vert, \Vert \bm{E}_{\mathcal{I}, \cdot} \Vert
\big \} \leq \left \Vert \bm{E} \right \Vert \leq C_g \sigma
\sqrt{n} \leq \cnoise C_g \sigma_r^\star \leq \frac{1}{2}\sigma_{r}^{\star},
\end{align*}
where the last relation holds provided that $\cnoise$ is sufficiently small.  Then we shall bound the
spectral norm of $( \bm{M}^{(\mathcal{I})}- \bm{M})
\bm{V}^{(\mathcal{I})}$ and $( \bm{M}^{(\mathcal{I})}- \bm{M})^{\top}
\bm{U}^{(\mathcal{I})}$.  The first one is upper bounded by
\begin{align}
 \big \Vert \big( \bm{M}^{(\mathcal{I})}- \bm{M} \big)
 \bm{V}^{(\mathcal{I})} \big \Vert & = \big \Vert \bm{E}_{\mathcal{I},
   \cdot} \bm{V}^{(\mathcal{I})} \big \Vert\overset{\text{(i)}}{ \leq
 }2 \big \Vert \bm{E}_{\mathcal{I}, \cdot} \bm{V}^{(\mathcal{I})}
 \bm{H}_{ \bm{V}}^{(\mathcal{I})} \big \Vert \leq 2 \big \Vert
 \bm{E}_{\mathcal{I}, \cdot} \bm{V}^{\star} \big \Vert + 2 \big \Vert
 \bm{E}_{\mathcal{I}, \cdot} \big( \bm{V}^{(\mathcal{I})} \bm{H}_{
   \bm{V}}^{(\mathcal{I})}- \bm{V}^{\star} \big) \big \Vert\nonumber
 \\ & \overset{\text{(ii)}}{ \leq } C_g \sigma \sqrt{I + r + \log n} + 2
 \beta_{1}, \label{eq:denoising-approx-1-10}
\end{align}
and the second one is upper bounded by 
\begin{align}
  & \big \Vert( \bm{M}^{(\mathcal{I})}- \bm{M})^{\top}
 \bm{U}^{(\mathcal{I})} \big \Vert  = \big \Vert \bm{E}_{\mathcal{I},
   \cdot}^{\top} \bm{U}_{\mathcal{I}, \cdot}^{(\mathcal{I})} \big
 \Vert\overset{\text{(iii)}}{ \leq }2 \big \Vert \bm{E}_{\mathcal{I},
   \cdot}^{\top} \bm{U}_{\mathcal{I}, \cdot}^{(\mathcal{I})} \bm{H}_{
   \bm{U}}^{(\mathcal{I})} \big \Vert\leq 2 \big \Vert
   \bm{E}_{\mathcal{I}, \cdot}^{\top} \bm{U}_{\mathcal{I}}^{\star} \big
   \Vert + 2 \Vert \bm{E}_{\mathcal{I}, \cdot} \Vert \big \Vert
   \bm{U}_{\mathcal{I}, \cdot}^{(\mathcal{I})} \bm{H}_{
   \bm{U}}^{(\mathcal{I})}- \bm{U}_{\mathcal{I}, \cdot}^{\star} \big
   \Vert\nonumber \\ 
   & \qquad \qquad  \overset{\text{(iv)}}{ \leq } 2 C_g \sigma
 \sqrt{n_{2} + \log n} \big \Vert \bm{U}_{\mathcal{I} , \cdot}^{\star} \big
 \Vert + C_g \sigma \sqrt{I + n_{2} + \log n} \big \Vert
 \bm{U}_{\mathcal{I}, \cdot}^{(\mathcal{I})} \bm{H}_{
   \bm{U}}^{(\mathcal{I})}- \bm{U}_{\mathcal{I}, \cdot}^{\star} \big
 \Vert\label{eq:denoising-approx-1-11}
\end{align}
Here (i) and (iii) hold true when $\sigma_{r}( \bm{H}_{
  \bm{V}}^{(\mathcal{I})})\geq1/2$ and $\sigma_{r}( \bm{H}_{
  \bm{U}}^{(\mathcal{I})})\geq1/2$, which can be verified following
the same analysis for proving~\Cref{lemma:denoising-basic-facts-crude}
(where we showed that $\sigma_{r}( \bm{H}_{ \bm{V}})\geq1/2$ and
$\sigma_{r}( \bm{H}_{ \bm{U}})\geq1/2$), and is omitted here for
brevity; while (ii) and (iv) follow
from~\Cref{lemma:gaussian-spectral}.  Collecting together
equations~(\ref{eq:denoising-approx-1-8}),
(\ref{eq:denoising-approx-1-9}) , (\ref{eq:denoising-approx-1-10}) and
(\ref{eq:denoising-approx-1-11}), we arrive at
\begin{align}
\label{eq:denoising-approx-1-12}  
 & \max \left\{ \big \Vert \bm{U}^{(\mathcal{I})} \bm{H}_{
  \bm{U}}^{(\mathcal{I})}- \bm{U} \bm{H}_{ \bm{U}} \big \Vert, \big
\Vert \bm{V}^{(\mathcal{I})} \bm{H}_{ \bm{V}}^{(\mathcal{I})}- \bm{V}
\bm{H}_{ \bm{V}} \big \Vert \right\} \nonumber \\ & \quad \leq 
\frac{ 20C_g }{\sigma_{r}^{\star}} \big[  \sigma \sqrt{I + r + \log n} + \sigma \sqrt{n_{2}} \big( \big \Vert
  \bm{U}_{\mathcal{I}, \cdot}^{\star} \big \Vert + \big \Vert
  \bm{U}_{\mathcal{I}, \cdot}^{(\mathcal{I})} \bm{H}_{
    \bm{U}}^{(\mathcal{I})}- \bm{U}_{\mathcal{I}, \cdot}^{\star} \big
  \Vert \big) \big] + \frac{6 \beta_1}{ \sigma_r^\star } .
\end{align}
Here we use the facts that $ \Vert \bm{U}_{\mathcal{I}, \cdot}^{\star}
\Vert \leq 1$ and $ \Vert \bm{U}_{\mathcal{I}, \cdot}^{(\mathcal{I})}
\bm{H}_{ \bm{U}}^{(\mathcal{I})}- \bm{U}_{\mathcal{I}, \cdot}^{\star}
\Vert \leq 3$.  We want to replace the spectral norm of $
\bm{U}_{\mathcal{I}, \cdot}^{(\mathcal{I})} \bm{H}_{
  \bm{U}}^{(\mathcal{I})}- \bm{U}_{\mathcal{I}, \cdot}^{\star}$ in
(\ref{eq:denoising-approx-1-12}) with that of $\bm{U}_{\mathcal{I},
  \cdot} \bm{H}_{ \bm{U}}- \bm{U}_{\mathcal{I}, \cdot}^{\star}$.  This
can be done by a self-bounding trick: by the triangle inequality and
(\ref{eq:denoising-approx-1-12}), we have
\begin{align*}
 & \big \Vert \bm{U}_{\mathcal{I}, \cdot}^{(\mathcal{I})} \bm{H}_{
    \bm{U}}^{(\mathcal{I})}- \bm{U}_{\mathcal{I}, \cdot}^{\star} \big
  \Vert \leq \big \Vert \bm{U}_{\mathcal{I}, \cdot} \bm{H}_{ \bm{U}}-
  \bm{U}_{\mathcal{I}, \cdot}^{\star} \big \Vert + \big \Vert
  \bm{U}^{(\mathcal{I})} \bm{H}_{ \bm{U}}^{(\mathcal{I})}- \bm{U}
  \bm{H}_{ \bm{U}} \big \Vert\\ & \qquad \leq \big \Vert
  \bm{U}_{\mathcal{I}, \cdot} \bm{H}_{ \bm{U}}- \bm{U}_{\mathcal{I},
    \cdot}^{\star} \big \Vert + \frac{ 20C_g }{\sigma_{r}^{\star}} \big[  \sigma \sqrt{I + r + \log n} + \sigma \sqrt{n_{2}} \big( \big \Vert
    \bm{U}_{\mathcal{I}, \cdot}^{\star} \big \Vert + \big \Vert
    \bm{U}_{\mathcal{I}, \cdot}^{(\mathcal{I})} \bm{H}_{
    \bm{U}}^{(\mathcal{I})}- \bm{U}_{\mathcal{I}, \cdot}^{\star} \big
    \Vert \big) \big] + \frac{6 \beta_1}{ \sigma_r^\star }.
\end{align*}
Notice that both sides of the above inequality include $ \Vert
\bm{U}_{\mathcal{I}, \cdot}^{(\mathcal{I})} \bm{H}_{
  \bm{U}}^{(\mathcal{I})}- \bm{U}_{\mathcal{I}, \cdot}^{\star} \Vert$.
Since $\sigma \sqrt{n_2}\ \leq \cnoise \sigma_{r}^{\star}$ for some sufficiently small $\cnoise>0$, we can
rearrange terms to reach
\begin{align} \hspace{-1.5ex}
 \big \Vert \bm{U}_{\mathcal{I}, \cdot}^{(\mathcal{I})} \bm{H}_{
   \bm{U}}^{(\mathcal{I})}- \bm{U}_{\mathcal{I}, \cdot}^{\star} \big
 \Vert \leq 2 \big \Vert \bm{U}_{\mathcal{I}, \cdot} \bm{H}_{
   \bm{U}}- \bm{U}_{\mathcal{I}, \cdot}^{\star} \big \Vert +
 \frac{30 C_g }{\sigma_{r}^{\star}} \left(\sigma \sqrt{I + r + \log n} + \sigma \sqrt{n_{2}} \big \Vert \bm{U}_{\mathcal{I},
   \cdot}^{\star} \big \Vert \right ) + \frac{9 \beta_1}{ \sigma_r^\star }.\label{eq:denoising-approx-1-13}
\end{align}
Substituting (\ref{eq:denoising-approx-1-13}) into
(\ref{eq:denoising-approx-1-12}) gives
\begin{align}
 & \max \left\{ \big \Vert \bm{U}^{(\mathcal{I})} \bm{H}_{
 	\bm{U}}^{(\mathcal{I})}- \bm{U} \bm{H}_{ \bm{U}} \big \Vert, \big
 \Vert \bm{V}^{(\mathcal{I})} \bm{H}_{ \bm{V}}^{(\mathcal{I})}- \bm{V}
 \bm{H}_{ \bm{V}} \big \Vert \right\} \nonumber \\
   & \qquad \leq
 \frac{ 40C_g }{\sigma_{r}^{\star}} \big[  \sigma \sqrt{I + r + \log n} + \sigma \sqrt{n_{2}} \big( \big \Vert
 \bm{U}_{\mathcal{I}, \cdot}^{\star} \big \Vert + \big \Vert \bm{U}_{\mathcal{I}, \cdot} \bm{H}_{ \bm{U}}-
 \bm{U}_{\mathcal{I}, \cdot}^{\star} \big \Vert \big) \big] + \frac{9 \beta_1}{ \sigma_r^\star }
\label{eq:denoising-approx-1-14}  
\end{align}
provided that $\sigma \sqrt{n_2}\ \leq \cnoise \sigma_{r}^{\star}$ holds for some sufficiently small $\cnoise>0$.

Finally, we can substitute (\ref{eq:denoising-approx-1-14}) back to
(\ref{eq:denoising-approx-1-6}) to reach
\begin{align*}
\beta_{1} & \leq 2 C_g^2 \frac{\sigma^2 }{\sigma_{r}^{\star}} \sqrt{n
  \left(I + r + \log n \right )} + C_g \sigma \sqrt{I + r + \log n} \big
\Vert \bm{V}^{(\mathcal{I})} \bm{H}_{ \bm{V}}^{(\mathcal{I})}- \bm{V}
\bm{H}_{ \bm{V}} \big \Vert \\ 
& \leq 250 C_g^2 \frac{\sigma^2
}{\sigma_{r}^{\star}} \sqrt{n \left(I + r + \log n \right )} + 9 C_g
\frac{\sigma \sqrt{I + r + \log n}}{\sigma_{r}^{\star}}\beta_{1},
\end{align*}
where we have used the facts that $ \Vert \bm{U}_{\mathcal{I},
  \cdot}^{\star} \Vert \leq 1$ and $ \Vert \bm{U}_{\mathcal{I}, \cdot}
\bm{H}_{ \bm{U}}- \bm{U}_{\mathcal{I}, \cdot}^{\star} \Vert \leq 3$.
Notice that $\beta_{1}$ appears in both side of the inequality. Given
that $\sigma \sqrt{I + r + \log n} \leq 2 \sigma
\sqrt{n} \leq 2 \cnoise \sigma_{r}^{\star}$ for some sufficiently small $\cnoise>0$, we can rearrange terms to deduce that
\begin{align}
\beta_{1} & \leq 300 C_g^2 \frac{\sigma^2 }{\sigma_{r}^{\star}} \sqrt{n
  \left(I + r + \log n \right )}.\label{eq:denoising-approx-1-16}
\end{align}
Then we substitute (\ref{eq:denoising-approx-1-14}) back to
(\ref{eq:denoising-approx-1-7}) to get
\begin{align}
\beta_{2} & \leq C_g \sigma \sqrt{I + n_{2} + \log n} \big \Vert
\bm{V}^{(\mathcal{I})} \bm{H}_{ \bm{V}}^{(\mathcal{I})}- \bm{V}
\bm{H}_{ \bm{V}} \big \Vert \nonumber \\
& \leq C_g \sigma \sqrt{I + n_{2} + \log n}
\left[ \frac{ 40C_g }{\sigma_{r}^{\star}} \big[  \sigma \sqrt{I + r + \log n} + \sigma \sqrt{n_{2}} \big( \big \Vert
\bm{U}_{\mathcal{I}, \cdot}^{\star} \big \Vert + \big \Vert \bm{U}_{\mathcal{I}, \cdot} \bm{H}_{ \bm{U}}-
\bm{U}_{\mathcal{I}, \cdot}^{\star} \big \Vert \big) \big] + \frac{9 \beta_1}{ \sigma_r^\star } \right]
\nonumber \\
\label{eq:denoising-approx-1-15}
& \leq 50 C_g^2 \sigma \sqrt{I + n_{2} + \log n}
\left[ \frac{\sigma}{\sigma_{r}^{\star}} \sqrt{I + r + \log n} +
  \frac{\sigma}{\sigma_{r}^{\star}} \sqrt{n_{2}} \left( \big \Vert
  \bm{U}_{\mathcal{I}, \cdot}^{\star} \big \Vert + \big \Vert
  \bm{U}_{\mathcal{I}, \cdot} \bm{H}_{ \bm{U}}- \bm{U}_{\mathcal{I},
    \cdot}^{\star} \big \Vert \right ) \right].
\end{align}
Here the last step follows from (\ref{eq:denoising-approx-1-16}) and
it holds true as long as $\cnoise>0$ is sufficiently small. Taking the bound
(\ref{eq:denoising-approx-1-16}) on $\beta_{1}$ and the bound
(\ref{eq:denoising-approx-1-15}) on $\beta_{2}$ collectively, we have
\begin{align}
\label{eq:E-V-error-error}  
\beta & \leq \beta_{1} + \beta_{2} \leq 400 C_g^2 \frac{\sigma^2 }{\sigma_{r}^{\star}} \sqrt{n \left( I + r + \log n \right) } + 100 C_g^2
\frac{\sigma^2 n_{2}}{\sigma_{r}^{\star}} \left( \big \Vert
\bm{U}_{\mathcal{I}, \cdot}^{\star} \big \Vert + \big \Vert
\bm{U}_{\mathcal{I}, \cdot} \bm{H}_{ \bm{U}}- \bm{U}_{\mathcal{I},
  \cdot}^{\star} \big \Vert \right ).
\end{align}

\paragraph{Step 4: putting everything together.}

Combining the bounds in the preceding steps, we know that with
probability exceeding $1-O(n^{-10})$
\begin{align*}
 & \big \Vert \bm{U}_{\mathcal{I}, \cdot} \bm{\Sigma} \bm{H}_{
    \bm{V}}- \bm{M}_{\mathcal{I}, \cdot} \bm{V}^{\star} \big \Vert
  \leq \alpha + \beta\\ 
  & \qquad \leq 3 C_g^2 \frac{\sigma^2 n}{\sigma_{r}^{\star}} \big  \Vert \bm{U}_{\mathcal{I}, \cdot}^{\star} \big  \Vert + 400 C_g^2 \frac{\sigma^2 }{\sigma_{r}^{\star}} \sqrt{n \left( I + r + \log n \right) } + 100 C_g^2
  \frac{\sigma^2 n_{2}}{\sigma_{r}^{\star}} \left( \big \Vert
  \bm{U}_{\mathcal{I}, \cdot}^{\star} \big \Vert + \big \Vert
  \bm{U}_{\mathcal{I}, \cdot} \bm{H}_{ \bm{U}}- \bm{U}_{\mathcal{I},
  	\cdot}^{\star} \big \Vert \right ) \\ 
  & \qquad \leq c_7 \frac{\sigma^2 n}{\sigma_{r}^{\star}} \big
  \Vert \bm{U}_{\mathcal{I}, \cdot}^{\star} \big \Vert + c_7
  \frac{\sigma^2 }{\sigma_{r}^{\star}} \sqrt{n \left(I + r + \log n
    \right )} + c_7 \frac{\sigma^2 n_{2}}{\sigma_{r}^{\star}} \big \Vert
  \bm{U}_{\mathcal{I}, \cdot} \bm{H}_{ \bm{U}}- \bm{U}_{\mathcal{I},
    \cdot}^{\star} \big \Vert,
\end{align*}
where $c_7 = 400 C_g^2$. Following the same analysis, we can also show that
\begin{align*}
 \big \Vert \bm{V}_{\mathcal{J}, \cdot} \bm{\Sigma} \bm{H}_{ \bm{U}}-
 \bm{M}_{\cdot, \mathcal{J}}^{\top} \bm{U}^{\star} \big \Vert &
 \leq c_7 \frac{\sigma^2 n}{\sigma_{r}^{\star}} \big \Vert
 \bm{V}_{\mathcal{J}, \cdot}^{\star} \big \Vert + c_7 \frac{\sigma^2
 }{\sigma_{r}^{\star}} \sqrt{n \left(J + r + \log n \right )} + c_7
 \frac{\sigma^2 n_{1}}{\sigma_{r}^{\star}} \big \Vert
 \bm{V}_{\mathcal{J}, \cdot} \bm{H}_{ \bm{V}}- \bm{V}_{\mathcal{J},
   \cdot}^{\star} \big \Vert,
\end{align*}
where $J = |\mathcal{J}|$. These complete the proof.

\subsubsection{Proof of~\Cref{lemma:denoising-est-crude}}
\label{appendix:proof-denoising-est}

We first use the triangle inequality to achieve
\begin{align*}
\left \Vert \bm{U}_{\mathcal{I}, \cdot} \bm{H}_{ \bm{U}}-
\bm{U}_{\mathcal{I}, \cdot}^{\star} \right \Vert & \leq
\frac{1}{\sigma_{r}^{\star}} \left \Vert \bm{U}_{\mathcal{I}, \cdot}
\bm{H}_{ \bm{U}} \bm{\Sigma}^{\star}- \bm{U}_{\mathcal{I},
  \cdot}^{\star} \bm{\Sigma}^{\star} \right \Vert \\ & \leq
\frac{1}{\sigma_{r}^{\star}} \left \Vert \bm{U}_{\mathcal{I}, \cdot}
\bm{\Sigma} \bm{H}_{ \bm{V}}- \bm{U}_{\mathcal{I}, \cdot}^{\star}
\bm{\Sigma}^{\star} \right \Vert + \frac{1}{\sigma_{r}^{\star}} \left
\Vert \bm{U}_{\mathcal{I}, \cdot} \bm{\Sigma} \bm{H}_{ \bm{V}}-
\bm{U}_{\mathcal{I}, \cdot} \bm{H}_{ \bm{U}} \bm{\Sigma}^{\star}
\right \Vert \\ & \leq \underbrace{\frac{1}{\sigma_{r}^{\star}} \left
  \Vert \bm{U}_{\mathcal{I}, \cdot} \bm{\Sigma} \bm{H}_{ \bm{V}}-
  \bm{M}_{\mathcal{I}, \cdot} \bm{V}^{\star} \right \Vert
}_{\eqqcolon\, \alpha_{1}} + \underbrace{\frac{1}{\sigma_{r}^{\star}}
  \left \Vert \bm{E}_{\mathcal{I}, \cdot} \bm{V}^{\star} \right \Vert
}_{\eqqcolon\, \alpha_{2}} + \underbrace{\frac{1}{\sigma_{r}^{\star}}
  \left \Vert \bm{U}_{\mathcal{I}, \cdot} \left( \bm{\Sigma} \bm{H}_{
    \bm{V}}- \bm{H}_{ \bm{U}} \bm{\Sigma}^{\star} \right ) \right
  \Vert }_{\eqqcolon\, \alpha_{3}},
\end{align*}
where the last line follows from $\bm{M} \bm{V}^{\star} =
\bm{U}^{\star} \bm{\Sigma}^{\star} + \bm{E} \bm{V}^{\star}$.  Thus,
the problem reduces to bounding $\alpha_{1}$, $\alpha_{2}$ and
$\alpha_{3}$.  Applying~\Cref{lemma:denoising-approx-1-crude} yields
\begin{align}
\label{eq:denoising-est-1}  
\alpha_{1} & \leq c_7 \frac{\sigma^2 n}{\sigma_{r}^{\star2}} \big
\Vert \bm{U}_{\mathcal{I}, \cdot}^{\star} \big \Vert + c_7 \frac{\sigma^2
}{\sigma_{r}^{\star2}} \sqrt{n \left(I + r + \log n \right )} + c_7
\frac{\sigma^2 n_{2}}{\sigma_{r}^{\star2}} \big \Vert
\bm{U}_{\mathcal{I}, \cdot} \bm{H}_{ \bm{U}}- \bm{U}_{\mathcal{I},
  \cdot}^{\star} \big \Vert.
\end{align}
Regarding $\alpha_{2}$, we can invoke
\Cref{lemma:gaussian-spectral} to show that, with probability
exceeding $1-O(n^{-10})$,
\begin{align}
\label{eq:denoising-est-2}  
\alpha_{2} \leq C_g \frac{\sigma}{\sigma_{r}^{\star}} \sqrt{I + r +
  \log n}.
\end{align}
We are now left with bounding $\alpha_{3}$, and it is seen from \Cref{lemma:denoising-approx-2-crude} that
\begin{align}
\alpha_{3} & \leq \frac{1}{\sigma_{r}^{\star}} \left \Vert
\bm{U}_{\mathcal{I}, \cdot} \right \Vert \left \Vert \bm{\Sigma}
\bm{H}_{ \bm{V}}- \bm{H}_{ \bm{U}} \bm{\Sigma}^{\star} \right \Vert
\leq \frac{2}{\sigma_{r}^{\star}} \left \Vert \bm{U}_{\mathcal{I},
  \cdot} \bm{H}_{ \bm{U}} \right \Vert \left \Vert \bm{\Sigma}
\bm{H}_{ \bm{V}}- \bm{H}_{ \bm{U}} \bm{\Sigma}^{\star} \right \Vert
\nonumber \\ & \leq \frac{2}{\sigma_{r}^{\star}} \left( \big \Vert
\bm{U}_{\mathcal{I}, \cdot}^{\star} \big \Vert + \big \Vert
\bm{U}_{\mathcal{I}, \cdot} \bm{H}_{ \bm{U}}- \bm{U}_{\mathcal{I},
  \cdot}^{\star} \big \Vert \right ) \big \Vert \bm{\Sigma} \bm{H}_{
  \bm{V}}- \bm{H}_{ \bm{U}} \bm{\Sigma}^{\star} \big
\Vert, \label{eq:denoising-est-3}
\end{align}
where we used $\sigma_{r}( \bm{H}_{
  \bm{U}})\geq1/2$
(see~\Cref{lemma:denoising-basic-facts-crude}). Recall the full SVD of
$\bm{M}$ in (\ref{eq:full-SVD}), which gives
\begin{align*}
 \bm{\Sigma} \bm{H}_{ \bm{V}}- \bm{H}_{ \bm{U}} \bm{\Sigma}^{\star} &
 = \bm{\Sigma} \bm{V}^{\top} \bm{V}^{\star}- \bm{U}^{\top}
 \bm{U}^{\star} \bm{\Sigma}^{\star} = \bm{U}^{\top}( \bm{U}
 \bm{\Sigma} \bm{V}^{\top}- \bm{U}^{\star} \bm{\Sigma}^{\star}
 \bm{V}^{\star\top}) \bm{V}^{\star} \\
 & = \bm{U}^{\top}( \bm{E}- \bm{U}_{\perp} \bm{\Sigma}_{\perp}
 \bm{V}_{\perp}^{\top}) \bm{V}^{\star} = \bm{U}^{\top} \bm{E}
 \bm{V}^{\star}.
\end{align*}
This allows us to bound
\begin{align}
 \left \Vert \bm{\Sigma} \bm{H}_{ \bm{V}}- \bm{H}_{ \bm{U}}
 \bm{\Sigma}^{\star} \right \Vert 
 & \overset{ \text{(i)} }{\leq} 2 \big \Vert( \bm{U} \bm{H}_{
  \bm{U}})^{\top} \bm{E} \bm{V}^{\star} \big \Vert \leq 2 \big \Vert(
 \bm{U} \bm{H}_{ \bm{U}}- \bm{U}^{\star})^{\top} \bm{E} \bm{V}^{\star}
 \big \Vert + 2 \big \Vert \bm{U}^{\star\top} \bm{E} \bm{V}^{\star}
 \big \Vert\nonumber \\ & \leq 2 \big \Vert \bm{U} \bm{H}_{ \bm{U}}-
 \bm{U}^{\star} \big \Vert \big \Vert \bm{E} \bm{V}^{\star} \big \Vert
 + 2 \big \Vert \bm{U}^{\star\top} \bm{E} \bm{V}^{\star} \big
 \Vert\nonumber \\
 & \overset{ \text{(ii)} }{\leq} 4 C_g \frac{\sigma}{\sigma_{r}^{\star}} \sqrt{n} \cdot C_g \sigma
\sqrt{n_{1} + r + \log n} + 2 C_g \sigma \sqrt{r + \log n} \nonumber \\
& \overset{ \text{(iii)} }{\leq} 2 C_g \sigma
\sqrt{r + \log n} + 7 C_g^2  \frac{\sigma^2 n}{\sigma_{r}^{\star}}. \label{eq:denoising-est-4} 
\end{align}
Here (i) holds since $\sigma_{r}( \bm{H}_{
	\bm{U}})\geq1/2$
(see~\Cref{lemma:denoising-basic-facts-crude}), (ii)
follows from \Cref{lemma:gaussian-spectral} as well as a direct
consequence of (\ref{eq:denoising-basic-2}):
\begin{align*}
 \left \Vert \bm{U} \bm{H}_{ \bm{U}}- \bm{U}^{\star} \right \Vert =
 \big \Vert \bm{U} \bm{U}^{\top} \bm{U}^{\star}- \bm{U}^{\star} \big
 \Vert = \big \Vert( \bm{U} \bm{U}^{\top}- \bm{U}^{\star}
 \bm{U}^{\star\top}) \bm{U}^{\star} \big \Vert \leq \big \Vert \bm{U}
 \bm{U}^{\top}- \bm{U}^{\star} \bm{U}^{\star\top} \big \Vert \leq 2 C_g
 \frac{\sigma}{\sigma_{r}^{\star}} \sqrt{n};
\end{align*}
while (iii) holds with the proviso that $\sigma \sqrt{n}
\lesssim \sigma_{r}^{\star}$.  Substituting (\ref{eq:denoising-est-4})
into (\ref{eq:denoising-est-3}) gives
\begin{align}
\alpha_{3} \leq \frac{2}{\sigma_{r}^{\star}} \left( \big \Vert
\bm{U}_{\mathcal{I}, \cdot}^{\star} \big \Vert + \big \Vert
\bm{U}_{\mathcal{I}, \cdot} \bm{H}_{ \bm{U}}- \bm{U}_{\mathcal{I},
  \cdot}^{\star} \big \Vert \right ) \big(2 C_g \sigma
  \sqrt{r + \log n} + 7 C_g^2  \frac{\sigma^2 n}{\sigma_{r}^{\star}} \big).\label{eq:denoising-est-5}
\end{align}

Taking the bounds (\ref{eq:denoising-est-1}) on $\alpha_{1}$,
(\ref{eq:denoising-est-2}) on $\alpha_{2}$ and 
(\ref{eq:denoising-est-5}) on $\alpha_{3}$ collectively, we achieve
\begin{align*}
 & \big \Vert \bm{U}_{\mathcal{I}, \cdot} \bm{H}_{ \bm{U}}-
 \bm{U}_{\mathcal{I}, \cdot}^{\star} \big \Vert  \leq \alpha_{1} +
 \alpha_{2} + \alpha_{3} \\
 & \qquad \leq (17C_g+c_7 \cnoise)\frac{\sigma}{\sigma_{r}^{\star}}
 \sqrt{I + r + \log n} + (c_7 +14 C_g^2) \frac{\sigma^2 n}{\sigma_{r}^{\star2}} \left(
 \big \Vert \bm{U}_{\mathcal{I}, \cdot}^{\star} \big \Vert + \big
 \Vert \bm{U}_{\mathcal{I}, \cdot} \bm{H}_{ \bm{U}}-
 \bm{U}_{\mathcal{I}, \cdot}^{\star} \big \Vert \right ).
\end{align*}
Here we have used the facts that $ \Vert \bm{U}_{\mathcal{I}, \cdot}^{\star}
\Vert \leq 1$ and $ \Vert \bm{U}_{\mathcal{I}, \cdot} \bm{H}_{
  \bm{U}}- \bm{U}_{\mathcal{I}, \cdot}^{\star} \Vert \leq 3$.  Notice
that the term $ \Vert \bm{U}_{\mathcal{I}, \cdot} \bm{H}_{ \bm{U}}-
\bm{U}_{\mathcal{I}, \cdot}^{\star} \Vert$ appears on both side of the
above inequality. Since we assume that $\sigma
\sqrt{n}\leq \cnoise \sigma_{r}^{\star}$ for some sufficiently small $\cnoise>0$, we can use the self-bounding
trick to obtain
\begin{align*}
 \big \Vert \bm{U}_{\mathcal{I}, \cdot} \bm{H}_{ \bm{U}}-
 \bm{U}_{\mathcal{I}, \cdot}^{\star} \big \Vert & \leq 20C_g 
 \frac{\sigma}{\sigma_{r}^{\star}} \sqrt{I + r + \log n} + (2c_7 +20 C_g^2)
 \frac{\sigma^2 n}{\sigma_{r}^{\star2}} \big \Vert
 \bm{U}_{\mathcal{I}, \cdot}^{\star} \big \Vert
\end{align*}
as claimed, if we take $c_8 = 2c_7+20C_g^2$. Following the same analysis, we can also show that
\begin{align*}
 \big \Vert \bm{V}_{\mathcal{J}, \cdot} \bm{H}_{ \bm{V}}-
 \bm{V}_{\mathcal{J}, \cdot}^{\star} \big \Vert \leq c_8
 \frac{\sigma^2 n}{\sigma_{r}^{\star2}} \big \Vert
 \bm{V}_{\mathcal{J}, \cdot}^{\star} \big \Vert + c_8
 \frac{\sigma}{\sigma_{r}^{\star}} \sqrt{J + r + \log n}.
\end{align*}


\subsubsection{Proof of equations in~\Cref{subsec:refine}}
\label{subsec:proof-sharpen-eqs}

This section provides the proof of equations \eqref{eq:denoising-UV-I-spectral} and \eqref{eq:denoising-UVSigma} in \Cref{subsec:refine}.

\paragraph{Proof of (\ref{eq:denoising-UV-I-spectral}).}
Applying the triangle inequality yields
\begin{align*}
 & \big \Vert \bm{U}_{\mathcal{I}, \cdot} \bm{R}_{ \bm{U}}-
 \bm{U}_{\mathcal{I}, \cdot}^{\star} \big \Vert  \leq \big \Vert
 \bm{U}_{\mathcal{I}, \cdot} \bm{H}_{ \bm{U}}- \bm{U}_{\mathcal{I},
   \cdot}^{\star} \big \Vert + \big \Vert \bm{U}_{\mathcal{I}, \cdot}(
 \bm{H}_{ \bm{U}}- \bm{R}_{ \bm{U}}) \big \Vert \leq \big \Vert
 \bm{U}_{\mathcal{I}, \cdot} \bm{H}_{ \bm{U}}- \bm{U}_{\mathcal{I},
   \cdot}^{\star} \big \Vert + \left \Vert \bm{U}_{\mathcal{I}, \cdot}
 \right \Vert \left \Vert \bm{H}_{ \bm{U}}- \bm{R}_{ \bm{U}} \right
 \Vert \\ 
 & \qquad \overset{\text{(i)}}{\leq} c_8 \frac{\sigma^2 n}{\sigma_{r}^{\star2}} \big \Vert
 \bm{U}_{\mathcal{I}, \cdot}^{\star} \big \Vert + c_8
 \frac{\sigma}{\sigma_{r}^{\star}} \sqrt{I + r + \log n} + \big( 3 \big
 \Vert \bm{U}_{\mathcal{I}, \cdot}^{\star} \big \Vert + 2 c_8
 \frac{\sigma}{\sigma_{r}^{\star}} \sqrt{I + r + \log n}
 \big) \cdot 16 C_g^2 \frac{\sigma^2 n}{\sigma_{r}^{\star2}}\\ 
 & \qquad \overset{\text{(ii)}}{\leq} (c_8 + 48 C_g^2)\frac{\sigma^2
   n}{\sigma_{r}^{\star2}} \big \Vert \bm{U}_{\mathcal{I},
   \cdot}^{\star} \big \Vert + 2c_8 \frac{\sigma}{\sigma_{r}^{\star}}
 \sqrt{I + r + \log n}.
\end{align*}
Here (i) follows from \Cref{lemma:denoising-est-crude},
(\ref{eq:U-I-spectral-norm}) and (\ref{eq:denoising-basic-9}), while
(ii) holds true when $\sigma \sqrt{n} \leq \cnoise \sigma_{r}^{\star}$
for some sufficiently small $\cnoise>0$.  We have proved
(\ref{eq:denoising-U-I-spectral}), and
(\ref{eq:denoising-V-I-spectral}) can be verified analogously.

\paragraph{Proof of (\ref{eq:denoising-UVSigma}). }

A weaker version of this result was established in \Cref{lemma:denoising-approx-1-crude}. Here we refine several steps in its proof using sharper perturbation
bounds (\ref{eq:denoising-UV-I-spectral}). We first improve (\ref{eq:denoising-approx-1-1}) to 
\begin{align}
 \big  \Vert \bm{U}^{\star\top} \bm{U}_{\perp} \bm{\Sigma}_{\perp} \bm{V}_{\perp}^{\top} \bm{V}^{\star} \big  \Vert &  \leq  \big  \Vert \bm{U}^{\star\top} \bm{U}_{\perp} \big  \Vert \big  \Vert \bm{\Sigma}_{\perp} \big  \Vert \big  \Vert \bm{V}^{\star\top} \bm{V}_{\perp} \big  \Vert \overset{\text{(i)}}{\leq} \big  \Vert \bm{U} \bm{R}_{ \bm{U}}- \bm{U}^{\star} \big  \Vert \big  \Vert \bm{\Sigma}_{\perp} \big  \Vert \big  \Vert \bm{V} \bm{R}_{ \bm{V}}- \bm{V}^{\star} \big  \Vert\nonumber \\
 &  \overset{\text{(ii)}}{\leq} \big(\frac{\sigma^2 n}{\sigma_{r}^{\star2}} + \frac{\sigma}{\sigma_{r}^{\star}} \sqrt{n_{1} + \log n} \big) \big(\frac{\sigma^2 n}{\sigma_{r}^{\star2}} + \frac{\sigma}{\sigma_{r}^{\star}} \sqrt{n_{2} + \log n} \big)\sigma \sqrt{n}, \label{eq:refine-alpha-1}
\end{align}
Here (i) follows from the relation between subspace distances (see e.g.~\cite[Section 2.2.2]{chen2020spectral}), while (ii) utilizes the refined bound (\ref{eq:denoising-UV-I-spectral}). 
We can improve (\ref{eq:denoising-approx-1-2}) to 
\begin{align}
 \big  \Vert \bm{U}^{\star\top} \bm{E}( \bm{V} \bm{V}^{\top}- \bm{I}_{n_{2}}) \bm{V}^{\star} \big  \Vert &  \leq  \big  \Vert \bm{U}^{\star\top} \bm{E} \big  \Vert \left \Vert  \bm{V} \bm{H}_{ \bm{V}}- \bm{V}^{\star} \right \Vert \nonumber \\
 &  \leq C_g  \sigma \sqrt{n_{2} + \log n} \cdot c_5' \big(\frac{\sigma^2 n}{\sigma_{r}^{\star2}} + \frac{\sigma}{\sigma_{r}^{\star}} \sqrt{n_{2} + \log n} \big), \label{eq:refine-alpha-2}
\end{align}
where we use the a direct consequence of (\ref{eq:denoising-U-I-spectral}):
\begin{align}
 \left \Vert  \bm{V} \bm{H}_{ \bm{V}}- \bm{V}^{\star} \right \Vert  &  \overset{\text{(a)}}{\leq}  \big  \Vert \bm{V} \bm{V}^{\top}- \bm{V}^{\star} \bm{V}^{\star\top} \big  \Vert \overset{\text{(b)}}{\leq} \big  \Vert \bm{V} \bm{R}_{ \bm{V}}- \bm{V}^{\star} \big  \Vert \overset{\text{(c)}}{\leq} c_5' \frac{\sigma^2 n}{\sigma_{r}^{\star2}} + c_5' \frac{\sigma}{\sigma_{r}^{\star}} \sqrt{n_{2} + \log n}.\label{eq:denoising-approx-1-5-refined}
\end{align}
Here (a) utilizes (\ref{eq:denoising-approx-1-5}), (b) follows again from the relation between subspace distances,  while (c) uses the refined bound in (\ref{eq:denoising-V-spectral}).
Putting (\ref{eq:refine-alpha-1}) and (\ref{eq:refine-alpha-2})
together,  we can improve (\ref{eq:denoising-approx-1-3}) to 
\begin{align}
 \big  \Vert \bm{\Sigma}^{\star} \bm{V}^{\star\top}( \bm{V} \bm{H}_{ \bm{V}}- \bm{V}^{\star}) \big  \Vert &  \leq  \big  \Vert \bm{U}^{\star\top} \bm{U}_{\perp} \bm{\Sigma}_{\perp} \bm{V}_{\perp}^{\top} \bm{V}^{\star} \big  \Vert +  \big  \Vert \bm{U}^{\star\top} \bm{E}( \bm{V} \bm{V}^{\top}- \bm{I}_{n_{2}}) \bm{V}^{\star} \big  \Vert\nonumber \\
 &  \leq 2C_g c_5'  \big(\frac{\sigma^2 n_{1}}{\sigma_{r}^{\star2}} + \frac{\sigma}{\sigma_{r}^{\star}} \sqrt{n_{2} + \log n} \big)^{2}\sigma_{r}^{\star}, \label{eq:refine-alpha-3}
\end{align}
provided that $\sigma \sqrt{n} \leq \cnoise \sigma_{r}^{\star}$ for some sufficiently small $\cnoise>0$. This allows
us to improve our original bound on $\alpha$ in (\ref{eq:denoising-approx-1-4})
to 
\begin{align}
\alpha \leq 2 C_g c_5'   \big(\frac{\sigma^2 n_{1}}{\sigma_{r}^{\star2}} + \frac{\sigma}{\sigma_{r}^{\star}} \sqrt{n_{2} + \log n} \big)^{2}\sigma_{r}^{\star} \left \Vert  \bm{U}_{\mathcal{I}, \cdot}^{\star} \right \Vert .\label{eq:refined-alpha}
\end{align}
We also need to refine our bound on $\beta$,  where we recall that $\beta \leq \beta_1 + \beta_2$ with
\begin{align*}
 \beta_1 = \big  \Vert \bm{E}_{\mathcal{I}, \cdot} \big( \bm{V}^{(\mathcal{I})} \bm{H}_{ \bm{V}}^{(\mathcal{I})}- \bm{V}^{\star} \big) \big  \Vert  \quad \text{and} \quad \beta_2 = \big  \Vert \bm{E}_{\mathcal{I}, \cdot} \big( \bm{V}^{(\mathcal{I})} \bm{H}_{ \bm{V}}^{(\mathcal{I})}- \bm{V} \bm{H}_{ \bm{V}} \big) \big  \Vert. 
\end{align*}
We improve our bound on $\beta_{1}$ to 
\begin{align}
\beta_{1} & \overset{ \text{(i)} }{ \leq } C_g \sigma \sqrt{I + r + \log n} \left \Vert  \bm{V} \bm{H}_{ \bm{V}}- \bm{V}^{\star} \right \Vert  + C_g \sigma \sqrt{I + r + \log n} \big  \Vert \bm{V}^{(\mathcal{I})} \bm{H}_{ \bm{V}}^{(\mathcal{I})}- \bm{V} \bm{H}_{ \bm{V}} \big  \Vert\nonumber \\
 & \overset{ \text{(ii)} }{ \leq } C_g c_5' \sigma \sqrt{I + r + \log n} \Big(\frac{\sigma^2 n}{\sigma_{r}^{\star2}} + \frac{\sigma}{\sigma_{r}^{\star}} \sqrt{n_{2} + \log n} \Big) + C_g  \frac{\sigma}{\sigma_{r}^{\star}} \sqrt{I + r + \log n} \cdot \beta_{1}\nonumber \\
 & \qquad + 40 C_g^2 \sigma \sqrt{I + r + \log n} \Big [\frac{\sigma}{\sigma_{r}^{\star}} \sqrt{I + r + \log n} + \frac{\sigma}{\sigma_{r}^{\star}} \sqrt{n_{2}} \left( \big  \Vert \bm{U}_{\mathcal{I}, \cdot}^{\star} \big  \Vert +  \big  \Vert \bm{U}_{\mathcal{I}, \cdot} \bm{H}_{ \bm{U}}- \bm{U}_{\mathcal{I}, \cdot}^{\star} \big  \Vert \right ) \Big ]\nonumber \\
 & \overset{\text{(iii)}}{ \leq } (100 C_g^2 + 3 C_g c_5')\sigma \sqrt{I + r + \log n} \big(\frac{\sigma^2 n_{1}}{\sigma_{r}^{\star2}} + \frac{\sigma}{\sigma_{r}^{\star}} \sqrt{I + n_{2} + \log n} \big).\label{eq:beta-1-refined}
\end{align}
Here (i) uses our original derivation in (\ref{eq:denoising-approx-1-6}), 
(ii) utilizes (\ref{eq:denoising-approx-1-5-refined}) and (\ref{eq:denoising-approx-1-14}), 
and (iii) follows $ \Vert \bm{U}_{\mathcal{I}, \cdot}^{\star} \Vert \leq 1$, 
$ \Vert \bm{U}_{\mathcal{I}, \cdot} \bm{H}_{ \bm{U}}- \bm{U}_{\mathcal{I}, \cdot}^{\star} \Vert \leq 2$, 
as well as the self-bounding trick that appeared several times in
the proof of \Cref{lemma:denoising-approx-1-crude},  which holds
provided that $\sigma \sqrt{I + r + \log n} \leq 2 \sigma \sqrt{n} \leq 2 \cnoise \sigma_{r}^{\star}$ for some sufficiently small $\cnoise>0$.
Then we can improve (\ref{eq:denoising-approx-1-14}) to 
\begin{align}
 &\big \Vert \bm{V}^{(\mathcal{I})} \bm{H}_{ \bm{V}}^{(\mathcal{I})}-
 \bm{V} \bm{H}_{ \bm{V}} \big \Vert 
  \overset{ \text{(a)} }{ \leq }
 \frac{ 40C_g }{\sigma_{r}^{\star}} \big[  \sigma \sqrt{I + r + \log n} + \sigma \sqrt{n_{2}} \big( \big \Vert
 \bm{U}_{\mathcal{I}, \cdot}^{\star} \big \Vert + \big \Vert \bm{U}_{\mathcal{I}, \cdot} \bm{H}_{ \bm{U}}-
 \bm{U}_{\mathcal{I}, \cdot}^{\star} \big \Vert \big) \big] + \frac{9 \beta_1}{ \sigma_r^\star } \nonumber \\ 
 & \qquad \overset{ \text{(b)} }{ \leq } 40 C_g \frac{\sigma}{\sigma_{r}^{\star}} \sqrt{I + r + \log n} + 40 C_g
 \frac{\sigma}{\sigma_{r}^{\star}} \sqrt{n_{2}} \big( \big \Vert
 \bm{U}_{\mathcal{I}, \cdot}^{\star} \big \Vert + c_8 \frac{\sigma^2
   n}{\sigma_{r}^{\star2}} \big \Vert \bm{U}_{\mathcal{I},
   \cdot}^{\star} \big \Vert + c_8 \frac{\sigma}{\sigma_{r}^{\star}}
 \sqrt{I + r + \log n} \big)\nonumber \\ 
 & \qquad \qquad + (900 C_g^2 + 27 C_g c_5') \frac{\sigma}{\sigma_r^\star} \sqrt{I + r + \log n} \big(\frac{\sigma^2 n_{1}}{\sigma_{r}^{\star2}} + \frac{\sigma}{\sigma_{r}^{\star}} \sqrt{I + n_{2} + \log n} \big) \nonumber \\
 & \qquad \overset{ \text{(c)} }{ \leq } 50 C_g
 \frac{\sigma}{\sigma_{r}^{\star}} \sqrt{I + r + \log n} + 50 C_g
 \frac{\sigma}{\sigma_{r}^{\star}} \sqrt{n_{2}} \big \Vert
 \bm{U}_{\mathcal{I}, \cdot}^{\star} \big \Vert\label{eq:loo-refined}
\end{align}
Here step (a) utilizes (\ref{eq:denoising-approx-1-14}),  step (b) follows from (\ref{eq:beta-1-refined}) and \Cref{lemma:denoising-est-crude}, while step (c) holds provided that $\sigma \sqrt{I + n_{2} + \log n} \leq \sigma \sqrt{n} \leq \cnoise \sigma_{r}^{\star}$ for some sufficiently small $\cnoise>0$. 
This allows us to improve the bound (\ref{eq:denoising-approx-1-15})
on $\beta_{2}$ and obtain 
\begin{align*}
\beta_{2} &  \leq C_g \sigma \sqrt{I + n_{2} + \log n} \big  \Vert \bm{V}^{(\mathcal{I})} \bm{H}_{ \bm{V}}^{(\mathcal{I})}- \bm{V} \bm{H}_{ \bm{V}} \big  \Vert \\
& \leq 50 C_g^2 \sigma \sqrt{I + n_{2} + \log n} \Big( \frac{\sigma}{\sigma_{r}^{\star}} \sqrt{I + r + \log n} + \frac{\sigma}{\sigma_{r}^{\star}} \sqrt{n_{2}} \big  \Vert \bm{U}_{\mathcal{I}, \cdot}^{\star} \big  \Vert \Big).
\end{align*}
Putting the above refined bounds on $\beta_{1}$ and $\beta_{2}$
together yields
\begin{align}
\beta &  \leq (150 C_g^2 + 3 C_g c_5') \sigma \sqrt{I + r + \log n} \big(\frac{\sigma^2 n_{1}}{\sigma_{r}^{\star2}} + \frac{\sigma}{\sigma_{r}^{\star}} \sqrt{I + n_{2} + \log n} \big) + 50 C_g^2 \frac{\sigma^2 n_{2}}{\sigma_{r}^{\star}} \big  \Vert \bm{U}_{\mathcal{I}, \cdot}^{\star} \big  \Vert, \label{eq:refined-beta}
\end{align}
where we use the fact that $ \Vert \bm{U}_{\mathcal{I}, \cdot}^{\star} \Vert \leq 1$.
Finally,  taking the refined bound (\ref{eq:refined-alpha}) on $\alpha$
and the refined bound (\ref{eq:refined-beta}) on $\beta$ collectively
yields 
\begin{align*}
 &  \big  \Vert \bm{U}_{\mathcal{I}, \cdot} \bm{\Sigma} \bm{H}_{ \bm{V}}- \bm{M}_{\mathcal{I}, \cdot} \bm{V}^{\star} \big  \Vert \leq \alpha + \beta\\
 & \qquad \leq c_7' \Big(\frac{\sigma^2 n_{1}}{\sigma_{r}^{\star2}} + \frac{\sigma}{\sigma_{r}^{\star}} \sqrt{n_{2} + \log n} \Big)^{2}\sigma_{r}^{\star} \big  \Vert \bm{U}_{\mathcal{I}, \cdot}^{\star} \big  \Vert + c_7' \sigma \sqrt{I + r + \log n} \big(\frac{\sigma^2 n_{1}}{\sigma_{r}^{\star2}} + \frac{\sigma}{\sigma_{r}^{\star}} \sqrt{I + n_{2} + \log n} \big).
\end{align*}
where $c_7' = 150 C_g^2 + 3 C_g c_5'$. We have proved (\ref{eq:denoising-U-Sigma}),  and the other claimed
bound (\ref{eq:denoising-V-Sigma}) can be proved analogously.

\subsection{Proof of corollaries for matrix denoising}

This section provide the proof for \Cref{corollary:full-matrix,corollary:inner-product-sub}.

\subsubsection{Proof of \Cref{corollary:full-matrix} \label{subsec:proof-full-matrix}}

To begin with, we decompose the error of interest as follows
\begin{align*}
 \bm{U} \bm{\Sigma} \bm{V}^{\top}- \bm{M}^{\star} &  =  \bm{U} \bm{\Sigma} \bm{V}^{\top}- \bm{U}^{\star} \bm{\Sigma}^{\star} \bm{V}^{\star\top} =  \bm{U} \bm{R}_{ \bm{U}} \bm{R}_{ \bm{U}}^{\top} \bm{\Sigma} \bm{V}^{\top}- \bm{U}^{\star} \bm{\Sigma}^{\star} \bm{V}^{\star\top}\\
 &  =  \big( \bm{U}^{\star} +  \bm{E} \bm{V}^{\star}( \bm{\Sigma}^{\star})^{-1} +  \bm{\Psi}_{ \bm{U}} \big) \big( \bm{M}^{\top} \bm{U}^{\star} +  \bm{\Delta}_{ \bm{V}} \big)^{\top}- \bm{U}^{\star} \bm{\Sigma}^{\star} \bm{V}^{\star\top}\\
 &  =  \bm{U}^{\star} \bm{U}^{\star\top} \bm{E} +  \bm{E} \bm{V}^{\star} \bm{V}^{\star\top} + \underbrace{ \bm{\Psi}_{ \bm{U}} \bm{\Sigma}^{\star} \bm{V}^{\star\top} +  \left( \bm{U} \bm{R}_{ \bm{U}}- \bm{U}^{\star} \right ) \bm{U}^{\star\top} \bm{E} +  \bm{U} \bm{R}_{ \bm{U}} \bm{\Delta}_{ \bm{V}}^{\top}}_{\eqqcolon\,  \bm{\Phi}}.
\end{align*}
For any $\mathcal{I}\subseteq[n_{1}]$ and $\mathcal{J}\subseteq[n_{2}]$, 
with probability exceeding $1-O(n^{-10})$ we have 
\begin{align*}
 & \big \Vert \bm{\Phi}_{\mathcal{I}, \mathcal{J}} \big \Vert \leq
  \big \Vert( \bm{\Psi}_{ \bm{U}})_{\mathcal{I}, \cdot}
  \bm{\Sigma}^{\star} \big \Vert \big \Vert \bm{V}_{\mathcal{J},
    \cdot}^{\star} \big \Vert + \big \Vert \bm{U}_{\mathcal{I}, \cdot}
  \bm{R}_{ \bm{U}}- \bm{U}_{\mathcal{I}, \cdot}^{\star} \big \Vert
  \big \Vert \bm{U}^{\star\top} \bm{E}_{\cdot, \mathcal{J}} \big \Vert
  + \big \Vert \bm{U}_{\mathcal{I}, \cdot} \big \Vert \big \Vert(
  \bm{\Delta}_{ \bm{V}})_{\mathcal{J}, \cdot} \big \Vert\\ 
  & \quad
  \overset{\text{(i)}}{\leq} c_1' \sigma \sqrt{I + r + \log n} \Big
           ( \frac{\sigma}{\sigma_{r}^{\star}} \sqrt{I + n_{2} + \log
               n} + \frac{\sigma^2 n_{1}}{\sigma_{r}^{\star2}} \Big )
           \big \Vert \bm{V}_{\mathcal{J}, \cdot}^{\star} \big \Vert + c_1'
           \Big( \frac{\sigma^2 n}{\sigma_{r}^{\star}} + \sigma \sqrt{r
             + \log n} \Big) \big \Vert \bm{U}_{\mathcal{I},
             \cdot}^{\star} \big \Vert \big \Vert \bm{V}_{\mathcal{J},
             \cdot}^{\star} \big \Vert \\
& \quad \quad + c_5' \Big(\frac{\sigma^2 n}{\sigma_{r}^{\star2}} \big \Vert
           \bm{U}_{\mathcal{I}, \cdot}^{\star} \big \Vert +
           \frac{\sigma}{\sigma_{r}^{\star}} \sqrt{I + r + \log n}
           \Big) \cdot C_g \sigma \sqrt{J + r + \log n} \\
& \quad \quad + \Big( 3 \big \Vert \bm{U}_{\mathcal{I}, \cdot}^{\star}
           \big \Vert + 2 c_8 \frac{\sigma}{\sigma_{r}^{\star}} \sqrt{I + r
             + \log n} \Big) \cdot c_2 \Big[ \sigma \sqrt{J + r + \log n} \Big(
             \frac{\sigma}{\sigma_{r}^{\star}} \sqrt{J + n_{1} + \log
               n} + \frac{\sigma^2 n}{\sigma_{r}^{\star2}} \Big) +
             \frac{\sigma^2 n}{\sigma_{r}^{\star}} \Vert
             \bm{V}_{\mathcal{J}, \cdot}^{\star} \Vert \Big]\\ 
& \quad \overset{\text{(ii)}}{\leq}  c_3' \frac{\sigma^2 }{\sigma_{r}^{\star}} \sqrt{
             \left(I + r + \log n \right ) \left(J + r + \log n \right
             )} + c_3' \sigma \sqrt{J + r + \log n}
           \big(\frac{\sigma}{\sigma_{r}^{\star}} \sqrt{J + n_{1}} +
           \frac{\sigma^2 n_{2}}{\sigma_{r}^{\star2}} \big) \big \Vert
           \bm{U}_{\mathcal{I}, \cdot}^{\star} \big \Vert \\
& \quad \quad + c_3' \sigma \sqrt{I + r + \log n} \big(
           \frac{\sigma}{\sigma_{r}^{\star}} \sqrt{I + n_{2}} +
           \frac{\sigma^2 n_{1}}{\sigma_{r}^{\star2}} \big) \big \Vert
           \bm{V}_{\mathcal{J}, \cdot}^{\star} \big \Vert + c_3'
           \big(\frac{\sigma^2 n}{\sigma_{r}^{\star}} + \sigma \sqrt{r
             + \log n} \big) \big \Vert \bm{U}_{\mathcal{I},
             \cdot}^{\star} \big \Vert \big \Vert \bm{V}_{\mathcal{J},
             \cdot}^{\star} \big \Vert.
\end{align*}
Here (ii) holds by taking $c_3' = 2 c_1' + 2 c_5' + 6 c_2 + 4 c_8 c_2$, and (i) follows from the
equations~\eqref{eq:denoising-UV-I-spectral}, \eqref{eq:U-I-spectral-norm}, \Cref{lemma:gaussian-spectral}, and~\Cref{prop:denoising_main_2},
as well as
\begin{align*}
& \big \Vert( \bm{\Psi}_{ \bm{U}})_{\mathcal{I}, \cdot}
 \bm{\Sigma}^{\star} \big \Vert \overset{\text{(iii)}}{=} \big \Vert( \bm{\Delta}_{
   \bm{U}})_{\mathcal{I}, \cdot} + \bm{U}_{\mathcal{I}, \cdot}
 \bm{R}_{ \bm{U}} \bm{\Delta}_{ \bm{\Sigma}} \big \Vert \leq \big
 \Vert( \bm{\Delta}_{ \bm{U}})_{\mathcal{I}, \cdot} \big \Vert + \big
 \Vert \bm{U}_{\mathcal{I}, \cdot} \big \Vert \big \Vert \bm{\Delta}_{
   \bm{\Sigma}} \big \Vert \\
&\qquad \overset{\text{(iv)}}{\leq} c_2 \sigma \sqrt{I + r + \log n} \bigg
       [\frac{\sigma}{\sigma_{r}^{\star}} \sqrt{I + n_{2} + \log n} +
         \frac{\sigma^2 n_{1}}{\sigma_{r}^{\star2}} \bigg ] + c_2
       \frac{\sigma^2 n}{\sigma_{r}^{\star}} \Vert
       \bm{U}_{\mathcal{I}, \cdot}^{\star} \Vert \\
& \qquad \qquad + \big( 3 \big \Vert \bm{U}_{\mathcal{I}, \cdot}^{\star} \big
       \Vert + 2 c_8 \frac{\sigma}{\sigma_{r}^{\star}} \sqrt{I + r + \log n}
       \big) \cdot c_6 \Big[\frac{\sigma^2 n}{\sigma_{r}^{\star}} + \sigma
         \sqrt{r + \log n} \Big] \\
& \qquad \overset{\text{(v)}}{\leq} c_1' \sigma \sqrt{I + r + \log n} \bigg
           [\frac{\sigma}{\sigma_{r}^{\star}} \sqrt{I + n_{2} + \log
               n} + \frac{\sigma^2 n_{1}}{\sigma_{r}^{\star2}} \bigg ]
           + c_1' \big(\frac{\sigma^2 n}{\sigma_{r}^{\star}} +\sigma
           \sqrt{r + \log n} \big) \big \Vert \bm{U}_{\mathcal{I},
             \cdot}^{\star} \big \Vert,
\end{align*}
where (iii) follows from \eqref{eq:URU-decomposition-Psi-U}, (iv) follows from~\Cref{prop:denoising_main_2}, 
equations~\eqref{eq:Delta-Sigma-bound} and \eqref{eq:U-I-spectral-norm}, while (v) holds by taking $c_1' = c_2 +3c_6 + 4c_6 c_8$.  Next,
applying~\Cref{lemma:gaussian-spectral} yields
\begin{align*}
 \big \Vert \bm{U}_{\mathcal{I}, \cdot}^{\star} \bm{U}^{\star\top}
 \bm{E}_{\cdot, \mathcal{J}} + \bm{E}_{\mathcal{I}, \cdot}
 \bm{V}^{\star} \bm{V}_{\mathcal{J}, \cdot}^{\star\top} \big \Vert &
 \leq C_g \sigma \sqrt{J + r + \log n} \big \Vert \bm{U}_{\mathcal{I},
   \cdot}^{\star} \big \Vert + C_g \sigma \sqrt{I + r + \log n} \big \Vert
 \bm{V}_{\mathcal{J}, \cdot}^{\star} \big \Vert.
\end{align*}
Combining the above bounds with the triangle inequality
yields
\begin{align*}
 \big \Vert( \bm{U} \bm{\Sigma} \bm{V}^{\top}- \bm{M}^{\star})_{\cdot,
   \mathcal{J}} \big \Vert &  \leq  \big \Vert \bm{U}_{\mathcal{I},
   \cdot}^{\star} \bm{U}^{\star\top} \bm{E}_{\cdot, \mathcal{J}} +
 \bm{E}_{\mathcal{I}, \cdot} \bm{V}^{\star} \bm{V}_{\mathcal{J},
   \cdot}^{\star\top} \big \Vert + \big \Vert \bm{\Phi}_{\mathcal{I},
   \mathcal{J}} \big \Vert \\
 & \leq c_3 \sigma \sqrt{J + r + \log n} \big \Vert
 \bm{U}_{\mathcal{I}, \cdot}^{\star} \big \Vert + c_3 \sigma \sqrt{I + r +
   \log n} \big \Vert \bm{V}_{\mathcal{J}, \cdot}^{\star} \big
 \Vert \\
 & \qquad + c_3 \frac{\sigma^2 n}{\sigma_{r}^{\star}} \big \Vert
 \bm{U}_{\mathcal{I}, \cdot}^{\star} \big \Vert \big \Vert
 \bm{V}_{\mathcal{J}, \cdot}^{\star} \big \Vert + c_3 \frac{\sigma^2
 }{\sigma_{r}^{\star}} \sqrt{ \left(I + r + \log n \right ) \left(J + r
   + \log n  \right )},
\end{align*}
where we take $c_3 = 2 C_g + 2 c_3' $. Here we
use the fact that $ \Vert \bm{U}_{\mathcal{I}, \cdot}^{\star} \Vert
 \leq  1$ and the assumption that $\sigma \sqrt{n} \leq \cnoise \sigma_r^\star$ holds for some sufficiently small $\cnoise>0$.

\subsubsection{Proof of \Cref{corollary:inner-product-sub}}
\label{subsec:proof-inner-product-sum}

In view of~\Cref{prop:denoising_main_1}, we can use the triangle
inequality to obtain
\begin{align*}
 \big \Vert \bm{U}_{\mathcal{I}, \cdot}^{\star\top}
 (\bm{U}_{\mathcal{I}, \cdot} \bm{R}_{ \bm{U}}- \bm{U}_{\mathcal{I},
   \cdot}^{\star}) \big \Vert \leq \underbrace{ \big \Vert
   \bm{U}_{\mathcal{I}, \cdot}^{\star\top} \bm{E}_{\mathcal{I}, \cdot}
   \bm{V}^{\star}( \bm{\Sigma}^{\star})^{-1} \big
   \Vert}_{\eqqcolon\alpha} + \underbrace{ \big \Vert
   \bm{U}_{\mathcal{I}, \cdot}^{\star\top}( \bm{\Psi}_{
     \bm{U}})_{\mathcal{I}, \cdot} \big \Vert}_{\eqqcolon\beta}.
\end{align*}
Thus, we have reduced the problem to bounding $\alpha$ and
$\beta$.

We begin by bounding $\alpha$. Applying~\Cref{lemma:gaussian-spectral}
yields
\begin{align*}
\alpha \leq \frac{1}{\sigma_{r}^{\star}} \big \Vert
\bm{U}_{\mathcal{I}, \cdot}^{\star\top} \bm{E}_{\mathcal{I}, \cdot}
\bm{V}^{\star} \big \Vert \leq C_g \frac{\sigma}{\sigma_{r}^{\star}}
\sqrt{r + \log n} \big \Vert \bm{U}_{\mathcal{I}, \cdot}^{\star} \big
\Vert.
\end{align*}
Recalling the definition~\eqref{lemma:gaussian-spectral} of
$\bm{\Psi}_{ \bm{U}}$, we can upper bound $\beta$ by
\begin{align*}
\beta & = \big \Vert \bm{U}_{\mathcal{I}, \cdot}^{\star\top}(
\bm{\Delta}_{ \bm{U}})_{\mathcal{I}, \cdot}( \bm{\Sigma}^{\star})^{-1}
+ \bm{U}_{\mathcal{I}, \cdot}^{\star\top} \bm{U}_{\mathcal{I}, \cdot}
\bm{R}_{ \bm{U}} \bm{\Delta}_{ \bm{\Sigma}}( \bm{\Sigma}^{\star})^{-1}
\big \Vert \\
& \leq \underbrace{ \big \Vert \bm{U}_{\mathcal{I}, \cdot}^{\star\top}
  \bm{U}_{\mathcal{I}, \cdot}^{\star} \bm{\Sigma}^{\star}
  \bm{V}^{\star\top} \left( \bm{V} \bm{H}_{ \bm{V}}-
  \bm{V}^{\star} \right )( \bm{\Sigma}^{\star})^{-1} \big
  \Vert}_{\eqqcolon\beta_{1}} + \underbrace{ \big \Vert
  \bm{U}_{\mathcal{I}, \cdot}^{\star\top} \bm{E}_{\mathcal{I}, \cdot}
  \left( \bm{V} \bm{H}_{ \bm{V}}- \bm{V}^{\star} \right )(
  \bm{\Sigma}^{\star})^{-1} \big \Vert}_{\eqqcolon\beta_{2}} \\
& \qquad + \underbrace{ \big \Vert \bm{U}_{\mathcal{I},
    \cdot}^{\star\top} \bm{U}_{\mathcal{I}, \cdot} \bm{R}_{ \bm{U}}
  \bm{\Delta}_{ \bm{\Sigma}}( \bm{\Sigma}^{\star})^{-1} \big
  \Vert}_{\eqqcolon\beta_{3}}.
\end{align*}
We now need to bound the three terms $\beta_{1}$, $\beta_{2}$ and
$\beta_{3}$.  Beginning with $\beta_{1}$, applying the
relation~\eqref{eq:refine-alpha-3} yields
\begin{align*}
\beta_{1} \leq \frac{1}{\sigma_{r}^{\star}} \big \Vert
\bm{U}_{\mathcal{I}, \cdot}^{\star} \big \Vert^{2} \left \Vert
\bm{\Sigma}^{\star} \bm{V}^{\star\top} \left( \bm{V} \bm{H}_{ \bm{V}}-
\bm{V}^{\star} \right ) \right \Vert \leq 2 C_g c_5' \big(\frac{\sigma^2
  n_{1}}{\sigma_{r}^{\star2}} + \frac{\sigma}{\sigma_{r}^{\star}}
\sqrt{n_{2} + \log n} \big)^{2} \big \Vert \bm{U}_{\mathcal{I},
  \cdot}^{\star} \big \Vert^{2}.
\end{align*}
Turning to $\beta_{2}$, applying the triangle inequality yields
\begin{align*}
\beta_{2} \leq \underbrace{ \big \Vert \bm{U}_{\mathcal{I},
    \cdot}^{\star\top} \bm{E}_{\mathcal{I}, \cdot} \big(
  \bm{V}^{(\mathcal{I})} \bm{H}_{ \bm{V}}^{(\mathcal{I})}-
  \bm{V}^{\star} \big)( \bm{\Sigma}^{\star})^{-1} \big
  \Vert}_{\eqqcolon\beta_{2, 1}} + \underbrace{ \big \Vert
  \bm{U}_{\mathcal{I}, \cdot}^{\star\top} \bm{E}_{\mathcal{I}, \cdot}
  \big( \bm{V}^{(\mathcal{I})} \bm{H}_{ \bm{V}}^{(\mathcal{I})}-
  \bm{V} \bm{H}_{ \bm{V}} \big)( \bm{\Sigma}^{\star})^{-1} \big
  \Vert}_{\eqqcolon\beta_{2, 2}}.
\end{align*}
The definition of $\bm{V}^{(\mathcal{I})}$ and $\bm{H}_{
  \bm{V}}^{(\mathcal{I})}$ can be found
in~\Cref{appendix:proof-denoising-approx-1}.  Since
$\bm{E}_{\mathcal{I}, \cdot}$ and $\bm{V}^{(\mathcal{I})}
\bm{H}_{\bm{V}}^{(\mathcal{I})}- \bm{V}^{\star}$ are independent, we
can invoke~\Cref{lemma:gaussian-spectral} to obtain
\begin{align*}
\beta_{2, 1} & \leq C_g \frac{\sigma}{\sigma_{r}^{\star}} \sqrt{r +
  \log n} \big \Vert \bm{U}_{\mathcal{I}, \cdot}^{\star} \big \Vert
\big \Vert \bm{V}^{(\mathcal{I})} \bm{H}_{ \bm{V}}^{(\mathcal{I})}-
\bm{V}^{\star} \big \Vert \leq C_g c_5' \frac{\sigma}{\sigma_{r}^{\star}}
\sqrt{r + \log n} \big(\frac{\sigma}{\sigma_{r}^{\star}} \sqrt{n_{2} +
  \log n} + \frac{\sigma^2 n}{\sigma_{r}^{\star2}} \big) \big \Vert
\bm{U}_{\mathcal{I}, \cdot}^{\star} \big \Vert,
\end{align*}
where we have used the fact that 
\begin{align*}
 \big \Vert \bm{V}^{(\mathcal{I})} \bm{H}_{ \bm{V}}^{(\mathcal{I})}-
 \bm{V}^{\star} \big \Vert & \leq c_5' \frac{\sigma^2
   n}{\sigma_{r}^{\star2}} + c_5' \frac{\sigma}{\sigma_{r}^{\star}}
 \sqrt{n_{2} + \log n}.
\end{align*}
(This bound can be established by an argument similar to that used in
proving the bound~\eqref{eq:denoising-approx-1-5-refined}).  In
addition, applying~\Cref{lemma:gaussian-spectral} and
equation~\eqref{eq:loo-refined} yields
\begin{align*}
\beta_{2, 2} & \leq \frac{1}{\sigma_{r}^{\star}} \big \Vert
\bm{U}_{\mathcal{I}, \cdot}^{\star\top} \bm{E}_{\mathcal{I}, \cdot}
\big \Vert \big \Vert \bm{V}^{(\mathcal{I})} \bm{H}_{
  \bm{V}}^{(\mathcal{I})}- \bm{V} \bm{H}_{ \bm{V}} \big \Vert \\
& \leq 50 C_g^2 \frac{\sigma}{\sigma_{r}^{\star}} \sqrt{n_{2} + \log n}
\big \Vert \bm{U}_{\mathcal{I}, \cdot}^{\star} \big \Vert 
\Big(\frac{\sigma}{\sigma_{r}^{\star}} \sqrt{I + r + \log n} +
\frac{\sigma}{\sigma_{r}^{\star}} \sqrt{n_{2}} \big \Vert
\bm{U}_{\mathcal{I}, \cdot}^{\star} \big \Vert \Big).
\end{align*}
Lastly, turning to $\beta_{3}$, it can be upper bounded as
\begin{align*}
\beta_{3} & \leq \big \Vert \bm{U}_{\mathcal{I}, \cdot}^{\star\top}
\bm{U}_{\mathcal{I}, \cdot} \bm{R}_{ \bm{U}} \bm{\Delta}_{
  \bm{\Sigma}}( \bm{\Sigma}^{\star})^{-1} \big \Vert \leq
\frac{1}{\sigma_{r}^{\star}} \big \Vert \bm{U}_{\mathcal{I},
  \cdot}^{\star} \big \Vert \big \Vert \bm{U}_{\mathcal{I}, \cdot}
\big \Vert \big \Vert \bm{\Delta}_{ \bm{\Sigma}} \big \Vert \\
& \leq \frac{1}{\sigma_{r}^{\star}} \big \Vert
\bm{U}_{\mathcal{I}, \cdot}^{\star} \big \Vert \big( 3 \big \Vert
\bm{U}_{\mathcal{I}, \cdot}^{\star} \big \Vert + 2 c_8
\frac{\sigma}{\sigma_{r}^{\star}} \sqrt{I + r + \log n} \big) \cdot c_6
\big(\frac{\sigma^2 n}{\sigma_{r}^{\star}} + \sigma \sqrt{r + \log n}
\big),
\end{align*}
where the last relation follows from the
relations~\eqref{eq:U-I-spectral-norm}
and~\eqref{eq:Delta-Sigma-bound}.

By combining the above bounds, we find that
\begin{align*}
 & \big \Vert \bm{U}_{\mathcal{I}, \cdot}^{\star\top}(
  \bm{U}_{\mathcal{I}, \cdot} \bm{R}_{ \bm{U}}- \bm{U}_{\mathcal{I},
    \cdot}^{\star}) \big \Vert \leq \alpha + \beta \leq \alpha
  + \beta_{1} + \beta_{2} + \beta_{3} \leq \alpha + \beta_{1} +
  \beta_{2, 1} + \beta_{2, 2} + \beta_{3} \\
  & \qquad \leq c_4 \bigg [ \frac{\sigma}{\sigma_{r}^{\star}} \sqrt{r
      + \log n} + \frac{\sigma^2 }{\sigma_{r}^{\star2}} \sqrt{n_{2}
      \left( n_{2} + I \right )} +
    \frac{\sigma^{3}n_{1}}{\sigma_{r}^{\star3}} \sqrt{I} \bigg ] \big
  \Vert \bm{U}_{\mathcal{I}, \cdot}^{\star} \big \Vert + c_4
  \frac{\sigma^2 n}{\sigma_{r}^{\star2}} \big \Vert
  \bm{U}_{\mathcal{I}, \cdot}^{\star} \big \Vert^{2},
\end{align*}
where we take $c_4 = 5 C_g c_5' + 100 C_g^2 + 5c_6 + 5 c_8 c_6$. Here we use the condition that $\sigma \sqrt{n} \leq \cnoise
\sigma_{r}^{\star}$ for some sufficiently small constant $\cnoise>0$. 



\section{Technical lemma} \label{appendix:technical_lemmas}
This section provides some technical lemmas that are useful in the
paper.
\begin{lemma}
\label{lemma:gaussian-spectral}
Let $\bm{E} \in \mathbb{R}^{n_{1}\times n_{2}}$ be a random matrix
with independent entries, and $\Vert E_{i,j}\Vert_{\psi_{2}} \leq
\sigma$.  Then there exists some sufficiently large constant $C_g>0$
such that
\begin{align}
\label{eq:lemma-gaussian-1}  
\left \Vert \bm{E} \right \Vert \leq C_g \sigma \sqrt{n} \qquad
\mbox{where $n \mydefn \max \{ n_{1}, n_{2} \}$}
\end{align}
holds with probability at least $1 - O(n^{-10})$.  For any fixed
orthonormal matrices $ \bm{U} \in \mathbb{R}^{n_{1}\times r_{1}}$ and
$\bm{V} \in \mathbb{R}^{n_{2}\times r_{2}}$, we have
\begin{align}
\label{eq:lemma-gaussian-2}  
\Vert \bm{U}^{\top} \bm{E} \bm{V} \Vert \leq \frac{1}{2} C_g \sigma
\sqrt{r_{1} + r_{2} + \log n}
\end{align}
with probability at least $1 - O(n^{-10})$.  As a consequence, for any
$ \bm{X} \in \mathbb{R}^{n_{1}\times m_{1}}$ and $\bm{Y} \in
\mathbb{R}^{n_{2}\times m_{2}}$,
\begin{align}
\label{eq:lemma-gaussian-3}  
\Vert \bm{X}^{\top} \bm{E} \bm{Y} \Vert \leq \frac{1}{2} C_g \sigma \left \Vert
\bm{X} \right \Vert \left \Vert \bm{Y} \right \Vert
\sqrt{\mathsf{rank} \left ( \bm{X} \right ) + \mathsf{rank} \left(
  \bm{Y} \right ) + \log n}
\end{align}
holds with probability at least $1 - O(n^{-10})$.
\end{lemma}

\begin{proof}
The first result (\ref{eq:lemma-gaussian-1}) is standard and is a
special case of (\ref{eq:lemma-gaussian-2}) by setting $ \bm{U} =
\bm{I}_{n_{1}}$ and $ \bm{V} = \bm{I}_{n_{2}}$. To bound the spectral
norm of $ \bm{U}^{\top} \bm{E} \bm{V}$, we first write it as
\begin{align*}
\Vert \bm{U}^{\top} \bm{E} \bm{V}\Vert = \sup_{ \bm{x} \in
  \mathcal{S}^{r_{1}}, \bm{y} \in \mathcal{S}^{r_{2}}} \bm{x}^{\top}
\bm{U}^{\top} \bm{E} \bm{V} \bm{y} = \sup_{ \bm{x} \in
  \mathcal{S}^{r_{1}}, \bm{y} \in \mathcal{S}^{r_{2}}} \sum_{i =
  1}^{n_{1}} \sum_{j = 1}^{n_{2}} E_{i,j} \left ( \bm{U}_{i,\cdot}
\bm{x} \right ) \left ( \bm{V}_{j,\cdot} \bm{y} \right )
\end{align*}
where $\mathcal{S}^{r} = \{ \bm{x} \in \mathbb{R}^{r}:\Vert
\bm{x}\Vert_{2} = 1\}$ denotes the unit sphere in $\mathbb{R}^{r}$. By
Hoeffding's inequality \cite[Theorem 2.6.3]{vershynin2016high}, we
know that with probability exceeding $1 - \delta$,
\begin{align*}
\sum_{i = 1}^{n_{1}}\sum_{j = 1}^{n_{2}}E_{i,j} \left (
\bm{U}_{i,\cdot} \bm{x} \right ) \left ( \bm{V}_{j,\cdot} \bm{y}
\right ) & \leq C_{\mathsf{b}}\sigma \left \Vert \bm{U} \bm{x} \right
\Vert _{2} \left \Vert \bm{V} \bm{y} \right \Vert _{2} \sqrt{\log
  \left (\delta^{-1} \right )} = C_{\mathsf{b}}\sigma \sqrt{\log \left
  (\delta^{-1} \right )}
\end{align*}
for some universal constant $C_{\mathsf{b}}>0$. Here the last relation
holds since $ \bm{U}$ and $ \bm{V}$ both have orthonormal columns.
The $1/4$-covering number of $\mathbb{S}^{r}$ is upper bounded by
$9^{r}$ \cite[Corollary 4.2.13]{vershynin2016high}. Let
$\mathcal{N}_{r_{1}}$ and $\mathcal{N}_{r_{2}}$ be the $1/4$- covering
of $\mathcal{S}^{r_{1}}$ and $\mathcal{S}^{r_{2}}$ satisfying
$|\mathcal{N}_{r_{1}}| \leq 9^{r_{1}}$ and $|\mathcal{N}_{r_{2}}| \leq
9^{r_{2}}$. Then with probability exceeding $1 - 9^{r_{1} + r_{2}}
\delta$,
\begin{align*}
\Vert \bm{U}^{\top} \bm{E} \bm{V} \Vert & \overset{\text{(i)}}{ \leq }
2 \max_{ \bm{x} \in \mathcal{N}_{r_{1}}, \bm{y} \in
  \mathcal{N}_{r_{2}}} \sum_{i = 1}^{n_{1}} \sum_{j = 1}^{n_{2}}
E_{i,j} \left ( \bm{U}_{i,\cdot} \bm{x} \right ) \left (
\bm{V}_{j,\cdot} \bm{y} \right ) \overset{\text{(ii)}}{ \leq } 2
C_{\mathsf{b}} \sigma \sqrt{\log \left (\delta^{-1} \right )}.
\end{align*}
Here (i) is the standard technique for bounding the spectral norm
using nets, see \cite[Section 4.4.1]{vershynin2016high} for details;
and (ii) follows from a union bound argument. By taking $\delta =
n^{-10} 9^{-r_{1} - r_{2}}$, we reach the desired bound
(\ref{eq:lemma-gaussian-2}). Finally it is straightforward to
establish the last result (\ref{eq:lemma-gaussian-3}) by considering
the SVDs of $\bm{X}$ and $\bm{Y}$ and applying
(\ref{eq:lemma-gaussian-2}), and we omit the proof here for brevity.
\end{proof}




\end{document}